\documentclass[a4paper,10pt,reqno]{amsart}
\usepackage{dsfont,amsmath,amsfonts,amscd,amssymb,graphicx,mathrsfs,eufrak,upgreek,pdflscape,colonequals,euscript,amsmath}
\usepackage[usenames]{color}
\usepackage{setspace}
\usepackage{array,multirow,booktabs,longtable}
\usepackage{mathtools}%% for psmallmatrix
\usepackage{comment}
\usepackage{cancel}
\usepackage{upref}
\renewcommand*\partial{\textnormal{\reflectbox{6}}}
\usepackage{nccmath}%%for medium sized math
\usepackage{calc}
\usepackage{enumitem}

\setlist[enumerate]{font=\normalfont}

\usepackage{array,multirow,booktabs,longtable}
\newcommand{\marginparstretch}{0.6}
\let\oldmarginpar\marginpar
\renewcommand\marginpar[1]{\-\oldmarginpar[\framebox{\setstretch{\marginparstretch}\begin{minipage}{\marginparwidth}{\raggedleft\tiny #1}\end{minipage}}]{\framebox{\setstretch{\marginparstretch}\begin{minipage}{\marginparwidth}{\raggedright\tiny #1}\end{minipage}}}}

\newcommand{\xdownarrow}[1]{%
  {\left\downarrow\vbox to #1{}\right.\kern-\nulldelimiterspace}
}

\usepackage[pagebackref]{hyperref}

\hypersetup{
  colorlinks   = true,          %Colors links instead of ugly boxes
  urlcolor     = blue,          %Color for external hyperlinks
  linkcolor    = purple,          %Color of internal links
  citecolor   = blue             %Color of citations
}

\makeatletter
\let\c@figure\c@table
\let\ftype@figure\ftype@table
\let\ext@figure\ext@table

\makeatother 
\usepackage{tikz,mathrsfs}
\usetikzlibrary{arrows,decorations.pathmorphing,decorations.pathreplacing,positioning,shapes.geometric,shapes.misc,decorations.markings,decorations.fractals,calc,patterns}

\tikzset{
        cvertex/.style={circle,draw=black,inner sep=1pt,outer sep=3pt},
        vertex/.style={circle,fill=black,inner sep=1pt,outer sep=3pt},
        DB/.style={circle,draw=black,circle,fill=black,inner sep=0pt, minimum size=4pt},
        DBlue/.style={circle,draw=black,circle,fill=white,inner sep=0pt, minimum size=4pt},
        DW/.style={circle,draw=black,inner sep=0pt, minimum size=4pt},
        star/.style={circle,fill=yellow,inner sep=0.75pt,outer sep=0.75pt},
        tvertex/.style={inner sep=1pt,font=\scriptsize},
        gap/.style={inner sep=0.5pt,fill=white}}

\newtheorem{thm}{Theorem}[section]
\newtheorem{prop}[thm]{Proposition}
\newtheorem{lemma}[thm]{Lemma}
\newtheorem{cor}[thm]{Corollary}

\newtheorem{claim}{Claim}

\theoremstyle{definition} 
\newtheorem{defin}[thm]{Definition}

\newtheorem{example}[thm]{Example}
\newtheorem{setup}[thm]{Setup}

\newtheorem{remark}[thm]{Remark}

\newtheorem{notation}[thm]{Notation}

\numberwithin{equation}{section}

\newcommand{\m}{\mathfrak{m}}
\newcommand{\n}{\mathfrak{n}}

\DeclareMathOperator{\Spec}{\rm Spec}
\DeclareMathOperator{\Sing}{\rm Sing}
\DeclareMathOperator{\ab}{\rm ab}

\def\filt{{\sf{filt}}}

\newcommand{\Bl}{\mathrm{Bl}}

\newcommand{\con}{\mathrm{con}}

\newcommand{\CA}{\mathrm{A}_{\con}}

\newcommand{\CAab}{\mathrm{A}^{\ab}_{\con}}

\newcommand\Curve{\mathrm{C}}

\renewcommand{\P}{\mathbb P}
\newcommand{\A}{\mathbb A}
\newcommand{\N}{\mathbb N}
\newcommand{\Z}{\mathbb Z}

\newcommand{\Ga}{\Gamma}

\newcommand{\C}{\mathbb C}

\newcommand{\la}{\uplambda}
\renewcommand{\mu}{\upmu}
\DeclareMathOperator{\ord}{ord}

\newcommand{\sff}{\mathsf{f}}
\newcommand{\sfg}{\mathsf{g}}
\newcommand{\sfh}{\mathsf{h}}

\newcommand{\sfm}{\mathsf{m}}

\newcommand{\bigO}{\scrO}

\newcommand{\kx}{\mathsf{r}}
\newcommand{\ky}{\mathsf{s}}

\newcommand{\scrA}{\EuScript{A}}
\newcommand{\scrB}{\EuScript{B}}
\newcommand{\scrC}{\EuScript{C}}

\newcommand{\scrE}{\EuScript{E}}
\newcommand{\scrF}{\EuScript{F}}
\newcommand{\scrG}{\EuScript{G}}

\newcommand{\scrI}{\EuScript{I}}
\newcommand{\scrJ}{\EuScript{J}}

\newcommand{\scrL}{\EuScript{L}}

\newcommand{\scrN}{\EuScript{N}}
\newcommand{\scrO}{\EuScript{O}}

\newcommand{\scrR}{\EuScript{R}}
\newcommand{\scrS}{\EuScript{S}}
\newcommand{\scrT}{\EuScript{T}}
\newcommand{\scrU}{\EuScript{U}}

\newcommand{\scrX}{\EuScript{X}}
\newcommand{\scrY}{\EuScript{Y}}
\newcommand{\scrZ}{\EuScript{Z}}

\DeclareMathOperator{\Neck}{\EuScript{N}\rm eck}
\DeclareMathOperator{\Mono}{\EuScript{M}\rm ono}
\DeclareMathOperator{\NeckPoly}{\EuScript{N}}
\DeclareMathOperator{\Orb}{\EuScript{O}\mathrm{rb}}

\DeclareMathOperator{\coh}{\rm coh}

\newcommand{\hatM}{\mathfrak{n}}
\newcommand{\dcyc}{\updelta}

\newcommand{\llangle}{\langle\kern -2.5pt\langle}
\newcommand{\rrangle}{\rangle\kern -2.5pt\rangle}
\newcommand{\ringtwo}{\C\llangle x,y\rrangle}
\newcommand{\llcurve}{\{\kern -3pt\{}
\newcommand{\rrcurve}{\}\kern -3pt\}}
\newcommand{\llsq}{[\![}
\newcommand{\rrsq}{]\!]}
\newcommand{\lcl}{(\kern -2.5pt(}
\newcommand{\rcl}{)\kern -2.5pt)}
\newcommand{\Jac}{\scrJ\mathrm{ac}}

\newcommand{\rad}{\mathrm{rad}}

\renewcommand{\phi}{\upvarphi}

\newcommand{\ring}{\C\llangle x,y\rrangle}
\newcommand{\xyring}{\C\llangle x,y_1,\hdots,y_{d-1}\rrangle}
\newcommand{\yring}{\C\llangle y_1,\hdots,y_{d-1}\rrangle}
\newcommand{\freering}{\C\langle x,y\rangle}
\newcommand{\al}{\upalpha}
\newcommand{\be}{\upbeta}
\newcommand{\ga}{\upgamma}

\newcommand{\xsf}{\mathsf{x}}
\newcommand{\ysf}{\mathsf{y}}
\newcommand{\gsf}{\mathsf{g}}
\newcommand{\hsf}{\mathsf{h}}
\newcommand{\jsf}{\mathsf{j}}
\newcommand{\ksf}{\mathsf{k}}

\newcommand{\Xsf}{\mathsf{X}}
\newcommand{\Ysf}{\mathsf{Y}}
\newcommand{\fringd}{\C\langle \mathsf{x}\rangle}
\newcommand{\ringd}{\C\llangle \mathsf{x}\rrangle}

\newcommand{\polyringd}{\C\langle x_1,\dots,x_d\rangle}
\newcommand{\polyringtwo}{\C\langle x,y\rangle}
\newcommand{\corank}{\scrC\mathrm{rk}}

\newcommand{\fJ}{{\mathfrak{J}}}
\newcommand{\Arad}{{\mathfrak{r}}}

\newcommand\GV[1]{\mathrm{GV}_{\! #1}}

\DeclareMathOperator{\LT}{LT}
\DeclareMathOperator{\LC}{LC}
\DeclareMathOperator{\LM}{LM}
\newcommand\Spoly{$S$\nobreakdash-\hspace{0pt}polynomial}
\newcommand\Spolys{$S$\nobreakdash-\hspace{0pt}polynomials}
\newcommand{\hot}{\textrm{h.o.t.}}

\DeclareMathOperator{\OL}{OL}

\newcommand{\mR}{\mathfrak{m}}

\newcommand{\Xbar}{\overline{X}}

\newcommand{\msf}{\mathsf{m}}
\newcommand{\nsf}{\mathsf{n}}
\newcommand{\fsf}{\mathsf{f}}
\newcommand{\asf}{\mathsf{a}}
\newcommand{\bsf}{\mathsf{b}}

\newcommand{\stub}{\mathrm{stub}}
\newcommand{\red}{\mathrm{red}}

\newcommand{\tf}{$3$\nobreakdash-fold}
\newcommand{\tfs}{$3$\nobreakdash-folds}
\newcommand{\Gaj}{\mathfrak{j}}

\newcommand{\caseA}{A}				% Z
\newcommand{\caseB}{B}		% X
\newcommand{\caseC}{C}			% Y
\newcommand{\nat}{\upvartheta}

\DeclareMathOperator{\scrHom}{\EuScript{H}\rm om}

\newcommand{\dd}{\mathds{D}}
\DeclareMathOperator{\Ext}{\mathrm{Ext}}

\newcommand{\Esix}[6]{%
\begin{tikzpicture}[scale=0.21]
\node at (0,0) [#1] {};
\node at (1,0) [#2] {};
\node at (2,0) [#3] {};
\node at (2,1) [#4] {};
\node at (3,0) [#5] {};
\node at (4,0) [#6] {};
\end{tikzpicture}
}

\newcommand{\Eseven}[7]{%
\begin{tikzpicture}[scale=0.21]
\node at (0,0) [#1] {};
\node at (1,0) [#2] {};
\node at (2,0) [#3] {};
\node at (2,1) [#4] {};
\node at (3,0) [#5] {};
\node at (4,0) [#6] {};
\node at (5,0) [#7] {};
\end{tikzpicture}
}

\newcommand{\Eeight}[8]{%
\begin{tikzpicture}[scale=0.21]
\node at (0,0) [#1] {};
\node at (1,0) [#2] {};
\node at (2,0) [#3] {};
\node at (2,1) [#4] {};
\node at (3,0) [#5] {};
\node at (4,0) [#6] {};
\node at (5,0) [#7] {};
\node at (6,0) [#8] {};
\end{tikzpicture}
}
\tikzset{
W/.style={circle,draw=black,circle,fill=white,inner sep=0pt, minimum size=4pt},
B/.style={circle,draw=black!80!white,circle,fill=black!80!white,inner sep=0pt, minimum size=4pt},
Or/.style={circle,draw=Orange,circle,fill=Orange,inner sep=0pt, minimum size=4pt},
P/.style={circle,draw=Pink,circle,fill=Pink,inner sep=0pt, minimum size=4pt},
R/.style={circle,draw=black!80!white,circle,fill=red!80!white,inner sep=0pt, minimum size=4pt},  
}

\begin{document}
\title[The Classification of Flops]{The Classification of Smooth $3$-fold Flops}
\author{Gavin Brown}
\address{Gavin Brown, Mathematics Institute, Zeeman Building, University of Warwick, Coventry, CV4 7AL, UK.}
\email{G.Brown@warwick.ac.uk}
\author{Michael Wemyss}
\address{Michael Wemyss, School of Mathematics and Statistics,
University of Glasgow, University Place, Glasgow, G12 8QQ, UK.}
\email{michael.wemyss@glasgow.ac.uk}
\begin{abstract}
This paper establishes the analytic classification of smooth, simple, \tf\ flops, and in the process provides both an intrinsic characterisation and a classification of their contraction algebras.    
\end{abstract}

%Primary: 58K40 Classification; finite determinacy of map germs
%Secondary: 16S15 Finite generation, finite presentability, normal forms (diamond lemma, term-rewriting)
%16S38 Rings arising from noncommutative algebraic geometry [See also 14A22]
%32S10 Invariants of analytic local rings
%14E30 Minimal model program (Mori theory, extremal rays)

\maketitle
\setcounter{tocdepth}{1}
%\tableofcontents        %% temporary for navigation during writing

\section{Introduction }
This paper classifies smooth \tf\ flops, namely those birational morphisms $\scrX\to\scrX_{\con}$ that contract a single nonsingular rational curve $\Curve$ that satisfies $K_{\scrX}\cdot \Curve =0$ with $\scrX$ a smooth \tf.
Such contractions are simply fundamental within both mathematics and theoretical physics: to give one example, such flops link any two birational Calabi--Yau \tfs\ \cite{KawaConnects}.  The techniques developed to do this are noncommutative, and their generality has many additional consequences, including to crepant divisor-to-curve contractions. 

Combined, the overall outcome is a full and final classification of flops, in a manner satisfyingly similar to the lists of Arnold \cite{Arnold}, Du Val \cite{DuVal1,DuVal2} and Klein \cite{Klein}, in which our families of flops exhibit moduli from length $3$ (equivalently, $E_6$) onwards. In contrast, this paper also shows that crepant divisor-to-curve contractions are much more rare: they stop at length $3$.

\subsection{Meta Strategy}\label{subsec:MetaStrategy}
One defining feature of any smooth, simple, $3$-fold flop
\[
\begin{tikzpicture}
\node at (-1,0) {$\scrX$};
\node at (1,0) {$\scrX^+$};
\node at (0,-1) {$\Spec\scrR$};
\draw[->] (-0.8,-0.25) -- node[left]{$\scriptstyle f\,\,$}(-0.2,-0.75);
\draw[->] (0.8,-0.25) -- node[right]{$\,\scriptstyle f^+$}(0.2,-0.75);
\end{tikzpicture}
\]
is its obvious symmetry.  Both the birational maps $\scrX\dashrightarrow\scrX^+$ and $\scrX^+\dashrightarrow\scrX$  are flops, which makes it manifestly impossible to distinguish or prefer $\scrX$ over $\scrX^+$.  As a consequence, any definition which attempts to `classify flops' up to analytic or complete local isomorphism needs to at the very least take into account this symmetry. 

Perhaps the most natural method, which avoids mention of the involution completely, is instead to classify the possible base rings $\scrR$ that can appear up to isomorphism; should that approach work, the good news is that by \cite{Pinkham} both $f$ and $f^+$ can be obtained canonically from $\scrR$ by blowing up the two generators $\pm 1$ of the class group $\mathrm{Cl}(\scrR)\cong\mathbb{Z}$.

The less good news is that the base rings $\scrR$, the extent to which we know them, are described by massive, highly non-trivial polynomials currently beyond the realms of human readability.  A fundamentally different approach is required, particularly when it turns out that an elegant classifying data structure does indeed exist. 

\medskip
To specify the setting, geometrically we are interested in the set $\scrF$ of all smooth $3$-folds $\scrX$ containing a $K_\scrX$-trivial curve $\Curve\cong\mathbb{P}^1$ which suitably locally (see \ref{cor:flops} and \S\ref{sec:BaseCategory}) can be contracted to give a flopping contraction, up to some choice of local isomorphism that is explained in \S\ref{sec:BaseCategory} but is standard and so we suppress from this introduction.    From this data the other side of the flop can be canonically constructed, so, exactly as above, understanding the set $\scrF$ suffices for our purposes.  

From here, the first step is to use the recent positive resolution of the Donovan--Wemyss conjecture \cite{JKM, KLW} (see also \cite{Booth}) which establishes a bijection
\[
\begin{tikzpicture}
\node at (0,0) {$\scrF \colonequals \{ \textnormal{all smooth $3$-fold flops}\}/\cong$};
\node at (0,-1) {$\scrA \colonequals \{\textnormal{finite dimensional contraction algebras}\}/\cong$};
\draw[<->] (-1.95,-0.3) -- node[left]{$\scriptstyle \sim$}(-1.95,-0.7);
\end{tikzpicture}
\]
This should be viewed as a theoretical classification, as it reduces the problem to understanding isomorphisms between specific finite dimensional algebras but says nothing about what $\scrF$ or $\scrA$ looks like, nor whether a classification of either is in fact possible.

From here, the strategy is remarkably simple. Taking the view that the set $\scrA$ is the set we wish to classify,  the main body of this paper constructs sets $\scrB$ and $\scrC$, where  
\begin{align*}
%\scrA &= \textnormal{set of things we want to classify}\\[2mm]
\scrB &= \{\textnormal{algebras satisfying some specific axioms, specified in \S\ref{subsec:NormalFormsIntro}}\}/\cong\\
\scrC &= \{\textnormal{the small list of algebras given in \ref{thm:introNFmain}\eqref{thm:introNFmain1}}\}/\cong
\end{align*}
and then produces a sequence of results which establishes that
\[
\scrA\subseteq \scrB\subseteq \scrC\subseteq \scrA.
\]
From this, $\scrA=\scrB=\scrC$ follows.   In particular we obtain  $\scrA=\scrC$, namely a full classification, and also $\scrA=\scrB$, an intrinsic characterisation. 

\medskip 
In more geometric terms, and as explained in \S\ref{subsec:2ndpaper}, the set $\scrC$ is in bijection with a set~$\scrS$ of \emph{some} $3$-fold flops.  The fact that $\scrA=\scrC$ then implies that the set $\scrF$ of all $3$-fold flops is in bijection with the subset $\scrS$, and so $\scrS$ realises the classification in geometric terms.

\subsection{Intrinsic Contraction Algebras (the set $\scrB$)}\label{subsec:NormalFormsIntro}
A feature common to both flops and crepant divisor-to-curve contractions is that by \cite{DW1} they admit an associated contraction algebra $\CA$.  This noncommutative algebra can be defined in various ways, via deformation theory of the exceptional curves, via cluster tilting, or via noncommutative singularity theory.

The fundamental problem is that contraction algebras are algebras that `come from geometry', and it is unclear how to \emph{intrinsically} characterise them within the wider class of all algebras.  This is the central problem solved here, from which almost all else follows.  

For the single-curve flops $\scrX\to\Spec\scrR$ under consideration in this paper, the most fundamental class, it is well-known that $\CA\cong\Jac(f)$ for some element $f\in\ringtwo$ in the complete free algebra; see e.g.\ \cite{VdBCY} and \S\ref{sec:differentiation}.  The following new  invariant is the key.

\begin{defin}\label{def:alglengthintro}
An element $f\in\ringtwo$ is said to have \emph{algebraic length $\ell$}, where $1\leq \ell\leq 6$, if there exists $z\in\scrZ(\Jac(f))$ such that $\Jac(f)/\lcl z\rcl \cong e\Uppi e$, for some ADE preprojective algebra $\Uppi$ and some primitive idempotent $e$ with associated number $\ell$ in \eqref{eqn:allDynkin}.  If no such central element exists, $f$ is said to be \emph{purely algebraic}.
\end{defin}

The number $\ell$ in \ref{def:alglengthintro} associated to the primitive idempotent $e$ is very natural indeed, coming from either the highest root in Lie theory, or the fundamental cycle in surface singularities; see \S\ref{subsec:AlgL} for details.  Regardless, one perhaps startling aspect of  \ref{def:alglengthintro} is that there is no appeal to generic central elements. Unlike in the geometry, where it is necessary to use generic elements to extract intrinsic invariants, on the noncommutative side plain old central elements suffice.  The reason is simple: the contraction algebra is the minimal classifying object, and so there is no redundancy in the information it contains. 

It is established in \ref{thm:ReductionPrecise} that the algebraic length is well defined and not multi-valued.  This then ensures that
\[
\scrB\colonequals \{\Jac(f)\mid \textnormal{$f$ has an algebraic length, and } \dim_\mathbb{C}\Jac(f)<\infty \}/\cong
\]
is a set.  The following summary theorem, which in the finite dimensional setting is simply the statement that $\scrA=\scrB$, asserts that algebraic length is indeed the intrinsic description of contraction algebras that we seek. 

\begin{thm}[{\ref{thm:Summary}}]\label{thm:IntroSummary}
If $\mathrm{B}$ is a $\mathbb{C}$-algebra, then the following statements are equivalent. 
\begin{enumerate}
\item $\mathrm{B}\cong\Jac(f)$ for some element $f\in\ringtwo$ that has an algebraic length.
\item $\mathrm{B}\cong\CA$ where $\CA$ is the contraction algebra of a smooth $3$-fold flop, or a smooth crepant divisor-to-curve contraction.
\end{enumerate}
In particular, contraction algebras of flops are characterised up to isomorphism as those $\Jac(f)$ where $f$ has an algebraic length, and where furthermore $\dim_\mathbb{C}\Jac(f)<\infty$.
\end{thm}

All the content in (2)$\Rightarrow$(1), which for flops is equivalent to the statement $\scrA\subseteq\scrB$, is that algebraic length is well defined and so $\scrB$ exists.   The proof of (1)$\Rightarrow$(2), which for flops is equivalent to $\scrB\subseteq\scrA$, involves vastly different techniques, and is a consequence of the two statements $\scrB\subseteq\scrC$ and $\scrC\subseteq \scrA$ which we now explain.

\subsection{Classification of Algebraic Length (the set $\scrC$)}

After further developing the techniques in noncommutative singularity theory initiated in \cite{BW3, BW4}, in \S\ref{sec:ParseOne}--\S\ref{sec:FunnellingOne} every algebra in the set $\scrB$, namely every algebra with algebraic length $\ell$, is shown to be isomorphic to one of the algebras stated in \ref{thm:introNFmain} below.  The set $\scrC$ is then defined to be the outcome of this process, namely the list, and so by design $\scrB\subseteq \scrC$.  All the content here, much like in the classical singularity theory of Arnold \cite{Arnold}, is in finding and establishing the list in the first place.

We then show in \S\ref{sec:FD}, still without using any geometry, precisely which algebras in the list are finite dimensional.  Again, as in classical singularity theory, the significance of these families is only fully realised \emph{after} classification: there is still no obvious reason why the following purely algebraic result should yet be classifying anything sensible.

To set notation, $\scrN_{i,j}$ are the necklace polynomials recalled in \S\ref{subsec:DefNeck}.  From the perspective of this introduction, $\scrN_{i,j}$ should simply be viewed as the natural noncommutative replacement of the commutative monomial $x^iy^j$.
\begin{thm}[{\ref{cor:classifyAcon}}]\label{thm:introNFmain} 
Suppose that $f\in\ringtwo$ has an algebraic length.
\begin{enumerate}
\item\label{thm:introNFmain1}  If $\dim_\C\Jac(f)<\infty$, then $\Jac(f)\cong\Jac(g)$ for some $g$ in the following list
\[
\renewcommand{\arraystretch}{1.3}
\medmath{\begin{array}{c@{\hspace{20pt}}l@{\hspace{2pt}}l@{\hspace{20pt}}l@{\hspace{2pt}}}
1 & x^2 + y^n && n\geq 2 \\
2 &x^2y+ y^{2n}+\upvarepsilon y^{2m-1} && 2n-1\geq m\geq 2  \\
3 &x^3 + xy^3 + y^n + y^{n+1}q& &q\in\C[y],\,n\geq 4  \\
%4 &x^3 + y^4+\scrN_{2,3}+\upalpha_1\scrN_{2,4}+\upalpha_2\scrN_{3,3}+ xy^7p +y^6q  && p,q\in\C[y]\\
4 &x^3 + y^4+\scrN_{2,3}+\sum_{(i,j)\in\scrT_4}\upalpha_{i,j}\scrN_{i,j} &&\\
5 &x^3+\scrN_{2,2}+y^5+xy^5p+y^6q  && p,q\in\C[y]\\
6 &x^3 + y^4+\scrN_{3,3}+\uplambda\scrN_{2,4}+\sum_{(i,j)\in\scrT_6}\upalpha_{i,j}\scrN_{i,j} && \uplambda\neq 0\\
\end{array}}
\renewcommand{\arraystretch}{1}
\]
where the length indexes the rows, $\upvarepsilon\in\{0,1\}$, $\upalpha_*\in\C$ and $\scrT_4,\scrT_6$ are defined in \eqref{eqn:DefS4} and \eqref{eqn:DefS6}.

\item\label{thm:introNFmain2}  If $\dim_\C\Jac(f)=\infty$, then $\Jac(f)\cong\Jac(g)$ for some $g$ in the following list
\[
\renewcommand{\arraystretch}{1.3}
\medmath{\begin{array}{c@{\hspace{20pt}}lll}
1& x^2 &  \\
2&x^2y + \upvarepsilon y^{2m-1} & m\geq2  \kern 135pt\\
3&x^3 + xy^3 & 
\end{array}\kern 33pt}
\renewcommand{\arraystretch}{1}
\]
where the length indexes the rows, and $\upvarepsilon\in\{0,1\}$.
\end{enumerate}
Conversely, every element in the lists in \eqref{thm:introNFmain1} and \eqref{thm:introNFmain2} has Jacobi algebra of length stated.
\end{thm}

The main part of the above theorem, namely \eqref{thm:introNFmain1} , establishes the definition of $\scrC$ and so the statement $\scrB\subseteq\scrC$.  The final converse statement for the list in \eqref{thm:introNFmain1}, which is equivalent to $\scrC\subseteq\scrB$, again takes longer: it is a combination of the statement $\scrC\subseteq\scrA$ explained in the next subsection, combined with $\scrA\subseteq\scrB$ established in the previous subsection.

\medskip
 Before proceeding, two remarks are in order. The first is simply that there are far fewer infinite dimensional algebras in \ref{thm:introNFmain} than might naively be expected. This has ramifications to the existence of crepant divisor-to-curve contractions below. The second is that every element in  \ref{thm:introNFmain}\eqref{thm:introNFmain2} can be viewed as a suitable $n\to\infty$ limit of some family in \ref{thm:introNFmain}\eqref{thm:introNFmain1}, a fact that remains mysterious.

\subsection{Classification of Flops}\label{subsec:2ndpaper}
Once the above algebraic steps which establish that $\scrA\subseteq\scrB$ and $\scrB\subseteq\scrC$ are complete, geometry can enter.  Inspecting the list in \ref{thm:introNFmain}, each element is a specific sum of necklace polynomials, and so the construction of flops from our companion paper \cite{BW5} then enters to show that \emph{every} Jacobi algebra in the list $\scrC$ can be realised as the contraction algebra of some $3$-fold flop, thus establishing $\scrC\subseteq\scrA$.  The chain of inclusions is then complete, and so finally $\scrA=\scrB=\scrC$.  

Translating the above into birational geometry, in the end simplicity reigns, and a classification of flops exists that is both uniform and elementary.   Removing all noncommutative technology from the statements (albeit, not the proofs), the following is the main geometric result.

\begin{thm}[{\ref{cor:flopsMain}}]\label{cor:flops}
Let $\scrX\to\scrX_{\con}$ be a smooth \tf\ flopping contraction, as recalled in \ref{defin:flop}. 
Then the completion of $\scrX$ along $\Curve$ is isomorphic to the completion of the variety defined by gluing two copies of $\mathbb{A}^3$ via the rule
\begin{eqnarray}\label{eq!glue_intro}
(a,z,y) & \stackrel{a\neq 0}{\longleftrightarrow} & (a^{-1}, \,a^{-1}z, \,\,a^{3}y + \partial E)
\end{eqnarray}
along the curve $y=z=0$, for some $\partial E\in\mathbb{C}[a^{\pm1},z]$ in the following list
\[
\renewcommand{\arraystretch}{1.5}
\medmath{\begin{array}{c@{\hspace{30pt}}r@{\hspace{2pt}}l@{\hspace{2pt}}c@{\hspace{2pt}}l@{\hspace{2pt}}c@{\hspace{2pt}}c@{\hspace{2pt}}c@{\hspace{20pt}}l@{\hspace{2pt}}l@{\hspace{2pt}}l@{\hspace{2pt}}c@{\hspace{2pt}}}
1 &a^2z &+& a^{2-n}z^{n-1}&&&&&
n\geq 2\\
2 &az^2 &+& a^{2-2n}z^{2n-1} &+& \upvarepsilon a^{3-2m}z^{2m-2}  &&&
 2n-1\geq m\geq 2,\, \upvarepsilon\in\{0,1\}\\
3&a^2z^2 &+& a^{-1}z^3 &+& a^{2-n}z^{n-1}&+&z^nQ_{3} & n\geq 4  \\ 
4& a^2z^2 &+& a^{-2}z^3 &+& a^{-1}z^4 &+& Q_4  &  \\ 
5&a^2z^2 &+& \phantom{a^{-2}}z^3 &+& a^{-3}z^4 &+& Q_5 &\\
6&a^2z^2 &+&  a^{-2}z^3 &+& (a^{-1} + \uplambda a^{-2})z^5  &+& Q_6  & \uplambda\in\C^* \\
\end{array}}
\renewcommand{\arraystretch}{1}
\]
where $Q_3, Q_4,Q_5,Q_6\in\mathbb{C}[a^{\pm1},z]$ are polynomials defined in \eqref{def:Qs}.  

Conversely, every member of the above list, for every choice of coefficients in the polynomials $Q_3, Q_4,Q_5,Q_6$, contracts locally, producing a smooth $3$-fold flopping contraction $\scrX\to\Spec\scrR$ of the given length, where $\scrR$ is complete local.  
\end{thm}

In the course of establishing \ref{cor:flops}, we show in \ref{thm:Summary} that Koll\'ar's length invariant \cite{Kollar} for geometric contractions is equal to the algebraic length defined in \ref{def:alglengthintro}.  Henceforth, we will use the word `length' unambiguously, to mean either.

In \ref{cor:flops}, for $\ell\geq 3$ the polynomials $Q_\ell$ should be treated as a shorthand for a choice of finitely many scalars, in which all scalars can be zero.
% \marginpar{This para now makes point of removing repetition; this is different to moduli, which is made below}
We will return in \S\ref{sec:ExistModuli} to explain that these scalars cannot all be removed within any fixed length $\ell\geq 3$ (that is, flops have moduli), but we pause to remark here that there are also some highly non-obvious isomorphisms within any fixed length $\ell\geq 3$.  Different parameters in e.g.\ $Q_4$ can happen to give isomorphic flopping neighbourhoods, analogous to the situation of Weierstrass models $y^2=x^3+ax+b$ of complex elliptic curves.

\medskip
The fact that the model \eqref{eq!glue_intro} is enough to describe all flopping neighbourhoods is a conclusion of the result, not an input.  The full and final classification in \ref{cor:flops} above includes Reid's $\ell=1$ Atiyah/Pagoda classification \cite{Pagoda} as the top line, albeit via a non-obvious isomorphism. Similarly, after further non-obvious changes in coordinates it of course also contains all known $\ell=2$ examples due to \cite{Laufer, Ferrari, AM, BW,Okke, Kawamata}, classified in \cite{BW3}, plus the four known examples of flops with $\ell\geq 3$ due to \cite{Ando, Karmazyn}.  

\subsection{Geometric Corollaries}
The existence of continuous parameters is hidden within \ref{thm:introNFmain} and \ref{cor:flops} as the choice of coefficients of a polynomial. A key benefit of the contraction algebra $\CA$ is that its isomorphism class is strictly easier to control than geometric data (e.g.\ the formal scheme defined by the gluing in \ref{cor:flops}), and so to prove that flops admit continuous parameters it simply suffices to verify when two noncommutative finite dimensional algebras are, or are not, isomorphic.

The next result establishes the existence of continuous parameters when $\ell=3$.  
\begin{prop}[{\ref{prop:moduliE65}}]\label{prop:moduliintro}
Consider
\begin{equation}
f = x^3 + xy^3 + y^5 + ay^6
\quad\textrm{and}\quad
g = x^3 + xy^3 + y^5 + by^6
\end{equation}
where $a,b\in\C$.
Then $\Jac(f)\cong\Jac(g)$ if and only if $a=b$.
\end{prop}
Similar results also establish the existence of continuous parameters for $\ell=4,5,6$, and details can be found in \ref{prop:modForl4}, \ref{prop:modForl5}, and \ref{prop:modForl6}.  We note that the continuous parameters makes \ref{thm:introNFmain} a list of forms, but not normal forms in the sense that there are still some isomorphisms hidden within the list. As in classical singularity theory, the existence of continuous parameters makes the quest for normal forms less pressing.
 
\medskip
It has for some time been an open question to  extract the numerical Gopakumar--Vafa (GV) invariants $\GV{1},\hdots,\GV{\ell}$ of a flopping contraction explicitly and easily from its corresponding contraction algebra $\CA$.  The next result completely solves this question, essentially by making it redundant: it turns out that basically all flops have the same GV (and thus DT/GW etc.) invariants.  This promptly extinguishes any hope that numerical curve invariants are in any way helpful for higher length flops, and it further demonstrates that the main result of \cite{BW} indeed captures the general behaviour.

\begin{cor}[{\ref{cor:allDims}}]\label{cor:allDimsIntro}
Let $\scrX\to\Spec\scrR$ be a smooth $3$-fold flop as in \ref{cor:flops}, with associated contraction algebra $\CA$.  Then $\dim_\mathbb{C}\CA$, and the associated \textnormal{GV} invariants, are the following, where the $n$ and $m$ in $\ell=1,2,3$ match those stated in \ref{thm:introNFmain}\eqref{thm:introNFmain1}.
\[
\medmath{
\begin{array}{l@{\hspace{20pt}}l@{\hspace{20pt}}l@{\hspace{10pt}}l@{\hspace{10pt}}l@{\hspace{10pt}}}
%\toprule
\ell & \dim_\C\CA & \GV{1},\hdots,\GV{\ell}&\\[2pt]
1 & n-1 & n-1 &&\\[3pt]
2&2m+4n-4 & 2m,n-1 & \mbox{if }\upvarepsilon=1 \mbox{ and } n\geq m\\
2& 6n-3& 2n+1,n-1 & \mbox{if }\upvarepsilon=1 \mbox{ and } n<m\\
2& 6n-3& 2n+1,n-1 & \mbox{if }\upvarepsilon=0 & \\[3pt]
3& 27 & 6,3,1 & \mbox{if }n= 4& \\
3& 4n+12 & 7,n-1,1 &\mbox{if }n>4\\[3pt]
4 & 60 & 6,5,2,1 \\[2pt]
5 & 125 & 8,6,4,2,1  \\[2pt]
6 & 200 & 6,6,4,3,2,1
\end{array}}
\]
\end{cor}
The static nature of the dimension of the contraction algebra $\CA$ when $\ell\geq 4$ is particularly remarkable.

The above also demonstrates that for $\ell\geq 4$, the length $\ell$ of the curve determines \emph{all} other numerical invariants.  As such, the length invariant $\ell$ returns at the end of the story as the hero, since it is just as good a numerical invariant as any other.  Of course, the whole point is that numerical invariants only take us so far: it is through understanding isomorphism classes of algebras, rather than simply studying their dimensions, that the classification of flops has been made possible.

\subsection{Divisor-to-Curve Contractions}
Since the main techniques of noncommutative singularity are general, this paper ends up also classifying all Jacobi algebras with algebraic length, not only the finite dimensional ones.  Thus there are additional applications to crepant divisor-to-curve contractions, that is, primitive Type III contractions.

Extreme care has to be taken in this more general setting, for two reasons. The first is that for crepant divisor-to-curve contractions the analogue of Katz--Morrison \cite{KM} is still open, although this creates at worst minor technical problems. The second, more serious reason, is that the following geometric example (which we believe to be unique) contradicts some papers and results on Type III contractions from the 1990s; we refer the reader to~\ref{rem:Issue90s} for details on how to correct those.

\begin{prop}[\ref{E6divcurveMain}]\label{E6divcurveIntro}
Consider  $f=x^5 - 3x^2yz + xy^3 + z^3 - t^2$   
and $\scrX_{\con}=\Spec\scrR$ for $\scrR\colonequals\C\llsq x,y,z,t\rrsq/(f)$.
\begin{enumerate}
\item The singular locus $C=\Sing(\scrX_\con)_{\mathrm{red}}$ is a singular curve germ.
\item Blowing up $C\subset\scrX_\con$ gives a crepant $\ell=3$ divisor-to-curve contraction $\scrX\to\scrY$ where $\scrX$ is smooth.
\item The contraction algebra of $\scrX\to\scrX_\con$ is isomorphic to $\Jac(x^3+xy^3)$.
\end{enumerate}
\end{prop}
The contraction in \ref{E6divcurveIntro} is constructed Zariski locally in \ref{ex:l3div} as a blowup of an ideal $X\to X_{\con}=\Spec\C[x,y,z,t]/(f)$, then passing to the formal fibre realises the contraction algebra.  The algebraic construction is mildly advantageous, since it allows us to consider neighbouring fibres, which are all of Type $A_1$.

With the above in hand, together with $\ell=2$ divisor-to-curve contractions constructed in \cite{BW3}, and the results of the previous subsections, the following is then immediate. Just as \ref{thm:introNFmain}\eqref{thm:introNFmain1} can be compared to Arnold \cite{Arnold}, Du Val \cite{DuVal1,DuVal2} and Klein \cite{Klein},  the following can be compared to Buchweitz--Greuel--Schreyer \cite{BGS}.

\begin{cor}[{\ref{cor:classifyAcon}\eqref{cor:classifyAcon2}}]\label{cor:DivCurveIntro}
Suppose that $\CA$ is the contraction algebra for some smooth crepant divisor-to-curve contraction. Then $\CA\cong\Jac(g)$ for some $g$ in the following list
\[
\renewcommand{\arraystretch}{1.2}
\medmath{\begin{array}{c@{\hspace{20pt}}lll}
1 & x^2 &  \\
2 &x^2y + \upvarepsilon y^{2m-1} & m\geq2, \upvarepsilon\in\{0,1\} \phantom{\quad+ (xy)^3 + \uplambda(xy^2)^2 +\, \scrO_7}\\
3 &x^3 + xy^3 & 
\end{array}}
\renewcommand{\arraystretch}{1}
\]
where the algebraic/Koll\'ar length $\ell$ indexes the rows.  Conversely, every algebra in the above list is the contraction algebra of a crepant divisor-to-curve contraction.
\end{cor}

A very remarkable aspect of \ref{cor:DivCurveIntro} is how few contraction algebras appear.  The analogue of the Donovan--Wemyss conjecture is still open in the setting of \ref{cor:DivCurveIntro}, namely there could still be multiple geometries which realise each contraction algebra listed, however there is enough information to establish the following non-existence result. The logic is simple: if a geometric contraction of large length existed, it would have to give an infinite dimensional contraction algebra $\CA\cong\Jac(f)$ with algebraic length $\ell\geq 4$.  But these are classified by \ref{cor:DivCurveIntro}, and there are none. In summary, the obstruction to such a contraction existing is noncommutative.  

\begin{cor}[\ref{cor:NoLength456Main}]\label{cor:NoLength456Intro}
Smooth crepant divisor-to-curve contractions of $\ell\geq 4$ do not exist.
\end{cor}
Pending the analogue of Katz--Morrison \cite{KM} or Donovan--Wemyss \cite{JKM, KLW} in this setting, the above  \ref{cor:NoLength456Intro} asserts that smooth crepant divisor-to-curve contractions stop at~$E_6$, which is precisely what happens for flips \cite{Mori}. This could either be a wild coincidence or it could hint at a much deeper relationship, but at this stage it is hard to say which.

\subsection{What Remains?}\label{sec:FAQ}
The classification of flops here is presented in two ways: via the forms in noncommutative singularity theory that realise the associated contraction algebra, and via the geometric gluing data that realises the smooth space $\scrX$ in the flopping contraction.  Notwithstanding that this is more than sufficient, it would in future be useful to have a Users Guide to Flops, akin to the fifteen characterisations of \cite{Durfee}, in which as much equivalent algebraic and geometric data is presented as possible.  This could include (1) the equations that define the base $\scrR$ (subject to the warnings in \S\ref{subsec:MetaStrategy}), (2) methodology that easily identifies a given flop in nature with one within the classification here, (3) the ability to recognise noncommutative normal forms and thus contraction algebras in a way similar to that in classical singularity theory, (4) a description of the matrix factorisations (or unprojection divisors) that realise the blowup, (5) a presentation of the NCCRs, together with (6) various other desirable properties such as establishing firm control of the dimension of the moduli.  As was the case in the development of surfaces, wish lists usually take some years to materialise; see also \cite{BW6} and for an overview \cite{Wemyss}.

\subsection{The Base Category}\label{sec:BaseCategory}
As opined by Reid \cite[\S4.4]{Reid:opinion}, we will not be overly concerned about the category in which we will work. We have to make a choice, so we state all results in the setting $\scrX\to\Spec\scrR$ where $\scrR$ is complete local, given that we are using \cite{VdB1d}. However, working \'etale locally, or with analytic varieties, or formal schemes, are perfectly reasonable alternative options.

\subsection{Notation and Conventions}\label{conventions}
The ground field, denoted $\C$, is algebraically closed of characteristic zero throughout.
The following notation is reserved and is used without further comment.

\begin{enumerate}
\item
Vector space dimension will be written  $\dim_{\mathbb{C}}V$.
\item
$\m=(x,y)$, a two-sided maximal ideal of the free algebra $\C\langle x,y\rangle$, whilst $\n=(x,y)$ is the unique two-sided maximal ideal of $\ringtwo$ (see \S\ref{sect: power notation}).
\item
The natural $\ring\twoheadrightarrow\Jac(f)$ will be written $\nat_f$, or simply $\nat$ when $f$ is clear.%\marginpar{macro \texttt{backslash-nat}}
\item
The Jacobson radical of $\Jac(f)$ will be written $\fJ_f$, or simply $\fJ$ when $f$ is clear.%\marginpar{macro \texttt{backslash-fJ}}
\item
$\ringd_i$ or $\fringd_i$ will denote the vector subspace consisting of homogeneous degree~$i$ polynomials. 
For a formal power series $g\in\ringd$ we denote the graded (necessarily polynomial) piece of degree~$i$ of $g$ by $g_i\in\ringd_i$.
\item
Write $g_{<d} = \sum_{i<d}g_i$ and $g_{>d}=\sum_{i>d}g_i$, with natural self-documenting variations such as $g_{\ge d}$. 
Thus, for example, $g=g_3+g_4+g_{\ge5}$ is a power series with no terms in degrees 0, 1 and~2, and no further conditions.
\item
The notation $f=g + \bigO_n$ means $f-g \in\mathfrak{n}^n$. That is, $f$ is a particular power series for which only terms of degree $<n$ are known, and specified by those of~$g$. This is additive in products: $\msf\times(g + \bigO_n) = \msf g + \bigO_{n+\deg(\msf)}$ for any monomial~$\msf$.
\item The notation $f=g + \hot$ is subtly different to $f=g + \bigO_n$. It indicates that $f$ may contain terms beyond the terms of $g$ which are larger with respect to the monomial order in \eqref{eqn:order}.
\end{enumerate}

\subsection*{Acknowledgements}
We place on record our appreciation of the very many mathematicians who have both directly and indirectly contributed to this work since it started over 12 years ago.  This includes, amongst others, Will Donovan, Osamu Iyama, Yujiro Kawamata, Shigefumi Mori and Miles Reid. Particular thanks are due to Allan Steel, whose conversation with M.W.\ in Sydney in 2023 changed our viewpoint on what was possible.

\subsection*{Funding} The authors were supported by EP/R034826/1, and by ERC Consolidator Grant 101001227 (MMiMMa).

\subsection*{Open Access} For the purpose of open access, the authors have applied a Creative Commons Attribution (CC:BY) licence to any Author Accepted Manuscript version arising from this submission.

%-----------------------------------------------------------------------------------------
\section{Preliminaries and Notation}

This section summarises, and in some cases amplifies, those tools and results of \cite[\S2]{BW2} used throughout, and is mainly used to fix notation.

\subsection{Polynomial and Power Series Notation}\label{sect: power notation}
We write $\ringd=\C\llangle x_1,\hdots,x_d\rrangle$ for formal noncommutative power series in $d$ variables, and $\fringd=\polyringd$ for the free algebra in $d$ variables. In practice, we work mostly in $\C\llangle x_1,x_2\rrangle$ or more simply $\C\llangle x,y\rrangle$. 

For either $f\in\fringd$ or $\ringd$
write $f$ in terms of its homogeneous pieces as
\[
f=f_0+f_1+f_2+f_3+f_4+\hdots,
\]
and define the \emph{order} of $f$ to be $\ord(f) = \min\{i\mid f_i\neq 0\}$, where by convention $\ord(0) = \infty$.   

We write $\m=(x_1,\hdots,x_d)$ for the two-sided maximal ideal of the free algebra $\fringd$ and $\hatM=(x_1,\hdots,x_d)\subset\ringd$.
For any $t\geq0$, set 
\[
\ringd_{\geq t}=\{f\in\ringd\mid f_i=0\mbox{ if }i<t\}=\{f\in\ringd\mid\ord(f)\geq t\} =\hatM^t
\] 
and note that this contains the zero element.

There is an isomorphism $\ringd\cong \varprojlim \fringd/\m^n$ sending $f\mapsto(f_{<n} + \m^n)_{n\geq 1}$, and continuous truncation maps
\[
\uppi_j\colon\ringd \longrightarrow \fringd/\m^j \quad(\cong\ringd/\hatM^j)
\]
where $\uppi_j(f)$ is the class of  $f_0 + \hdots + f_{j-1}$.

%-----------------------------------------------------------------------------------------
\subsection{Closure and Cyclic Permutation}
For any subset $\scrS\subset \ringd$, its \emph{closure} is 
\[
\overline{\scrS}\colonequals \bigcap_{i= 1}^{\infty}(\scrS + \hatM^i).
\]
In other words, $b\in\overline{\scrS}$ if and only if for all $i\geq 1$, there exists $s_i\in\scrS$ such that $b-s_i\in\hatM^i$.

\begin{defin}\label{DWZ cyclic fact}
Two elements $f,g\in \ringd$ are called cyclically equivalent, or $f$ is said to cyclically permute to $g$, if for every $i\ge0$, the homogeneous piece $(f-g)_i = f_i-g_i$ is a finite sum of commutators, that is, equals
$\sum_{j=1}^k \uplambda_j(a_jb_j-b_ja_j)$ for elements $a_j,b_j\in\ringd$ and $\uplambda_j\in\C$. 
We write $f\sim g$ in this case.
\end{defin}
The condition may be stated equivalently on $n$-jets: $f\sim g$ if and only if for every $n\ge0$, $f-g= \sum_{j=1}^k \uplambda_j(a_jb_j-b_ja_j) + \bigO_{n+1}$ for elements $a_j,b_j\in\ringd$ and $\uplambda_j\in\C$ (noting that $f=f_{\le n} + \bigO_{n+1}$).

%-----------------------------------------------------------------------------------------

\subsection{Jacobi Algebras}\label{sec:differentiation}
Consider the $\C$-linear  `strike off' maps $\partial_{i}\colon \ringd\to\ringd$ which act on monomials via the rule
\[
\partial_{i}(m)=\begin{cases} n & \text{if } m = x_in \\ 0 & \text{otherwise.} \end{cases}
\]
This determines $\C$-linear maps $\dcyc_{i}\colon \ringd\to\ringd$ which on monomials send
\[
x_{i_1}\hdots x_{i_t} \mapsto 
\sum_{j=1}^t \partial_{i}(x_{i_j}x_{i_{j+1}}\hdots x_{i_t}\cdot x_{i_1}\hdots x_{i_{j-1}}).
\]
The $\dcyc_{i}$ are called the \emph{cyclic derivatives}.  There is one such map, for each variable $x_i$.

\begin{defin}\label{defin:Jacobi}
For $f\in\ringd$, the Jacobi algebra is defined to be
\[
\Jac(f) = \frac{\ringd}{\lcl\dcyc_{1}f,\dots, \dcyc_{d}f\rcl}
\]
where $\lcl\dcyc_{1}f,\dots, \dcyc_{d}f\rcl\,\reflectbox{\colonequals}\, \lcl \updelta f\rcl$ is the closure of the two-sided ideal $(\updelta f)\colonequals (\dcyc_{1}f,\dots, \dcyc_{d}f)$.
\end{defin}

\subsection{Radicals and Coranks}\label{subsec:RadicalsandDim}
For any ring $\scrR$, write $\rad(\scrR)$ for its Jacobson radical. If $I$ is any ideal of $\scrR$ contained in $\rad(\scrR)$, then $\rad(\scrR/I)=\rad(\scrR)/I$ \cite[4.6]{Lam}.  
It is clear that $\rad(\ringd)=\n$.  If $f\in\ringd_{\geq 2}$, which will be the case throughout, then $\lcl \updelta f\rcl$ is contained in~$\n$, and so $\rad(\Jac(f))=\n/\lcl \updelta f\rcl$.  
We set $\fJ_f=\rad(\Jac(f))$ throughout, and more simply write $\fJ$ when $f$ is clear from the context.  Further, 
\[
\fJ^n = \big(\n^n+\lcl \updelta f \rcl \big)/\lcl \updelta f\rcl \textnormal{ for } n\ge2.
\]
From this, $\corank_i(f)\colonequals \dim_\C\fJ^i/\fJ^{i+1}$ is called the $i$th corank of $f$.

%-----------------------------------------------------------------------------------------
\subsection{Noncommutative Standard Basis}
\label{sec:ncsb}

The basic idea of \emph{standard bases} (also known as \emph{local Gr\"obner bases}) for non-commutative power series rings is natural, conceptually a close relation of the Bergman--Newman--Bokut--Shirshov Diamond Lemma \cite[1.2]{bergman}. 
Hironaka \cite{hironaka} and Mora \cite{mora88,mora94}
present the commutative case $\C\llsq\xsf\rrsq$, in a sense dual to Buchberger's well-known global polynomial version \cite{buchberger70}.
In the noncommutative case, Gerritzen and Holtkamp \cite{GH98} define the process for $\ringd$ and prove the main results. Following \cite{GH98}, below we set up notation and review the main ideas and results in $\ring$.

\subsubsection{Monomial orders, leading terms, standard basis, $\n$-adic bases}
\label{sec:stdbasis}

The set of monomials $W\subset\ring$
is an $\hatM$-adic basis for $\ring$, meaning just that
$\left\{ \uppi_j(\msf) \mid \msf \in W\right\} \setminus\{0\}$
is a basis for $\freering/\m^j$ for every $j\ge1$.
This is called the \emph{standard $\hatM$-adic basis} of $\ring$.
Any $f\in\ring$ has a unique expansion in the standard
$\hatM$-adic basis:
\[
f=\sum_{\msf\in W}c_{\msf}\msf
\quad\text{for unique coefficients }c_{\msf}\in\C.
\]
Once and for all, fix the local graded monomial order on $\ring$ used throughout:
for distinct monomials $\msf=\msf_1\msf_2\hdots \msf_d$ and
$\nsf=\nsf_1\hdots \nsf_e\in\freering$ (where each $\msf_i, \nsf_j\in\{x,y\}$), write
\begin{equation}
\msf<\nsf
\qquad\text{if and only if}\qquad
\begin{array}{rll}
\text{either}& d < e \\
 \text{ or} &
d = e \text{ and } \msf_k = x,\ \nsf_k=y
\end{array}\label{eqn:order}
\end{equation}
where $k\in\N$ is the minimum index $i$ for which the $i$th factor $\msf_i$ of $\msf$ and $\nsf_i$ of $\nsf$ disagree.

The \emph{leading term} $\LT(f)=c_{\msf}\msf$ of a nonzero $f=\sum c_{\msf}\msf\in\ring$
is the unique term with $c_{\msf}\ne0$ and $\msf\in W$ minimal with respect to this order.  By convention, the leading term of zero is zero. When $\LT(f)=c_{\msf}\msf$, the \emph{leading coefficient} of~$f$ is $\LC(f) = c_{\msf}$, the \emph{leading monomial} of~$f$ is $\LM(f)=\msf$, and $f$ is \emph{normalised} if $\LC(f)=1$.

For any $\Ga\subset\ring$, the set of leading monomials of $\Ga$ is $\LM(\Ga)=\{\LM(h)\mid h\in\Ga\}$.
For any ideal $I\subset\ring$ write $\scrI=\lcl I\rcl$ for the closure of $I$ and note that  $\LM(I) =\LM(\scrI)$.
For a set of monomials $X\subset\ring$, write $\Xbar\subset X$ for the unique minimal monomial generating set of the ideal generated by~$X$.

A \emph{weak standard basis} of $I$ is a set $\scrS=\{h_1,h_2,\hdots\}\subset\scrI$ for which $\LM(\scrS)$ generates the ideal $\LM(I)$ (the countability assumption is harmless), and $\scrS$ is a \emph{standard basis} if $\LM(\scrS)$ is a minimal generating set
(in particular, with no repetitions of leading monomials).
A standard basis $\scrS$ of $I$ is called \emph{simple} if each $\LC(h_i)=1$, and \emph{reduced} if the only monomial of  $\overline{\LM(I)}$ to appear in $h_i$ is $\LM(h_i)$.
A simple reduced standard basis of $I$ exists and is unique.
Any weak standard basis has the following properties:
\begin{enumerate}
\item\label{GI2}
$\lcl\scrS\rcl=\scrI$; equivalently
$f\in\scrI \Longleftrightarrow$ $f$ reduces to zero under $\scrS$
(in the $\hatM$-adic limit, as explained below).
\item\label{GI3} 
$W\setminus \left(W\LM(\scrS)W\right)$ describes a monomial $\hatM$-adic basis for $\ring/\scrI$.
\item\label{GI4}
Each vector space $(\scrI\cap\hatM^i)/(\scrI\cap\hatM^{i+1})$
has a monomial basis
\[
\left\{
\msf v\nsf \mid v\in\LM(\scrS)\text{ and } \msf,\nsf\in W \text{ with }
\deg \msf + \deg \nsf =i-\deg v \right\}.
\]
\end{enumerate}
A weak standard basis $\scrS$ is easily transformed into a simple reduced standard basis by scaling, removing multiples, and then reducing: that is, if $h\in\scrS$ and $u$ is a monomial from $\overline{\LM(I)}$ that lies in the support of $h$ but $u\not=\LM(h)$, then
$u=\msf\LM(g)\nsf$ is some multiple of some $g\in\scrS$, and $h\mapsto h-\la \msf g\nsf$ for suitable $\la\in\C$ removes $u$ from the support of $h$ while preserving leading term and $h\in(\scrS)$. Repeating inductively for such $u$ of minimal degree gives the reduction, preserving only $h\in\lcl\scrS\rcl$. 
Such reductions are the reason we have only $\scrS\subset\scrI$, not necessarily $\scrS\subset I$. 

More generally, we may reduce any $f\in\ring$ by $\scrS$, inductively pushing terms of minimal degree that are divisible by some $\LM(g)$ into higher degree to leave a power series having no term divisible by any $g\in\scrS$: that is
\begin{enumerate}[resume]
\item\label{GI5}
Any $f\in\ring$ may be expressed in the $\n$-adic basis \eqref{GI3} by reduction under $\scrS$.
\end{enumerate}

\subsubsection{Calculating a standard basis}
We describe the usual device for systematically `cancelling off low order terms'.
Set $W^* = W\setminus\{1\}$.
For $X\subset W$, the \emph{set of overlaps} is
\[
\mathsf{ol}(X) = \left\{  (\asf_0,\asf_1,\asf_2) \mid \asf_0,\asf_1,\asf_2\in W^* \text{ and } \asf_0\asf_1, \asf_1\asf_2\in\Xbar\right\}.
\]
This is the natural notion of overlaps: if $X = \{ x^2,xy\}$ then $\mathsf{ol}(X) = \{ (x,x,x), (x,x,y) \}$.

For $\scrS\subset\ring$, the \emph{set of overlaps} is
\begin{eqnarray*}
\OL(\scrS) &=& \left\{ (\asf,f) \mid \asf=(\asf_0,\asf_1,\asf_2)\in\mathsf{ol}(\LM(\scrS))\text{ and }\right. \nonumber\\
&& \qquad\qquad \left. f=(f_1,f_2)\text{ with }\LM(f_1)=\asf_0\asf_1,\ \LM(f_2)=\asf_1\asf_2 \right\}.
\end{eqnarray*}
The set of \Spolys{} $S(\scrS) \subset (\scrS)$ is
\[
S(\scrS) = \left\{ \LC(f_1)\asf_0f_2 - \LC(f_2)f_1\asf_2 \mid ( (\asf_0,\asf_1,\asf_2), (f_1,f_2)) \in \OL(\scrS) \right\}.
\] 

The analogue of the Buchberger criterion \cite[4.1 and 5]{GH98} is similar to the commutative case in practice. 
Starting with the given generators $\scrS=\{f_1,\hdots,f_s\}$ of $I$, work in increasing degree,
calculating \Spolys{} and reducing them.
If some \Spoly{} $H\in S(\scrS)$ does not reduce to zero under $\scrS$,
then include a normalised and reduced multiple of $H$ in $\scrS$,
reduce the elements of $\scrS$ up to the current degree and continue.
There is no termination guarantee, but if $\dim\ring/\lcl I\rcl<\infty$
then of course the procedure terminates with a finite set $\scrS$
that satisfies properties \eqref{GI2}--\eqref{GI4} above.
If only a weak standard basis is required, then it is not necessary to perform reductions.

%-----------------------------------------------------------------------------------------
\subsection{Unitriangular Automorphisms}
For any fixed $\sff_1,\dots,\sff_d\in\hatM\subset\ringd$, and $i\geq 1$, consider
\[
\phi_i\colon \fringd/\m^i\to \fringd/\m^i
\]
which sends $x_k+\m^i\mapsto x_k+(\sff_k)_{<i} +\m^i$ for each $1\le k\le d$.  On the truncated finite dimensional algebras $\fringd/\m^i$, since both the domain and codomain of $\phi_i$ have the same dimension, clearly each $\phi_i$ is an algebra isomorphism provided that it is surjective.
This is equivalent to $\phi_2$ being surjective.  Often this condition will be clear, so will not be mentioned.  Regardless, when $\phi_2$ is surjective, as in \cite[2.4]{DWZ} the $\phi_i$ combine to induce an automorphism $\phi\colon \ringd\to\ringd$.

\begin{defin}\label{unitriangular def}
Consider $\sff_1,\dots,\sff_d\in\hatM$ defining $\phi_i$ as above. If $\phi_2$ is surjective, then the induced $\phi\colon \ringd\to\ringd$ is called a \emph{unitriangular automorphism}. %By convention, when we say $\phi$ has depth $e\le0$ then $\phi$ is the identity map.
\end{defin}
We will abuse notation slightly and write 
\[\arraycolsep=2pt
\begin{array}{rcl}
\ringd&\to&\ringd\\
x_k&\mapsto& x_k+\sff_k
\end{array}
\]
for $\phi$, since it is induced by such morphisms on the truncations $\fringd/\m^i$.  For $e\ge1$ we say that $\phi$ has \emph{depth} $e$ provided that $\sff_1,\dots,\sff_d\in\hatM^{e+1}$.  Note that if $\phi$ has depth $\geq1$, then $\phi_2$ sends $x_i\mapsto x_i$ and so is clearly surjective, and thus $\phi$ is automatically an isomorphism.

\begin{lemma}\label{lem!simauto}\label{prop:scaling}
Let $\upphi$ be an automorphism of $\ringd$, and suppose that either
\begin{enumerate}
\item
$\upphi$ is a unitriangular automorphism, or
\item
$\upphi$ is a scaling automorphism $x\mapsto \upalpha x$ and $y\mapsto \upbeta y$, where $\upalpha,\upbeta\in\C^*$.
\end{enumerate}
If $f\sim g$, then $\upphi(f)\sim\upphi(g)$ and $\Jac(f)\cong\Jac(g)\cong\Jac(\upphi(f))\cong\Jac(\upphi(g))$.  Furthermore, $\Jac(h)\cong\Jac(\uplambda h)$ for any $\uplambda\in\C^*$
\end{lemma}

\begin{proof}
By the characterisation of \ref{DWZ cyclic fact}, for any $n\ge0$, $f-g = \sum_{i=1}^k\uplambda_i(a_ib_i-b_ia_i) + \bigO_{n+1}$.  
For both cases of automorphism $\upphi$, $\upphi(\n^{n+1})\subset\n^{n+1}$, so 
\[
\upphi(f) - \upphi(g) =
\sum_{i=1}^k\uplambda_i(\upphi(a_i)\upphi(b_i)-\upphi(b_i)\upphi(a_i)) + \bigO_{n+1}
\]
for any $n\ge1$. Thus $\upphi(f)\sim\upphi(g)$.  The fact that $\sim$ implies isomorphism of Jacobi algebras is \cite[3.3]{DWZ}.  Since the automorphisms in (1) and (2) fix the ground field $\C$, there is isomorphism $\Jac(f)\cong\Jac(\upphi(f))$ by \cite[3.7]{DWZ}.  The last statement involving $\uplambda$ is obvious, since the relations given by differentiating $\uplambda h$ are scalar multiples of the relations given by differentiating $h$.
\end{proof}

The following will be convenient: it is shorthand for $f$ and $g$ being equal up to and including degree $t-1$, and cyclically equivalent in all higher degrees.
\begin{notation}
For $t\geq 1$, write $f\sim_{t} g$ if and only if $f\sim g$ and $f - g \in \n^{t}$.
\end{notation}
A form of the usual Tschirnhausen trick works in the noncommutative context, taking care with cyclic equivalence.  Below, the following will mostly be applied with $k=3$, with the indexing $n+1\colonequals m+k$.

\begin{lemma}[Tschirnhausen]\label{tschirnhausen}
Suppose that $g\in\xyring$ satisfies
\[
g\sim_{m+k-1} x^k+ x^{k-1}p + xq + r +  \scrO_{m+k}
\]
for some $k\ge1$, $p,q\in\xyring$, $r\in\yring$, where $\ord(p)=m\ge2$ and $\ord(q)\ge k$.  Then the depth $m-1$ unitriangular $\upvarphi\colon x\mapsto x-\frac{1}{k}p, \,\,y\mapsto y$ satisfies
\[
\upvarphi(g)\sim_{m+k-1} x^k + xq+r + \scrO_{m+k}.
\]
\end{lemma}
\begin{proof}
Set $p=p(x,y)$ and $q=q(x,y)$, using the abbreviation $y=y_1,\hdots,y_{d-1}$, where it helps to emphasise that the arguments of $p$ and $q$ are not more general expressions.  Also, write $\sim$ for $\sim_{m+k-1}$.

First, note throughout that $2m+k-2\ge m+k$ and observe that 
\begin{align}
\left(x-\tfrac{1}{k}p\right)^k &= x^k - \tfrac{1}k\left(x^{k-1}p + x^{k-2}px + x^{k-3}px^2 + \hdots + px^{k-1}\right) + \bigO_{2m+k-2} \nonumber \\
&\sim x^k - x^{k-1}p + \bigO_{m+k}.\label{eq!xxk}
\end{align}
Second, $x^{k-1}p(x-\tfrac{1}{k}p,y)-x^{k-1}p(x,y) \in\n^{2m+k-2}$ since replacing one factor of $x$ in $p$ by the expression $\tfrac{1}{k}p$ leads to a series of order $(k-1)+(\ord(p)-1)+\ord(p)$. Thus
\begin{align}
\left(x-\tfrac{1}{k}p\right)^{k-1}p(x-\tfrac{1}{k}p,y) &=
 x^{k-1}p(x-\tfrac{1}{k}p,y) + \bigO_{2m+k-2} \nonumber\\
&= x^{k-1}p(x,y) + \bigO_{m+k}.\label{eq!xkm1p}
\end{align}
Third, $xq(x-\tfrac{1}{k}p,y)-xq(x,y) \in\n^{m+k}$ since replacing one factor of $x$ in $q$ by the expression $\tfrac{1}{k}p$ gives a series of order $1+(\ord(q)-1)+\ord(p)$. Thus
\begin{align}
\left(x-\tfrac{1}{k}p\right)q(x-\tfrac{1}{k}p,y) &= xq(x-\tfrac{1}{k}p,y) + \bigO_{m+k} \nonumber\\
&= xq(x,y) + \bigO_{m+k}.\label{eq!xq}
\end{align}
Now expanding $\upvarphi(g)$, it follows that
\begin{align*}
\upvarphi(g)
&\sim \upvarphi(x^k+ x^{k-1}p + xq + r +  \scrO_{m+k})  \tag{by e.g.\ \ref{lem!simauto}}\\
&= (x-\tfrac{1}{k}p)^k + (x-\tfrac{1}{k}p)^{k-1}p(x-\tfrac{1}{k}p,y)+ (x-\tfrac{1}{k}p)q(x-\tfrac{1}{k}p,y) + r(y) + \scrO_{m+k}\\
&= (x-\tfrac{1}{k}p)^k + x^{k-1}p(x,y)+ xq(x,y) + r+ \scrO_{m+k} \tag{by \eqref{eq!xkm1p} and \eqref{eq!xq}}\\
&\sim (x^k - x^{k-1}p) + x^{k-1}p+ xq + r+ \scrO_{m+k} \tag{by \eqref{eq!xxk}} \\
&=x^k + xq + r + \scrO_{m+k}.\qedhere
\end{align*}
\end{proof}

In future sections we will start with a given~$f$, and produce an infinite sequence of automorphisms which chase terms into higher and higher degrees. The following will allow us to take limits which are compatible with cyclic permutation.

%Unstated convention below:  \sff_{n} = \sff_{n-1} + (assessing the new degree n terms) + \scrO_{n+1}.  
%	so \sff_6 = (stub in deg 5) + (monomials in deg 6 to be assessed) + \scrO_7.

\begin{prop}\label{main chase lemma}
Start with a fixed $\sff_N$ of order $3$, where the index $N\geq 6$.  Assume that for every $n\geq N$ there are automorphisms $\uppsi_{n,1},\hdots,\uppsi_{n,m_n}$ each of depth $\geq n - 5$, 
%[calibrate: to fix the deg N=6 terms into the state we want, we need depth 1, so quadratics and more. Yes, this is what we use] 
and power series $\sff_{n,0},\hdots,\sff_{n,m_n}$ such that $\sff_n\sim_{n-1}\sff_{n,0}$ and
\[
\uppsi_{n,1}(\sff_{n,0}) \sim_{n-1} \sff_{n,1},\quad
\uppsi_{n,2}(\sff_{n,1}) \sim_{n-1} \sff_{n,2},\quad\hdots\quad,
\uppsi_{n,m_n}(\sff_{n,m_n-1}) \sim_{n-1} \sff_{n,m_n}.
\]
That is for each $n\geq N$, take $\sff_n$ (which is either $\sff_N$, or has been constructed inductively), cycle only in degrees $\geq n-1$ then apply a depth $\geq n-5$ automorphism, take this output, cycle only in degrees $\geq n-1$ then apply a depth $\geq n-5$ automorphism, etc., finitely often.

Then  $\sff_\infty \colonequals \lim_{n\geq N} \sff_n$ exists, and $\Jac(\sff_N) \cong \Jac (\sff_\infty)$.
\end{prop}

\begin{proof}
We first claim that 
\begin{enumerate}
\item The composition $\upphi_n \colonequals \uppsi_{n,m_n}\circ\hdots\circ\uppsi_{n,1}$ has depth $\geq n-5$, 
\item For all $n\geq N$, with $\sff_{n+1}\colonequals \sff_{n,m_{n}}$, (a) $\upphi_n(\sff_n) \sim \sff_{n+1}$, (b) $\upphi_n(\sff_n) - \sff_{n+1} \in \n^{n-1}$.
\end{enumerate}
(1) is well known (see e.g.\ \cite[2.2(2)]{BW3}).  For (2), by assumption we may write
\begin{align*}
\sff_n - \sff_{n,0} &= s_{n-1}^0 + s_n^0 + \scrO_{n+1}\\
\uppsi_{n,1}(\sff_{n,0}) - \sff_{n,1} &= s_{n-1}^1 + s_n^1 +\scrO_{n+1}\\
\vdots&\\
\uppsi_{n,m_n}(\sff_{n,m_n-1}) - \sff_{n,m_n} &= s_{n-1}^{m_n} + s_n^{m_n} +\scrO_{n+1} \end{align*}
where all $s_i^j$ have degree $i$, and all $s_i^j \sim 0$.  Now
\begin{align*}
\uppsi_{n,1}(\sff_n) 
&= \uppsi_{n,1}(\sff_{n,0}) + s_{n-1}^0 + s_n^0 + \scrO_{n+1}  \tag{since $\mathrm{depth}\,\uppsi\geq n-5\geq 1$}\\
&= \sff_{n,1} + (s_{n-1}^1 + s_{n-1}^0) + (s_n^0+s_n^1) + \scrO_{n+1}   \tag{again $\mathrm{depth}\,\uppsi\geq n-5\geq 1$}
\end{align*}
thus by induction
\[
\upphi_n(\sff_n) = \sum_{j=0}^{m_n} s_{n-1}^j + \sum_{j=0}^{m_n} s_{n}^j + \sff_{n,m_n}  + \scrO_{n+1} 
\]
Since $\sff_{n+1}\colonequals \sff_{n,m_n}$,  both (2)(a)(b) follow easily.

Now reindex, by setting $\sfg_1 = \sff_N, \sfg_2 = \sff_{N+1},\hdots$ and $\upvarphi^1=\upphi_N, \upvarphi^2=\upphi_{N+1}$, etc.  Then $\upvarphi^i$ has depth $\geq N+(i-1)-5 \geq i$  (since $N\geq 6$), $\upvarphi^i(\sfg_i) - \sfg_{i+1} \sim 0$, and  
\[
\upvarphi^i(\sfg_i) - \sfg_{i+1} =  \upphi_{N+i-1}(\sff_{N+i-1}) - \sff_{N+i}  \in \n^{N+i-2} \subset \n^{i+1}.  \tag{since $N-2\ge 4$}  
\]
These are the three properties that allow us to apply \cite[2.9]{BW3} directly, and conclude that $\lim \sfg_i=\lim_{n\geq N} \sff_n\colonequals \sff_\infty$ exists, and that there exists an automorphism $F$ such that $F(\sfg_1)\sim \sff_\infty$.  The last claimed isomorphism follows from the chain 
\[
\Jac(f_N)=\Jac(\sfg_1)\cong\Jac(F(\sfg_1))= \Jac(\sff_\infty),
\]
where the first and third follow from the easy fact that cyclic equivalence induces equality of Jacobi algebras, and the middle follows from the slightly harder fact that Jacobi algebras are preserved under automorphisms \cite[3.7]{DWZ}.
\end{proof}

%-----------------------------------------------------------------------------------------
\section{Normal Forms in Degrees 4 and 5}\label{sec:ParseOne}
 
All elements $f\in\ringtwo$ with $\ord(f)\leq 3$, except those satisfying $f\cong x^3+\scrO_4$, were put in normal form in \cite{BW3}; these are the cases $\ell=1,2$ (Types $A$ and $D$) in \ref{thm:ReductionPrecise} below.  This section considers the remaining elements of the form $f\cong x^3+\scrO_4$, and via changes in coordinates describes the options for the degree $4$ and $5$ terms of $f$.

\subsection{Normal Forms in Degree 4}
%The convention is that the coranks stated are for the stubs.
\begin{thm}\label{thm:uptodeg4}
If $f=x^3+\scrO_4\in\ring$, then $f\cong g+\scrO_5$ where $g$ is one of the following named cases, and either $\upvarepsilon=0$ or $\upvarepsilon=1$. 
\[
g=
\left\{
\begin{array}{lll}
x^3 + xy^3 + \upvarepsilon y^4&\textrm{\qquad case }E_6\\%&[2,3,4,4]\\
x^3 + xyxy&\textrm{\qquad case }\caseC\\%&[2,3,4,6]\\				% C = Y
x^3 + y^4&\textrm{\qquad case }\caseB\\%& [2,3,4,5]\\				% B = X
x^3 &\textrm{\qquad case }\caseA\\%&[ 2, 3, 5, 8 ]					% A = Z
\end{array}
\right.
\]
\end{thm}
\begin{proof}
%See emails to Gavin, 21st Feb 2024. And Deg4CleanUp, "examples from here" 
By cyclic permutation term-by-term, there exists $\upalpha, \upbeta, \upgamma, \updelta, \upzeta, \upeta\in\C$ such that
\begin{equation}\label{eq!deg4h4}
f\sim x^3 +
(\upalpha x^4
+ \upbeta x^3y
+ \upgamma x^2y^2
+ \updelta xyxy
+ \upzeta xy^3
+ \upeta y^4)
+\bigO_5.
\end{equation}
Suppose first that $\upeta\ne0$, and consider $\upvarphi_y\colon y\mapsto ay + bx$ for $a,b\in\C$, which implicitly fixes~$x$.
Up to cyclic permutation, the terms in $\upvarphi_y(f)$ with $xyxy$ and $y^4$ are precisely
\[
(a^2\updelta + a^2b\upzeta + 2a^2b^2\upeta) xyxy + a^4\upeta y^4.
\]
We may choose non-zero $a$ such that $a^4\upeta=1$, and $b$ such that $\updelta + b\upzeta + 2b^2\upeta=0$. For this choice $\upvarphi_y$ is an automorphism which when applied to $f$ removes the $xyxy$ term whilst normalising the coefficient on $y^5$; consequently, after relabelling the coefficients 
\begin{align}\label{eq!phiyf}
\upvarphi_y(f) &\sim x^3 + \upalpha\, x^4 + \upbeta\,x^3y + \upgamma\,x^2y^2  + \upzeta\,xy^3 + y^4 + \scrO_5\nonumber\\
&= x^3 + x^2(\upalpha\, x^2 + \upbeta\,xy + \upgamma\,y^2)  + \upzeta\,xy^3 + y^4 + \scrO_5.
\end{align}
Applying Tschirnhausen \ref{tschirnhausen} to \eqref{eq!phiyf}, using further \ref{lem!simauto}, there exists $\upvarphi_x$ with
\[
\upvarphi_x\upvarphi_y(f) \sim x^3  + \upzeta\,xy^3 + y^4 + \scrO_5.
\]
When $\upzeta=0$, we are in case~$\caseB$.  If $\upzeta\neq 0$, then applying \ref{prop:scaling} to $\upphi\colon x\mapsto ax$ and $y\mapsto by$ with $a,b\in\C^*$ chosen so that $a^3=\upzeta ab^3=b^4$ gives
\[
\upphi\upvarphi_x\upvarphi_y(f) \sim a^3(x^3  + xy^3 + y^4) + \scrO_5.
\]
This is case $E_6$ (with $\upvarepsilon=1$), since the non-zero scalar $a^3$ may be cancelled by \ref{prop:scaling}, without effecting the isomorphism class of the Jacobi algebra.

Now suppose that $\upeta=0$ in \eqref{eq!deg4h4}, so that
\[
f \sim x^3 + \upalpha\, x^4 + \upbeta\,x^3y + \upgamma\,xyxy+\updelta\,x^2y^2  + \upzeta\,xy^3 + \scrO_5.
\]
When $\upzeta\neq 0$, the unitriangular automorphism $\upvarphi_y\colon y\mapsto y -\tfrac{\upgamma}{\upzeta}x$ exists, and satisfies
\[
\upvarphi_y(f)\sim   x^3 +  
\tfrac{\upalpha \upzeta^2 - \upbeta\upgamma \upzeta + \upgamma^2\updelta}{\upzeta^2}\,x^4+
\tfrac{\upbeta \upzeta + \upgamma^2 - 2\upgamma\updelta}{\upzeta}\,x^3y+
\tfrac{-2\upgamma + \updelta}{1}\,x^2y^2+
\upzeta \,xy^3 + \scrO_5
\]
in which the $xyxy$ has been removed at the cost of terms involving $x^2$. Tschirnhausen \ref{tschirnhausen} then determines a unitriangular automorphism $\upvarphi_x$ such that
\[
\upvarphi_x\upvarphi_y(f) \sim x^3  + \upzeta\,xy^3 + \scrO_5.
\]
As above, an easy scaling then gives case $E_6$ with $\upvarepsilon=0$.  

Finally suppose that $\upeta=\upzeta=0$ in \eqref{eq!deg4h4}.  In this case, Tschirnhausen \ref{tschirnhausen} determines a unitriangular $\upvarphi_x$ with $\upvarphi_x(f)\sim x^3 + \upgamma xyxy+\scrO_5$. 
If $\upgamma\ne0$, then an easy scaling gives case~$\caseC$. If $\upgamma=0$, we are in case~$\caseA$.
\end{proof}

\subsection{Degree 5 Cleanup}\label{subsec:deg5clean}
This section proceeds to describe the possible degree five terms that can appear in each of the Cases $E_6$, $\caseA$, $\caseB$, and $\caseC$ in \ref{thm:uptodeg4}.  

One fact used repeatedly is that every degree five element is, up to $\sim$, of the form
\begin{equation}
x^2h_3 + \upalpha\,xyxy^2 + \upbeta\,xy^4 + \upgamma\,y^5\label{eqn:deg5general}
\end{equation}
for some $h_3\in\freering_3$ and $\upalpha, \upbeta, \upgamma\in\C$.
\begin{prop}[Case $E_6$]\label{prop!caseE6}
If $f = x^3 + xy^3 + \upvarepsilon y^4 + \scrO_5$, then $f\cong x^3 + xy^3 + \upvarepsilon y^4 + \uplambda y^5 + \scrO_6$, for some $\uplambda\in\mathbb{C}$.
\end{prop}
\begin{proof}
Consider $\upvarphi\colon y\mapsto y+ s xy + t y^2$ where $s,t\in\C$, which is an automorphism regardless of the choice of $s$ and $t$. Then
\[
\upvarphi( x^3 + xy^3 + \upvarepsilon y^4)\sim  x^3 + xy^3 + \upvarepsilon y^4 +  sx^2y^3 + 2sxyxy^2 + (4\upvarepsilon s + 3t)xy^4 + 4\upvarepsilon  ty^5+\bigO_6.
\]
Now by \eqref{eqn:deg5general} there exists $h_3\in\freering_3$ and $\upalpha, \upbeta, \upgamma\in\C$ such that 
\[
f\sim x^3 + xy^3 + \upvarepsilon y^4 +  (x^2h_3 + \upalpha\,xyxy^2 + \upbeta\,xy^4 + \upgamma\,y^5)+\bigO_6.
\]
We may choose $s$ and $t$ such that $2s+\upalpha=0$ and $\upbeta+4\upvarepsilon s + 3t=0$, and so
\[
\upvarphi(f)\sim x^3 + xy^3 + \upvarepsilon y^4 +  x^2(h_3+sy^3) + (\upgamma+ 4\upvarepsilon  t)y^5+\bigO_6.
\]
Set $\uplambda=\upgamma+ 4\upvarepsilon  t$, then applying Tschirnhausen \ref{tschirnhausen} yields the result.
\end{proof}

\begin{prop}[Case $\caseA$]\label{prop!caseA}
If $f = x^3 + \scrO_5$, then $f\cong g + \scrO_6$ where $g$ is one of the following:
\[
g= 
\left\{ 
\begin{array}{ll}
x^3 +  y^5 \\% &[ 2, 3, 5, 7, 10 ] \\ 
x^3 +  xyxy^2\\%  &[ 2, 3, 5, 7, 10 ] \\ 
x^3 +  xyxy^2 + y^5\\% &[ 2, 3, 5, 7, 9] \\ 
x^3 + xy^4\\ % 	& [ 2, 3, 5, 7,9 ]\\
x^3\\%	& [ 2, 3, 5, 8, 13 ]		 
\end{array}
\right.
\]
\end{prop}
\begin{proof}
%See Line 151 of Eallclearing.
By \eqref{eqn:deg5general} there exists $h_3\in\freering_3$ and $\upalpha, \upbeta, \upgamma\in\C$ such that 
\begin{equation}\label{eq!Zh5}
f\sim x^3 +  (x^2h_3 + \upalpha\,xyxy^2 + \upbeta\,xy^4 + \upgamma\,y^5)+\bigO_6.
\end{equation}
Consider first the case $\upgamma\neq 0$.  Under this assumption, the unitriangular automorphism $\upvarphi_y\colon y\mapsto y - \frac{\upbeta}{5\upgamma}x $ exists and satisfies
\[
\upvarphi_y(f)\sim  x^3 + x^2g_3 +
    \tfrac{5\upalpha\upgamma - \upbeta^2}{5\upgamma}\,xyxy^2 + 
    \upgamma\,y^5 +\scrO_6
\]
for some $g_3$. The point is that $xy^4$ has been removed at the cost of a coefficient on $xyxy^2$ plus some extra terms which have been absorbed into the $x^2$.  We may then use Tschirnhausen \ref{tschirnhausen} to supply a $\upvarphi_x$ which removes the $x^2g_3$ term, and thus 
\[
\upvarphi_x\upvarphi_y(f)\sim  x^3  +
    \upmu\,xyxy^2 + 
    \upgamma\,y^5 +\scrO_6.
\]
If $\upmu=0$, then a simple rescaling gives the first case.  If $\upmu\neq 0$, then the scaling $x\mapsto ax, y\mapsto by$ with $a,b\in\C^*$ satisfying
$a^3=\upmu a^2b^3=\upgamma b^5$  gives the third case.

Suppose instead that $\upgamma=0$ in \eqref{eq!Zh5}.  If $\upbeta\neq 0$, then  the unitriangular automorphism $\upvarphi_y\colon y\mapsto y - \frac{\upalpha}{2\upbeta}x $ exists and satisfies
\[
\upvarphi_y(f)\sim  x^3 + x^2h_3 +
 \upbeta xy^4 +\scrO_6
\]
for some $h_3$. The point is that $xyxy^2$ has been removed at the cost of absorbing more terms into the $x^2$ factor.  Now \ref{tschirnhausen} followed by an easy rescaling gives the fourth case. 

Hence we can assume that $\upgamma=\upbeta=0$ in \eqref{eq!Zh5}.  When $\upalpha=0$ then \ref{tschirnhausen} gives the fifth case; when $\upalpha\neq 0$ then \ref{tschirnhausen} followed by an easy rescaling gives the second case. 
\end{proof}

\begin{prop}[Case $\caseB$]\label{prop:CaseBsplit}
If $f=x^3+y^4+\scrO_5$, then $f\cong g + \scrO_6$ where $g$ is one of the following:
\[
g=
\left\{\begin{array}{lll}
x^3 + y^4 + xyxy^2 &\textrm{\qquad case }E_7\\%&[ 2, 3, 4, 5, 6 ]\\   
x^3 + y^4  	&\textrm{\qquad case }E_{8,6}^*\\%&[ 2, 3, 4, 5, 7 ]
\end{array}
\right.
\]
%with option for necklace $N_{23}$ instead.
\end{prop}
\begin{proof}
%Then Line 446 summarises
As before, by \eqref{eqn:deg5general}  there are $h_3\in\freering_3$ and $\upalpha, \upbeta, \upgamma\in\C$ with
\[
f\sim x^3 + y^4 + x^2h_3 + \upalpha\,xyxy^2 + \upbeta\,xy^4 + \upgamma\,y^5 +\scrO_6.
\]
Applying the automorphism $\upvarphi_y\colon y\mapsto y+ syx + ty^2$ with $s,t\in\C$ gives
\[
\upvarphi_y(f)\sim x^3 + y^4 + x^2h_3 + \upalpha\,xyxy^2 + (\upbeta+4s)xy^4 + (\upgamma+4t)y^5 +\scrO_6.
\]
Choosing $s=-\upbeta/4$ and $t=-\upgamma/4$, then composing with the $\upvarphi_x$ supplied by Tschirnhausen \ref{tschirnhausen} to clear the $x^2h_3$ term gives
\[
\upvarphi_x\upvarphi_y(f)\sim x^3 + y^4 + \upalpha\,xyxy^2 +\scrO_6. 
\]
If $\upalpha=0$ then we are in case $E_{8,6}^*$, whereas if $\upalpha\neq 0$ then scaling via $x\mapsto ax, y\mapsto by$ with $a,b\in\C^*$ satisfying $a^3=b^4=\upalpha a^2b^3$ gives case $E_7$.
\end{proof}

In contrast, the following is mildly more involved.

\begin{prop}[Case $\caseC$]\label{prop!caseC}
If $f=x^3 + xyxy + \scrO_5$, then $f\cong g + \scrO_6$ where $g$ is one of the following:
\[
g=
\left\{\begin{array}{lll} 
x^3 + xyxy + \upbeta\,xy^4 &\\%&[ 2, 3, 4, 6, 7 ] \\
x^3 + xyxy + y^5    &\textrm{\qquad case }E_{8,5}\\%&[ 2, 3, 4, 6, 7 ]\\
\end{array}
\right.
\]
for some $\upbeta\in\C$.
\end{prop}
\begin{proof}
%Line 2063 
Again by \eqref{eqn:deg5general} there exists $h_3\in\freering_3$ and $\upalpha, \upbeta, \upgamma\in\C$ with
\begin{equation}
f\sim x^3 + xyxy + x^2h_3 + \upalpha\,xyxy^2 + \upbeta\,xy^4 + \upgamma\,y^5 +\scrO_6.\label{eqn:CaseC1}
\end{equation}
Suppose that $\upgamma\neq0$.
The strategy is first to remove the $xy^4$ term, then those degree~4 terms which are created but are divisible by $x^2$, then the $xyxy^2$ term, and finally any remaining terms in degree~5 divisible by $x^2$. First apply $\upvarphi_y\colon y\mapsto y-\tfrac{\upbeta}{5\upgamma}x$ to give
\[
\upvarphi_y(f) \sim x^3 + xyxy + x^2\big(h_2+ j_{3}\big) + \upnu\,xyxy^2 + \upgamma\,y^5 +\scrO_6,
\]
for some $\upnu\in\C$, some element $j_{3}$ of degree three, where $h_2 \colonequals \tfrac{\upbeta^2}{25\upgamma^2}x^2 - \tfrac{2\upbeta}{5\upgamma}xy$.
That is, $xy^4$ has been removed at the cost of a coefficient on $xyxy^2$, plus some some additional (degree 2 and degree 3) terms on the $x^2$.

Next, attack the $x^2h_2$ term.  Applying $\upvarphi_x\colon x\mapsto x - h_2/3$ as in \ref{tschirnhausen} gives
\[
\upvarphi_x\upvarphi_y(f) \sim x^3 + xyxy + x^2 k_3 + \upmu\,xyxy^2  + \upgamma\,y^5 +\scrO_6
\]
for some element $k_3$ of degree~3 and $\upmu\in\C$.  Applying $\uppsi_y\colon y\mapsto y - \tfrac{\upmu}{2}y^2$ then gives
\[
\uppsi_y\upvarphi_x\upvarphi_y(f)\sim x^3 + xyxy + x^2k_3  + \upgamma\,y^5 +\scrO_6,
\]
from which applying the $\uppsi_x$ supplied by Tschirnhausen \ref{tschirnhausen} to clear the $x^2k_3$ term gives
\[
\uppsi_x\uppsi_y\upvarphi_x\upvarphi_y(f)\sim x^3 + xyxy + \upgamma\,y^5 +\scrO_6.
\]
The scaling $x\mapsto ax, y\mapsto by$ with $a,b\in\C^*$ satisfying $a^3=a^2b^2=\upgamma b^5$ gives the second case $E_{8,5}$.

Now suppose that $\upgamma=0$ in \eqref{eqn:CaseC1}. 
Applying $\upvarphi_y\colon y\mapsto y - \tfrac{\upalpha}{2}y^2$ gives
\[
\upvarphi_y(f)\sim x^3 + xyxy + x^2h_3  + \upbeta\,xy^4 +\scrO_6,
\]
after which the $\upvarphi_x$ supplied by Tschirnhausen \ref{tschirnhausen} gives the first case.
\end{proof}

%------------------------------------------------------------------------------------------------------------
\section{Centres of Nilpotent Jacobi Quotients}\label{sec:centres}\label{subsec:computeZ}

To set notation, $\scrZ(\Jac(f))$ is the centre of $\Jac(f)$.
In \S\ref{sec:AlgLength}, we control $\scrZ(\Jac(f))$ by considering the images of central elements in the quotients $A_k$ defined below.
The analysis of $A_k$ in this section is an easy and elementary application of the theory of noncommutative standard bases of ideals explained in \S\ref{sec:ncsb}.  

\begin{notation}\label{not!jacqt}
For fixed $f\in\ringtwo$ and $k\geq 1$, we use the following notation throughout.

\begin{enumerate}
\item Recall from \S\ref{subsec:RadicalsandDim} that $\fJ$ denotes the Jacobson radical of $\Jac(f)$.
\item Set $A_k \colonequals \Jac(f) / \fJ^k$, namely $\Jac(f)$ modulo the $k$th power of its radical. 
\item Consider the natural maps $\uppi_k\colon \Jac(f)\rightarrow A_k$ and $\uprho_{k+1}\colon A_{k+1}\rightarrow A_{k}$, and note that $\uppi_k = \uprho_{k+1}\circ\uppi_{k+1}$.
\item
Write $\Arad\colonequals \uppi_k(\fJ)\subset A_k$ for the Jacobson radical of $A_k$,
noting that $\uppi_k(\fJ^j)=\Arad^j$ for any~$j\ge1$.
\item\label{not!Ik}
Write $A_k = \ring/I_k$ where $I_k=(\dcyc f) + \m^k$.
The standard basis $\scrS$ of $I_k$ is finite because $\dim_{\C}A_k<\infty$.
It consists of particular polynomials $h_1,\hdots,h_s$ of $\ord h_i\le \deg h_i\le k-2$
together with all monomials of degree~$k-1$ not divisible by any $\LT(h_i)$.
We abbreviate this hereafter by saying `$\scrS=\{h_1,\hdots,h_s\}$ and monomials of degree~$k-1$', without specifying those additional monomials (their precise form plays no role in arguments).
\end{enumerate}
All of these notions depend on $f$, but this is suppressed from the notation.
\end{notation}

\begin{remark}\label{rk:sbn^k}
By \S\ref{sec:stdbasis}\eqref{GI3}, for any (weak) standard basis $\scrS$ of $I_k$,
the set of (classes in $A_k$ of) all monomials of degree $\le k-1$ that are not divisible by the leading term of any element of~$\scrS$ is a $\C$-vector space basis of $A_k$. Expressing elements in this basis using reduction \S\ref{sec:stdbasis}\eqref{GI5} is the fundamental tool.

It is important to be clear when standard basis is a simple tool and when it is not.
For any given ideal $I\subset\ring$, the standard basis of $I$ is likely to be very complicated,
 even if finite. In \S\ref{sec:FD}, knowing the standard basis would show at once that $\dim\ring/\lcl I\rcl<\infty$, but alas it cannot reasonably be calculated.

By contrast, for small $k$ the standard basis of $I+\n^k$ is very easy to calculate, either by eye or in one or two simple steps.
For example, consider
$f = \tfrac{1}3x^3 + xy^3 + \tfrac{1}4\la y^4 + \tfrac{1}5\mu y^5 + \bigO_6$.
The ideal $\lcl\dcyc f\rcl + \n^5$ is precisely $(\sfg_1,\sfg_2)+\n^5$, where
$\sfg_1=x^2+y^3$ and $\sfg_2=xy^2+yxy+y^2x+\la y^3 + \mu y^4$ are the derivatives truncated at order~5.
There are \Spolys
\[
x\sfg_1-\sfg_1x = xy^3 - y^3x \buildrel{\sfg_2}\over{\equiv} -yxy^2-y^2xy-\la y^4 - y^3x
\buildrel{\sfg_2}\over{\equiv} 0
\]
(truncating as we go) so that $xy^3\equiv y^3x$, and 
\[
x\sfg_2 - \sfg_1y^2 = xyxy + xy^2x + \la xy^3  \buildrel{\sfg_2}\over{\equiv} xyxy - yxyx - y^2x^2 - \la y^3x + \la xy^3
 \buildrel{\sfg_1}\over{\equiv} xyxy - yxyx.
\]
It is immediate that all other \Spolys\ lie in $\n^5$ and so need no consideration. Thus the standard basis of $(\sfg_1,\sfg_2)+\n^5$ is
\[
\scrS = \{\sfg_1, \sfg_2, xyxy-yxyx\}
\textrm{ and monomials of degree 5}
\]
where the monomials of degree~5 are in fact $y^2xyx$, $y^4x$, $y^5$, those not divisible by any of the leading terms $x^2$, $xy^2$, $xyxy$. 
It follows that an $\n$-adic $\C$-basis of $\ring/I+\n^5$ is given by the 14 monomials $1,x,y,xy,yx,y^2,xyx,yxy,y^2x,y^3,yxyx,y^2xy,y^3x,y^4$.
\end{remark}
Throughout \S\ref{subsec:computeZ}--\ref{sec:AlgLength}, the use of standard bases is restricted to very simple calculations as in \ref{rk:sbn^k}. As such, typically we give only quick outlines of proof, and only in harder cases.

\begin{lemma}
\label{lem:Ak2}
$\Arad^{k-1}\subset\scrZ(A_k)$.
\end{lemma}
\begin{proof}
If $z\in\Arad^{k-1}$ then each of $xz,zx,yz,zy\in\Arad^k$ are zero in~$A_k$, so that $z$ commutes with $x$ and $y$ trivially.
\end{proof}

The remainder of this section controls the centre of particular $A_k$ for various classes of~$f$.

\begin{prop}\label{prop:z1=0}
Let $f\in\ringtwo$.
If $f\cong\upvarepsilon x^3 + \bigO_4$ for $\upvarepsilon\in\{0,1\}$, then $\scrZ(A_k)\cap\Arad\subset\Arad^2$ for any $k\ge3$.
\end{prop}
\begin{proof}
The ideal $I_k$ is generated by the derivatives $\dcyc_xf = 3\upvarepsilon x^2 + \bigO_3$
and $\dcyc_yf=0+\bigO_3$ together with all monomials of degree~$k$.
Therefore the standard basis of $I_k$ lies within~$\Arad^3$, with the exception of a leading term $x^2$ in the case $\upvarepsilon\ne0$. Thus there is a $\C$-basis of $A_k$ comprising $1,x,y,xy,yx,y^2$ (and $x^2$ in the case $\upvarepsilon=0$) together with some elements of $\Arad^3$.

If $z=\la_1x+\la_2y+\bigO_2\in\scrZ(A_k)\cap\Arad$, then $xz-zx = \la_2(xy-yx)+\bigO_3$ and $yz-zy=\la_1(yx-xy)+\bigO_3$,
so $\la_1=\la_2=0$ since they both appear as coefficients of basis elements in an expression of zero.
\end{proof}
The proof strategy from here on is to determine a (weak) standard basis $\{h_1,h_2,\dots,h_s\}$ of $A_k$, by \ref{prop:z1=0} consider an arbitrary central element $z\in\Arad^2$ expressed with respect to the corresponding monomial $\C$-basis, and expand the commutators $xz-zx=yz-zy=0$.

\begin{prop}\label{prop:ZA4(1)}
Let $f\in\ringtwo$.
If $f \cong x^3 + \scrO_4$, then $\scrZ(A_4)\cap\Arad=\Arad^3$.
\end{prop}

\begin{proof}
Consider $z\in\scrZ(A_4)\cap\Arad$. By \ref{prop:z1=0} we may suppose $z\in\Arad^2$.
By \ref{thm:uptodeg4}, $f\cong x^3 + h +\scrO_5$ for $h\in\{xy^3+\upvarepsilon y^4,xyxy,y^4,0\}$ and $\upvarepsilon=0$ or~$1$.
The standard basis of $I_4$ therefore comprises some element of the form $x^2+\bigO_3$, some element $m_3+\bigO_4$ with $m_3\in\{xy^2+yxy+y^2x+4\upvarepsilon y^3,xyx,y^3,0\}$, and monomials of degree~4, % not divisible by $x^2$ or~$\LT(m_3)$, if nonzero
since any \Spoly{} involving $x^2$ has order $\ge4$. The corresponding $\C$-basis of $A_4$ consists of all monomials of degree $\le3$ that are not divisible by $x^2$ or $\LT(m_3)$. Thus $z$ has the form $z=\la_1xy+\la_2yx+\la_3y^2+\bigO_3$, and
\begin{eqnarray*}
xz-zx &=& \la_1(x^2y-xyx)+\la_2(xyx-yx^2)+\la_3(xy^2-y^2x) \nonumber\\
&=& (\la_2-\la_1)xyx+\la_3xy^2-\la_3y^2x 
\end{eqnarray*}
since the standard basis includes $x^2+\bigO_3$, so that $x^2$ reduces to zero in $A_4$, and
\[
yz-zy = \la_1(yxy-xy^2)+\la_2(y^2x-yxy) = (\la_1-\la_2)yxy -\la_1xy^2+\la_2y^2x.
\]
When $\LT(m_3)\in\{xyx,y^3,0\}$, all the monomials involved in these two expressions are elements of a $\C$-basis or zero, and it is immediate that $\la_1=\la_2=\la_3=0$ and $z\in\Arad^3$ as claimed. When $\LT(m_3)=xy^2$, the expressions above reduce further to
\[
 (\la_2-\la_1)xyx-\la_3yxy-2\la_3y^2x-4\la_3\upvarepsilon y^3
 \textrm{ and }
(\la_1-\la_2)yxy +\la_1yxy+(\la_1+\la_2)y^2x+4\la_1\upvarepsilon y^3
\]
and again $\la_1=\la_2=\la_3=0$ and $z\in\Arad^3$ follow immediately.
The result follows by \ref{lem:Ak2}.
\end{proof}

\begin{cor}\label{cor:ZAk}
If $f\cong x^3 + \bigO_4$, then
$\scrZ(A_k)\cap\Arad\subset\Arad^3$ for any $k\ge4$.
\end{cor}

\begin{proof}
By \ref{sec:stdbasis}\eqref{GI3}, the standard basis of $I_k$ determines a monomial $\n$-adic basis $\scrB$ of~$A_k$.
The natural map $\uprho\colon A_k\rightarrow A_4$ is injective on the span of monomials in $\scrB$ of degree $\le3$.
Consider $z\in\scrZ(A_k)\cap\Arad$. Using $\scrB$, we may write $z=h+\bigO_3$ with $h$ expressed in
the elements of $\scrB$ of degree~$\le2$. The image $\uprho(z)\in\scrZ(A_4)\cap\Arad$ has the same form $\uprho(z) = h + \bigO_3$, and so $h=0$ by \ref{prop:ZA4(1)}.
\end{proof}

Thus we may strengthen the proof strategy and restrict attention to arbitrary central elements $z\in\Arad^3$.

\begin{prop}\label{lem:ZA5}
Let $f\in\ringtwo$.
If $f$ belongs to cases $\caseA$, $\caseB$ or $\caseC$, namely
\[
f\cong
\left\{
\begin{array}{lll}
x^3 + \scrO_5&\caseA\\
x^3 + y^4+ \scrO_5&\caseB\\
x^3 + xyxy+ \scrO_5&\caseC\\
\end{array}
\right.
\]
then $\scrZ(A_5)\cap\Arad=\Arad^4$.
\end{prop}
\begin{proof}
Case $\caseA$. By \ref{prop!caseA}, $f\cong x^3+h+\bigO_6$ for $h\in\{y^5,xyxy^2+\upvarepsilon y^5, xy^4,0\}$ and $\upvarepsilon=0$ or~1. The standard basis of $I_5$ therefore comprises an element $x^2+\bigO_4$, an element $m_4+\bigO_5$ with $m_4\in\{y^4,xy^2x+yxy^2+y^2xy+5\upvarepsilon y^4, xy^3+yxy^2+y^2xy+y^3x,0\}$, and monomials of degree~5.
The corresponding $\C$-basis of $A_5$ consists of all 19 monomials of degree $\le4$ that are not divisible by $x^2$, with $\LT(m_4)$ also removed if nonzero.

By \ref{cor:ZAk}, we may consider $z\in\scrZ(A_5)\cap\Arad^3$. Such $z$ has the form
\[
z= 
\mu_1xyx+\mu_2xy^2+\mu_3yxy+\mu_4y^2x+\mu_5y^3 + \bigO_4.
\]
Since $x^2\equiv \bigO_4$ in $A_5$, the commutator $xz-zx$ equals
\[
\mu_3 (xyxy-yxyx) + (\mu_4-\mu_2) xy^2x + \mu_5 (xy^3-y^3x) + \bigO_5.
\]
Reducing further by any of the possible $\LT(m_4)$ leaves these same coefficients on (possibly different) $\C$-basis elements, so 
$\mu_2=\mu_4$ and $\mu_3=\mu_5=0$. Imposing these, $yz-zy$ equals
\[
\mu_1 (yxyx-xyxy)+ \mu_2 (yxy^2 + y^3x - xy^3 - y^2xy) + \bigO_5
\]
so, even if $m_4$ reduces $xy^3$, we have $\mu_1=\mu_2=0$, and $z\in\Arad^4$.

\medskip
Case $\caseB$. 
By \ref{prop:CaseBsplit}, $f\cong x^3+y^4+h+\bigO_6$ for $h\in\{xyxy^2,0\}$. The standard basis of $I_5$ therefore comprises an element of form $x^2 + \bigO_4$, an element $y^3+\bigO_4$, and monomials of degree~5. 
The corresponding $\C$-basis of $A_5$ consists of all monomials of degree $\le4$ that are not divisible by $x^2$ or $y^3$. 
The proof now follows as Case~$\caseA$ above with $\mu_5=0$ (and using that now $xy^3,y^3x$ reduce into $\bigO_5$).

\medskip
Case $\caseC$. 
By \ref{prop!caseC}, $f\cong x^3+xyxy+h+\bigO_6$ for $h\in\{\upbeta xy^4,y^5\}$ with $\upbeta\in\C$. The standard basis of $I_5$ therefore comprises an element of form $x^2+\tfrac{2}3yxy+\bigO_4$, an element $xyx + m_4+\bigO_5$ with $m_4\in\{\tfrac{\upbeta}2(xy^3+yxy^2+y^2xy+y^3x),\tfrac{5}2y^4\}$, and monomials of degree~5. (The four \Spolys{} on the overlaps of $x^2$ and $xyx$ all reduce into $\bigO_5$ at once.) 
The corresponding $\C$-basis of $A_5$ consists of all 16 monomials of degree $\le4$ that are not divisible by $x^2$ or~$xyx$. 
Thus, using \ref{cor:ZAk}, we consider $z\in\scrZ(A_5)\cap\Arad^3$ in the form
\[
z=
\mu_1xy^2+\mu_2yxy+\mu_3y^2x+\mu_4y^3 + \bigO_4
\]
so that $yz-zy$ equals
\[
(\mu_1-\mu_2)yxy^2 - \mu_1xy^3+(\mu_2-\mu_3)y^2xy+\mu_3y^3x + \bigO_5
\]
whence $\mu_1=\mu_2=\mu_3=0$, and then $xz-zx=\mu_4(xy^3-y^3x)+\bigO_5$ 
gives $\mu_4=0$, and $z\in\Arad^4$ as claimed.
\end{proof}

\begin{prop}\label{lem:ZA6}
Let $f\in\ringtwo$.
\begin{enumerate}
\item\label{lem:ZA61}
(Within $\caseB$) If $f\cong x^3 + y^4 +\scrO_6$ (the case $E_{8,6}^*$),
then $\scrZ(A_6)\cap\Arad=\Arad^5$.
\item\label{lem:ZA62} (Within $\caseC$)
If $f\cong x^3 + xyxy + \upbeta\,xy^4 +\scrO_6$ for some $\upbeta\in\C$,  and $z\in\scrZ(A_6)\cap\Arad$,
then $z=\uplambda xy^2x + \bigO_5$ for some $\uplambda\in\C$.
\end{enumerate}
\end{prop}

\begin{proof}
\eqref{lem:ZA61}
The standard basis of $I_6$ comprises an element $x^2+\bigO_5$, an element $y^3+\bigO_5$ and monomials of degree~6.
The corresponding $\C$-basis of $A_6$ consists of all 23 monomials of degree $\le5$ that are not divisible by $x^2$ or~$y^3$. 
By \ref{lem:ZA5}, $z\in\scrZ(A_5)\cap\Arad$ has the form
\[
z= \la_1xyxy + \la_2 xy^2x + \la_3 yxyx + \la_4 yxy^2 + \la_5 y^2xy + \bigO_5
\]
so that $xz-zx$ equals
\[
(\la_3-\la_1) xyxyx + \la_4 xyxy^2 + \la_5 xy^2xy - \la_4 yxy^2x - \la_5 y^2xyx + \bigO_6
\]
so $\la_4=\la_5=0$ and $\la_1=\la_3$, and then $yz-zy=-\la_1xyxy^2-\la_2 xy^2xy 
+ \la_2 yxy^2x + \la_1 y^2xyx$ 
so that $\la_1=\la_2=\la_3=0$ and $z\in\Arad^5$ as claimed.

\medskip
\noindent
\eqref{lem:ZA62}
Rescaling $x,y,\upbeta$ gives $f\cong \tfrac{1}3x^3 + \tfrac{1}2(xy)^2 + \upbeta xy^4 + \bigO_6$.
The following claim is elementary.
\begin{claim}\label{claimSB}
If $\upbeta=0$ then the standard basis of $I_6$ is $\scrS=\{x^2 + yxy, xyx, xy^2xy,yxy^2x\}$
together with monomials of degree~6.

If $\upbeta\ne0$ then the standard basis of $I_6$ is
\[
\scrS = \left\{
\begin{array}{ll}
x^2 + yxy + by^4,&
    xyx + bxy^3 + byxy^2 + by^2xy + by^3x, \\
    xy^2xy - yxy^2x, &
    xy^3x + \tfrac{b - 1}{b}yxy^2x
\end{array}
\right\}
\]
together with monomials of degree~6.
\end{claim}

\medskip\noindent
Suppose $\upbeta=0$. By Claim~\ref{claimSB},
there is a $\C$-basis of $A_6$ consists of all 23 monomials of degree $\le5$ that are not divisible by any of $x^2$, $xyx$, $xy^2xy$ and $yxy^2x$. 

By \ref{lem:ZA5}, $z\in\scrZ(A_5)\cap\Arad$ has the form
\begin{equation}\label{eq:zform}
z= \la_1xy^2x + \la_2 xy^3 + \la_3 yxy^2 + \la_4 y^2xy + \la_5 y^3x + \la_6y^4 + \bigO_5
\end{equation}

If $\upbeta=0$ then
$yz-zy=$
\begin{align}
&-\la_1xy^2xy - \la_2 xy^4 +\la_1yxy^2x + (\la_2-\la_3) yxy^3 + (\la_3-\la_4) y^2xy^2 \nonumber\\
&\qquad + (\la_4-\la_5) y^3xy + \la_5 y^4x + \bigO_6 \nonumber\\
&=
- \la_2 xy^4 + (\la_2-\la_3) yxy^3 + (\la_3-\la_4) y^2xy^2 + (\la_4-\la_5) y^3xy + \la_5 y^4x + \bigO_6 \nonumber\label{eq:yzminuszy}
\end{align}
so that $\la_2=\la_3=\la_4=\la_5=0$. With those, $z=\la_1xy^2x +  \la_6y^4 + \bigO_5$ and so
\[
xz-zx = \la_1x^2y^2x +  \la_6xy^4  - \la_1xy^2x^2 +  \la_6y^4x  + \bigO_6 = \la_6xy^4 +  \la_6y^4x  + \bigO_6 
\]
so $\la_6=0$ and $z=\la_1xy^2x+\bigO_6$, as required.

\medskip\noindent
Now suppose $\upbeta\ne0$.
By Claim~\ref{claimSB},
there is a $\C$-basis of $A_6$ consists of all 23 monomials of degree $\le5$ that are not divisible by any of $x^2$, $xyx$, $xy^2xy$ and $xy^3x$. 
Thus $z\in\scrZ(A_5)\cap\Arad$ again has the form~\eqref{eq:zform}
and the identical argument follows.
\end{proof}

Finally we consider the case $f\cong 0+\scrO_4$. The ideas are no different, except that this case does not benefit from the preparation lemmas of \S\ref{sec:ParseOne}.
\begin{lemma}\label{lem:preplast}
If $f\in\ring$ has $\ord(f)=4$, then the degree~$3$ leading monomials of $\dcyc_xf$ and $\dcyc_yf$ are one of the following pairs:
\[
(x^3,h),\quad (x^2y,x^3),\quad (xy^2,x^2y),\quad (yxy,xyx),\quad (y^3,xy^2), \quad (0,y^3)
\]
where $h\in\{0,x^3,x^2y,xyx,xy^2,y^3\}$.
\end{lemma}
\begin{proof}
Write $f\sim g_4 + \bigO_5$ where $g_4=\al_1x^4 + \al_2x^3y + \al_3x^2y^2 + \al_4xyxy + \al_5xy^3 + \al_6y^4$
for some $\al_1,\hdots,\al_6$ not all zero. Now  
\begin{equation}\label{eq:dfpairs4}
\medmath{
\begin{pmatrix}
\dcyc_xg_4 \\ \dcyc_yg_4
\end{pmatrix} =
\begin{pmatrix}
4x^3 & x^2y+xyx+yx^2 & xy^2+y^2x & 2yxy & y^3 & 0 \\
0 & x^3 & x^2y+yx^2 & 2xyx & xy^2+yxy+y^2x & 4y^3
\end{pmatrix}
\begin{pmatrix}
\al_1 \\\vdots\\ \al_6
\end{pmatrix}.
}
\end{equation}
The statement follows, since the entries of the $2\times 6$ matrix in \eqref{eq:dfpairs4} are written in increasing monomial order.
\end{proof}

\begin{prop}\label{prop:ZA4(0)}
If $f\in\ringtwo_{\ge4}$, then either
\begin{enumerate}
\item
$\scrZ(A_4)\cap\Arad=\Arad^3$, or
\item
$f = \al(x^2y^2-xyxy) + \bigO_5$ for some $\al\in\C^*$ and $\scrZ(A_4)\cap\Arad = (xy-yx) + \Arad^3$.
\end{enumerate}
\end{prop}

\begin{proof}
We write $h_1$, $h_2$ for the normalised $\dcyc_xf$, $\dcyc_yf$ respectively. Their leading terms $m_1,m_2$ form one of the pairs of \ref{lem:preplast}, or are both zero if $\ord(f)\ge5$.
The ideal $I_4=(m_1,m_2)+\Arad^4$ has a standard basis $\{m_1,m_2\}$ and monomials of degree~4, since any \Spoly{} among $m_1,m_2$ lies in $\Arad^4$. Therefore a $\C$-basis of $A_4$ is given by the monomials of degree $\leq 3$ after removing $m_1$ and $m_2$.

Write $z\in\scrZ(A_4)\cap\Arad^2$ in the form $z=\la_1x^2 + \la_2xy+\la_3yx+\la_4y^2+\bigO_3$
and consider
\begin{equation}\label{eq:z2derivs}
\begin{array}{rcl}
xz-zx &=& \phantom{-}\la_2 x^2y + (\la_3-\la_2)xyx - \la_3yx^2 + \la_4xy^2 \phantom{{}+ (\la_2-\la_3)yxy } - \la_4y^2x\\
yz-zy &=& -\la_1 x^2y \phantom{{}+(\la_3-\la_2)xyx} + \la_1yx^2 - \la_2xy^2 + (\la_2-\la_3)yxy + \la_3y^2x
\end{array}
\end{equation}
When $\ord(f)\ge5$, the monomials of degree $\le3$ form a $\C$-basis of $A_4$, and so
$\la_1,\hdots,\la_4$ are all zero and $z\in\Arad^3$, so suppose $\ord(f)=4$. We must reduce \eqref{eq:z2derivs} by
the two entries of~\eqref{eq:dfpairs4} corresponding to the first non-vanishing~$\al_i$.

If $\al_1\ne0$ then $m_1=x^3$ and $h_1$ makes no reduction,
while $h_2$ makes a reduction only in the cases $m_2=x^2y$, $xyx$ or $xy^2$.
We show that $\la_1,\hdots,\la_4$ are all zero in each of these cases (using notation $c_i\in\C$ for the scaled coefficients):
\begin{enumerate}
\item
If $m_2=x^2y$, then $h_2= x^2y+yx^2 + c_1xyx + c_2(xy^2+yxy+y^2x)+c_3y^3+\bigO_4$ eliminates $x^2y$ from the $\C$-basis, and the resulting coefficient of $yx^2$ in $yz-zy$ is~$2\la_1$. Thus $\la_1=0$, and so in fact there was no reduction in $yz-zy$ (since $x^2y$ did not actually appear) and so also $\la_2=\la_3=0$.
With that, similarly there is no reduction in $xz-zx$, and so also $\la_4=0$.

\item
If $m_2=xyx$, then $h_2= xyx + c_1(xy^2+yxy+y^2x)+c_2y^3+\bigO_4$ and there is no reduction of $yz-zy$ so $\la_1=\la_2=\la_3=0$. Thus the coefficient of $xyx$ in $xz-zx$ is zero, so there is no reduction there either and so also $\la_4=0$.

\item
If $m_2=xy^2$, then  $h_2= xy^2+yxy+y^2x+\hot$. 
The $x^2y$ and $yx^2$ terms do not reduce, so $\la_1=\la_2=\la_3=0$. Then $xz-zx$ reduces to
${}-\la_4yxy +\hot$ so $\la_4=0$.
\end{enumerate}
If $\al_1=0$ and $\al_2\ne0$, then $h_2$ makes no reduction ($m_2=x^3$) and reducing \eqref{eq:z2derivs} by $h_1=x^2y+xyx+yx^2+\hot$ gives $yz-zy\equiv \la_1xyx+\hot$ Thus $\la_1=0$ and so in fact $yz-zy$ does not reduce, whence $\la_2=\la_3=0$
so also $xz-zx$ does not reduce and $\la_4=0$.

The case $\al_1=\al_2=0$ and $\al_3\ne0$ is where the alternative arises. We reduce \eqref{eq:z2derivs} by $h_1=xy^2+y^2x+c_1yxy+c_2y^3$ and $h_2=x^2y+yx^2 + c_3xyx+c_4(xy^2+yxy+y^2x) + c_5y^3$. At once $xz-zy\equiv (\la_3-\la_2-c_3\la_2)xyx-(\la_2+\la_3)yx^2+\hot$ and $yz-zy\equiv 2\la_1yx^2 + \hot$, so $\la_1=\la_2+\la_3=0$.
Setting $\la_3=-\la_2$, the pair \eqref{eq:z2derivs} then reduces further to
\[
\begin{array}{l}
-(2+c_3)\la_2xyx  +(c_1c_2\la_2 - c_1\la_4-c_4\la_2)yxy-2\la_4y^2x + (c_2c_4\la_2 -c_2\la_4- c_5\la_2)y^3 \\
 \quad\textrm{and}\quad (2+c_1)\la_2 yxy + c_2\la_2y^3.
\end{array}
\]
Thus $\la_4=0$. If $\la_2=0$ we are done. Otherwise $\la_2\ne0$ and so $c_2=c_4=c_5=0$ and $c_1=c_3=-2$.  Noting that $c_1=c_3=2\al_4/\al_3$, we have that $f = \al_3(x^2y^2-xyxy)+\bigO_5$ and $z=\la_2(xy- yx)\in\scrZ(A_4)$, giving the alternative result.

The remaining cases are quick.
If $\al_{\le3}=0$ and $\al_4\ne0$, then $m_1=yxy$ and $m_2=xyx$ so $yz-zy$ reduces by $h_1$ only to show $\la_1=\la_2=\la_3=0$;
then $\la_4=0$ since $xz-zx$ is already reduced.
If $\al_{\le4}=0$ and $\al_5\ne0$, then $m_1=y^3$ and $m_2=xy^2$ so at once $\la_1=\la_2=\la_3=0$;
then $\la_4=0$ after simple $xz-zx$ reduction.
Finally if $\al_{\le5}=0$ and $\al_6\ne0$, then there are no reductions and all $\la_i=0$.
\end{proof}

%------------------------------------------------------------------------------------------------------------
\section{Algebraic Length}\label{sec:AlgLength}
This section proves \ref{thm:ReductionPrecise}, which asserts that if there exists~$z\in\scrZ(\Jac(f))$ such that $\Jac(f)/z$ is isomorphic to $e\Uppi e$ for some idempotent $e$ in some ADE preprojective algebra~$\Uppi$, then $f$ must belong to a very small number of cases from \S\ref{subsec:deg5clean}.  This will form the first key step towards the classification of contraction algebras.

%-------------------------------------------------------------------
\subsection{Leading Terms of Central Elements}
Recall the notation $\uppi_k\colon\Jac(f)\rightarrow A_k$ from \ref{not!jacqt}.
The following lemma is clear since each $\uppi_k$ is a surjective homomorphism.
\begin{lemma}\label{lem:Zmodk}
If $z\in\scrZ(\Jac(f))$, then $\uppi_k(z)\in\scrZ(A_k)$ for any $k\ge1$.
\end{lemma}

\begin{prop}\label{lem:keycentreallcases}
Let $f\in\ringtwo$ and $z\in\scrZ(\Jac(f))\cap\fJ$.
\begin{enumerate}
\item\label{lem:keycentreallcases1} If $f \cong x^3 + \scrO_4$, then $\uppi_3(z)=0$.
\item\label{lem:keycentreallcases10} If $f\cong 0+\scrO_4$, then $\uppi_3(z)=\al(xy-yx)$ for some $\al\in\C$.
\item\label{lem:keycentreallcases2} If $f$ belongs to cases $\caseA$, $\caseB$ or $\caseC$, namely
\[
f\cong
\left\{
\begin{array}{lll}
x^3 + \scrO_5&\caseA\\
x^3 + y^4+ \scrO_5&\caseB\\
x^3 + xyxy+ \scrO_5&\caseC\\
\end{array}
\right.
\]
then $\uppi_4(z)=0$.
\item\label{lem:keycentreallcases3} (Within $\caseB$)
If $f$ belongs to case $E_{8,6}^*$, namely $f\cong x^3 + y^4 +\scrO_6$, then $\uppi_5(z)=0$.
\item\label{lem:keycentreallcases4} (Within $\caseC$)
If $f\cong x^3 + xyxy + \upbeta\,xy^4 +\scrO_6$ for some $\upbeta\in\C$,  then  $\uppi_5(z)=\uplambda xy^2x$ for some $\uplambda\in\C$.
\end{enumerate}
\end{prop}

\begin{proof}
The strategy to impose conditions on $\uppi_k(z)$ is simple: consider the image of $z$ under the composition
\[
\uppi_k\colon \Jac(f) \buildrel{\uppi_{k+1}}\over{\longrightarrow} A_{k+1} \buildrel{\uprho_{k+1}}\over{\longrightarrow} A_{k}
\]
using the fact that $\uppi_{k+1}(z)\in\scrZ(A_{k+1})$ and apply the results of \S\ref{subsec:computeZ}.

\noindent
(1) By \ref{prop:ZA4(1)} we have $\scrZ(A_4)\cap\n=\n^3$, so $\uppi_4(z)\in\n^3$.
Thus $\uppi_3(z) = \uprho_4(\uppi_4(z)) = 0$.

\noindent
(2) By \ref{prop:ZA4(0)} we have either $\scrZ(A_4)\cap\n=\n^3$, so $\uppi_3(z)=0$ as in \eqref{lem:keycentreallcases1} (and $\al=0$ gives the result), or
$\uppi_4(z) = \al(xy-yx)+\bigO_3$ for some $\al\ne0$, so $\uppi_3(z)=\al(xy-yx)$.

\noindent
(3) By \ref{lem:ZA5} we have $\scrZ(A_5)\cap\n=\n^4$, so $\uppi_5(z)\in\n^4$ and $\uppi_4(z)=0$ follows as before.

\noindent
(4) By \ref{lem:ZA6}\eqref{lem:ZA61} we have $\scrZ(A_6)\cap\n=\n^5$, so $\uppi_6(z)\in\n^5$ and $\uppi_5(z)=0$.

\noindent
(5) By \ref{lem:ZA6}\eqref{lem:ZA62}, if $z\in\scrZ(A_6)\cap\n$, then $z=\uplambda xy^2x + \bigO_5$ and
the result follows.
\end{proof}

%-------------------------------------------------------------------
\subsection{Algebraic Length}\label{subsec:AlgL}
Recall first the following ADE Dynkin diagrams, where the vertices have been labelled with the data of  the rank of the highest root.
\begin{equation}\label{eqn:allDynkin}
\begin{array}{c}
\begin{tikzpicture}
\node at (0,0) {$A_n$};
\node at (2.3,0) {$
\begin{tikzpicture}[xscale=0.6]
 \node (0) at (0,0) [DW] {};
 \node (1) at (1,0) [DW] {};
 \node (2) at (2,0) [DW] {};
 \node (3) at (4,0) [DW] {};
 \node (4) at (5,0)[DW] {};
 \node at (3,0) {$\cdots$};
 \node (0a) at (0,-0.3) {$\scriptstyle 1$};
 \node (1a) at (1,-0.3) {$\scriptstyle 1$};
 \node (2a) at (2,-0.3) {$\scriptstyle 1$};
 \node (3a) at (4,-0.3) {$\scriptstyle 1$};
 \node (4a) at (5,-0.3) {$\scriptstyle 1$};
\draw [-] (0) -- (1);
\draw [-] (1) -- (2);
\draw [-] (2) -- (2.6,0);
\draw [-] (3.4,0) -- (3);
\draw [-] (3) -- (4);
\end{tikzpicture}$};
\node at (5,0) {$D_n$};
\node at (7.3,0) {\begin{tikzpicture}[xscale=0.6,yscale=0.7]
 \node (0) at (0,0) [DW] {};
 \node (1) at (1,0) [DW] {};
 \node (1b) at (1,0.75) [DW] {};
 \node (2) at (2,0) [DW] {};
 \node (3) at (4,0) [DW] {};
 \node (4) at (5,0)[DW] {};
 \node at (3,0) {$\cdots$};
 \node (0a) at (0,-0.3) {$\scriptstyle 1$};
 \node (1a) at (1,-0.3) {$\scriptstyle 2$};
 \node (1ba) at (1.3,0.75) {$\scriptstyle 1$};
 \node (2a) at (2,-0.3) {$\scriptstyle 2$};
 \node (3a) at (4,-0.3) {$\scriptstyle 2$};
 \node (4a) at (5,-0.3) {$\scriptstyle 1$};
\draw [-] (0) -- (1);
\draw [-] (1) -- (2);
\draw [-] (2) -- (2.6,0);
\draw [-] (3.4,0) -- (3);
\draw [-] (3) -- (4);
\draw [-] (1) -- (1b);
\end{tikzpicture}};
\node at (0,-1.5) {$\begin{tikzpicture}[yscale=0.7,xscale=0.6]
\node at (-0.5,0.75) {$E_6$};
 \node (m1) at (-1,0) [DW] {};
 \node (0) at (0,0) [DW] {};
 \node (1) at (1,0) [DW] {};
 \node (1b) at (1,0.75) [DW] {};
 \node (2) at (2,0) [DW] {};
 \node (3) at (3,0) [DW] {};
\node (m1a) at (-1,-0.3) {$\scriptstyle 1$};
 \node (0a) at (0,-0.3) {$\scriptstyle 2$};
 \node (1a) at (1,-0.3) {$\scriptstyle 3$};
 \node (1ba) at (1.3,0.75) {$\scriptstyle 2$};
 \node (2a) at (2,-0.3) {$\scriptstyle 2$};
 \node (3a) at (3,-0.3) {$\scriptstyle 1$};
\draw [-] (m1) -- (0);
\draw [-] (0) -- (1);
\draw [-] (1) -- (2);
\draw [-] (2) -- (3);
\draw [-] (1) -- (1b);
\end{tikzpicture}$};
\node at (3.5,-1.5) {$\begin{tikzpicture}[yscale=0.7,xscale=0.6]
\node at (-0.5,0.75) {$E_7$};
 \node (m1) at (-1,0) [DW] {};
 \node (0) at (0,0) [DW] {};
 \node (1) at (1,0) [DW] {};
 \node (1b) at (1,0.75) [DW] {};
 \node (2) at (2,0) [DW] {};
 \node (3) at (3,0) [DW] {};
 \node (4) at (4,0) [DW] {};
\node (m1a) at (-1,-0.3) {$\scriptstyle 2$};
 \node (0a) at (0,-0.3) {$\scriptstyle 3$};
 \node (1a) at (1,-0.3) {$\scriptstyle 4$};
 \node (1ba) at (1.3,0.75) {$\scriptstyle 2$};
 \node (2a) at (2,-0.3) {$\scriptstyle 3$};
 \node (3a) at (3,-0.3) {$\scriptstyle 2$};
 \node (4a) at (4,-0.3) {$\scriptstyle 1$};
\draw [-] (m1) -- (0);
\draw [-] (0) -- (1);
\draw [-] (1) -- (2);
\draw [-] (2) -- (3);
\draw [-] (3) -- (4);
\draw [-] (1) -- (1b);
\end{tikzpicture}$};
\node at (8.2,-1.5) {$\begin{tikzpicture}[yscale=0.7,xscale=0.6]
\node at (-0.5,0.75) {$E_8$};
 \node (m1) at (-1,0) [DW] {};
 \node (0) at (0,0) [DW] {};
 \node (1) at (1,0) [DW] {};
 \node (1b) at (1,0.75) [DW] {};
 \node (2) at (2,0) [DW] {};
 \node (3) at (3,0) [DW] {};
 \node (4) at (4,0) [DW] {};
  \node (5) at (5,0) [DW] {};
\node (m1a) at (-1,-0.3) {$\scriptstyle 2$};
 \node (0a) at (0,-0.3) {$\scriptstyle 4$};
 \node (1a) at (1,-0.3) {$\scriptstyle 6$};
 \node (1ba) at (1.3,0.75) {$\scriptstyle 3$};
 \node (2a) at (2,-0.3) {$\scriptstyle 5$};
 \node (3a) at (3,-0.3) {$\scriptstyle 4$};
 \node (4a) at (4,-0.3) {$\scriptstyle 3$};
 \node (5a) at (5,-0.3) {$\scriptstyle 2$};
\draw [-] (m1) -- (0);
\draw [-] (0) -- (1);
\draw [-] (1) -- (2);
\draw [-] (2) -- (3);
\draw [-] (3) -- (4);
\draw [-] (4) -- (5);
\draw [-] (1) -- (1b);
\end{tikzpicture}$};
\end{tikzpicture}
\end{array}
\end{equation}
To each of these graphs is an associated preprojective algebra $\Uppi$ (see e.g.\ \cite{CBH}),
which is finite dimensional.  The vertices correspond to primitive idempotents $e$ in $\Uppi$, so that each primitive idempotent of $\Uppi$ comes equipped with the number in \eqref{eqn:allDynkin}.

Our main new algebraic invariant which will aid classification is the following. 

\begin{defin}\label{def:alglength}
We say $f\in\ringtwo$ has \emph{algebraic length $\ell$}, where $1\leq \ell\leq 6$, if there exists $z\in\scrZ(\Jac(f))$ such that $\Jac(f)/\lcl z\rcl \cong e\Uppi e$, for some ADE preprojective algebra $\Uppi$ and some primitive idempotent $e$ with associated number $\ell$ above.  If no such central element exists, we say that $f$ is \emph{purely algebraic}.
\end{defin}

The fact that the algebraic length is well-defined will be established in \ref{thm:ReductionPrecise} below.
\begin{remark}\label{rem:notgeneric}
Later (in \ref{thm:Summary}) we will show that the algebraic length is precisely Koll\'ar's length invariant for flops.  With that in mind, one truly remarkable aspect is that \ref{def:alglength} does not require the central element to be generic to be well defined.  This is a manifestation of the fact that the contraction algebra is the minimal classification object: it, unlike the $3$-fold geometry, is very small object, and has comparatively few central elements.
\end{remark}

\begin{remark}\label{def!6algebras}
Whilst the following remark cannot be used until after the classification in \S\ref{sec:MAIN}, it is helpful to keep in mind that by Katz--Morrison \cite{KM} we should expect that the $e\Uppi e$ in the following cases
\[
\begin{array}{c}
\begin{tikzpicture}[scale=0.625]
\node at (0,0) {$
\begin{tikzpicture}[scale=0.45]
 \node (0) at (0,0) [DBlue] {};
 \node (0a) at (0,-0.5) {$\scriptstyle 1$};
 \node[draw=none] at (0,1) {};
\end{tikzpicture}$};
\node at (2,0) {\begin{tikzpicture}[scale=0.4]
 \node (0) at (0,0) [DB] {};
 \node (1) at (1,0) [DBlue] {};
 \node (1b) at (1,0.75) [DB] {};
 \node (2) at (2,0) [DB] {};
 \node (0a) at (0,-0.5) {$\scriptstyle 1$};
 \node (1a) at (1,-0.5) {$\scriptstyle 2$};
 \node (1ba) at (1.4,0.75) {$\scriptstyle 1$};
 \node (2a) at (2,-0.5) {$\scriptstyle 1$};
\draw [-] (0) -- (1);
\draw [-] (1) -- (2);
\draw [-] (1) -- (1b);
\end{tikzpicture}};
\node at (5.1,0) {$\begin{tikzpicture}[scale=0.4]
 \node (m1) at (-1,0) [DB] {};
 \node (0) at (0,0) [DB] {};
 \node (1) at (1,0) [DBlue] {};
 \node (1b) at (1,0.75) [DB] {};
 \node (2) at (2,0) [DB] {};
 \node (3) at (3,0) [DB] {};
\node (m1a) at (-1,-0.5) {$\scriptstyle 1$};
 \node (0a) at (0,-0.5) {$\scriptstyle 2$};
 \node (1a) at (1,-0.5) {$\scriptstyle 3$};
 \node (1ba) at (1.4,0.75) {$\scriptstyle 2$};
 \node (2a) at (2,-0.5) {$\scriptstyle 2$};
 \node (3a) at (3,-0.5) {$\scriptstyle 1$};
\draw [-] (m1) -- (0);
\draw [-] (0) -- (1);
\draw [-] (1) -- (2);
\draw [-] (2) -- (3);
\draw [-] (1) -- (1b);
\end{tikzpicture}$};
\node at (9,-0) {$\begin{tikzpicture}[yscale=0.4,xscale=0.4]
 \node (m1) at (-1,0) [DB] {};
 \node (0) at (0,0) [DB] {};
 \node (1) at (1,0) [DBlue] {};
 \node (1b) at (1,0.75) [DB] {};
 \node (2) at (2,0) [DB] {};
 \node (3) at (3,0) [DB] {};
 \node (4) at (4,0) [DB] {};
\node (m1a) at (-1,-0.5) {$\scriptstyle 2$};
 \node (0a) at (0,-0.5) {$\scriptstyle 3$};
 \node (1a) at (1,-0.5) {$\scriptstyle 4$};
 \node (1ba) at (1.4,0.75) {$\scriptstyle 2$};
 \node (2a) at (2,-0.5) {$\scriptstyle 3$};
 \node (3a) at (3,-0.5) {$\scriptstyle 2$};
 \node (4a) at (4,-0.5) {$\scriptstyle 1$};
\draw [-] (m1) -- (0);
\draw [-] (0) -- (1);
\draw [-] (1) -- (2);
\draw [-] (2) -- (3);
\draw [-] (3) -- (4);
\draw [-] (1) -- (1b);
\end{tikzpicture}$};
\node at (13.25,0) {$\begin{tikzpicture}[yscale=0.4,xscale=0.4]
 \node (m1) at (-1,0) [DB] {};
 \node (0) at (0,0) [DB] {};
 \node (1) at (1,0) [DB] {};
 \node (1b) at (1,0.75) [DB] {};
 \node (2) at (2,0) [DBlue] {};
 \node (3) at (3,0) [DB] {};
 \node (4) at (4,0) [DB] {};
  \node (5) at (5,0) [DB] {};
\node (m1a) at (-1,-0.5) {$\scriptstyle 2$};
 \node (0a) at (0,-0.5) {$\scriptstyle 4$};
 \node (1a) at (1,-0.5) {$\scriptstyle 6$};
 \node (1ba) at (1.4,0.75) {$\scriptstyle 3$};
 \node (2a) at (2,-0.5) {$\scriptstyle 5$};
 \node (3a) at (3,-0.5) {$\scriptstyle 4$};
 \node (4a) at (4,-0.5) {$\scriptstyle 3$};
 \node (5a) at (5,-0.5) {$\scriptstyle 2$};
\draw [-] (m1) -- (0);
\draw [-] (0) -- (1);
\draw [-] (1) -- (2);
\draw [-] (2) -- (3);
\draw [-] (3) -- (4);
\draw [-] (4) -- (5);
\draw [-] (1) -- (1b);
\end{tikzpicture}$};
\node at (18,0) {$\begin{tikzpicture}[yscale=0.4,xscale=0.4]
 \node (m1) at (-1,0) [DB] {};
 \node (0) at (0,0) [DB] {};
 \node (1) at (1,0) [DBlue] {};
 \node (1b) at (1,0.75) [DB] {};
 \node (2) at (2,0) [DB] {};
 \node (3) at (3,0) [DB] {};
 \node (4) at (4,0) [DB] {};
  \node (5) at (5,0) [DB] {};
\node (m1a) at (-1,-0.5) {$\scriptstyle 2$};
 \node (0a) at (0,-0.5) {$\scriptstyle 4$};
 \node (1a) at (1,-0.5) {$\scriptstyle 6$};
 \node (1ba) at (1.4,0.75) {$\scriptstyle 3$};
 \node (2a) at (2,-0.5) {$\scriptstyle 5$};
 \node (3a) at (3,-0.5) {$\scriptstyle 4$};
 \node (4a) at (4,-0.5) {$\scriptstyle 3$};
 \node (5a) at (5,-0.5) {$\scriptstyle 2$};
\draw [-] (m1) -- (0);
\draw [-] (0) -- (1);
\draw [-] (1) -- (2);
\draw [-] (2) -- (3);
\draw [-] (3) -- (4);
\draw [-] (4) -- (5);
\draw [-] (1) -- (1b);
\end{tikzpicture}$};
\end{tikzpicture}
\end{array}
\]
should play a central role, where $e$ corresponds to the marked vertex.   Indeed, reading left to right, the $e\Uppi e$ corresponding to the above cases are precisely the algebras
\begin{equation}\label{eqn:6keyalgebras}
\arraycolsep=4pt
\begin{array}{cccccc}
\Upgamma_A&\Upgamma_{D,4}&\Upgamma_{E,6}&\Upgamma_{E,7}&\Upgamma_{E,8,5}&\Upgamma_{E,8,6}\\[2mm]
\mfrac{\mathbb{C}\langle x,y,z\rangle}
{\begin{array}{c}
x+y+z,\\
x, y, z
\end{array}}
&
\mfrac{\mathbb{C}\langle x,y,z\rangle}
{\begin{array}{c}
x+y+z,\\
x^2, y^2, z^2
\end{array}}
&
\mfrac{\mathbb{C}\langle x,y,z\rangle}
{\begin{array}{c}
x+y+z,\\
x^2, y^3, z^3
\end{array}}
&
\mfrac{\mathbb{C}\langle x,y,z\rangle}
{\begin{array}{c}
x+y+z,\\
x^2, y^3, z^4
\end{array}}
&
\mfrac{\mathbb{C}\langle x,y,z\rangle}
{\begin{array}{c}
x^2+y+z,\\
\mathrm{(see\,\, below)}
\end{array}}
&
\mfrac{\mathbb{C}\langle x,y,z\rangle}
{\begin{array}{c}
x+y+z,\\
x^2, y^3, z^5
\end{array}}
\end{array}
\end{equation}
where $\Upgamma_{E,8,5}$ can be presented as
\begin{equation}\label{eqn:theawkwardone}
\Upgamma_{E,8,5}=
\mfrac{\mathbb{C}\langle x,y,z\rangle}{\left(\begin{array}{c}
x^2+y+z,\\ x^4,\,\, z^2+xyx,\,\, yxy+y^2x+yx^3 \end{array}\right)}
\cong
\mfrac{\freering}{\left(\begin{array}{c}
x^2 + xy^2 + yxy + y^2x,\\
xyx - yxy^2 - y^2xy,\ y^4\end{array}\right)}
\end{equation} 
The reason \ref{def:alglength} is stated generally is that we cannot yet appeal to Katz--Morrison: it is only after classification that many equivalent characterisations of algebraic length can be made. For more details, see \ref{rem:KMok} and \cite{BW6}.
\end{remark}

For an ADE preprojective algebra $\Uppi$ and idempotent $e$, for $t\geq 1$ set $\Upgamma_t=e\Uppi e/\Gaj^t$ where $\Gaj$ is the radical of $e\Uppi e$, and consider the sequence 
\[
\filt\colonequals [\dim_\C\Upgamma_1,\dim_\C \Upgamma_2, \dim_\C \Upgamma_3,\hdots].
\]
\begin{lemma}\label{lem:allSurfaceData}
For any given \textnormal{ADE} preprojective algebra $\Uppi$ and idempotent $e$ with associated number $\ell$ in \eqref{eqn:allDynkin}, the following statements hold.
\begin{enumerate}
\item\label{lem:allSurfaceData1} If $\ell=1$, then $e\Uppi e$ is commutative.
\item\label{lem:allSurfaceData2} If $\ell=2$, then $\filt$ starts $[1,3,a,b]$, where $a\leq 5$ and $b\leq 7$.  
\item\label{lem:allSurfaceData3} If $\ell=3$, then $\filt$ starts $[ 1, 3, 6, 9, c ]$, where $c\in\{11,12\}$.
\item\label{lem:allSurfaceData4} If $\ell=4$, then $\filt$ starts $[ 1, 3, 6, 10, 14, 18, d]$, where $d\in\{21,22\}$. 
\item\label{lem:allSurfaceData5} If $\ell=5$, then $\filt$ equals $[ 1, 3, 6, 10, 15, 20, 25, 30, 34, 37, 39, 40 ]$.
\item\label{lem:allSurfaceData6} If $\ell=6$, then $\filt$ equals $[ 1, 3, 6, 10, 15, 21, 27, 33, 39, 45, 50, 54, 57, 59, 60 ]$.
\end{enumerate}
\end{lemma}
\begin{proof}
(1) This is essentially just \cite[8.2]{BW3}, the point being that $e\Uppi e$ can be viewed as a factor of an endomorphism ring of a rank one CM module over a Kleinian singularity, which is always commutative.\\
(2) Although this algebraic statement can be proved combinatorially using the techniques below, here we give a direct (albeit outrageous) geometric proof, since this avoids the proliferation of cases.   Consider any vertex $v$ in some ADE $\Updelta$ which is labelled $2$ in \eqref{eqn:allDynkin}, and consider the corresponding $e\Uppi e$.  Using \cite[p366--367]{Pinkham} as explained in \cite[\S1]{KM}, we may build a smooth $3$-fold flopping contraction $\upphi\colon Y\to\Spec R$ such that there exists $c\in R$ for which under the base change
\[
\begin{tikzpicture}
\node (A) at (0,0) {$S$};
\node (B) at (2,0) {$Y$};
\node (a) at (0,-1) {$\Spec(R/c)$};
\node (b) at (2,-1) {$\Spec R $};
\draw[->] (A)--(B);
\draw[->] (a)--(b);
\draw[->] (A)--node[left]{$\scriptstyle \upvarphi$}(a);
\draw[->] (B)--node[right]{$\scriptstyle \upphi$}(b);
\end{tikzpicture}
\]
the surface $R/c$ is Kleinian of type $\Updelta$, and $S$ is its partial resolution where we have only extracted the curve corresponding to $v$. It is important that to emphasise that $c$ need not be generic.  Regardless, since $\upphi$ is a smooth flopping contraction, by \cite{DW1} there is an associated contraction algebra $\Jac(g)$ for some $g\in\ringtwo$. Furthermore by \cite[(3.C)]{DW1}, which does not require $c$ to be generic, necessarily $\Jac(g)/c \cong e\Uppi e$.  

The main result of \cite{KM} asserts that $\upphi$ is generically a Type $D$ flop.  But contraction algebras of those are now known by \cite[8.9]{BW3}, so by \cite[4.11]{BW3} it follows that $\Jac(g)$ grows at worst as $[1,3,5,7,\hdots]$.  Thus its quotient $\Jac(g)/c\cong e\Uppi e$ grows at worst as $[1,3,5,7,\hdots]$. But since $\Jac(g)$ is a contraction algebra, $c$ is in $\n^2$ by \cite[8.6]{BW2}, thus the basis of $\Jac(g)/c$ must contain $1,x,y$ and so $\dim_\C\Upgamma_2\geq 3$.  Combining, it follows that the first two numbers in $\filt$ are precisely $1,3$, and the next two are at worst $5,7$.  \\
(4) There are only three times the number $4$ appears in \eqref{eqn:allDynkin}, so we simply exhaust all cases.  The idempotents $\Eseven{B}{B}{W}{B}{B}{B}{B}$ and $\Eeight{B}{B}{B}{B}{B}{W}{B}{B}$ 
%$E_7$ middle and $E_8$ long 
both have $\filt=[ 1, 3, 6, 10, 14, 18, 21, 23, 24 ]$, whilst  
%$E_8$ short 
the idempotent $\Eeight{B}{W}{B}{B}{B}{B}{B}{B}$ has 
 $\filt=[ 1, 3, 6, 10, 14, 18, 22, 25, 27, 28 ]$.  These facts follow immediately from the presentations of the algebras $e\Uppi e$ given in \ref{rem:ell4presentations} below.\\
(3) There are only five times the number $3$ appears in \eqref{eqn:allDynkin}, so here again it is easier to just exhaust all cases.  The idempotents $\Esix{B}{B}{W}{B}{B}{B}$, $\Eseven{B}{W}{B}{B}{B}{B}{B}$, and $\Eeight{B}{B}{B}{B}{B}{B}{W}{B}$
%$E_6$ middle, $E7$ short, $E_8$ long 
all have $\filt=[ 1, 3, 6, 9, 11, 12 ]$.  Again, the first of these follows the known presentation of $e\Uppi e$ in  e.g.\ \cite[Theorem 1]{Melit} (see also \cite[7.2]{BW3}), and the others follow similarly, using knitting and the method e.g.\ in \cite[7.4]{DW5}. 
In a similar way, the idempotent $\Eseven{B}{B}{B}{B}{B}{W}{B}$ has $\filt=[ 1, 3, 6, 9, 12, 14, 15 ]$,  whilst the idempotent $\Eeight{B}{B}{B}{W}{B}{B}{B}{B}$ has $\filt=[ 1, 3, 6, 9, 12, 14, 15, 16 ]$.\\
(5) and (6) are precisely the idempotents $\Eeight{B}{B}{B}{B}{W}{B}{B}{B}$ and $\Eeight{B}{B}{W}{B}{B}{B}{B}{B}$ respectively. The result follows a consequence of the presentation in \eqref{eqn:theawkwardone}, and the last presentation in \eqref{eqn:6keyalgebras}.
\end{proof}

\begin{remark}\label{rem:ell4presentations}
Below we will need just slightly more information on the case $\ell=4$.  In each of the three cases in \ref{lem:allSurfaceData}\eqref{lem:allSurfaceData4}, the algebra $e\Uppi e$ can be presented as a factor of $\polyringtwo$ by the following ideals
\[
\begin{tabular}{cl}
$\Eseven{B}{B}{W}{B}{B}{B}{B}$ & $\medmath{(x^2,y^3,(x+y)^4)}$\\

$\Eeight{B}{B}{B}{B}{B}{W}{B}{B}$ &
$\medmath{(x^3, y^4, y^2x^2, x^2yx + xyx^2 + y^2, xyxy + yx^2y + yxyx, y^2xyx^2, xy^2 - y^2x)}$\\

$\Eeight{B}{W}{B}{B}{B}{B}{B}{B}$ & $\medmath{(xy^3 - y^3x, y^6x, xyxy + xy^2x + yxyx + y^3, x^2, yxy^2 + y^2xy + y^3x, y^4xyx + y^6, y^8)}$
\end{tabular}
\]
where the first is $\Upgamma_{E,7}$ in \eqref{eqn:6keyalgebras}, and the others are obtained using the same technique; we have made no attempt to simplify the presentations.  Whilst it turns out that the three algebras presented above are mutually non-isomorphic, we will not require this, however we will require the claim that after factoring by the fourth power of the radical, they become isomorphic. Indeed, removing all terms of degree $\geq 4$ in the above, it follows that in each case the algebra $\Upgamma_4$ can be presented as a factor of $\polyringtwo$ by the following ideals
\[
\begin{tabular}{cl}
$\Eseven{B}{B}{W}{B}{B}{B}{B}$ & $\medmath{(x^2,y^3,\m^4)}$\\

$\Eeight{B}{B}{B}{B}{B}{W}{B}{B}$ &
$\medmath{(x^3, y^2, xy^2 - y^2x,\m^4)=(x^3, y^2,\m^4)}$\\

$\Eeight{B}{W}{B}{B}{B}{B}{B}{B}$ & $\medmath{(y^3, x^2, \m^4)}$
\end{tabular}
\]
These are all clearly isomorphic to $\polyringtwo\big/(x^2, y^3,\m^4)$.
\end{remark}

\subsection{Summary of Strategy}
The main strategy is to leverage the simple observation (in \ref{lem:Zmodk}) that any central element remains central modulo any power of the radical.

Given $f\in\ringtwo$, recall the notation $\fJ$, $A_k$ and $\uppi_k$ from \ref{not!jacqt}.

\begin{remark}\label{rem:overallstrategy}
Suppose that there exists some $z\in\scrZ(\Jac(f))\cap\fJ$ with $\Jac(f)/\lcl z\rcl\cong e\Uppi e$, as in \ref{def:alglength}.  As above \ref{lem:allSurfaceData}, for any $k\geq 1$ set $\Upgamma_k \colonequals e\Uppi e/\Gaj^k$, then 
\begin{equation}\label{eq!Akiso}
\Upgamma_k\cong \big(\Jac(f)/\lcl z\rcl\big) \big/ \fJ^k \cong A_k\big/\lcl \uppi_k(z)\rcl
\end{equation}
where $\uppi_k(z)\in\scrZ(A_k)$.  That is, for all $k\geq 1$ the central element $\uppi_k(z)$ slices the finite dimensional $A_k$ to obtain the `surface data' $\Upgamma_k$.

The overall strategy to prove that there is no isomorphism $\Jac(f)/\lcl z\rcl\cong e\Uppi e$ is to contradict \eqref{eq!Akiso} for some $k$, using knowledge of $\uppi_k(z)$ from \ref{lem:keycentreallcases}.  Since by \ref{lem:allSurfaceData} the dimension of the left hand side $\Upgamma_k$ of \eqref{eq!Akiso} is known in all cases, a simple dimension count usually shows that $A_k\big/\lcl \uppi_k(z)\rcl$ is too big, and so the isomorphism in \eqref{eq!Akiso} cannot exist.  
\end{remark}

We illustrate \ref{rem:overallstrategy} first in the following, which generalises \cite[8.8]{BW3} from $D_4$ to cover all length two vertices in all ADE diagrams.

\begin{prop}\label{prop:Never1or2}
If $f\cong x^3 + \bigO_4$ has an algebraic length $\ell$, then $\ell\geq 3$.
\end{prop}
\begin{proof}
By \ref{lem:keycentreallcases}\eqref{lem:keycentreallcases1} $\uppi_3(z)=0$ and so $A_3\big/\lcl \uppi_3(z)\rcl =  A_3 =\polyringtwo/(x^2, \m^3)$, which is noncommutative of dimension $6$.  If $\ell=1$ then by \ref{lem:allSurfaceData}\eqref{lem:allSurfaceData1} $e\Uppi e$ is commutative, thus so is its factor $\Upgamma_3$, and so an isomorphism \eqref{eq!Akiso} with $k=3$ is not possible. Thus $\ell\neq 1$.  If $\ell=2$ then by \ref{lem:allSurfaceData}\eqref{lem:allSurfaceData2} the left hand side of \eqref{eq!Akiso} with $k=3$ has dimension $\le5$, but the right hand side has dimension $6$; thus again an isomorphism \eqref{eq!Akiso} with $k=3$ is not possible, which rules out $\ell=2$.
\end{proof}

\subsection{The Possible Algebraic Lengths}\label{sec:possalglengths}

\begin{prop}\label{prop:orderfournotalg}
If $f\in\ringtwo$ with $\ord(f)\geq 4$, then $f$ is purely algebraic.
\end{prop}
\begin{proof}
Suppose that $f$ has algebraic length $\ell$, and consider  \eqref{eq!Akiso}  with $k=3$, which by \ref{lem:keycentreallcases}\eqref{lem:keycentreallcases10} reads $\Upgamma_3\cong A_3/(\al(xy-yx))$ for some $\al\in\C$.

If $\al=0$, then $A_3=\polyringtwo\big/\m^3$. Now $\Jac(f)$ is not commutative since its factor $A_3$ is not, so by \ref{lem:allSurfaceData}\eqref{lem:allSurfaceData1} $\ell\geq 2$.  Clearly $\dim_\C A_3=7$, whereas inspecting \ref{lem:allSurfaceData}\eqref{lem:allSurfaceData2}--\eqref{lem:allSurfaceData6} shows that  $\dim_\C\Upgamma_3\leq 6$, thus an isomorphism $\Upgamma_3\cong A_3$ is impossible and $f$ is purely algebraic.

If $\al\ne0$, $A_3/\lcl\uppi(z)\rcl\cong\C[x,y]/\m^3$ is commutative of dimension~6; in particular, $\ell\ne2$ by \ref{lem:allSurfaceData}\eqref{lem:allSurfaceData2}.
In each case $\ell\ge3$, the algebra $\Upgamma_3$ is not commutative.
Indeed \cite[3.16]{DW1} shows this for the Katz--Morrison cases listed in \ref{def!6algebras} (which includes all cases for $\ell=5,6$), while the case $\ell=4$ follows from \ref{rem:ell4presentations}, and the case $\ell=3$ is very similar.
\end{proof}

\begin{prop}[Case $\caseA$]\label{prop!x3y5centre}
If $f\cong x^3 + \scrO_5$, then $\dim_\C A_4\big/\lcl \uppi_4(z)\rcl = 11$
for any $z\in \scrZ(\Jac(f))\cap\fJ$, and as a consequence $f$ is purely algebraic.
\end{prop}

\begin{proof}
By \ref{lem:keycentreallcases}\eqref{lem:keycentreallcases2} $\uppi_4(z)=0$, thus $A_4\big/\lcl \uppi_4(z)\rcl = A_4$.  But clearly, given the form of $f$, $A_4=\polyringtwo/(x^2,\m^4)$, which evidently has dimension $11$.  The first claim follows.  

If $f$ has algebraic length $\ell$, then $\ell\geq 3$ by \ref{prop:Never1or2}.  But for these $\ell$, the data within \ref{lem:allSurfaceData}\eqref{lem:allSurfaceData3}--\eqref{lem:allSurfaceData6} shows that $\dim_\C\Upgamma_4\leq 10$.  But this contradicts \eqref{eq!Akiso} with $k=4$, since the left hand side has dimension $\leq 10$, whereas the right hand side has dimension $11$.  It follows that such an isomorphism \eqref{eq!Akiso} can never exist, and so $f$ is purely algebraic.
\end{proof}

\begin{prop}[Case $E_6$]\label{prop!caseE6only3}
If $f\cong x^3 + xy^3 + \upvarepsilon y^4+\scrO_5$ has algebraic length~$\ell$, then $\ell=3$.
\end{prop}
\begin{proof}
Necessarily $\ell\geq 3$ by \ref{prop:Never1or2}.  Now by \ref{prop!caseE6}, $f\cong x^3 + xy^3 + \upvarepsilon y^4 + \uplambda y^5 + \scrO_6$ where $\uplambda\in\C$.
From this, by standard bases as in \ref{rk:sbn^k}, we have $\dim_\C A_4=10$ and $\dim_\C A_5=14$. If $\ell=5,6$ then $\dim_\C\Upgamma_5=15$ by \ref{lem:allSurfaceData}\eqref{lem:allSurfaceData5}--\eqref{lem:allSurfaceData6}, but this would contradict \eqref{eq!Akiso} with $k=5$ since the left hand side of \eqref{eq!Akiso} has dimension $15$, but the right hand side is a factor of $A_5$, and thus has dimension $\leq 14$.  

The only remaining options are $\ell=3,4$, so consider the isomorphism \eqref{eq!Akiso} with $k=4$. We will rule out $\ell=4$ in both the following subcases. Consider $z\in\scrZ(\Jac(f)$ as in \ref{def:alglength}, the definition of~$\ell$.\\
\emph{Case $\uppi_4(z)\neq 0$}. Since  $\dim_\C A_4=10$, necessarily $\dim_\C A_4/\lcl \uppi_4(z)\rcl \leq 9$.  If $\ell=4$ then due to the $10$ in \ref{lem:allSurfaceData}\eqref{lem:allSurfaceData4}, the left and right sides of  \eqref{eq!Akiso} (with $k=4$) would have different dimensions.  Hence $\ell\neq 4$.\\
\emph{Case $\uppi_4(z)=0$}. In this case, \eqref{eq!Akiso} with $k=4$ reduces to the statement
\begin{equation}
\Upgamma_4\cong A_4= \frac{\polyringtwo}{(3x^2+y^3,  xy^2 + yxy + y^2x + 4\upvarepsilon y^3,\m^4)}.  \label{eqn:A4notl4} 
\end{equation}
If $\ell=4$, then by \ref{rem:ell4presentations} the left hand side is always isomorphic to $\polyringtwo\big/ (x^2,y^3,\m^4)$.  Such an isomorphism \eqref{eqn:A4notl4} cannot exist: the left-hand side $\Upgamma_4$ has a nilpotent $x$ of order $2$, but the right-hand side does not, as we prove next.
\begin{claim}\label{v2claim}
There is no element $v\in\scrJ\setminus\scrJ^2$ in $A_4$ with $v^2=0$.
\end{claim}
To verify the claim, suppose that $v=\al_1 x+ \al_2 y+ \al_3xy + \al_4yx + \al_5y^2 + \bigO_3\in A_4$ satisfies $v^2=0$ for some $\al_i\in\C$, using the standard monomial $\C$-basis to express~$v$ (cf.~\ref{rk:sbn^k}). First note that applying the reduction $x^2 = -\tfrac{1}3y^3$,
\[
v^2 = \al_1^2x^2 + \al_1\al_2(xy+yx) + \al_2^2 y^2 + \bigO_3 
= \al_1\al_2(xy+yx) + \al_2^2 y^2 + \bigO_3
\]
so that $\al_2=0$. Then using the same $x^2$ reduction to push terms into $\m^4$,
\begin{eqnarray*}
v^2 &=& -\tfrac{1}3\al_1^2 y^3 + \al_1(\al_3+\al_4) xyx + \al_1\al_5(xy^2+y^2x) + \bigO_4 \nonumber\\
&=& \al_1(\al_3+\al_4) xyx - \al_1\al_5 yxy -\tfrac{1}3\al_1^2 y^3 + \bigO_4 
\end{eqnarray*}
so that $\al_1=0$. Thus $v\notin \scrJ\setminus\scrJ^2$ and the claim holds.

In both cases $\ell\neq4$.  Hence, if $f$ has an algebraic length, necessarily $\ell=3$.
\end{proof}

\begin{prop}[Case $\caseB$]\label{prop!caseBfirstlength}
If $f=x^3+y^4+\scrO_5$ has an algebraic length $\ell$, then $\ell=4$ or $\ell=6$.
\end{prop}
\begin{proof}
Necessarily $\ell\geq 3$ by \ref{prop:Never1or2}.  To rule out $\ell=3,5$, consider the isomorphism \eqref{eq!Akiso}  with $k=4$, which reads $\Upgamma_4\cong A_4/\lcl \uppi_4(z)\rcl$. By  \ref{lem:keycentreallcases}\eqref{lem:keycentreallcases2} $\uppi_4(z)=0$, so \eqref{eq!Akiso} becomes $\Upgamma_4\cong A_4$.

Now $A_4=\polyringtwo\big/ (x^2,y^3,\m^4)$ has dimension $10$.  Since by  \ref{lem:allSurfaceData}\eqref{lem:allSurfaceData3} $\dim_\C\Upgamma_4\leq 9$ if $\ell=3$, it follows that $\ell= 3$ would contradict $\Upgamma_4\cong A_4$.   Hence $\ell\neq 3$. If $\ell=5$, then $A_4$ must be isomorphic to $\Upgamma_4$ where $\Upgamma$ is the $\ell=5$ algebra in \eqref{eqn:theawkwardone}.   But whilst $A_4$ admits an element $v\in\scrJ\setminus \scrJ^2$ such that $v^2=0$, $\Upgamma_4$ for the $\ell=5$ algebra in \eqref{eqn:theawkwardone} does not, by a similar calculation to Claim~\ref{v2claim} above.  Hence $\ell\neq 5$. \end{proof}

\begin{prop}[{Within case $\caseB$}]\label{prop:withinB}
The following statements hold.
\begin{enumerate}
\item \textnormal{(Case $E_7$)} If $f\cong x^3 + y^4 + xyxy^2 + \scrO_6$ has algebraic length $\ell$, then $\ell=4$.
\item \textnormal{(Case $E_{8,6}^*$)} If $f\cong x^3 + y^4 +\scrO_6$ has algebraic length $\ell$, then $\ell=6$.
\end{enumerate}
\end{prop}
\begin{proof}
 (1) Note that $A_5=\polyringtwo\big/(\dcyc_xf,\dcyc_yf,\m^5)$, which has dimension 15.   By \ref{prop!caseBfirstlength} either $\ell=4$ or $\ell=6$, so consider the isomorphism \eqref{eq!Akiso} with $k=5$. We will rule out $\ell=6$ in both the following subcases. \\
\emph{Case $\uppi_5(z)\neq 0$}. Since  $\dim_\C A_5=15$, necessarily $\dim_\C A_5/\lcl \uppi_5(z)\rcl \leq 14$.  If $\ell=6$ then due to the $15$ in \ref{lem:allSurfaceData}\eqref{lem:allSurfaceData6}, the left and right sides of  \eqref{eq!Akiso} (with $k=5$) would have different dimensions.  Hence $\ell\neq 6$.\\
\emph{Case $\uppi_5(z)=0$}. In this case, \eqref{eq!Akiso} with $k=5$ reduces to the statement
\begin{equation}
\Upgamma_5\cong A_5= \frac{\polyringtwo}{(3x^2+yxy^2 +y^2xy,  4y^3 + xy^2x +xyxy+yxyx,\m^5)}.\label{eqn:A5notl6}
\end{equation}
If $\ell=6$, then by \ref{def!6algebras} the left hand side is isomorphic to $\polyringtwo\big/ (x^2,y^3,\m^5)$.  Such an isomorphism \eqref{eqn:A5notl6} cannot exist, as we claim that the right hand side has no element $v\in\scrJ\setminus\scrJ^2$ such that $v^2=0$, whereas the left hand side clearly does.  

The claim follows as before. Set $\sfg_1 = x^2 + \tfrac{1}3(yxy^2+y^2xy)\in I_5$. Suppose that $v=\al_1 x+ \al_2 y+ \al_3xy + \al_4yx + \al_5y^2 + \bigO_3\in A_5$ for some $\al_i\in\C$ has $v\in\scrJ\setminus \scrJ^2$ and $v^2=0$.  Expressing $v^2$ in the $\C$-basis, we 
have $v^2-a^2\sfg_1 = \al_1\al_2(xy+yx) + \al_2^2y^2 + \bigO_3$, so that $\al_2=0$, after which
have $v^2-\al_1^2\sfg_1 - \al_1\al_3\sfg_1y - \al_1\al_4\sfg_1$ is equal to
\[
(\al_1\al_3 + \al_1\al_4)xyx + \al_1\al_5(xy^2 + y^2x) + \al_3^2xyxy + \hot
\]
thus necessarily also $\al_1=0$.  

In both cases $\ell\neq6$.  Hence, if $f$ has an algebraic length $\ell$, then $\ell=4$.

\noindent
(2)  Again by \ref{prop!caseBfirstlength} the only options are $\ell=4,6$, and here we rule out $\ell=6$ by considering  the isomorphism \eqref{eq!Akiso} with $k=5$.  By \ref{lem:keycentreallcases}\eqref{lem:keycentreallcases3}  $\uppi_5(z)=0$, and so $\dim_\C A_5\big/\lcl \uppi_5(z)\rcl=\dim_\C A_5=15$.  Given this, if $\ell=4$ then due to the $14$ in \ref{lem:allSurfaceData}\eqref{lem:allSurfaceData4}, the left and right sides of \eqref{eq!Akiso} (with $k=5$) would have different dimensions.  It thus follows that $\ell=6$.
\end{proof}

\begin{prop}[Case $\caseC$]\label{prop!caseCfirstlength}
If $f \cong x^3 + xyxy + \scrO_5$ has algebraic length~$\ell$, then $\ell=5$.
\end{prop}
\begin{proof}
By definition $A_4=\polyringtwo\big/(x^2+\tfrac{2}{3}yxy, xyx, \m^4)$, and the ideal has standard basis consisting of the two truncated derivatives $x^2 + \tfrac{2}3yxy$ and $xyx$ together with all those monomials of degree~4 not divisible by $x^2$ or $xyx$, and $\dim_\C A_4=10$ follows.
The fact that there is no $v\in A_4$, with $v\in\scrJ\setminus \scrJ^2$ follows easily as before.

The proof now concludes easily: if an algebraic length exists, necessarily $\ell\geq 3$ by \ref{prop:Never1or2}.  The isomorphism \eqref{eq!Akiso} reads $\Upgamma_4\cong A_4/\lcl \uppi_4(z)\rcl$, which by \ref{lem:keycentreallcases}\eqref{lem:keycentreallcases2} equals $A_4$ and thus has dimension $10$; inspecting \ref{lem:allSurfaceData}\eqref{lem:allSurfaceData3} then rules out $\ell=3$.  Now combining \ref{rem:ell4presentations} with the presentation of $\Upgamma_{E,8,6}$ in \eqref{eqn:6keyalgebras}, all the possible $\Upgamma_4$ for either $\ell=4$ or $\ell=6$ have an element $v\in\scrJ\setminus\scrJ^2$ with $v^2=0$. Hence the above claim rules out $\ell=4,6$. 
\end{proof}

\begin{prop}[{Within case $\caseC$}]\label{prop:withingC}
If $f\cong x^3 + xyxy + \upbeta\,xy^4 +\scrO_6$ for some $\upbeta\in\C$, then $f$ is purely algebraic.
\end{prop}
\begin{proof}
Note first that $\dim_\C A_5=16$, regardless of $\upbeta$: the leading terms of the standard basis in degree $\le4$ are $x^2$ and~$xyx$.  If $f$ has an algebraic length, by \ref{prop!caseCfirstlength} necessarily $\ell=5$.  We thus suppose $\ell=5$ and  consider the isomorphism \eqref{eq!Akiso}  with $k=5$, which reads $\Upgamma_5\cong A_5\big/ \lcl \uppi_5(z)\rcl$.  By \ref{lem:keycentreallcases}\eqref{lem:keycentreallcases4} $\uppi_5(z)=\uplambda xy^2x$ for some $\uplambda\in\C$.   When $\uplambda=0$ the contradiction is then immediate, since the left hand side $\Upgamma_5$ has dimension $15$ by \ref{lem:allSurfaceData}\eqref{lem:allSurfaceData5}, whereas the right hand side has dimension $16$.  To obtain a contradiction when $\uplambda\neq 0$, observe then that  $A_5\big/\lcl \uppi_5(z)\rcl  =\polyringtwo\big/(\dcyc_xf,\,\dcyc_yf, \,xy^2x,\, \m^5)$.  This clearly abelianises to an algebra of dimension of $8$ (with $x^2$ and $xy^3$ pushed into higher degree, or zero), whereas the presentation of $\Upgamma_5$ in \eqref{eqn:6keyalgebras} (for $\ell=5$) abelianises to an algebra of dimension $7$ (with also $y^4$ set to zero).  Thus, regardless of $\uplambda$, an isomorphism cannot exist, and so  $f$ must be purely algebraic.
\end{proof}

\subsection{Summary}

The following is the main result of this section. The statement is slightly subtle: it does not claim that every element listed in the list has the prescribed algebraic length, it only claims that if $\Jac(f)$ admits a central element with associated number $\ell$, then $f\cong g$ where $g$ is some element in the set of possible $g$ stated.

\begin{thm}\label{thm:ReductionPrecise}
If $f\in\ringtwo$ has algebraic length $\ell$, then $\Jac(f)\cong \Jac(g)$ where $g$ is some element in the following list, where all $\upvarepsilon_i\in\{0,1\}$.
\[
\renewcommand{\arraystretch}{1.2}
\medmath{\begin{array}{c@{\hspace{20pt}}l@{\hspace{1pt}}l@{\hspace{15pt}}l}
1 & x^2 + \upvarepsilon_1 y^n && n\geq 2 \\
2 & x^2y+\upvarepsilon_2 y^{2n}+\upvarepsilon_3 y^{2m-1} &&  2n-1\geq m\geq 2 \\
3 & x^3 + xy^3 + \upvarepsilon_1 y^4 + \upalpha y^5& + \,\scrO_6 & \upalpha\in\mathbb{C} \\
4 & x^3 + y^4 + xyxy^2  & +\, \scrO_6\\
5 & x^3 + xyxy + y^5 & +\, \scrO_6\\
6 & x^3 + y^4 & +\, \scrO_6\\
%\bottomrule
\end{array}}
\renewcommand{\arraystretch}{1}
\]
In particular, the following statements hold.
\begin{enumerate}
\item The algebraic length is well defined.  That is, if there exist $z_1,z_2\in\scrZ(\Jac(f))$ such that $\Jac(f)/\lcl z_1\rcl \cong e_1\Uppi_1 e_1$ and $\Jac(f)/\lcl z_2\rcl \cong e_2\Uppi_2e_2$ for some \textnormal{ADE} preprojective algebras $\Uppi_1,\Uppi_2$ and some primitive idempotents $e_1$ and $e_2$ with associated numbers $\ell_1$ and $\ell_2$ in \eqref{eqn:allDynkin}, then $\ell_1=\ell_2$.
\item If $f$ and $h$ have different algebraic lengths, then $\Jac(f)\ncong\Jac(h)$. 
\end{enumerate}
\end{thm}
\begin{proof}
If $f$ has a linear term then $\Jac(f)=0$, which does not have algebraic length. Hence we can assume that $f\in\ringtwo_{\geq 2}$.

Since $f$ has algebraic length $\ell$, by \ref{prop:orderfournotalg} $f$ must contain some term of degree $\leq 3$. In the notation of \S\ref{subsec:RadicalsandDim}, $\corank_2(f)\leq 3$. There are three options; the key is that the first two are somewhat implicit in \cite{BW3}, whilst the third relies on all of \S\ref{sec:ParseOne} and \S\ref{sec:AlgLength}. 
\setlength{\leftmargini}{2em}
\begin{enumerate}
\item[(A)] $\corank_2(f)\leq 1$, which by \cite[1.2, above 4.11]{BW3} holds if and only if $f\cong x^2 + \upvarepsilon y^n$.  It is clear by inspection that these Jacobi algebras have a central element $y$ with $\ell=1$. Further, a central element with $\ell\geq 2$ is not possible since then $A_2\big/ \lcl \uppi_2(z)\rcl\cong \Upgamma_2$, but the left hand side is a factor of $\C\llsq y\rrsq\big/(y^2)$, and thus has dimension $\leq 2$, but the right hand side has $\dim_\C\Upgamma_2=3$ by inspecting \ref{lem:allSurfaceData}\eqref{lem:allSurfaceData2}--\eqref{lem:allSurfaceData6}.   Thus every central element of Jacobi algebras within this class can only have $\ell=1$.

\item[(B)] $\corank_2(f)=2$, which by  \cite[1.3]{BW3} (with $x$ and $y$ swapped) holds if and only if $f\cong x^2y+\upvarepsilon_2y^{2n}+\upvarepsilon_3 y^{2m-1}$.  By \cite[1.4]{BW3} there is a known central element with $\ell=2$.  By standard bases (see e.g.\ \cite[4.11]{BW3}) these $\Jac(f)$ grow at worst as $[1,3,5,\hdots]$. Thus any central quotient $e\Uppi e$ grows at worst as $[1,3,5,\hdots]$, from which \ref{lem:allSurfaceData}\eqref{lem:allSurfaceData3}--\eqref{lem:allSurfaceData6} shows that $\ell\geq 3$ is not possible.  A central quotient with $\ell=1$ is also not possible, but for different reasons:  since the Jacobi algebras within this class are known contraction algebras, we may apply \cite[8.6]{BW3} to see that every central element has no linear terms. Thus $\Upgamma_2\cong A_2\big/ \lcl \uppi_2(z)\rcl = A_2$, but the left hand side is commutative for $\ell=1$ (by \ref{lem:allSurfaceData}\eqref{lem:allSurfaceData1}), but clearly $A_2$ is not commutative.  Thus every central element of Jacobi algebras within this class can only have $\ell=2$.

\item[(C)] $\corank_2(f)=3$, which by  \cite[below 1.3]{BW3} holds if and only if $f\cong x^3+\scrO_4$.  By \ref{thm:uptodeg4} this family splits into Cases $E_6$, $\caseA$, $\caseB$ and $\caseC$.  The assumption that $f$ has algebraic length $\ell$, together with \ref{prop!x3y5centre}, rules out Case $\caseA$.  The only options remaining are the following:
\begin{enumerate}
\item Case $E_6$, which by \ref{prop!caseE6only3} can only admit a central element with $\ell=3$.
\item Case $\caseB$, first subcase (of two) in \ref{prop:CaseBsplit}.  By \ref{prop:withinB}, this family ($E_7$) can only admit a central element with $\ell=4$.
\item Case $\caseB$, final subcase in \ref{prop:CaseBsplit}.  By \ref{prop:withinB}, this family ($E_{8,6}^*$) can only admit a central element with $\ell=6$.

\item Case $\caseC$, which by \ref{prop!caseCfirstlength} can only admit a central element with $\ell=5$.  Then \ref{prop:withingC} further refines this, by showing that within Case $\caseC$, only those of the form $x^3 + xyxy + y^5+\scrO_6$ ($E_{8,5}$) can possibly have $\ell=5$, given that the others are purely algebraic.

\end{enumerate}
\end{enumerate}
Given these facts, the proof is routine.  The key is that since $\corank_2(f)\leq 3$ the element $f$ must appear in at least one of (A), (B), (C)(a)(b)(c)(d),  each possible number $\ell$ appears in only one of these options, and within each option, the number is unique.  

For example, if $f$ has algebraic length $1$, then as $f$ must appear somewhere, the only option is that $f$ belongs to (A).  This populates the first line in the list.  Similarly, if $f$ has algebraic length $2$, it must appear somewhere, and (B) is the only option.  This populates the second line.  The proof of the first statement, namely populating the list, continues in this manner.  

\noindent 
(1) This statement now follows using uniqueness within each option. For example, if $f$ admits a central element with $\ell=3$, then by above $f$ must belong to (C)(a). Further, every central element of Jacobi algebras within this class can only have $\ell=3$.

\noindent
(2) This follows immediately from (1).
\end{proof}

\begin{remark}
The implication in \ref{thm:ReductionPrecise} is strict: it is not true that every element in bottom line ($x^3 + y^4+\, \scrO_6$, case $\ell=6$) admits a desired central element.  Only after making the $\ell=6$ family smaller, and after significant extra work, will \ref{thm:ReductionPrecise} finally be upgraded to an if and only if statement; this upgrade does not appear until right at the end, in \ref{rem:converseAlgLength}.  See also \ref{rem:AlmostconverseAlgLength}.  
\end{remark}

%--------------------------------------------------------------------------------------------------------
\section{Funnelling the Normal Forms}\label{sec:FunnellingOne}

For each $f$ in the final four rows of the list in \ref{thm:ReductionPrecise}, this section uses the low degree terms in each case to simplify, or `funnel', the monomials which can appear within $\scrO_6$. Section~\ref{subsec:E6funnel} contains the definitive answer for length three, to arbitrarily high degree. However, to avoid duplication in \S\ref{sec:NecklacIFY} later, in \ref{thm:Length456HalfWay} we only funnel lengths $\ell=4,5,6$ up to small degree, since this is all that is required to determine finite dimensionality of the Jacobi algebra in \S\ref{sec:FD}.   The full funnel for the cases $\ell=4,5,6$, which is needed to construct and classify flops, will be deduced as a consequence of \S\ref{sec:NecklacIFY}.

\subsection{Length Three}\label{subsec:E6funnel}
Set $f_{\rm stub} =  x^3+xy^3+\upvarepsilon y^4+\uplambda y^5$, where $\upvarepsilon\in\{0,1\}$, and throughout this subsection consider $f\cong f_{\rm stub}+\scrO_6$.   The aim is to simplify $f$ into the form $\Jac(f)\cong \Jac(x^3  + xy^3 + p)$ for some monic $p\in\mathbb{C}\llsq y\rrsq_{\geq 4}$, using \ref{main chase lemma}.

Below, Tschirnhausen \ref{tschirnhausen} will be used to clear terms containing $x^2$, and up to cyclic equivalence every monomial not containing $x^2$ is either of the form $xy^{a_1}\hdots xy^{a_t}$ or $y^a$.  Thus to achieve the aim of this section will require all terms of the form $xy^{a_1}\hdots xy^{a_t}$ to be cleared into higher degree, in a way compatible with \ref{main chase lemma}.  

\begin{notation}\label{tonotation}
Set $f=f_{\rm stub} +g_n+\scrO_{n+1}$ with $n\geq 6$, and where $\deg(g_n)=n$.
\begin{enumerate}
\item\label{tonotation1} For degree $n$ monomials $\sfm_1, \sfm_2$, write $\mathsf{m}_1\to \mathsf{m}_2$ if for all $\upalpha\in\C$ there exists a depth $\geq n-5$ unitriangular automorphism $\upphi$ such that
\begin{equation}
\upphi(f) \sim_{n-1} f_{\rm stub}  + ( g_n-\upalpha\mathsf{m}_1+\upmu\upalpha\mathsf{m}_2) +\scrO_{n+1}\label{eqn:arrownotation}
\end{equation}
for some $\upmu\neq0$.  That is, after applying $\upphi$ to $f$ and cyclically permuting in degrees $\geq n-1$ we obtain the right hand side of \eqref{eqn:arrownotation}, which has removed $\upalpha\mathsf{m}_1$ from $g_n$ at the cost of adding in $\upmu\upalpha\mathsf{m}_2$.  Very occasionally (e.g.\ in \S\ref{subsec:L5fullfunnel}) it will be necessary to remember the coefficient $\upmu$, and if so we will write $\mathsf{m}_1\to \upmu\mathsf{m}_2$.
\item\label{tonotation2} For implicit use later, note that if we add arbitrary degree $6,\hdots,n-1$ terms to both $f$ on the left hand side and to $f_{\stub}$ on the right hand side this does not effect the conclusion of \eqref{eqn:arrownotation}, since the depth of $\upphi$ is $\geq n-5$.  
%CALIBRATION: depth n-5 is x --> x + O_{n-4}.  Thus in deg 6 term, if subs into one place, get degree 5 + (n-4) = n+1
\item\label{tonotation2REV} If  $\mathsf{m}_1\to \mathsf{m}_2$, then  $\mathsf{m}_2\to \mathsf{m}_1$, simply by setting $\upbeta\colonequals -\upmu\upalpha$.
\item\label{tonotation3} Also note that if $\mathsf{n}$ is a monomial such that $\mathsf{n}\sim\mathsf{m}$ and $\mathsf{m}\to  \mathsf{m}_2$, then $\mathsf{n}\to\mathsf{m}_2$.
\item Write $\mathsf{m}_1\to \scrO_{n+1}$ if the same conditions in (1) hold, with $\mathsf{m}_2=0$.
\item Write $\mathsf{m}_1\to \mathsf{m}_2,\mathsf{m}_3,\upvarepsilon\mathsf{m}_4$ if in \eqref{eqn:arrownotation}, $\upmu\mathsf{m}_2$ is replaced with $\upmu_2\upalpha\mathsf{m}_2+\upmu_3\upalpha\mathsf{m}_3+\upvarepsilon\upmu_4\upalpha\mathsf{m}_4$, where recall that throughout this subsection $\upvarepsilon\in\{0,1\}$. 
\item Set $(a_1,\hdots,a_n)=xy^{a_1}\hdots xy^{a_t}$, and for simplicity write $(a_1,\hdots,a_n)\to (b_1,\hdots,b_s)$ and $(a_1,\hdots,a_n)\to \scrO_{n+1}$ etc. 
\end{enumerate}
\end{notation}
The notation $\mathsf{m}_1\to \mathsf{m}_2,\mathsf{m}_3$ should be read `$\mathsf{m}_1$ is replaced by a linear combination of $\mathsf{m}_2$ and $\mathsf{m}_3$, in a way compatible with \ref{main chase lemma}'.
\begin{lemma}\label{lem:E6funnelA}
Consider $xy^{a_1}\hdots xy^{a_{t}}=(a_1,\hdots,a_t)$ of degree $n$, where $t\geq 3$.
\begin{enumerate}
\item\label{lem:E6funnelA1} $(1,1,a_3,\hdots,a_{t})\to \upvarepsilon (3,a_3,\hdots,a_{t})$.
\item\label{lem:E6funnelA2} $(2,a_2,\hdots,a_{t})\to (1,a_2,\hdots,a_{t}+1), \upvarepsilon(a_2,\hdots,a_{t-1},a_t+3)$.
\item\label{lem:E6funnelA3} If $a_1\geq 3$, then $(a_1,a_2,\hdots,a_t)\to $
\[
(a_1-2,a_2,\hdots,a_{t-1},a_t+2),\,(a_1-1,a_2,\hdots,a_{t-1},a_t+1),\,\upvarepsilon(a_2,\hdots,a_{t-1},a_1+a_t+1).
\]
\item\label{lem:E6funnelA4} If $n\geq 5$, then $xy^{n-1}\to\upvarepsilon y^{n}$.
\end{enumerate}

\end{lemma}
\begin{proof}
(1) Write $(1,1,a_3,\hdots,a_{t})=xyxyxA$, and apply the automorphism $y\mapsto y-\upalpha xAx$, $x\mapsto x+ \frac{\upalpha}{3}(y^2xA + Axy^2)$, which has depth $n-4$.\\
(2) Write $(2,a_2,\hdots,a_{t})=xy^2xB$, and send $x\mapsto x + \frac{\upalpha}{3}By^2$, $y\mapsto y-\upalpha xB$.\\
(3) Write $(a_1,a_2,\hdots,a_{t})=xy^2C$ and send $y\mapsto y-\upalpha C$.\\
(4) Apply the automorphism $y\mapsto y - \tfrac{\upalpha}{3}y^{n-3}$. 
\end{proof}

In what follows, we say that a monomial $xy^{a_1}\hdots xy^{a_t}$ has width $t$.
\begin{cor}\label{lem:E6funnelC}
If $t\geq 3$, then $(a_1,\hdots,a_t)\to \mathsf{m}_1,\hdots,\mathsf{m}_s$ where each $\mathsf{m}_i$ has width $< t$. 
\end{cor}
\begin{proof}
The proof proceeds by induction on $c=\mathrm{min}\{a_1,\hdots,a_t\}$.  When $c=1$, by \ref{tonotation}\eqref{tonotation2} via cyclic rotation if necessary we can assume that the monomial is $(a_1,1,a_3,\hdots,a_t)$. If $a_1=1$ we are done by \ref{lem:E6funnelA}\eqref{lem:E6funnelA1}, and if $a_1=2$ apply \ref{lem:E6funnelA}\eqref{lem:E6funnelA2} then \ref{lem:E6funnelA}\eqref{lem:E6funnelA1}.  In all other cases, apply \ref{lem:E6funnelA}\eqref{lem:E6funnelA3} and use induction.   Similarly, when $c=2$, by \ref{tonotation}\eqref{tonotation2} we can assume that $a_1=2$.  But then \ref{lem:E6funnelA}\eqref{lem:E6funnelA2} followed by the $a_1=1$ result just proved gives the statement. 

Hence we can assume that $c>2$, that $a_1$ achieves this minimum, and by induction that the result is true for all monomials $(b_1,\hdots,b_t)$ for which $b_1<a_1$.  Applying the move in  \ref{lem:E6funnelA}\eqref{lem:E6funnelA3} to $(b_1,\hdots,b_t)$, followed by the inductive hypothesis on the first two terms produced (noting that the third term already has smaller width) proves the statement.
\end{proof}

Applying \ref{lem:E6funnelC} repeatedly will result only in monomials of width $2$.  To clear these is slightly more delicate, since it involves tracking coefficients.

\begin{lemma}\label{lem:E6funnelD}
If $(a_1,a_2)$ has degree $n\geq 6$, then $(a_1,a_2)\to xy^{n-1}$.
\end{lemma}
\begin{proof}
We provide the proof when $n=2m$ is even; the proof when $n$ is odd is very similar.  

Consider first the automorphism $\upphi$ sending $x\mapsto x-\tfrac{\upalpha}{3}y^{n-2}$ and $y\mapsto y + \upalpha xy^{n-4}$, 
% both terms in \n^{n-3}, so depth n-4.
then
\[
\upphi(f_{\stub})\sim_{n} f_{\stub}+\upalpha xyxy^{n-3} + \upalpha xy^2xy^{n-4}+4\upvarepsilon\upalpha xy^{n-1}+\scrO_{n+1}
\]
Consider next the automorphisms $\upphi_i$ defined by $y\mapsto y+\upalpha y^{i}xy^{n-4-i}$, 
% degree n-3, which is same as for \upphi!  
fixing $x$, defined for all $i$ such that $1\leq i\leq m-2$.  Then for $i$ such that $0\leq i\leq m-3$,
\[
\upphi_i(f_{\stub})\sim_{n} f_{\stub}+\upalpha xy^{i}xy^{n-2-i} + \upalpha xy^{i+1}xy^{n-3-i}+ \upalpha xy^{i+2}xy^{n-4-i} +4\upvarepsilon\upalpha xy^{n-1}+\scrO_{n+1}
\]
% largest i + 2 = (m-3) + 2 = m - 1
% smallest n - 4 - i = 2m - 4 - (m - 3) = m - 1,
%         so far RHS is (m-1,m-1).
whereas for $i=m-2$ 
\[
\upphi_{m-2}(f_{\stub})\sim_{n} f_{\stub}+ 2\upalpha xy^{m-2}xy^{m}+ \upalpha xy^{m-1}xy^{m-1} +4\upvarepsilon\upalpha xy^{n-1}+\scrO_{n+1}
\]
%first x*y^power*x*  matches the index on the map 
We next build a matrix where the columns are indexed by $xy^ixy^{n-2-i}$ for $i=1,\hdots,m-1$, (namely, all possible $(a_1,a_2)$ with degree $n$), and the rows by $\upphi,\upphi_1,\hdots,\upphi_{m-2}$, where the entries capture the coefficients on $xy^ixy^{n-2-i}$ in the above three displayed equations. The first row is $[1,1,0,\hdots,0]$, the second row is $[1,1,1,0,\hdots,0]$, the third $[0,1,1,1,0,\hdots,0]$, and the pattern continues until the bottom row, which is $[0,\hdots,0,2,1]$.  This matrix has non-zero determinant, so by taking suitable linear combinations of the depth $n-4$ automorphisms $\upphi,\upphi_1,\hdots,\upphi_{m-2}$, any particular $xy^ixy^{n-2-i}$ can be pushed into higher degree with (at worst) adding to the $xy^{n-1}$ term.
% n - 2 - (m-1) = 2m -  2 - m + 1 = m - 1 
%
%For $N\ge2$, let $K_N(x,y)=(k_{ij})$ be the $N\times N$ matrix with
%\[
%\medmath{
%k_{ij} = 
%\begin{cases}
%0 & \textrm{if }|i-j| > 1 \\
%1 & \textrm{if $|i-j|\le 1$ and $j<n$} \\
%y & i=n-1, j=n \\
%x & i = j = n
%\end{cases}}
%\]
%where $x$ and $y$ are indeterminates. Thus
%\[
%K_2(x,y) =
%\begin{pmatrix}
%1&y\\1&x\\
%\end{pmatrix} 
%\quad \textnormal{and} \quad
%K_3(x,y) =
%\begin{pmatrix}
%1&1&0\\
%1&1&y\\
%0&1&x
%\end{pmatrix}
%\]
%and analogously for $n\ge4$.
%Clearly $\det K_2(x,y)=x-y$ and $\det K_3(x,y)=-y$, and for $N\ge4$ by expanding along the first row gives
%\[
%\det K_N(x,y) = \det K_{N-1}(x,y) - \det K_{N-2}(x,y)
%\]
%so that inductively $\det K_N(x,y) \in\{\pm x, \pm y, \pm (x-y)\}$.
%In particular, $\det K_N(x,y)=0$ if and only if $xy(x-y)=0$.
%The matrices above are either $K_{2n}(2,1)$ or $K_{2n+1}(1,2)$, so are always nonsingular,
\end{proof}

The following is the main result of this subsection.
\begin{thm}\label{thm:TypeEGeneric}
$\Jac(f_{\rm stub} + \scrO_6)\cong \Jac(x^3  + xy^3 + y^4p)$ for some monic $p\in\mathbb{C}\llsq y\rrsq$.
\end{thm}
\begin{proof}
By cyclically permuting $g=f_{\rm stub} + \scrO_6$ so that the degree six terms begin with either an $x^2$, or are monomials of the form $xy^{a_1}\hdots xy^{a_t}$, or are $y^6$, we may write
\[
g\sim_6 f_{\rm stub} + \uplambda_6y^6 + (x^2H + \sum_{t=1}^s \upmu xy^{a_1}\hdots xy^{a_t})  + \scrO_7 \colonequals \sff_6.
\]
where $(-)$ has degree six.  The above is the fixed $\sff_N$, where here $N=6$, in \ref{main chase lemma}.  Now, for a fixed $n\geq N$, we will assume that $\sff_n$ has been constructed and has the form 
\begin{equation}
\sff_n = f_{\rm stub} + \sum_{t=6}^{n}\uplambda_t y^t + (x^2H' + \sum_{t=1}^{s'} \upmu xy^{a_1}\hdots xy^{a_t} )+\scrO_{n+1},\label{eqn:sffnl3}
\end{equation}
where $(-)$ has degree $n$.  From this, we claim that we can apply finitely many `$\to$ moves' to $\sff_n$ to produce a power series $\sff_{n+1}$ which has the same form as in \eqref{eqn:sffnl3}, but where $n$ is replaced by $n+1$.  Once this has been achieved, \ref{main chase lemma} applies, and so $\Jac(g)\cong\Jac(\sff_6)\cong\Jac(\sff_\infty)$ where $\sff_\infty \colonequals \lim_{n\geq 6} \sff_n$.  This $\sff_\infty$ clearly has the form $x^3  + xy^3 + p$ for some $\mathbb{C}\llsq y\rrsq_{\geq 4}$.  If $p=0$ we are done, hence suppose that the lowest degree coefficient in $p$ is $c_n\neq0$.  Choose $a,b\in\C^*$ such that $b^{9/2-n}=c_n$ and $a=b^{3/2}$. Then $a^3=ab^3=c_nb^n$, and the scaling automorphism $x\mapsto ax$ and $y\mapsto by$ replaces $p$ by a monic polynomial.

To prove the claim, by \ref{tonotation}\eqref{tonotation2} we can freely use \ref{lem:E6funnelA}--\ref{lem:E6funnelD}.  Now first use Tschirnhausen \ref{tschirnhausen} to push $x^2H'\to \scrO_{n+1}$, then use \ref{lem:E6funnelC} repeatedly to push the $\sum$ onto width two monomials, and from there use \ref{lem:E6funnelD} to push these onto the width one monomial $xy^{n-1}$. By \ref{lem:E6funnelA}\eqref{lem:E6funnelA4}, this then pushes onto (at worst) $y^n$, leaving only some scalar times $y^n$ in degree $n$.  We then cyclically permute the degree $n+1$ terms so that they begin with either an $x^2$, or are monomials of the form $xy^{a_1}\hdots xy^{a_t}$, or are $y^{n+1}$.  Overall, this has only used finitely many `$\to$ moves', and has cyclically permuted only in degrees $n$ and $n+1$.  This proves the claim, and completes the proof.
\end{proof}

\subsection{Lengths 4--6}
The following result is the minimum information required in \S\ref{sec:FD} to establish that all $\Jac(f)$ with algebraic length $\ell\geq 4$ are finite dimensional.  This result will be greatly improved in \S\ref{sec:NecklacIFY}.

\begin{thm}\label{thm:Length456HalfWay}
If $f\in\ringtwo$ has algebraic length $\ell\geq 4$, then $\Jac(f)\cong \Jac(g)$ where $g$ is some element in the following list, where all $\upalpha_*\in\C$ and $\uplambda\in\C^*$.
\[
\renewcommand{\arraystretch}{1.2}
\medmath{\begin{array}{c@{\hspace{30pt}}c@{\hspace{2pt}}c@{\hspace{2pt}}c@{\hspace{2pt}}c@{\hspace{4pt}}c@{\hspace{2pt}}c@{\hspace{4pt}}c@{\hspace{2pt}}c@{\hspace{2pt}}c@{\hspace{2pt}}c@{\hspace{4pt}}c@{\hspace{2pt}}l@{\hspace{2pt}}c@{\hspace{2pt}}c@{\hspace{2pt}}c@{\hspace{2pt}}c@{\hspace{2pt}}c@{\hspace{2pt}}c@{\hspace{2pt}}}
4 && x^3 &+& y^4 &+& xyxy^2 &+& \upalpha_1(xy)^3 &+& \upalpha_2(xy^2)^2  &+& \scrO_7\\
5 && x^3 &+& (xy)^2 &+& y^5 &+& \upalpha_3(xy^2)^2 &+& \upalpha_4y^6 &+& \scrO_{7}\\
6 && x^3 &+& y^4 &&&+& (xy)^3&+& \uplambda(xy^2)^2 &+& \upalpha_5(xy)^3y&+& \upalpha_6(xy)^4&+&\scrO_9
\end{array}}
\renewcommand{\arraystretch}{1}
\]
\end{thm}
\begin{proof}
For $\ell=4$, by \ref{thm:ReductionPrecise} $\Jac(f)\cong\Jac(g)$ where $g=x^3+y^4+xyxy^2+\scrO_6$.  By Tschirnhausen \ref{tschirnhausen} the  the $x^3$ term can be used to remove from $g$ any monomial in degree six  divisible by $x^2$, at the cost of higher order terms. Similarly, the $y^4$ term can be used to remove any monomial in $g$ of degree six divisible by $y^3$, at the cost of higher order terms.  Up to cyclic rotation, in degree six this leaves only the monomials $(xy)^3$ and $(xy^2)^2$.\\
For $\ell=5$,  by \ref{thm:ReductionPrecise} $\Jac(f)\cong\Jac(g)$ where $g=x^3+xyxy+y^5+\scrO_6$.  Set $f_{\stub}=x^3+xyxy+y^5$, then by Tschirnhausen \ref{tschirnhausen} the $x^3$ term can again be used to clear any monomial in degree six divisible by $x^2$ into higher order, and similarly the $xyxy$ term can be used to clear any monomial in $g$ of degree six divisible by $xyx$ into higher order.  Up to cyclic rotation, this only leaves $xy^5, (xy^2)^2$, and $y^6$ in degree six.  Applying the automorphism $\upphi$ defined by $x\mapsto x - \tfrac{2\upalpha}{3}yxy - \tfrac{\upalpha^2}{3}yx^2y$, $y\mapsto y + \upalpha xy$ (for any $\upalpha\neq 0$),
\[
\upphi(f_{\stub})\sim_6 f_{\stub}  - \tfrac{4\upalpha}{3}(xy^2)^2 + 5\upalpha xy^5+\scrO_7
\]
Thus, choosing $\upalpha$ appropriately, we can also choose between $(xy^2)^2$ and $xy^5$.  Replacing $xy^5$ by $(xy^2)^2$ proves the statement.\\
For $\ell=6$, by \ref{thm:ReductionPrecise} $\Jac(f)\cong\Jac(g)$ where $g=x^3+y^4+\scrO_6$.  By Tschirnhausen \ref{tschirnhausen} the $x^3$ term can be used to push any monomial in degree six divisible by $x^2$ into higher order, and the $y^4$ term can be used to push any monomial in degree six divisible by $y^3$ into higher order. Up to cyclic rotation, this only leaves $(xy)^3$ and $(xy^2)^2$ in degree six, thus $\Jac(f)\cong\Jac(h)$ where $h=x^3+y^4+\upalpha (xy)^3 + \upbeta(xy^2)^2+\scrO_7$
for some $\upalpha,\upbeta\in\C$.  Repeating the above analysis in degree seven, pushing every monomial divisible by $x^2$ and every monomial divisible by $y^3$ into higher order, then up to cyclic permutation, this only leaves $(xy)^3y$. Thus $\Jac(f)\cong \Jac(x^3+y^4+\upalpha (xy)^3 + \upbeta(xy^2)^2+\upalpha_5(xy)^3y+\scrO_8)$ for some $\upalpha_5\in\C$.

We next claim, since $f$ has algebraic length, that both $\upalpha$ and $\upbeta$ must be non-zero.  

\medskip
The proof of the claim is by considering the quotient of $\Jac(f)$ by a central element $z\in\scrZ(\Jac(f))\cap\n$, in a similar way to \S\ref{sec:centres}--\ref{sec:AlgLength}. It splits into cases.\\
\emph{Case 1:  $\upalpha=\upbeta=0$}, so that $h=x^3+y^4+\scrO_7$.  In this case, then in a similar way to \ref{lem:keycentreallcases}\eqref{lem:keycentreallcases3} but raised one degree, $\uppi_6(z)=0$ for any central $z$. Thus $A_6\big/\lcl \uppi_6(z)\rcl=A_6=\tfrac{\C\langle x,y\rangle}{(x^2,y^3,\m^6)}$, which has dimension $22$.   This contradicts \eqref{eq!Akiso} with $k=6$, since for $\ell=6$ we have $\dim_\C\Upgamma_6=21$ by \ref{lem:allSurfaceData}\eqref{lem:allSurfaceData6}.\\
\emph{Case 2:  $\upalpha\neq 0, \upbeta=0$}, so that after rescaling 
%choose $b=\upalpha^{-1/3}$ and $a=\upalpha^{-4/9}$, so that $a^3=b^4=\upalpha a^3b^3$. 
$h\cong x^3+y^4+ (xy)^3 +\upalpha'_5(xy)^3y+\scrO_8$.  As in \S\ref{sec:centres}, it is elementary to check that $\uppi_6(z)=\upgamma y^3$ for some $\upgamma\in\C$.  If $\upgamma=0$, then $\dim_\C A_6\big/\lcl \uppi_6(z)\rcl=22$, which contradicts $\dim_\C\Upgamma_6=21$ as in Case~1. If $\upgamma\neq 0$, then $A_6\big/\lcl \uppi_6(z)\rcl=\C\langle x,y\rangle/I_6$ where $I_6=(x^2+yxyxy,y^3,xyxyx,\m^6)$.  The given generators of $I_6$ are already a standard basis, so determine a monomial $\C$-basis $\scrB$ of $\C\langle x,y\rangle/I_6$ with 21 elements. We show that this is not isomorphic to $\Upgamma_6$ following \S\ref{sec:possalglengths}: there is no $v\in\scrJ\setminus\scrJ^2$ in $A_6\big/\lcl \uppi_6(z)\rcl$ with $v^2=0$.
Indeed expressing $v$ in the ordered basis $\scrB$ as
\[
\medmath{v = \al_1x + \al_2y + \al_3xy + \al_4yx + \hdots + \al_8yxy + \hdots + \al_{10}xyxy + \al_{11}xy^2x + \al_{12}yxyx + \hdots}
\]
then $v^2$ expressed in $\scrB$, after reducing by $x^2 \equiv - yxyxy$ and $y^3\equiv xyxyx\equiv 0$, is
\[
\medmath{v^2 = \al_1\al_2(xy+yx) + \al_2^2y^2 + \al_1(\al_3+\al_4)xyx + \hdots + (-\al_1^2 + \al_2(\al_{10}+\al_{12}) + \al_8(\al_3+\al_4))yxyxy + \hdots}
\]
so that $\al_2=0$, from which it follows at once that $\al_1=0$, and the claim holds.\\
\emph{Case 3: $\upalpha =0, \upbeta\neq 0$}, so that after rescaling 
%Want a^3 = b^4 = A*a^2*b^4
%so choose a = root(1/A), then second equation holds.  Choosing b to be the 4th root of a^3, the first holds. Thus, we can scale!
$h\cong x^3+y^4+ (xy^2)^2 +\upalpha'_5(xy)^3y+\scrO_8$.  In a similar manner to the previous case, now $\uppi_6(z)=\upgamma x^2$. Either $\upgamma=0$ in which case we obtain an immediate contradiction since $\dim_\C A_6\big/\lcl \uppi_6(z)\rcl=22$, or $\upgamma\neq 0$ in which case $A_6\big/\lcl \uppi_6(z)\rcl=\C\langle x,y\rangle/J$ where $J=(x^2,yxyxy,y^3+xyxyx,\m^6)$.  Whilst this quotient again has dimension 21, it too is not isomorphic to $\Upgamma_6=\tfrac{\C\langle x,y\rangle}{(x^2, y^3,(x+y)^5,\m^6)}$ for $\ell=6$. This again contradicts \eqref{eq!Akiso} with $k=6$.

\medskip

As the above three cases have obtained a contradiction, we conclude that $\upalpha,\upbeta\neq 0$, proving the claim.  Thus, since $\upalpha\neq0$, the same rescaling as in Case 2 
%choose $b=\upalpha^{-1/3}$ and $a=\upalpha^{-4/9}$, so that $a^3=b^4=\upalpha a^3b^3$. 
then shows that $\Jac(f)\cong\Jac(f_{\stub} +\upalpha'_5(xy)^3y+\scrO_8)$, where $f_{\stub}\colonequals x^3+y^4+(xy)^3+ \uplambda(xy^2)^2$, $\uplambda\in\C^*$ and $\upalpha'_5\in\C$.

In degree eight, again pushing every monomial divisible by $x^2$ and every monomial divisible by $y^3$ into higher order, then up to cyclic permutation the only remaining monomials are $xy(xy^2)^2$ and $(xy)^5$.  The automorphism $\upphi$ sending $x\mapsto x + \upalpha y^2xyxy + \tfrac{2\upalpha\uplambda}{3}y^3xy^2$ and $y\mapsto y + \upalpha yxy - \upalpha xy^2 + \upalpha^2y^2xyx - \upalpha^2xy^2xy$ (which will become more conceptual in \S\ref{subsec:l6funnel}) for some $\upalpha\in\C$ then gives
\[
\upphi(f_{\stub}+\upalpha_5(xy)^3y)\sim_8  f_{\stub}+\upalpha_5(xy)^3y+   3\upalpha (xy)^4 + 2\upalpha\uplambda xy(xy^2)^2+\scrO_9
\]
Thus, since $\uplambda\neq 0$, choosing $\upalpha$ appropriately we can remove any $xy(xy^2)^2$ term in degree nine at the cost of loading more onto the $(xy)^4$ coefficient, and at the cost of higher order terms. The statement follows.
\end{proof}

\begin{cor}\label{cor:SemiFinalAlgLength}
If $f\in\ringtwo$ has algebraic length $\ell$, then $\Jac(f)\cong \Jac(g)$ where $g$ is some element in the following list, where all $\upvarepsilon_*\in\{0,1\}$ and all $\upalpha_*\in\C$.
\[
\renewcommand{\arraystretch}{1.3}
\medmath{\begin{array}{l@{\hspace{15pt}}lll}
1 & x^2 + \upvarepsilon_1 y^n & n\geq 2, \upvarepsilon\in\{0,1\} \\
2 &x^2y+\upvarepsilon_2 y^{2n}+\upvarepsilon_2y^{2m-1} &   2n-1\geq m\geq 2 \\
3 &x^3 + xy^3 + y^4p & p\in\C\llsq y\rrsq  \\
4 & x^3 + y^4 + xyxy^2 + \upalpha_1(xy)^3 + \upalpha_1(xy^2)^2+\scrO_7\\
5 & x^3 + xyxy + y^5  + \upalpha_3(xy^2)^2 + \upalpha_4y^6 + \scrO_{7}\\
6 & x^3 + y^4 + (xy)^3+ \uplambda(xy^2)^2+ \upalpha_5(xy)^3y  + \scrO_8& \uplambda\in\C^*
\end{array}}
\renewcommand{\arraystretch}{1}
\]
\end{cor}
\begin{proof}
Consider the elements in \ref{thm:ReductionPrecise}. For $\ell=1,2$ there is nothing to do.  For $\ell=3$, the form in \ref{thm:ReductionPrecise} can be improved into the line in the list above using \ref{thm:TypeEGeneric}.  For $\ell=4,5,6$, the forms in \ref{thm:ReductionPrecise} can be improved into the lines stated using \ref{thm:Length456HalfWay}.  
\end{proof}

%-----------------------------------------------------------------------------------------
\section{Bounding the Dimension}\label{sec:FD}

Within the set of $f\in\ringtwo$ which have algebraic length $\ell$, this section determines those that satisfy $\dim_\mathbb{C}\Jac(f)<\infty$.  The answer is surprising: all but one $f$ with $\ell=3$, and \emph{all} $f$ with $\ell\geq 4$ satisfy this strong finiteness condition.

The main novelty is the development of a weak version of Hironaka's standard bases, adapted to the noncommutative setting, which is used to establish upper bounds on the dimension.  Those upper bounds do turn out to be sharp, but to bypass unnecessary and cumbersome (and very hard to read) standard bases verifications, we avoid establishing the precise dimension in this section, deliberately leaving this to \ref{cor:allDims} where more elegant geometric arguments can be used.

\subsection{Weak Standard Basis Strategy}
If $J\subset\ring$ is any two-sided ideal, then the ideal of leading terms
$\LT(J) = \left\{ \LT(h) \mid h \in J\right\}$ satisfies the numerical property 
\[
\dim_\mathbb{C}\frac{\ring}{J + \m^k} = \dim_\mathbb{C}\frac{\ring}{\LT(J) + \m^k}
\]
for any $k\ge0$.

Since $\LT(J)$ is defined using all elements of $J$, it follows given \emph{any} finite number of elements $\mathsf{g}_1,\hdots, \mathsf{g}_s\in J$ that
$( \LT(\mathsf{g}_1),\hdots,\LT(\mathsf{g}_s)) + \m^k \subset \LT(J) + \m^k$,
and so
\begin{equation}
\dim_\mathbb{C}\frac{\ring}{J + \m^k} \le \dim_\mathbb{C}\frac{\ring}{(\LT(\mathsf{g}_1),\hdots,\LT(\mathsf{g}_s)) + \m^k}.\label{eqn:aimGB}
\end{equation}
In what follows, we set $J=\lcl\dcyc f\rcl$ and the strategy is to exhibit $\scrS=\{\mathsf{g}_1,\hdots, \mathsf{g}_s\}\subset J$ for which there exists $N\in\N$ with the right hand side of \eqref{eqn:aimGB} $\leq N$ for all $k$.   Once this is established, it follows immediately that $\dim_\mathbb{C}\Jac(f)\leq N$.  The fact that the $\mathsf{g}_i$ may be found by running standard bases algorithms is interesting, but not relevant here; this section merely exhibits the $\mathsf{g}_1,\hdots, \mathsf{g}_s\in J$ and produces the upper bound $N$.

Throughout this section, congruences $g\equiv h$ of elements $g,h\in\ring$ are taken modulo~$\lcl\dcyc f\rcl$,
where sometimes the more precise decorated notation $g\buildrel{\mathsf{a}}\over{\equiv}h$ is used for congruence modulo the ideal generated by only an element~$\mathsf{a}$.

\subsection{Length Three}\label{subsec:l3bound}
If $f\in\ringtwo$ satisfies $\ell=3$, then by \ref{thm:TypeEGeneric} $f\cong x^3+xy^3+p(y)$ for some monic $p\in\mathbb{C}\llsq y\rrsq_{\geq 4}$.  It turns out that the dimension of $\Jac(f)$ depends only on the order of $p$.  The following is the easiest case, namely $p=0$.

\begin{lemma}\label{lem:l3infinite}
If  $\Jac(f)\cong \Jac(x^3 + xy^3)$, then $\dim_\mathbb{C}\Jac(f)=\infty$.
\end{lemma}
\begin{proof}
After rescaling we can assume that $f=\tfrac{1}3x^3 + xy^3$.
The Jacobian ideal $\lcl\dcyc f\rcl$  is the closed ideal generated by two elements $\mathsf{g}_1\colonequals  x^2+y^3$ and $\mathsf{g}_2\colonequals xy^2+yxy+y^2x$, and we simply compute its standard basis.

First note that $x \mathsf{g}_1- \mathsf{g}_1x = xy^3-y^3x$, so that $y^3$ commutes with $x$ modulo~$\mathsf{g}_1$ (which we use without further comment), and furthermore
\[
x\mathsf{g}_1-\mathsf{g}_1x = xy^3-y^3x\buildrel{\mathsf{g}_2}\over{\equiv} -yxy^2-y^2xy-y^3x\buildrel{\mathsf{g}_2}\over{\equiv} (y^2xy+y^3x)-y^2xy-y^3x=0.
\]
Consider $x\mathsf{g}_2-\mathsf{g}_1y^2=xyxy+xy^2x-y^5\buildrel{\mathsf{g}_2}\over{\equiv} xyxy-yxyx-y^2x^2-y^5\buildrel{\mathsf{g}_1}\over{\equiv}xyxy-yxyx$
and set $\mathsf{g}_3\colonequals xyxy-yxyx\in \lcl\dcyc f\rcl$.
The remaining overlaps are trivial:
\begin{align*}
&\mathsf{g}_1yxy - x\mathsf{g}_3=y^4xy+xyxyx\buildrel{\mathsf{g}_3}\over{\equiv} y^4xy + yxyx^2\buildrel{\mathsf{g}_1}\over{\equiv}0,
\quad\text{and $xy\mathsf{g}_3-\mathsf{g}_3xy$ equals}\\
 &\quad -xy^2xyx+yxyx^2y\equiv ((yx)^3+y^2x^2yx)-yxy^5\equiv (y^2x)^2 - y^6x -y^4xy^2\buildrel{\mathsf{g}_2}\over{\equiv}0.
\end{align*}
Thus $\{\mathsf{g}_1,\mathsf{g}_2,\mathsf{g}_3\}$ is the standard basis, and so by \ref{sec:stdbasis} a monomial basis of $\Jac(f)$ is given by the classes of the monomials not divisible by any of the leading terms $x^2, xy^2, xyxy$. Thus a monomial basis of $\Jac(f)$ is given by the monomials
\begin{equation}
\begin{array}{c}
\renewcommand{\arraystretch}{1.2}
\newcommand{\cw}[1]{4pt}
\medmath{
\begin{array}{llllllll}
&&&xyx&yxyx&\hdots&y^{n-3}xyx&\hdots \\
&&xy&yxy&y^2xy&\hdots&y^{n-2}xy&\hdots\\
&x&yx&y^2x&y^3x&\hdots&y^{n-1}x&\hdots\\
1&y&y^2&y^3&y^4&\hdots&y^{n}&\hdots
\end{array}
}\label{eqn:E6full}
\end{array}
\renewcommand{\arraystretch}{1}
\end{equation}
and so $\dim_\mathbb{C}\Jac(f)=\infty$.
\end{proof}
In all other cases, namely those for which $p\neq 0$, we show that that $\dim_\mathbb{C}\Jac(f)<\infty$.
 
\begin{prop}\label{prop:l3finite}
Suppose $f\in\ringtwo$ has $\ell=3$.
Consider $g=x^3+xy^3+y^4p$ with $p\in\C\llsq y\rrsq$ and $\Jac(f)\cong\Jac(g)$,
as given by \ref{cor:SemiFinalAlgLength}. Suppose $p\neq 0$ and consider the set
\[
\scrS = \{x^2, xy^2, xyxy, y^{n-1}xyx, y^{n+2}xy, y^{n+4}x, y^{n+7}\}\subset\ring
\]
where $n\colonequals \ord(p)+4\ge5$.
Then the following statements hold:
\begin{enumerate}
\item\label{E6dimi}
$\scrS\subset\LT(J)$ where $J=\lcl\dcyc g\rcl$.
\item\label{E6dimii}
$\dim_\mathbb{C} \ring/\lcl \scrS\rcl = 4n+12$ and  $\n^{n+7}\subset \lcl \scrS\rcl$.
\end{enumerate}
In particular, $\Jac(f)$ is finite dimensional, $\dim_\mathbb{C}\Jac(f)\le 4n+12$, and $\n^{n+7}\subset\lcl\dcyc f\rcl$.
\end{prop}

\begin{proof}
(1) After rescaling we can assume that $g = \tfrac{1}3x^3 + xy^3 + \tfrac{1}{n}y^n + \tfrac{\upmu}{n+1}y^{n+1} + \bigO_{n+2}$ for some $\upmu\in\C$.  
The Jacobian ideal $J=\lcl\dcyc g\rcl$ is the closed ideal generated by two elements
\[
\mathsf{g}_1 \colonequals x^2 +y^3 \qquad\textrm{and}\qquad 
\sfg_2\colonequals xy^2 + yxy + y^2x + y^{n-1} + \upmu y^n + \scrO_{n+1}.
\]
As in \ref{lem:l3infinite}, by $x\mathsf{g}_1-\mathsf{g}_1x\in J$, $x$ and $y^3$ commute modulo~$J$, which we do without further comment.
To prove \eqref{E6dimi} we construct seven elements of $J$ whose leading terms are the elements of~$\scrS$.  With $\mathsf{g}_1, \sfg_2$ as stated, consider $\sfg_3 \colonequals x\sfg_2-\sfg_2x+y^2\mathsf{g}_1-\mathsf{g}_1y^2\in J$ so that
\[
\sfg_3  = xyxy - yxyx + xy^{n-1} - y^{n-1}x + \upmu xy^n - \upmu y^nx + \scrO_{n+2}
\]
where the higher order terms of $\sfg_3$ are not required below.

\begin{lemma}\label{lem!E6g4}
If $n\ge6$, then there is an element $\sfg_4 = y^{n-1}xyx + \be y^{n+3} + \scrO_{n+4}\in(\mathsf{g}_1,\sfg_2)$ for some $\be\in\C$ with $\be^2+\be+1\ne0$.
\end{lemma}

We use that $x^2\equiv-y^3$ to substitute for $x^2$ in expressions without further comment.

\begin{proof}
In the proof of this lemma, we work to precision $\bigO_{n+4}$ except where specified.
Consider the expression 
\begin{align}
xy\sfg_3-\sfg_3xy &= -xy^2xyx + xyxy^{n-1} - xy^nx + \upmu xyxy^n - \upmu xy^{n+1}x \nonumber\\
&\qquad{} +yxyx^2y - xy^{n-1}xy + y^{n-1}x^2y - \upmu xy^nxy + \upmu y^nx^2y \nonumber\\
&\buildrel{\mathsf{g}_1}\over{\equiv} (-xy^2xyx+yxyx^2y) + (xyxy^{n-1} - xy^{n-1}xy) - xy^nx - y^{n+3} \nonumber\\
&\qquad{} + \upmu\big(xyxy^n - xy^nxy\big) - \upmu xy^{n+1}x  \label{xyCCxy}
\end{align}
since $y^nx^2y\equiv - y^{n+4}$.
Consider first the terms in \eqref{xyCCxy} of degree~6:
\begin{align}
-xy^2xyx+yxyx^2y &\buildrel{\sfg_2}\over\equiv \big((yx)^3+y^2x^2yx+y^{n-1}xyx+\upmu y^nxyx\big) - yxy^5 \nonumber\\
&\buildrel{\mathsf{g}_1}\over{\equiv} \big((yx)^3 - y^6x + y^{n-1}xyx + \upmu y^nxyx\big) - y^4xy^2  \label{Ddeg6terms}\\
&\buildrel{\sfg_2}\over{\equiv} \big((yx)^3 - \cancel{y^6x} + y^{n-1}xyx + \upmu y^nxyx\big) + \big(y^5xy+\cancel{y^6x} + y^{n+3}\big).\nonumber
\end{align}
The first term of \eqref{Ddeg6terms} is
\begin{align}
(yx)^3 = y(xyxy)x &\buildrel{\sfg_3}\over{\equiv} y^2xyx^2 - yxy^{n-1}x + y^nx^2 - \upmu yxy^nx + \upmu y^{n+1}x^2  \nonumber\\
&\buildrel{\mathsf{g}_1}\over{\equiv} -y^5xy - yxy^{n-1}x - y^{n+3} - \upmu yxy^nx \label{yx3terms}
\end{align}
so that \eqref{Ddeg6terms} and \eqref{yx3terms} together give
\begin{equation}\label{Ddeg6reduced}
-xy^2xyx+yxyx^2y \equiv y^{n-1}xyx - yxy^{n-1}x + \upmu\big(y^nxyx-yxy^nx \big).
\end{equation}
Combining \eqref{xyCCxy} and \eqref{Ddeg6reduced}, collected into degree~$n+2$
and $n+3$ parts as indicated, gives
\begin{align}
xy\sfg_3-\sfg_3xy &\equiv \big(xyxy^{n-1} - xy^{n-1}xy + y^{n-1}xyx - yxy^{n-1}x - xy^nx\big) \nonumber\\
&\qquad{} + \big({-y^{n+3}} + \upmu(xyxy^n - xy^nxy + y^nxyx - yxy^nx)  \big). \label{Dmidpoint}
\end{align}
To reduce \eqref{Dmidpoint}, consider the residues $n$ mod $3$ separately. Using only $x^2\equiv -y^3$, easy congruences are indicated in the following table.
\[
\medmath{\begin{array}{c@{\hspace{20pt}}lllll}
& xyxy^{n-1} & -xy^{n-1}xy & y^{n-1}xyx & -yxy^{n-1}x & -xy^nx \\[4pt]
n\equiv0 & y^{n-3}xyxy^2 & -y^{n-3}xy^2xy & y^{n-1}xyx & -y^{n-2}xy^2x & +y^{n+3} \\
n\equiv1 & y^{n-1}xyx & +y^{n+3} & y^{n-1}xyx & +y^{n+3} & -y^{n-1}xyx \\
n\equiv2 & y^{n-2}xyxy & -y^{n-2}xyxy & y^{n-1}xyx & -y^{n-1}xyx & -y^{n-2}xy^2x
\end{array}}
\]
To precision $\bigO_{2n-2}$, any substitution using $\sfg_2$ in monomials of degree $\ge n+2$ replaces $xy^2$ by $-yxy-y^2x$,
and to precision $\bigO_{2n}$ any substitution using $\sfg_3$ replaces $xyxy$ by $yxyx$.
Since $n\ge6$ we have $2n-2\ge n+4$, so we do not lose precision by these steps.
Making these substitutions, and using $\sfg_3$ twice rather than $\sfg_2$ to simplify factors $xyxy^2$, refines the table above to the following.
\[
\medmath{\begin{array}{c@{\hspace{20pt}}ccccc}
& xyxy^{n-1} & -xy^{n-1}xy & y^{n-1}xyx & -yxy^{n-1}x & -xy^nx \\[4pt]
n\equiv0 & y^{n-1}xyx & y^{n-1}xyx-y^{n+3} & y^{n-1}xyx & y^{n-1}xyx-y^{n+3} & y^{n+3} \\
n\equiv1 & y^{n-1}xyx & y^{n+3} & y^{n-1}xyx & y^{n+3} & -y^{n-1}xyx \\
n\equiv2 & y^{n-1}xyx & -y^{n-1}xyx & y^{n-1}xyx & -y^{n-1}xyx & y^{n-1}xyx-y^{n+3}
\end{array}}
\]
Adding up the rows, it follows that the three cases $n\equiv 0, 1, 2$ mod $3$, together with the additional term $-y^{n+3}$ from \eqref{Dmidpoint}, are thus
\[
4y^{n-1}xyx - 2y^{n+3},\quad
y^{n-1}xyx + y^{n+3}
\quad\textrm{and}\quad
y^{n-1}xyx - 2 y^{n+3}
\]
respectively. Thus, there is an element $\sfh_4\equiv xy\sfg_3-\sfg_3xy$ of the form
\[
\sfh_4 = y^{n-1}xyx + \be y^{n+3} + \upmu\big(xyxy^n - xy^nxy + y^nxyx - yxy^nx\big) + \bigO_{n+4}.
\]
In particular, $y^nxyx = y\sfh_4 + \bigO_{n+4} \equiv 0 + \bigO_{n+4}$ up to our working precision $\bigO_{n+4}$.
We use this to push all remaining degree $n+3$ terms of $\sfh_4$ into higher degree, as follows.

Once more, split by residue $n\mod 3$. Applying $x^2\equiv -y^3$ and $xyxy\equiv yxyx$
and $y^nxyx\equiv 0$ (up to precision $\bigO_{n+4}$)
gives the following:
\[
\medmath{\begin{array}{c@{\hspace{20pt}}c@{\hspace{15pt}}c@{\hspace{15pt}}c@{\hspace{15pt}}}
& xyxy^n & -xy^nxy & -yxy^nx \\[4pt]
n\equiv0 & y^nxyx\equiv0 & -y^nx^2y\equiv0 & y^{n+1}x^2\equiv0 \\
n\equiv1 & y^{n}xyx\equiv0 & -y^{n}xyxy\equiv0 & -y^nxyx\equiv0 \\
n\equiv2 & y^{n-2}xyxy^2 & -y^{n-2}xy^2xy & -y^{n-1}xy^2x 
\end{array}}
\]
The three terms in the case $n\equiv2$ are also eliminated by applying $\sfg_2$ or $\sfg_3$ once and then $y^nxyx\equiv 0$. Thus $\sfh_4\equiv y^{n-1}xyx+\upbeta y^{n+3} + \bigO_{n+4} \equalscolon \sfg_4$, which completes the proof.
\end{proof}

The proof above contains the main calculation also for the case $n=5$.
\begin{cor}
If $n=5$, then there is an element $\sfg_4 = y^4xyx + y^7x + \upbeta y^8 + \bigO_9\in(\sfg_1,\sfg_2)$ for some $\upbeta\in\C$ with $\upbeta^2+\upbeta+1\ne0$.
\end{cor}
\begin{proof}
The proof of \ref{lem!E6g4} is valid for $n=5$ up to \eqref{Dmidpoint}.
Commuting $xy^3\equiv y^3x$, the first 4 terms of \eqref{Dmidpoint} cancel at once.
Then, setting $\upbeta=-2$,
\[
-xy^nx-y^{n+3}\equiv -y^3xy^2x-y^8\equiv y^4xyx + y^5x^2 + y^7x -y^8 + \bigO_9 \equiv y^4xyx + y^7x + \upbeta y^8 + \bigO_9
\]
and the remaining terms reduce to zero degree~8 leaving terms in degree~$\ge9$. 
\end{proof}

We resume the general assumption $n\ge5$.
Having identified $\sfg_1,\sfg_2,\sfg_3,\sfg_4\in J$, with leading terms $x^2, xy^2, xyxy, y^{n-1}xyx$ respectively,  construct $\sfg_5\in J$ as
\[
y^{n-1}xy\mathsf{g}_1 - \sfg_4x = y^{n-1}xy^4 - \be y^{n+3}x + \scrO_{n+5}
\buildrel{\mathsf{g}_1}\over{\equiv} y^{n+2}xy - \be y^{n+3}x + \scrO_{n+5} \equalscolon \sfg_5.
\]
(Note that in the case $n=5$, the additional term $y^7x$ appears in $\sfg_4x$ as $y^7x^2\equiv y^{10}$,
and so does not contribute to precision $\bigO_{n+5}$.)
To construct $\sfg_6$, we first identify an intermediate element $\sfh_6\in J$ defined by
\begin{align}
y^{n+2}\sfg_2 - \sfg_5y &= \big(y^{n+3}xy + y^{n+4}x + y^{2n+1} + \upmu y^{2n+2}\big) + \be y^{n+3}xy + \scrO_{n+6} \nonumber \\
&\buildrel{\sfg_5}\over{\equiv}  (1+\be+\be^2) y^{n+4}x + \scrO_{n+6}\equalscolon \sfh_6\label{def:E6F0} 
\end{align}
where $2n+1\ge n+6$ since $n\ge5$.  

\begin{lemma}
There is an element $\sfg_6 = y^{n+4}x + \updelta y^{n+6} + \bigO_{n+7}\in(\mathsf{g}_1,\sfg_2)$
for some $\updelta\in\C$.
\end{lemma}
\begin{proof}
With $\sfh_6\in (\mathsf{g}_1,\sfg_2)$ as defined by \eqref{def:E6F0},
consider a term $\updelta\msf$ of degree $n+6$ in $\sfh_6$ with $\updelta\in\C$.
If $\msf$ is divisible by $x^2$, then we reduce it by $\mathsf{g}_1$ to an element of order $n+7$.
If $\msf = y^axy^b\nsf$ with $a\ge0$, $b\ge2$ and some monomial $\nsf$, then we reduce it by $\sfg_2$ to
\[
\msf \equiv -y^{a+1}xy^{b-1}\nsf - y^{a+2}xy^{b-2}\nsf + \bigO_{n+7}.
\]
If $\msf = y^axyxy\nsf$ with $a\ge0$, then we reduce it by $\sfg_3$ to
\[
\msf \equiv y^{a+1}xyx\nsf + \bigO_{n+7}.
\]
Applying these three reduction steps repeatedly leaves us to consider only the cases
where $\msf$ is one of $y^{n+3}xyx$, $y^{n+4}xy$, $y^{n+5}x$ or $y^{n+6}$.
In the first three cases, $\msf$ reduces into $\bigO_{n+7}$ by $y^4\sfg_4$, $\sfh_6y$ and $y\sfh_6$ respectively.
All that remains in degree~$n+6$ is $\updelta\msf=\updelta y^{n+6}$ for some $\updelta\in\C$, so dividing by $1+\be+\be^2\ne0$ proves the lemma.
\end{proof}

Finally, setting $\sfg_7 \colonequals y^{n+4}\mathsf{g}_1 - \sfg_6x = y^{n+7} + \scrO_{n+8}\in J$
completes the proof of \ref{prop:l3finite}\eqref{E6dimi}.

\noindent
(2) It is immediate by comparison with \eqref{eqn:E6full} that the $4(n+3)$ monomials
\[
\renewcommand{\arraystretch}{1.2}
\medmath{
\begin{array}{llllllllllllll}
&&&xyx&yxyx&\hdots&y^{n-2}xyx \\
&&xy&yxy&y^2xy&\hdots&y^{n-1}xy&y^{n}xy&y^{n+1}xy\\
&x&yx&y^2x&y^3x&\hdots&y^{n}x&y^{n+1}x&y^{n+2}x&y^{n+3}x\\
1&y&y^2&y^3&y^4&\hdots&y^{n+1}&y^{n+2}&y^{n+3}&y^{n+4}&y^{n+5}&y^{n+6}
\end{array}
}
\renewcommand{\arraystretch}{1}
\]
are exactly those monomials not divisible by an element of~$\scrS$
and thus a subset forms a basis of $\Jac(g)\cong\Jac(f)$.

Furthermore, it is immediate that every monomial of degree~$n+7$ is divisible by some element of $\scrS$, and so $\n^{n+7}\subset(\scrS)$. Now suppose $h\in\n^{n+7}$ and consider its leading monomial $\msf=\LM(h)$. By definition of $\scrS$, there are monomials $\asf,\bsf$ so that $\msf = \asf\LM(\sfg_i)\bsf$ for some $i\in\{1,\hdots,7\}$. Thus $h-\LC(h)\asf\sfg_i\bsf$ is a sum of terms of degree $\ge n+7$ strictly greater than $\msf$ in the monomial order. Inductively cancelling off leading terms in this way constructs a sequence $h-\sum_{k=1}^{s} \al_k\asf_k\sfg_{i_k}\bsf_k$ which converges to zero as $s\rightarrow\infty$. By this reduction process $h\in\lcl\dcyc g\rcl$, thus $\n^{n+7}\subset\lcl\dcyc g\rcl$. The same then follows for $\lcl\dcyc f\rcl$ as claimed, because $\n/\lcl \dcyc g\rcl$ and $\n/\lcl \dcyc f\rcl$ are the respective Jacobson radicals, and their orders of nilpotency are preserved under isomorphism.
\end{proof}

We take the case $n=4$ separately, since the result is very slightly different:
it has $y^{10}\in\scrS$ rather than~$y^{n+7}=y^{11}$, and so the dimension bound reduces by one.
\begin{prop}\label{prop:l3finite4}
Suppose $f\in\ringtwo$ has $\ell=3$.
Consider $g=x^3+xy^3+y^4 + p$ with $p\in\C\llsq y\rrsq$, $\ord(p)\ge5$ and $\Jac(f)\cong\Jac(g)$,
as given by \ref{cor:SemiFinalAlgLength}. Consider the set
\[
\scrS = \{x^2, xy^2, xyxy, y^3xyx, y^6xy, y^8x, y^{10}\}\subset\ring.
\]
Then the following statements hold:
\begin{enumerate}
\item\label{E6Fourdimi}
$\scrS\subset\LT(J)$ where $J=\lcl\dcyc g\rcl$.
\item\label{E6Fourdimii}
$\dim_\mathbb{C} \ring/\lcl \scrS\rcl = 27$ and  $\n^{10}\subset \lcl \scrS\rcl$.
\end{enumerate}
In particular, $\Jac(f)$ is finite dimensional, $\dim_\mathbb{C}\Jac(f)\le 27$, and $\n^{10}\subset\lcl\dcyc f\rcl$.
\end{prop}
We outline the minor differences for the case $n=4$ from the general case above.
\begin{proof}
As in the proof of \ref{prop:l3finite}, we may write $g=x^3+xy^3+y^4 + \tfrac{1}5\la y^5 + \bigO_6$ and construct particular elements $\sfg_1,\hdots,\sfg_7\in\lcl\dcyc g\rcl$ according to the table:
\[
\renewcommand{\arraystretch}{1.2}
\medmath{\begin{array}{ccclll}
i & \ord & \bigO_k & \textrm{\Spoly} & \qquad\sfg_i+\bigO_k \\
1&2& \infty & \dcyc_xg& x^2 + y^3 \\
2&3& 5 & \dcyc_yg &  xy^2 + yxy + y^2x + y^3 + \la y^4 + \bigO_5   \\
3&4& 6 & x\sfg_2-\sfg_1y^2 & xyxy - yxyx + \la y^3xy - a y^4x  \\
4&6& 8 & xy\sfg_3-\sfg_3xy -y\sfg_3x  &  y^3xyx + \la y^4xyx + \la y^6x + y^7 + \bigO_8  \\
5&8& 9 & y^3xy\sfg_1-\sfg_4x  &  y^6xy - y^7x + \la y^4xy^4 + \bigO_9  \\
6&9& 10 & \tfrac{1}3 (y^6\sfg_2-\sfg_5y - 2y\sfg_5) &  y^8x + \tfrac{1}3 y^9 + \bigO_{10}  \\
7&10& 11 & \tfrac{1}9(y^8\sfg_1-\sfg_6x + \tfrac{1}3y\sfg_6) & y^{10} + \bigO_{11}
\end{array}}
\renewcommand{\arraystretch}{1}
\]
where some reductions of \Spolys\ are indicated, and the $\sfg_3,\sfg_4,\sfg_5$ require further straightforward
reductions by $\sfg_1,\sfg_2$ (substituting for occurrences of $x^2$ and $xy^2$)
to put them in the form presented.
The result follows at once, as in \ref{prop:l3finite}.
\end{proof}

\subsection{Length Four}\label{subsec:l4bound}
If $f\in\ringtwo$ satisfies $\ell=4$, then by \ref{cor:SemiFinalAlgLength}
$\Jac(f)\cong\Jac(g)$ where after harmlessly rescaling we can assume that
\begin{equation}
g = \tfrac{1}3x^3 + \tfrac{1}4y^4 + xyxy^2 + \tfrac{1}3\uplambda(xy)^3 + \tfrac{1}2\upmu(xy^2)^2 +\scrO_{7}.\label{eqn:l4rescaled}
\end{equation} The Jacobian ideal $J=\lcl\dcyc f\rcl$ of $g$ is thus the closed ideal generated by the two elements $\mathsf{g}_1 \colonequals x^2 + yxy^2 + y^2xy + \uplambda yxyxy + \upmu y^2xy^2 + \scrO_{6}$ and $\mathsf{g}_2 \colonequals y^3 + xyxy + xy^2x + yxyx + \uplambda xyxyx + \upmu xy^2xy + \upmu yxy^2x + \scrO_{6}$. 
Consider the element $\mathsf{g}_3\in J$ defined by
\begin{align*}
\mathsf{g}_3 \colonequals x\mathsf{g}_1-\mathsf{g}_1x  \nonumber 
&= 
xyxy^2 + xy^2xy  -yxy^2x  -y^2xyx   \nonumber \\
&\qquad {}+ \uplambda xyxyxy  - \uplambda yxyxyx+ \upmu xy^2xy^2  - \upmu y^2xy^2x + \scrO_{7}
\end{align*}
and also the expression 
\begin{eqnarray*}
 \mathsf{g}_3y-xyx\mathsf{g}_2 &=& 
xy^2xy^2  -yxy^2xy  -y^2xyxy  -xyx^2yxy  -xyx^2y^2x \nonumber\\
&&\ \  {} -xyxyxyx + \uplambda xyxyxy^2  - \uplambda yxyxyxy+ \upmu xy^2xy^3   -\upmu y^2xy^2xy + \scrO_{8} \nonumber\\
&\buildrel{\sfg_1,\sfg_2}\over{\equiv}& xy^2xy^2  -yxy^2xy  -y^2xyxy \nonumber \\
&&\ \  {}-xyxyxyx+ \uplambda xyxyxy^2  - \uplambda yxyxyxy   -\upmu y^2xy^2xy + \scrO_{8}\equalscolon \sfg_4
\end{eqnarray*}
where we use $\mathsf{g}_1$ to substitute for $x^2$, pushing those terms into degree $\ge9$,  and $\mathsf{g}_2$ to substitute for $y^3$ pushing that term into degree~8, and we set $\mathsf{g}_4\in J$ to be the result.
Clearly $\LT(\mathsf{g}_1) = x^2$, $\LT(\mathsf{g}_2) = y^3$, $\LT(\mathsf{g}_3) = xyxy^2$ and $\LT(\mathsf{g}_4) = xy^2xy^2$. 

\begin{notation}
We write $f=g + \bigO_{n,m}$ to mean that $f-g = p_{n+m} + \bigO_{n+m+1}$ where $p_{n+m}$ is homogeneous of
degree $n+m$ and has at most $n$ factors of~$x$ and at least $m$ factors of~$y$.
This notation does not behave well with respect to the monomial order, but in the context here it
elegantly handles terms that we do not need to consider in detail: if $f=(xy)^5 + \bigO_{4,6}$, then
any term of $f_{4+6}$ is either behind $(xy)^5$ in the monomial order, or has a factor of $x^2$ and so reduces by $\sfg_1$ into $\bigO_{11}$, and so {\em after reduction} $\LT(f)=(xy)^5$.
\end{notation}

Consider $\sfh_5\colonequals xy^2\mathsf{g}_4-\mathsf{g}_4xy^2 \in J$ which is 
\begin{eqnarray*}
\sfh_5 &=& xy^2(-yxy^2xy  -y^2xyxy) -(-yxy^2xy  -y^2xyxy)xy^2  +\scrO_{4,6}\nonumber \\
&=&  -xy^3xy^2xy - xy^4xyxy + yxy^2xyxy^2 + y^2xyxyxy^2 +\scrO_{4,6}
\end{eqnarray*}
where the precision bound ${}+\scrO_{4,6}$ holds simply because each term of degree~10 in $\sfh_5$ is a product of
$xy^2$ (with a single $x$ factor) and a term of degree~7 in $\mathsf{g}_4$, all of which have at most~3 factors of~$x$.
The key is to understand how the four degree~9 terms of $\sfh_5$ reduce under $\mathsf{g}_1,\mathsf{g}_2,\mathsf{g}_3,\mathsf{g}_4$. This is taken care of in the following lemma.
\begin{lemma}\label{lem:xy34x}
With notation as above, the following statements hold.
\begin{enumerate}
\item\label{lem:xy3x}
$xy^3x\equiv\scrO_{8}$ mod $(\mathsf{g}_1,\mathsf{g}_2)$, and so $xy^3xy^2xy\equiv \scrO_{12}$.
\item\label{lem:xy4x}
$xy^4x\equiv -(xy)^3x +\scrO_{8}$ mod $(\mathsf{g}_1,\mathsf{g}_2)$, and so $-xy^4xyxy\equiv(xy)^5+\scrO_{11}$.
\item\label{lem:D3C}
$yxy^2xyxy^2\equiv -y^2xy^2(xy)^2+\scrO_{4,6}$ mod $(\mathsf{g}_1,\mathsf{g}_2,\mathsf{g}_3)$.
\item\label{lem:D4C}
$-y^2xyxyxy^2\equiv y^2xy^2(xy)^2+\scrO_{4,6}$ mod $(\mathsf{g}_1,\mathsf{g}_2,\mathsf{g}_3,\mathsf{g}_4)$.
\end{enumerate}
Therefore
$\sfh_5$ reduces under $\sfg_1,\sfg_2,\sfg_3,\sfg_4$ to an element $\sfg_5\in J$ with $\LT(\sfg_5)=(xy)^5$.
\end{lemma}
\begin{proof}
\noindent
(1) Since every non-leading term of $\mathsf{g}_2$ starts or ends with $x$, each of the terms of $xy^3x - x\mathsf{g}_2x$ is divisible by $x^2$ and has degree $\ge6$. 
Substituting for $x^2$ using $\mathsf{g}_1$ pushes each term into $\scrO_{8}$, since $\mathsf{g}_1-x^2$ has order~4.
Thus $xy^3x\equiv\scrO_{8}$ modulo $\mathsf{g}_1$ and $\mathsf{g}_2$.
Multiplying that by $y^2xy$ on the right completes the claim.\\
%Substituting for $y^3$ using $B$ gives a sum of terms, each of degree at least $e+6$,  up to precision $+\bigO(e+8)$. Since every term of $B-y^3$ starts or ends with $x$, each of the terms is divisible by $x^2$. Substituting for $x^2$ in each term using $A$ pushes each term into $\bigO(e+8)$.
\noindent
(2) The expression $xy^4x - xy\mathsf{g}_2x$ has order~7, and every term of degree 7 is divisible by $x^2$
except the term $-xy(xyxy)x$ that arises from the unique term $xyxy$ of $\mathsf{g}_2$ in degree~4 that does not end in~$x$.
%(The term $\upmu xy^2xy$ does not lead to an $x^2$ factor, but it leads to a term in degree~8.) 
Substituting for $x^2$ using $\mathsf{g}_1$ pushes those terms into degree at least~9, giving the result.\\
(3) Reducing $xyxy^2$ by $\sfg_3$ and then applying parts (\ref{lem:xy3x}--\ref{lem:xy4x}) gives
\[
yxy^2(xyxy^2) \buildrel{\sfg_3}\over\equiv yxy^2(-xy^2xy+yxy^2x+y^2xyx) + \scrO_{4,6}  \equiv -yxy^2xy^2xy+\scrO_{4,6}
\]
and reducing that by $\sfg_4$ and then by $\sfg_1,\sfg_2$ gives
\[
-y(xy^2xy^2)xy \buildrel{\sfg_4}\over{\equiv}y^2xy^2(xy)^2 + y^3(xy)^3 + y(xy)^3x^2y + \scrO_{4,6} 
\equiv y^2xy^2(xy)^2 + \scrO_{4,6}
\]
where the precision bound holds because all terms of degree~4 in $\mathsf{g}_1$ have 1 factor of~$x$ and those of $\mathsf{g}_2$ have 2 factors of~$x$.

\noindent
(4) Reducing $xyxy^2$ by $\sfg_3$ twice followed by $\sfg_4$ (and applying \eqref{lem:xy3x} to push $xy^3x$ and $\sfg_1$ to push $x^2$ into higher degree) gives
\begin{eqnarray*}
y^2xy(xyxy^2) &\buildrel{\mathsf{g}_3}\over{\equiv}&-y^2(xyxy^2)xy+y^2xy^2xy^2x+y^2\cancel{xy^3x}yx+\scrO_{4,6} \nonumber\\
%&\buildrel{\mathsf{g}_3}\over{\equiv}&y^2(-xy^2xy+yxy^2x+y^2xyx)xy+y^2xy^2xy^2x+\scrO_{4,6} \nonumber\\
&\buildrel{\mathsf{g}_3}\over{\equiv}& \big({-y^2xy^2xyxy} + y^3xy^2\cancel{x^2y} + y^4xy\cancel{x^2y}\big) + y^2(xy^2xy^2)x +\scrO_{4,6}\nonumber\\
%&\buildrel{\mathsf{g}_1}\over{\equiv}& -y^2xy^2xyxy + y^2(xy^2xy^2)x +\scrO_{4,6}\nonumber\\
%&\buildrel{\mathsf{g}_4}\over{\equiv}&-y^2xy^2xyxy + y^2((\mathsf{g}_4)_{\le7}-(xy^2)^2)x +\scrO_{4,6}\nonumber\\
&\buildrel{\sfg_4}\over{\equiv}&-y^2xy^2xyxy + \scrO_{4,6}
\end{eqnarray*}
where the last line holds because (a) every non-leading term $m$ of degree~6 in $\mathsf{g}_4$ starts in $y$ (so $y^2mx$ pushes into degree~10 mod~$\mathsf{g}_2$ with at most 4 factors of $x$), and (b) 
every term of degree~7 in $\mathsf{g}_4$ either either ends in $x$ (so $y^2mx$ pushes into degree~12 mod~$\mathsf{g}_1$),  or starts in $y$ (so $y^2mx$ pushes into degree~11 mod~$\mathsf{g}_2$).
\end{proof}

With the above preparation, the following is clear.
\begin{prop}\label{prop:l4finite}
Suppose $f\in\ringtwo$ has $\ell=4$. Consider $g$ with $\Jac(f)\cong\Jac(g)$, as given by \ref{cor:SemiFinalAlgLength} rescaled as in \eqref{eqn:l4rescaled}.  Given the set 
\[
\scrS = \{x^2, y^3, xyxy^2, (xy^2)^2, (xy)^5\}
\]
the following statements hold.
\begin{enumerate}
\item\label{E7dimi}
$\scrS\subset\LT(J)$ where $J=\lcl\dcyc g\rcl$.
\item\label{E7dimii}
$\dim_\mathbb{C} \ring/\lcl \scrS\rcl = 60$ and  $\n^{15}\subset \lcl \scrS\rcl$.
\end{enumerate}
In particular, $\Jac(f)$ is finite dimensional, $\dim_\mathbb{C}\Jac(f)\le 60$, and $\n^{15}\subset\lcl\dcyc f\rcl$.
\end{prop}
\begin{proof}
\eqref{E7dimi} By above, $\scrS$ is the set of leading terms of $\mathsf{g}_1,\mathsf{g}_2,\mathsf{g}_3,\mathsf{g}_4,\mathsf{g}_5\in J$. \\
\noindent
\eqref{E7dimii}
There are exactly $60=2\times2\times3\times5$ monomials $y^\al(xy^2)^\be(xy)^\ga x^\varepsilon$ with $\be,\varepsilon\le1$, $\al\le2$, $\ga\le4$  that are not divisible by one of the monomials of $\scrS$.
In particular, since these $60$ monomials have degrees in the range $\left\{0,1,\hdots,14\right\}$, 
any monomial $\msf\in\mR^{15}$ is divisible by one of the monomials of $\scrS$.
\end{proof}

\subsection{Length Five}\label{subsec:l5bound}
This is substantially easier than the other cases.

\begin{prop}\label{prop:l5finite}
Suppose $f\in\ringtwo$ has $\ell=5$.
Consider $g= \tfrac{1}3x^3 + \tfrac{1}2(xy)^2 + \tfrac{1}5y^5 + \tfrac{1}2\upalpha(xy^2)^2 + \tfrac{1}6\upbeta y^6+ \scrO_{7}$ with $\Jac(f)\cong\Jac(g)$ as given by \ref{cor:SemiFinalAlgLength} after rescaling.  Given the set 
\[
\scrS = \{x^2, xyx, xy^4, xy^2xy, xy^3xy^2x, y^4xy^3x, y^8xy^3, y^9xy^2x, y^{13}x, (xy^3)^4x, y^{21}\}
\]
the following statements hold.
\begin{enumerate}
\item\label{E85dimi}
$\scrS\subset\LT(J)$ where $J=\lcl\dcyc g\rcl$.
\item\label{E85dimii}
$\dim_\mathbb{C} \ring/\lcl \scrS\rcl = 125$ and  $\n^{21}\subset \lcl \scrS\rcl$.
\end{enumerate}
In particular, $\Jac(f)$ is finite dimensional, $\dim_\mathbb{C}\Jac(f)\le 125$, and $\n^{21}\subset\lcl\dcyc f\rcl$.
\end{prop}
\begin{proof}
(1) The Jacobian ideal $J=\lcl\dcyc g\rcl$ of $g$ is the closed ideal generated by the two elements
$\mathsf{g}_1\colonequals x^2 + yxy + \upalpha y^2xy^2 + \scrO_{6}$ and $
\mathsf{g}_2 \colonequals xyx + y^4 + \upalpha xy^2xy + \upalpha yxy^2x + \upbeta y^5 + \scrO_{6}$. Therefore $J$ also contains
\[
x\mathsf{g}_2  - \mathsf{g}_1yx =  xy^4 +  \upalpha \cancel{x^2y^2}xy + \upalpha \cancel{xyx}y^2x  - yxy^2x + \bigO_7 \buildrel{\sfg_1,\sfg_2}\over\equiv xy^4 - yxy^2x + \scrO_{7} \equalscolon\sfg_3
\]
and similarly
\[
xy\mathsf{g}_1 - \mathsf{g}_2x = xy^2xy  - y^4x -  \upalpha xy^2\cancel{xyx} - \upalpha yxy^2\cancel{x^2} + \bigO_7  \buildrel{\sfg_1,\sfg_2}\over\equiv xy^2xy - y^4x + \scrO_{7} \equalscolon \mathsf{g}_4.
\]
Therefore $J$ contains elements of all the following forms, since the right-hand expression
is evidently in $J$, and congruence modulo $\mathsf{g}_1,\mathsf{g}_2$ preserves membership of $J$.
\[
\renewcommand{\arraystretch}{1.3}
\medmath{
\begin{array}{c@{\hspace{2pt}}c@{\hspace{2pt}}l@{\hspace{2pt}}c@{\hspace{2pt}}l@{\hspace{2pt}}c@{\hspace{2pt}}c@{\hspace{2pt}}}
\mathsf{g}_5 &\colonequals& xy^3xy^2x - y^4xy^3 + \scrO_{10} & =& -xy^2\mathsf{g}_3 + \mathsf{g}_4y^3 \\
\mathsf{g}_6 &\colonequals& y^4xy^3x + (xy^3)^2xy + \scrO_{11} & =& - \mathsf{g}_5x + xy^3xy^2\mathsf{g}_1  \\
\mathsf{g}_7 &\colonequals& y^8xy^3 + (xy^3)^3x + \scrO_{14} & =& -y^4\mathsf{g}_5 + \mathsf{g}_6y^2x  \\
\mathsf{g}_8 &\colonequals& y^9xy^2x + (xy^3)^3xy + \scrO_{15} & =& -y^8\mathsf{g}_3 + \mathsf{g}_7y \\
\mathsf{g}_9 &\colonequals& y^{13}x + (xy^3)^3xy^2 + \scrO_{16} & =& -y^9\mathsf{g}_4 + \mathsf{g}_8y  \\
\mathsf{g}_{10} &\colonequals& (xy^3)^4x - y^{17} + \scrO_{18} & =& \phantom{-}\mathsf{g}_9yx - y^{13}\mathsf{g}_2  \\
\mathsf{g}_{11} &\colonequals& y^{21}  + \scrO_{22} & =& - y^4\mathsf{g}_{10} + \mathsf{g}_6(y^3x)^3 
\end{array}}
\renewcommand{\arraystretch}{1}
\]
The leading terms of $\mathsf{g}_1,\hdots,\mathsf{g}_{11}\in J$ are the 11 monomials of $\scrS$, which proves \eqref{E85dimi}.\\
\eqref{E85dimii}
Given (1), this is a very elementary count, and we refrain from giving details here.
\end{proof}

\subsection{Length Six}\label{subsec:l6bound}
\begin{prop}\label{prop:l6finite}
Suppose $f\in\ringtwo$ has $\ell=6$.  As given by \ref{cor:SemiFinalAlgLength}, after rescaling consider  $g=\tfrac{1}3x^3 + \tfrac{1}4y^4 + \tfrac{1}3(xy)^3 + \tfrac{1}2\uplambda(xy^2)^2 + \upmu(xy)^3y+ \tfrac{1}4\upnu(xy)^4 + \scrO_9$ with $\Jac(f)\cong\Jac(g)$, where $\uplambda,\upmu,\upnu\in\C$ and $\uplambda\not=0$.  Given the set 
\begin{align*}
\scrS = &\{
    x^2,
    y^3,
    (xy)^3,
    xy(xy^2)^2,
    xyxy^2(xy)^2x,
    (xy^2)^2(xy)^2x,
    (xy^2)^4, \\
    &\qquad\qquad\qquad\qquad
    y(xy^2)^3(xy)^2y,
    (xy^2)^3(xyxy^2)^2,
    (xyxy^2)^4,
    (xy^2xy)^4xy
\}
\end{align*}
the following statements hold.
\begin{enumerate}
\item\label{E86dimi}
$\scrS\subset\LT(J)$ where $J=\lcl\dcyc g\rcl$.
\item\label{E86dimii}
$\dim_\mathbb{C} \ring/\lcl \scrS\rcl = 200$ and  $\n^{26}\subset \lcl \scrS\rcl$.
\end{enumerate}
In particular, $\Jac(f)$ is finite dimensional, $\dim_\mathbb{C}\Jac(f)\le 200$, and $\n^{26}\subset\lcl\dcyc f\rcl$.
\end{prop}
\begin{proof}
(1) We outline the elementary steps, which can be carried out by hand.

\medskip
\paragraph{\bf Step 1}
Set $\sfh_1=\dcyc_xf$ and $\sfh_2=\dcyc_yf$, and then define power series
$\sfh_3,\hdots,\sfh_{7}\in (\dcyc f)$ via the following \Spolys\ (with an additional reduction in the case of $\sfh_6$) after reduction by $\sfh_1,\sfh_2$ in the leading polynomials, as follows.
\[
\medmath{\begin{array}{ccclll}
 & \ord & \bigO_k & \textrm{\Spoly} & \textrm{leading polynomial of $\sfh_i$} \\
\sfh_1&2& 8 & & x^2  \\
\sfh_2&3& 8 & &  y^3   \\
\sfh_3&6& 9 & x\sfh_1-\sfh_1x & xyxyxy + \la xy^2xy^2 - yxyxyx - \la y^2xy^2x  \\
\sfh_4&8& 11 & \tfrac{1}\la \big(xy\sfh_3-\sfh_3xy)  &  xyxy^2xy^2 - \tfrac{1}{\la }xy^2xyxyx - xy^2xy^2xy  \\
\sfh_5&10& 13 & -\la (xy\sfh_4-\sfh_3yxy^2)  &  xyxy^2(xy)^2x + \la x(yxy)^3 - \la (yx)^5 - 
        \la^2(y^2x)^2yxy^2   \\
\sfh_6&11& 14 & -\tfrac{1}\la (xyxy^2xyxy\sfh_1-\sfh_5x) &   \\
&& & \hfill -xyxy^3xy^2\sfh_1- \sfh_4xyx &  (xy^2)^2(xy)^2x - (yx)^4y^2x - \la (y^2x)^2yxy^2x  \\
\sfh_7&12& 15 & \tfrac{1}\la (xy^2xy^2\sfh_3-\sfh_6y) & (xy^2)^4 + \tfrac{1}{\la }(yx)^4y^2xy + 
        (y^2x)^2yxy^2xy 
\end{array}}
\]
These realise the first 7 elements of $\scrS$ as their leading terms.

\medskip
\paragraph{\bf Step 2}
Inter-reduce $\sfh_1,\hdots,\sfh_7$ by their 7 leading terms to determine 
$\sfg_1,\hdots,\sfg_7\in (\dcyc f)$ as follows.
\[
\renewcommand{\arraystretch}{1.2}
\medmath{\begin{array}{l}
\sfg_1 =
x^2 + yxyxy + \la y^2xy^2 + \mu yxyxy^2 + \mu yxy^2xy + 
        \mu y^2xyxy - \la {\upnu}y(xy^2)^2 + {\upnu}y^2xyxyx + \bigO_8
\\
\sfg_2 =
y^3 + xyxyx + \la xy^2xy + \la yxy^2x + \mu xyxy^2x + \mu xy^2xyx 
        - \la \mu (xy^2)^2 + 2\mu yxyxyx + \la \mu y^2xy^2x\\ 
        \qquad\quad - 
        \mu^2xyxy^2xy - \mu^2xy^2xyxy - \la {\upnu}(xy^2)^2x + 
        \mu^2yxyxy^2x  + \mu^2yxy^2xyx \\ 
        \qquad\qquad+ \la \mu^2y(xy^2)^2 - 
        \la {\mu}^2y^2xy^2xy + \bigO_8
\\
\sfg_3 =
(xy)^3 + \la (xy^2)^2 - yxyxyx - \la y^2xy^2x + {\mu}xyxy^2xy 
        + {\mu}xy^2xyxy - {\mu}yxyxy^2x - {\mu}yxy^2xyx  
        \\\qquad{} -\la {\mu}y(xy^2)^2 + \la {\mu}y^2xy^2xy - \tfrac{{\mu}^2}{\la }xy^2xyxyx - 
        {\mu}^2xy^2xyxy^2 + (-\la {\upnu} - {\mu}^2)(xy^2)^2xy + \la {\upnu}y(xy^2)^2x \\\qquad{}\quad 
        + {\mu}^2y^2xyxy^2x + {\mu}^2y^2xy^2xyx + \bigO_9
\\
\sfg_4 =
xy(xy^2)^2 - \tfrac{1}{\la }xy^2xyxyx - (xy^2)^2xy - 
        \tfrac{{\mu}}{\la }xy^2xyxy^2x - \tfrac{{\mu}}{\la }(xy^2)^2xyx  - 
        \tfrac{{\mu}^2}{\la }xyxy^2xyxy^2 \\\qquad{}+ \tfrac{{\mu}^2}{\la }xy^2xyxy^2xy + 
        \tfrac{{\mu}^2}{\la }(xy^2)^2xyxy + {\upnu}(xy^2)^2xy^2x + 
        {\mu}^2y(xy^2)^3 - {\mu}^2y^2xy^2xyxy^2 + \bigO_{11}
\\
\sfg_5 =
xyxy^2xyxyx + \la (xy^2)^2xyxy + \la^2y(xy^2)^3 - 
        \la^2y^2xy^2xyxy^2 + {\mu}xyxy^2xyxy^2x + 
        \la {\mu}(xy^2)^2xyxy^2 \\\qquad{}+ \la {\mu}(xy^2)^2xy^2xy - 
        2\la {\mu}y(xy^2)^3x + 2\la {\mu}y^2xy^2xyxy^2x - 
        {\mu}^2xy^2xyxy^2xyxy \\\qquad{}+ (-\la {\upnu} - {\mu}^2)(xy^2)^2xyxy^2x  - 
        {\mu}^2(xy^2)^3xyx - 2\la {\mu}^2y(xy^2)^2xyxy^2 + (-\la^2{\upnu} 
        - \la {\mu}^2)y(xy^2)^3xy \\\qquad{}+ \la {\mu}^2y^2xyxy^2xyxy^2 + 
        \la^2{\upnu}y^2xy^2xyxy^2xy + \la {\mu}^2y^2(xy^2)^3x + \bigO_{13}
\\
\sfg_6 =
(xy^2)^2xyxyx + \la y(xy^2)^3x - \la y^2xy^2xyxy^2x + 
        {\mu}(xy^2)^2xyxy^2x + {\mu}(xy^2)^3xyx + 
        {\mu}^2(xy^2)^3xyxy \\\qquad{}- 2{\mu}^2y(xy^2)^2xyxy^2x + (-\la {\upnu} 
        - {\mu}^2)y(xy^2)^3xyx + {\mu}^2y^2xyxy^2xyxy^2x + 
        \la {\upnu}y^2xy^2xyxy^2xyx + \bigO_{14}
\\
\sfg_7 =
(xy^2)^4 - y(xy^2)^3xy + y^2xy^2xyxy^2xy 
        + \la xy^2xyxy^2xyxy^2x  - 2\tfrac{{\mu}^2}{\la }(xy^2)^3xyxy^2 \\\qquad{}+ 
        2\tfrac{{\mu}^2}{\la }y(xy^2)^2xyxy^2xy - \tfrac{{\mu}^2}{\la }y^2xyxy^2xyxy^2xy + \tfrac{{\mu}^2}{\la }y^2xy^2xyxy^2xyxy + \bigO_{15}
\end{array}}
\renewcommand{\arraystretch}{1}
\]

\paragraph{\bf Step 3}
Determine power series $\sfh_8,\hdots,\sfh_{11}\in (\dcyc f)$, and their reductions $\sfg_8,\hdots,\sfg_{11}\in (\dcyc f)$ as follows.
Each reduction $\equiv$ only uses $\sfg_1,\sfg_2$ to simplify the leading polynomial.
\[
\renewcommand{\arraystretch}{1.2}
\medmath{\begin{array}{r@{\hspace{4pt}}c@{\hspace{4pt}}l}
&&\textrm{\Spoly\ $+$ small term reductions $\equiv$ leading polynomial ${}+\hot$}\\
\sfh_8 &\colonequals& xy^2\sfg_7-\sfg_7xy^2 + x\sfg_2(xy^2)^3xy
        - x\sfg_2yxy^2xyxy^2xy \\
        &=& y(xy^2)^3xyxy^2 - y^2xy^2xyxy^2xyxy^2 +\hdots + \bigO_{18} \equalscolon \sfg_8\\
\sfh_9 &\colonequals& -\tfrac{1}\la (y^2\sfg_8-\sfg_2(xy^2)^3xyxy^2
+ \sfg_2yxy^2xyxy^2xyxy^2 - \sfg_3xy^2xyxy^2xyxy^2) \\
      &&   \hspace{8em}- xy^2\sfg_4xy^2xyxy^2
        - \tfrac{1}{\la }xyxy\sfg_1y^2(xy^2)^2xyxy^2
        - \tfrac{1}{\la }(xy^2)^2xyxy\sfg_1y^2xyxy^2 \\
&\equiv&
(xy^2)^3(xyxy^2)^2 - (yxy^2x)^3yxy^2 +
    (y^2xyx)^3y^2xy - y^2(xy^2)^2xyxy^2xyxy^2x + \bigO_{20}\equalscolon \sfg_9 \\
\sfh_{10} &\colonequals&  
-x\sfg_9 + \sfg_1y^2(xy^2)^2xyxy^2xyxy^2
        - \tfrac{1}{\la }xyxyxy\sfg_1y^2xyxy^2xyxy^2 + xyxy^2\sfg_1y^2(xy^2)^2xyxy^2
\\
&\equiv&
(xyxy^2)^4 - 
    (xy^2xy)^4 + 
    (yxy^2x)^4 - 
    (y^2xyx)^4 + \bigO_{21} \equalscolon \sfg_{10}\\
\sfh_{11} &\colonequals&  
-xy\sfg_{10} + \sfg_3y(xyxy^2)^3 - \sfg_{10}xy
        + \sfg_4xy^2x(yxy^2x)^2 + 2\tfrac{{\mu}^2}{\la^2}xyxy\sfg_8xy^2
        - 2\tfrac{{\mu}^2}{\la^2}xyxy^2xy\sfg_8 \\
&\equiv&        
(xy^2xy)^4xy - xy(yxy^2x)^4 + \bigO_{23}\equalscolon \sfg_{11}
\end{array}}
\renewcommand{\arraystretch}{1}
\]
The key, non-trivial, part of the above is the construction of $\sfh_9$, where the stated \Spoly\ naturally
has order $17$, but all terms of degrees $17$ and $18$ reduce as indicated
to leave a power series of order~$19$.

The elements of $\scrS$ are clearly the leading terms of $\sfg_1,\hdots,\sfg_{11}$.
\end{proof}

\section{Necklaceification}\label{sec:NecklacIFY}

This section proves that if $f$ has algebraic length, then $\Jac(f)\cong\Jac(g)$ where $g$ is a finite linear combination of necklace polynomials, recalled in \S\ref{subsec:DefNeck}.  In many respects this section is the technical heart of the paper, since the algebraic approach above needs to be interfaced directly with geometry, which is precisely what necklace polynomials allow.  The concluding result of the section, \ref{cor:NecklaceSummary},  is particularly remarkable as there is no obvious reason for why the results in this section should be true.

In \S\ref{subsec:IntrinsMain} we will use \ref{cor:NecklaceSummary} to directly argue that each $\Jac(f)$ of algebraic length $\ell$ is the contraction algebra of some smooth $3$-fold flop, or crepant divisor-to-curve contraction.

\subsection{Necklace Polynomials}\label{subsec:DefNeck}
Given two integers $j,k\geq 0$ such that $n\colonequals j+k\geq 1$, consider the set $\Mono_{j,k}$ consisting of degree $n$ monomials in $\ringtwo$ which have precisely $j$ occurrences of $x$, and thus precisely $k$ occurrences of $y$.  The cyclic group $\mathbb{Z}_n$ acts on the set $\Mono_{j,k}$ via rotation,  partitioning $\Neck_{j,k}$ into the set of cyclic equivalence classes, written $\Orb_{j,k}$.  To each $m\in\Orb_{j,k}$ we choose the unique representative $p_m$ which is smallest with respect to the local graded monomial order $<$ in \eqref{eqn:order}.  The \emph{free necklace polynomial} $\NeckPoly_{j,k}$ is defined to be
\[
\NeckPoly_{j,k}(x,y)\colonequals  \frac{1}{j+k}\sum_{m\in\Orb_{j,k}}|m| \cdot p_m\quad\in\freering
\]
For calibration, and also since they are used often below, $\NeckPoly_{3,0}=\tfrac{1}{3}x^3$, whereas 
\[
\begin{array}{c@{\hspace{2pt}}c@{\hspace{2pt}}l@{\hspace{40pt}}r@{\hspace{2pt}}c@{\hspace{2pt}}l@{\hspace{2pt}}}
\NeckPoly_{2,2} &=& x^2y^2 + \tfrac{1}{2}xyxy & \NeckPoly_{2,4} &=& x^2y^4 + xyxy^3 +\tfrac{1}{2}xy^2xy^2\\
\NeckPoly_{2,3} &=& x^2y^3 + xyxy^2& \NeckPoly_{3,3} &=& x^3y^3 + x^2yxy^2 + x^2y^2xy + \tfrac{1}{3}xyxyxy
\end{array}
\]
 
\medskip
Returning to the aim of the section, it is evident by inspection of the list in  \ref{cor:SemiFinalAlgLength} that $\Jac(f)$, for every $f$ of length $\ell=1,2$, is isomorphic to $\Jac(g)$ where $g$ is a finite sum of necklace polynomials. Further, after some work that establishes that the $p(y)$ in \ref{cor:SemiFinalAlgLength} is a polynomial and not a power series (done in \ref{cor:NecklaceSummary} below), the same is true for $\ell=3$.  The remainder of this section considers the cases, in order, $\ell=5,4,6$. Each has very different properties, and so extreme care is required.

\subsection{Length Five} \label{subsec:L5fullfunnel}
Given any $f$ with $\ell=5$,  by \ref{thm:ReductionPrecise} $\Jac(f)\cong\Jac(x^3+xyxy+y^5+\scrO_6)$.  The first problem is that $xyxy$ is not a necklace polynomial.  The following fixes this, and should be viewed as the necklace version of \ref{prop!caseC} (where in the proof there, $\upgamma\neq 0$).

\begin{lemma}\label{lem:Ell5necksmall}
If $f\in\ringtwo$ has $\ell=5$, then $\Jac(f)\cong\Jac(x^3+\NeckPoly_{2,2}+y^5+\scrO_6)$.
\end{lemma}
\begin{proof}
We have $f\cong x^3+xyxy+y^5+\scrO_6$ by \ref{thm:ReductionPrecise}, and applying the scaling automorphism $x\mapsto 32x$ and $y\mapsto 8y$
%\marginpar{this scaling doesn't seem correct, but it is!!} 
shows that $f\cong x^3+\frac{1}{2}xyxy+y^5+\scrO_6$. Then sending $x\mapsto x + \tfrac{1}{3}y^2$ 
\begin{equation}
f\cong x^3+(\tfrac{1}{2}xyxy+x^2y^2)+\tfrac{2}{3}xy^4+y^5+\scrO_6=x^3+\NeckPoly_{2,2}+\tfrac{2}{3}xy^4+y^5+\scrO_6\label{eqn:ell5funnelneckprep}
\end{equation}
That is, $xyxy$ has turned into $\NeckPoly_{2,2}$, at the cost of an $xy^4$ term.  The proof now simply follows the proof of \ref{prop!caseC}, with slightly different coefficients that account for the new $x^2y^2$ term. Summarising, applying $x\mapsto x - \tfrac{1}{3}(\tfrac{2}{75}x^2 - \tfrac{2}{5}xy)$ and $y\mapsto y -\tfrac{2}{15}x - \tfrac{8}{45}y^2$ to the right hand side of \eqref{eqn:ell5funnelneckprep} gives $f\cong x^3+\frac{1}{2}xyxy+y^5+x^2h_3+\scrO_6$ for some degree three $h_3$.  One final application of Tschirnhausen \ref{tschirnhausen} pushes $x^2h_3$ into higher degree, as required.
\end{proof}

For the remainder of this subsection, $f_{\rm stub}\colonequals x^3+\NeckPoly_{2,2}+y^5$.  Given this, we now recycle the notation $\to$ from 
\ref{tonotation}, noting that the clearing lemmas \ref{lem:E6funnelA}--\ref{lem:E6funnelD} are no longer valid since $f_{\rm stub}$ has changed.  The following lemma replaces these, for the new $f_{\rm stub}$.  Recall that $(a_1,\hdots,a_t)$ denotes $xy^{a_1}\hdots xy^{a_t}$, and below we will also sometimes drop commas and brackets and write simply $a_1\hdots a_t$, provided that no confusion arises.

\begin{lemma}\label{Length5Neckpush}
For $f_{\rm stub}\colonequals x^3+\NeckPoly_{2,2}+y^5$, the following statements hold.
\begin{enumerate}
\item\label{Length5Neckpush1} Given any element $F$ of degree $n-3$,  $xyxF\to\scrO_{n+1}$.
\item\label{Length5Neckpush2} If $t=1$ and $a_1\geq 5$, then $(a_1)\to \tfrac{1}{5}(2,a_1-3),\tfrac{2}{15}(3,a_1-4)$.
\item\label{Length5Neckpush3} It $t\geq 2$ and $a_t=4$, then $(a_1,a_2,\hdots,a_t) \to \scrO_{n+1}$.
\item\label{Length5Neckpush4} If $t\geq 2$ and $a_t\geq 5$, then $(a_1,a_2,\hdots,a_t) \to$
\[
\tfrac{2}{15}(2, a_1, a_2,\hdots,a_{t-1},a_{t}-3), \tfrac{2}{15}(3, a_1, a_2,\hdots,a_{t-1},a_t-4), \tfrac{1}{15}(2, a_1+1,a_2,\hdots,a_{t-1},a_t-4)
\]
\item\label{Length5Neckpush5}  If $t\geq 3$ and $(a_1,\hdots,a_t)$ can be rotated so that there are two consecutive $2$'s, then  $(a_1,\hdots,a_t)\to\scrO_{n+1}$.
\item\label{Length5Neckpush6} $(2323)\to\scrO_{15}$, $(23233)\to\scrO_{19}$, and $(232323)\to\scrO_{22}$. 
\end{enumerate}
\end{lemma}
\begin{proof}
(1) Sending $y\mapsto y - \upalpha F$ subtracts $\upalpha xyxF$ in degree $n$, at the cost of adding the degree $n$ term $-\upalpha x^2Fy-\upalpha x^2yF$.  Tschirnhausen \ref{tschirnhausen} pushes this term into $ \scrO_{n+1}$. \\
(2) Set $n=a_1+1$ and consider $\upphi$ which sends $y\mapsto y - \frac{\upalpha}{5} xy^{a_1-4}$. Then
\[
\upphi(f_{\stub})\sim f_{\stub} - \tfrac{\upalpha}{5}x^2(xy^{a_1-3} + yxy^{a_1-4} +y^{a_1-4}xy)-\upalpha xy^{a_1}+x^2h_{n-2}+\scrO_{n+1}
\]
where $h_{n-2}$ exists and has degree $n-2$ if and only if $a_1=5$. 
%Thus the degree $n$ term $xy^{a_1}$ has been removed at the cost of the three terms of degree $n-1$, plus possibly some $x^2h_{n-2}$ which will play no role.  
Applying $\uppsi\colon x\mapsto x + \tfrac{\upalpha}{15}(xy^{a_1-3} + yxy^{a_1-4} +y^{a_1-4}xy)$ then gives
\[
\uppsi\upphi(f_{\stub})\sim (f_{\stub}-\upalpha xy^{a_1}) +\tfrac{\upalpha}{5}xyxy^{a_1-2} + \tfrac{\upalpha}{5}xy^2xy^{a_1-3} + \tfrac{2\upalpha}{15}xy^3xy^{a_1-4} + x^2g_{n-2} + \scrO_{n+1}
\]
for some $g_{n-2}$ of degree $n-2$.  Applying \eqref{Length5Neckpush1} then pushes $xyxy^{a_1-2}\to\scrO_{n+1}$, and applying Tschirnhausen \ref{tschirnhausen} pushes $ x^2g_{n-2}\to\scrO_{n+1}$.  The coefficients in the statement are read from the displayed equation above, using the convention in \ref{tonotation}\eqref{tonotation1}.\\
(3)  % see deg 13
Write $(a_1,\hdots,a_t)=(x\msf x)y^4$ and send 
$y\mapsto y - \tfrac{\upalpha}{5} x\msf x$, $x\mapsto x +\tfrac{\upalpha}{15}(x\msf xy + yx\msf x + \msf x^2y)$.
% same formula as in (4) below, with C=\msf x
This subtracts $\upalpha y^4(x\msf x)$, at the cost of adding in terms of degree $n$ that are all divisible by $x^2$ or $xyx$. Applying \eqref{Length5Neckpush1} and Tschirnhausen \ref{tschirnhausen} pushes those into $ \scrO_{n+1}$.\\
(4) Write $(a_1,\hdots,a_t)= x\msf y^4$, where $\msf \colonequals y^{a_1}(a_2,\hdots,a_{t-1},a_t-4)$, then in a manner very similar to parts \eqref{Length5Neckpush2} and \eqref{Length5Neckpush3}, send $y\mapsto  y - \tfrac{\upalpha}{5} x\msf $, and $x\mapsto x +\tfrac{\upalpha}{15}(x\msf y + yx\msf  + \msf xy)$.\\
(5) Since $t\geq 3$, we may rotate and write $(a_1,\hdots,a_t)=xy^2x\msf xy^2$ for some $\msf $.  Applying $x\mapsto  x - {\upalpha}yx\msf xy$ subtracts $\upalpha xy^2x\msf xy^2$ at the cost of adding $-3\upalpha x^2 yx\msf xy$ in degree $n-1$, and  also adding $- {\upalpha}xyxy^3x\msf  - {\upalpha}xyx\msf xy^3$ in degree $n$ . Note that both of the latter terms are divisible by $xyx$.  The trick is to rotate $- 3\upalpha x^2 yx\msf xy \sim -3\upalpha xyx\msf xyx$, so then applying $y\mapsto y + 3\upalpha \msf xyx$ replaces $- 3\upalpha x^2 yx\msf xy \sim -3\upalpha xyx\msf xyx$ in degree $n-1$ with $+3\upalpha x^2(\msf xyxy + y\msf xyx)$, at the cost of adding in an additional  term in degree $n$, also divisible by $xyx$.  Finally, applying $x\mapsto x-\upalpha (\msf xyxy + y\msf xyx)$ removes $+3\upalpha x^2(\msf xyxy + y\msf xyx)$ and thus everything from degree $n-1$, at the cost of adding in yet more terms of degree $n$, all of which (simply by the form of the last automorphism) are all divisible by $xyx$.  Thus, overall, we have removed $\upalpha xy^2x\msf xy^2$ from degree $n$, and replaced it by a sum of terms, all in degree $n$, which are all divisible by $xyx$. One final application of \eqref{Length5Neckpush1} yields the result.\\
 (6) By \eqref{Length5Neckpush4} we have $(227) \to (2224), (3223),(2323)$.  But by \eqref{Length5Neckpush5} all $(227)$, $(2224)$, $(3223)$ clear into higher degree, thus so does $(2323)$.  The other two statements follow using $(2237)\to (22234), (32233), (23233)$ and   $(22327)\to (222324),(322323),(232323)$.
\end{proof}

\begin{thm}\label{thm!funnel5}
Set $f_{\rm stub}=x^3+\scrN_{2,2}+y^5$.  If $g= f_{\rm stub}+ \scrO_6$, then there exists $p,q\in\C [y]$  such that $\Jac(g)\cong \Jac(x^3+\scrN_{2,2}+y^5+xy^5p+y^6q)$.
\end{thm}
\begin{proof}
The proof first proceeds inductively through each degree $n=6,\hdots,21$ in order, examining all the monomials up to cyclic equivalence of degree $n$.  We will prove that for every such fixed $n$, we can push every monomial except possibly $xy^{n-1}$ and $y^n$ into $\scrO_{n+1}$. 

From here, fix such an $n$.  In each case, Tschirnhausen \ref{tschirnhausen} pushes all monomials divisible by $x^2$ into higher degree, and \ref{Length5Neckpush}\eqref{Length5Neckpush1} pushes all monomials divisible by $xyx$ into higher degree. This leaves only $y^n$, and monomials of the form $(a_1,\hdots,a_t)$ where $t\geq 1$ and all $a_i\geq 2$.  Note that since $n\geq 6$, any such width one monomial $xy^{a_1}$ has $a_1\geq 5$. 

Now, given any such $(a_1,\hdots,a_t)$, if some $a_i\geq 4$, then after cyclic rotation if necessary, by \ref{Length5Neckpush}\eqref{Length5Neckpush2}--\eqref{Length5Neckpush4} we can either push this monomial into higher degree (if $a_i=4$), 
%which here forces $t\geq 2$
or into the same degree $n$, but with higher width.  Simply repeating this process pushes all $(a_1,\hdots,a_t)$ of degree $n$ either into $\scrO_{n+1}$, or onto degree $n$ monomials of the form $(b_1,\hdots,b_s)$, where all $b_i\in\{2,3\}$.
This leaves very few monomials to analyse.  

Indeed, after the above moves, the second column in the following table summarises (up to cyclic equivalence) which monomials remain, which of course depends on the degree~$n$. The third column shows either how to clear these using \ref{Length5Neckpush}, or marks with a $*$ those monomials which need to be analysed further.
\[
\medmath{
\begin{array}{c@{\hspace{10pt}}>{\arraybackslash$} p{7cm} <{$}l}
 {\rm Deg}&{\rm Remaining}&\textnormal{Clear using \ref{Length5Neckpush}(-)}\\
       6 &  y^6, (22) & *,*\\
        7 & y^7,  (23) & *,*  \\
        8 &  y^8,  (33) &  *,* \\
        9 &  y^9,  (222) &  *,\eqref{Length5Neckpush5}   \\
        10 & y^{10},  (223) &  *,\eqref{Length5Neckpush5} \\
        11 & y^{11},   (233) &   *,*\\
        12 & y^{12},   (333),(2222) & *,*,\eqref{Length5Neckpush5}   \\
        13 & y^{13},   (2223) &  *,\eqref{Length5Neckpush5} \\
        14 & y^{14},   (2323),(2233) & *,\eqref{Length5Neckpush6},\eqref{Length5Neckpush5}  \\
        15 & y^{15},   (2333),(22222) & *,*, \eqref{Length5Neckpush5}   \\
        16 & y^{16},   (3333),(22223) &  *,*, \eqref{Length5Neckpush5} \\
        17 & y^{17},   (22323),(22233) & *,\eqref{Length5Neckpush5},\eqref{Length5Neckpush5} \\
        18 &  y^{18},  (22333),(23233),(222222) & *,\eqref{Length5Neckpush5},\eqref{Length5Neckpush6},\eqref{Length5Neckpush5}   \\
        19 & y^{19},  (23333)(222223) & *,*,\eqref{Length5Neckpush5}  \\
        20 & y^{20},  (33333),(222323),(222233),(223223) & *,*,\eqref{Length5Neckpush5},\eqref{Length5Neckpush5},\eqref{Length5Neckpush5}   \\
	21 & y^{21},  (223233),(222333),(223323),(232323),(2222222) &*, \eqref{Length5Neckpush5},\eqref{Length5Neckpush5},\eqref{Length5Neckpush5},\eqref{Length5Neckpush6},\eqref{Length5Neckpush5}  \\
\end{array}}
\]
We next claim that all entries $*$ in the above table of degree $n$, which are not $y^n$, can be pushed onto $xy^{n-1}$.  For $n=6,7,8$ this is proved as follows: 
\begin{equation}\label{eqn:signs}
\begin{array}{lcl}
xy^5\stackrel{\eqref{Length5Neckpush2}}{\to} \{\tfrac{1}{5}(22),\tfrac{2}{15}(31)\}\stackrel{\eqref{Length5Neckpush1}}{\to}\tfrac{1}{5}(22)&\mbox{thus}& (22)\to xy^5  \\
xy^6\stackrel{\eqref{Length5Neckpush2}}{\to} \{\tfrac{1}{5}(23),\tfrac{2}{15}(32)\}=\tfrac{5}{15}(23)&\mbox{thus}& (23)\to xy^6  \\
xy^7\stackrel{\eqref{Length5Neckpush2}}{\to} \{\tfrac{1}{5}(24),\tfrac{2}{15}(33)\}\stackrel{\eqref{Length5Neckpush3}}{\to}\tfrac{2}{15}(33)&\mbox{thus}& (33)\to xy^7  
\end{array}
\end{equation}
All remaining cases are similar, but involve longer chains. The point is that repeatedly applying  \ref{Length5Neckpush}\eqref{Length5Neckpush2}--\eqref{Length5Neckpush4} pushes $xy^{n-1}$ onto a sum of the non-$y^n$ monomials in the above table, all of degree $n$.  In degree $11,15,19$ the term $(23\hdots3)$ appears at the end of this process (noting that all signs on the coefficients in \ref{Length5Neckpush}\eqref{Length5Neckpush2}\eqref{Length5Neckpush4} are positive, so as in \eqref{eqn:signs} terms can only add, and never cancel), and in degree $12,16,20$ the term $(3\hdots3)$ appears at the end.  Since all other terms in the above table can be cleared, it follows in degree $n=11,15,19$ that  $xy^{n-1}\to(23\hdots3)$, thus $(23\hdots3)\to xy^{n-1}$, and in degree $n=12,16,20$  that $xy^{n-1}\to(3\hdots3)$, thus $(3\hdots3)\to xy^{n-1}$.

Hence, proceeding through the degrees in increasing order, we have shown that  $\Jac(g)\cong\Jac(x^3+\scrN_{2,2}+y^5+xy^5p+y^6q+\scrO_{22})$ for some $p,q\in\C[y]$.  Applying \ref{prop:l5finite} to $\sff=x^3+\scrN_{2,2}+y^5+xy^5p+y^6q+\scrO_{22}$, the 21st power of the radical belongs to its Jacobi ideal. Since the $\scrO_{22}$ part of $\sff$ differentiates to terms of degree $\geq21$, the tail $\scrO_{22}$ does not contribute to the relations, and so can be set to zero. 
\end{proof}

\subsection{Length Four}
Throughout this subsection, $f_{\rm stub}\colonequals x^3 + y^4+\scrN_{2,3}$.  Given this, we again recycle the notation $\to$ from 
\ref{tonotation}, noting that any previous clearing lemma is no longer valid since $f_{\rm stub}$ has changed. In what follows $A$ is a non-empty tuple of $1$s and $2s$, and e.g.\ $1A2 = (1,a_1,\hdots,a_t,2)$, which as before is shorthand for the monomial $xyxy^{a_1}\hdots xy^{a_t}xy^2$.  

\begin{lemma}\label{Length4Neckpush}
For $f_{\rm stub}\colonequals x^3 + y^4+\scrN_{2,3}$, the following statements hold.
\begin{enumerate}
\item\label{Length4Neckpushm1} For any $n\geq 6$, $x^2g_{n-2}\to\scrO_{n+1}$ and $y^3g_{n-3}\to\scrO_{n+1}$, where $g_i$ has degree $i$
\item\label{Length4Neckpush0} For any $t\geq 1$,  $(1,a_1,\hdots,a_t,2)\to(1,2,a_1,\hdots,a_t)$.
\item\label{Length4Neckpush1} If $1A2\sim 12A$, then $1A2\to\scrO_{n+1}$.
\item\label{Length4Neckpush2} If $a_i=2$ for all $1\leq i\leq t$ where $t\geq 3$, then $(a_1,\hdots,a_t)\to\scrO_{n+1}$.
\item\label{Length4Neckpush3}  $(111222)\to\scrO_{16}$, $(112212)\to\scrO_{16}$, $(112122)\to\scrO_{16}$, and $(121212)\to\scrO_{16}$.
\end{enumerate}
\end{lemma}
\begin{proof}
(1) The case  $x^2g_{n-2}\to\scrO_{n+1}$ is Tschirnhausen \ref{tschirnhausen}, and  the case $y^3g_{n-3}\to\scrO_{n+1}$ is essentially the same, using the lowest degree $y$ term $y^4$ in $f_{\stub}$ and the automorphism from Tschirnhausen  \ref{tschirnhausen} (which only involves $y$, so the $x^3$ term in $f_{\stub}$ does not contribute) to push the top degree terms involving $y^3$ into higher order.\\ 
(2) Write $1A2=x \msf xy^2$ where $\msf\colonequals yxy^{a_1}\hdots xy^{a_t}$, then apply $y\mapsto y + \upalpha \msf$.  As this is a move only on $y$, the $x^3$ in $f_{\stub}$ does not contribute, whilst the $y^4$ in $f_{\stub}$ adds a term cyclically equivalent to $4\upalpha xy^{a_1}\hdots xy^{a_{t-1}}xy^{a_t+4}$.  Further, the $x^2y^3$ in $\scrN_{2,3}$ produces (after substitution) only elements in top degree $n$ that are divisible by $x^2$ and so we are down to analysing  what $xyxy^2$ produces under substituting $ \upalpha \msf$ for $y$.  The first $y$ substitution gives the target monomial  $\upalpha (1,a_1,\hdots,a_t,2)$, the second $y$ substitution gives $\upalpha (1,1,a_1,\hdots,a_{t-1},a_t+1)$, whilst the third $y$ subs gives $\upalpha (1,2,a_1,\hdots,a_t)$.  Everything else is in higher degree.
	
Then apply $y\mapsto y - xy^{a_1}\hdots xy^{a_{t-1}}xy^{a_t+1}$, which is just $\msf$ rotated once. The $y^4$ in $f_{\stub}$ now adds a term cyclically equivalent to $-4\upalpha xy^{a_1}\hdots xy^{a_{t-1}}xy^{a_t+4}$, which cancels the one above.  The $x^2y^3$ still only produces (after substitution) elements in top degree $n$ that are divisible by $x^2$. Now the substitutions into $xyxy^2$ give two terms in top degree $n$ divisible by $x^2$, and one term in top degree $n$ equal to $= -\upalpha (1,1,a_1,\hdots,a_{t-1},a_t+1)$, which cancels the one above.

Overall, the automorphism $y\mapsto y +\upalpha \msf- \upalpha xy^{a_1}\hdots xy^{a_{t-1}}xy^{a_t+1}$ produces the term $\upalpha (1,a_1,\hdots,a_t,2)$ at the cost of adding $\upalpha (1,2,a_1,\hdots,a_t)$, plus further degree $n$ terms all of which are divisible by $x^2$ or $y^3$.  These latter monomials clear into $\scrO_{n+1}$ by \eqref{Length4Neckpushm1}.\\
(3) Is an immediate consequence of \eqref{Length4Neckpush0}, since in this case the overall automorphism in \eqref{Length4Neckpush0} produces $2\upalpha (1,a_1,\hdots,a_t,2)$, plus terms in degree $n$ that clear into $\scrO_{n+1}$.\\
(4) This is very similar to \eqref{Length4Neckpush0}. Write $xy^2\hdots xy^2=x\msf xy^2$ where $\msf=y^2xy^2\hdots xy^2$, and apply $y\mapsto y +\upalpha \msf$.  Again the $x^3$ in $f_{\stub}$ does not contribute, whilst now the $y^4$ in $f_{\stub}$ adds a term cyclically equivalent to $\sim 4\upalpha xy^{2}\hdots xy^2xy^7$.  The $x^2y^3$ part of $\scrN_{2,3}$ produces only elements in top degree $n$ that are divisible by $x^2$, and now the first $y$ substitution into $xyxy^2$ gives $(2,\hdots,2)$, whilst the second and third substitutions are both divisible by $y^3$. Then apply $y\mapsto y - \upalpha y^1xy^2\hdots xy^3$, which is again just $\msf$ rotated once. This clears the lower degree term produced by $y^4$ at the cost of adding terms in top degree, all divisible by $y^3$ or $x^2$.  Overall,  $y\mapsto y +\upalpha \msf- \upalpha y^1xy^2\hdots xy^3$ produces $\upalpha xy^2\hdots xy^2$ at the cost of degree $n$ terms, all divisible by $x^2$ or $y^3$.  Again, by \eqref{Length4Neckpushm1}, these clear into $\scrO_{n+1}$.\\
(5) The proof uses linear algebra, in a similar way to \ref{lem:E6funnelD}. Applying \eqref{Length4Neckpushm1} to $111222$, with coefficients as in the proof there exists an automorphism $\upphi_1$ such that
\[
\upphi_1(f_{\stub})\sim f_{\stub}+\upalpha 111222 + \upalpha 112212 +\scrO_{n+1}
\]
Similarly, applying \eqref{Length4Neckpush0} to $112212, 112122, 121212$ respectively, there exists automorphisms $\upphi_2,\upphi_3,\upphi_4$ such that
\[
\begin{array}{r@{\hspace{4pt}}r@{\hspace{4pt}}l@{\hspace{4pt}}}
\upphi_2(f_{\stub})&\sim& f_{\stub}+\upalpha 112212 + \upalpha 112122 +\scrO_{n+1}\\
\upphi_3(f_{\stub})&\sim& f_{\stub}+\upalpha 112122 + \upalpha 121212+\scrO_{n+1} \\
\upphi_4(f_{\stub})&\sim& f_{\stub}+\upalpha 121212 + \upalpha 112212 +\scrO_{n+1}
\end{array}
\]
Consider the matrix where the columns are indexed by $111222, 112212, 112122, 121212$, and the rows by $\upphi_1,\upphi_2,\upphi_3,\upphi_4$, where the entries capture the coefficients in the above four displayed equations. The first row is $[1,1,0,0]$, the second row is $[0,1,1,0]$, the third $[0,0,1,1]$, and the fourth row is $[0,1,0,1]$.  This matrix has non-zero determinant, so by taking suitable linear combinations of $\upphi_1,\upphi_2,\upphi_3, \upphi_4$, all of the terms  $111222$, $112212$, $112122$ and $121212$ can be pushed into higher degree.
\end{proof}

\begin{lemma}\label{Length4NeckHowToNeck}
Suppose that $\msf$ is one of the monomials in $\scrN_{i,j}$, where $i+j\geq 7$.  If every other monomial $\msf'$ in $\scrN_{i,j}$ satisfies $\msf'\to\scrO_{n+1}$, then $\msf\to \scrN_{i,j}$.  
\end{lemma}
\begin{proof}
This follows simply since if we have the ability to push $\msf'$ with arbitrary coefficient into $\scrO_{n+1}$ without effecting other degree $n$ terms, we have the ability to bring it back with arbitrary coefficient.  Thus, we may bring back all the other monomials in $\scrN_{i,j}$, with coefficients as prescribed in the necklace polynomial, and so $\msf\to\scrN_{i,j}$.
\end{proof}

The following deals with monomials $\msf_1,\msf_2$ of different degrees, and so the previous notation $\to$ does not apply.  
\begin{notation}\label{not:LeadsTo}
Suppose that we are given $f=f_{\rm stub}+g_6+\hdots+g_n+\scrO_{n+1}$ for some $n\geq 7$,  where the $g_i$ are arbitrary terms of degree $i$.  
\begin{enumerate}
\item Given monomials $\msf_1$, $\msf_2$ of degrees $\geq 6$,  we will write $\msf_1\leadsto_n\msf_2$ if there exists an automorphism $\upphi$ such that
\[
\upphi(f) \sim f - \upalpha \msf_1 + \upalpha\upmu \msf_2 +\scrO_{n+1} 
\]
That is, $\upalpha \msf_1$ has been subtracted from $f$ at the cost of adding $\upalpha\upmu \msf_2$, and up to cyclic permutation no other terms of degree $\leq n$ are effected
\item\label{not:LeadsToREV} Just as in \ref{tonotation}\eqref{tonotation2REV}, if  $\msf_1\leadsto_n\msf_2$ then  $\msf_2\leadsto_n\msf_1$.
\item\label{not:LeadsToCOMP}  If $\msf_1\leadsto_n\msf_2$ and $\msf_2\leadsto_n\msf_3$, then $\msf_1\leadsto_n\msf_3$.
\end{enumerate}
\end{notation}

\begin{lemma}\label{lem:LeadsTo}
If $f=f_{\rm stub}+g_6+\hdots+g_n+\scrO_{n+1}$ for some $n\geq 7$,  where the $g_i$ are arbitrary terms of degree $i$, then the following statements hold.
\begin{enumerate}
\item\label{lem:LeadsToTsc} $x^2g_{n-3}\leadsto_n\scrO_{n+1}$ for all $n\geq 6$.
\item\label{lem:LeadsToDC} (`Double clear') $y^6g_{n-7}\leadsto_n\scrO_{n+1}$ for all $n\geq 7$.
\item\label{lem:LeadsToSum2} $23 \leadsto_8 \scrO_{9}$ and $26 \leadsto_{12} 1123 \leadsto_{12} \scrO_{13}$. Consequently
\begin{enumerate}
\item\label{lem:LeadsTo1} $\phantom{x}y^{6\phantom{1}}\leadsto_{8\phantom{0}} 14\phantom{4}\leadsto_{8\phantom{0}} 1111$.
\item\label{lem:LeadsTo4} $\phantom{x}y^{9\phantom{1}}\leadsto_{12} 17\phantom{4}\leadsto_{12} 1114\leadsto_{12} 111111$.
\end{enumerate}
\item\label{lem:LeadsTo2} $xy^{7\phantom{1}}\leadsto_{10} 114\leadsto_{10} 11111$, and $y^{8\phantom{1}}\leadsto_{10} 25\leadsto_{10} 1122$.
\item\label{lem:LeadsTo5} $xy^{10}\leadsto_{14} 117\leadsto_{14} 11114\leadsto_{14} 1111111$, and $y^{11}\leadsto_{14} 28\leadsto_{14} 1125\leadsto_{14} 111122$.
\end{enumerate}
\end{lemma}
\begin{proof}
(1)  This is identical to \ref{Length4Neckpush}\eqref{Length4Neckpushm1}, noting simply that there is a degree two gap between $x^3$ and the next term $\scrN_{2,3}$ in $f_{\stub}$ containing $x$. Thus, the clearing automorphism in Tschirnhausen \ref{tschirnhausen} extends the clear by an extra degree.\\
(2)  Apply $y\mapsto y+\tfrac{\upalpha}{4}y^3g_{n-7}$. The $y^4$ term creates ${\upalpha}y^6g_{n-7}$ in degree $n-1$, and the $\scrN_{2,3}$ terms of $f$ create terms of top degree $n$, all of which are divisible by $y^3$. All other terms of $f$ are irrelevant, as they substitute into $\scrO_{n+1}$, unless $n=7$ and the $y^4$ term double substitutes to give $\tfrac{3\upalpha^2}{8}y^8$ in top degree. Regardless, in all cases, all the new top degree terms are divisible by $y^3$ and so can be cleared into $\scrO_{n+1}$ using \ref{Length4Neckpush}\eqref{Length4Neckpushm1}, so $y^6g_{n-7}\leadsto_{n}\scrO_{n+1}$. \\
(3)  For the first statement $n=8$, and apply $y\mapsto y + \tfrac{\upalpha}{4}xy^2x$ to $f$. The $x^3$ term of $f_{\stub}$ does not contribute, and the $y^4$ term of $f$ produces a term cyclically equivalent to ${\upalpha}xy^2xy^3$ of degree $7$. The $\scrN_{2,3}$ terms of $f$ produce top degree $n=8$ terms all divisible by $x^2$, which clear into $\scrO_{n+1}$ by \ref{Length4Neckpush}\eqref{Length4Neckpushm1}, and all higher degree terms of $f$ are irrelevant since they directly substitute into $\scrO_{n+1}$. Thus $23 \leadsto_{8} \scrO_{9}$.
%ff := &+CleanUp( 3*Neck(F,3,0) +  Neck(F,2,3) + 4*Neck(F,0,4) + a*Neck(F,2,4) + b*Neck(F,3,3) + c*Neck(F,0,6));
%set := CleanUp(OurEvaluate(ff,x ,y + t/4*x*y^2*x ,8)-ff);
%[
%    1/2*t*x^3*y^2*x*y^2,
%    1/4*t*x^2*y*x*y^2*x*y,
%    1/4*t*x^2*y^2*x^2*y^2,
%    1/4*t*x*y*x^2*y^2*x*y,
%    1/4*t*x^2*y*x*y*x*y^2,
%
%    t*x*y^2*x*y^3
%]

Next, $n=12$. Applying $y\mapsto y + \tfrac{\upalpha}{4} xy^2xy^3$ to $f$, the $x^3$ term of $f_{\stub}$ does not contribute, the $y^4$ term adds in a single term cyclically equivalent to $\upalpha xy^2xy^6$ (which is our target, in degree $10$), and the $\scrN_{2,3}$ term contributes $x^2g_{n-3}$ terms plus $\tfrac{\upalpha}{4}xyxyxy^2xy^3$, all in degree $11$. The degree six terms of $f$, which are unspecified, after substitution all produce degree $n$ terms, but just since $y^3$ divides $\tfrac{\upalpha}{4}xy^2xy^3$, these top degree $n$ terms produced are all divisible by $y^3$.  The terms of $f$ of degree $7$ and higher are irrelevant, since they substitute directly into $\scrO_{n+1}$.  Now \eqref{lem:LeadsToTsc} clears the $x^2g_{n-3}$ terms, and \ref{Length4Neckpush}\eqref{Length4Neckpushm1} clears all terms in top degree $n$, thus $26 \leadsto_{12} 1123$. For the last statement, apply $y\mapsto y + \tfrac{\upalpha}{4}xyxyxy^2x$.  The $y^4$ term of $f$ produces ${\upalpha}xyxyxy^2xy^3$ of degree $11$, whilst the $\scrN_{2,3}$ terms produce top degree $n=12$ terms all divisible by $x^2$, which clear into $\scrO_{n+1}$ by \ref{Length4Neckpush}\eqref{Length4Neckpushm1}. Thus $1123 \leadsto_{12} \scrO_{13}$.

\noindent
(a) Apply $y\mapsto y+\tfrac{\upalpha}{4}y^3$ to $f$. The $y^4$ term gives $\upalpha y^6$ in degree $6$ and $\tfrac{3\upalpha^2}{8}y^8$ in top degree $8$.  The $\scrN_{2,3}$ term of $f$ gives $\tfrac{3\upalpha}{4}x^2y^5 + \tfrac{\upalpha}{2}xyxy^4 + \tfrac{\upalpha}{4}xy^2xy^3$, and the degree $6$ terms of $f$ give various degree $8$ terms, all of which are divisible by $y^3$.  Thus, all terms produced in top degree are divisible by $y^3$, so clear by \ref{Length4Neckpush}\eqref{Length4Neckpushm1}, the $x^2y^5$ clears by \eqref{lem:LeadsToTsc}, and the $xy^2xy^3$ clears by the first statement of \eqref{lem:LeadsToSum2}, which we have already proved. Thus, overall,   $y^{6}\leadsto_{8}14$. The claim $14\leadsto_{8} 1111$ is easier, since simply applying $y\mapsto y + \tfrac{\upalpha}{4}xyxy$ to $f$ produces $\upalpha xyxy^4$ in degree $7$, plus $\tfrac{\upalpha}{4}(xy)^4$ in degree $8$, plus other terms of top degree $8$ all of which are divisible by $x^2$, so clear by \ref{Length4Neckpush}\eqref{Length4Neckpushm1}.
%ff := &+CleanUp(x^3 + y^4 + Neck(F,2,3) + a*Neck(F,2,4) + b*Neck(F,3,3));
%set := CleanUp(OurEvaluate(ff,x ,y + t/4*y^3 ,   8)-ff);  
%[
%    3/4*b*t*x^3*y^5,
%    1/2*b*t*x^2*y*x*y^4,
%    1/4*b*t*x^2*y^2*x*y^3,
%    1/4*b*t*x^2*y^3*x*y^2,
%    1/2*b*t*x^2*y^4*x*y,
%    a*t*x^2*y^6,
%    1/4*b*t*x*y*x*y*x*y^3,
%    3/4*a*t*x*y*x*y^5,
%    1/2*a*t*x*y^2*x*y^4,
%    1/4*a*t*x*y^3*x*y^3,
%    3/8*t^2*y^8,
%
%    3/4*t*x^2*y^5,
%    1/2*t*x*y*x*y^4,
%    1/4*t*x*y^3*x*y^2,
%
%    t*y^6
%]
%ff := &+CleanUp(x^3 + y^4 + Neck(F,2,3) + a*Neck(F,2,4) + b*Neck(F,3,3));
%set := CleanUp(OurEvaluate(ff,x ,y + t/4*x*y*x*y ,   8)-ff);  
%[
%    1/4*t*x^3*y*x*y^3,
%    1/2*t*x^2*y*x*y*x*y^2,
%    1/4*t*x^2*y^2*x*y*x*y,
%    1/4*t*x*y*x^2*y*x*y^2,
%    1/4*t*x*y*x*y*x*y*x*y,
%    t*x*y*x*y^4
%]

\noindent
(b) This is similar. For the first, apply $f\mapsto f+\tfrac{\upalpha}{4}y^6$.  The $y^4$ term of $f$ produces $\upalpha y^9$, the $\scrN_{2,3}$ term produces $\tfrac{3\upalpha}{4}x^2y^8 + \tfrac{\upalpha}{2}xyxy^7 + \tfrac{\upalpha}{4}xy^2xy^6$.  The degree $6$ terms of $f$, which are not specified, create degree $n-1=11$ terms all divisible by $y^6$ (which are cleared by~\eqref{lem:LeadsToDC}), and the unspecified degree $7$ terms of $f$ create top degree $n=12$ terms all divisible by $y^6$, thus $y^3$ (and are cleared by \ref{Length4Neckpush}\eqref{Length4Neckpushm1}).  Now the created $x^2y^8$ term can be cleared easily, at the cost of creating additional terms in degree $12$ all divisible by $y^8$ (which are cleared by \ref{Length4Neckpush}\eqref{Length4Neckpushm1}).  The $xy^2xy^6$ terms clears into $\scrO_{13}$ by the second statement in \eqref{lem:LeadsToSum2}. Thus, overall, $y^9\leadsto_{12}xyxy^7$.  The other  $17\leadsto_{12} 1114\leadsto_{12} 111111$ are easier: for the first apply $y\mapsto y+\tfrac{\upalpha}{4}xyxy^4$ and argue in a similar way, and for the second apply  $y\mapsto y+\tfrac{\upalpha}{4}(xy)^4$.

\noindent
(4) This is all very similar to the above. In summary, for $xy^{7}\leadsto_{10} 114$ apply $y\mapsto y + \tfrac{\upalpha}{4}xy^4$, for $114 \leadsto_{10} 11111$ apply $y\mapsto y + \tfrac{\upalpha}{4}(xy)^3$.  For $y^{8}\leadsto_{10} 25$ apply $y\mapsto y + \tfrac{\upalpha}{4}y^5$, and for $25\leadsto_{10} 1122$ apply $y\mapsto y + \tfrac{\upalpha}{4}xy^2xy^2$.

\noindent
(5) This is also very similar to the above. For $xy^{10}\leadsto_{14} 117$ apply $y\mapsto y + \tfrac{\upalpha}{4}(xy)y^6 $, for $117\leadsto_{14} 11114$ apply $y\mapsto y + \tfrac{\upalpha}{4}(xy)^3y^3$ and for $11114\leadsto_{14} 1111111$ apply $y\mapsto y + \tfrac{\upalpha}{4}(xy)^5$.  For $y^{11}\leadsto_{14} 28$ apply $y\mapsto y + \tfrac{\upalpha}{4}y^8$, for $28\leadsto_{14} 1125$ apply $y\mapsto y + \tfrac{\upalpha}{4}xy^2xy^5$, and for $1125\leadsto_{14} 111122$ apply $y\mapsto y + \tfrac{\upalpha}{4}xyxyxy^2xy^2$.
\end{proof}

To set notation, consider the set $\scrT_4$ (where $\scrT$ stands for tail)
 defined to be
\begin{equation}
\scrT_4\colonequals \{ (3,3), (2,4),(4,4), (4,6),(0,11),(5,5),(6,6),(7,7) \}\label{eqn:DefS4}
\end{equation}
In the following, there are two different options for the length four funnel.  Roughly speaking, the first makes length four look aesthetically similar to length five in \ref{thm!funnel5}, whilst the second makes length four look aesthetically similar to length six (which will be proved in \ref{thm!funnel6}).  Neither option is preferred, as such, but the second is mildly easier when  constructing flops in \cite{BW5}.  
\begin{thm}\label{thm!funnel4}
Set $f_{\rm stub}=x^3 + y^4+\scrN_{2,3}$.  If $g= f_{\rm stub}+ \scrO_6$, then the following hold.
\begin{enumerate}
\item\label{thm!funnel4A} There exists $p,q\in\C[ t]$ and $\upalpha_1,\upalpha_2\in\C$ such that $\Jac(g)\cong \Jac(h)$, where 
\[
h=f_{\stub}+\upalpha_1\scrN_{2,4}+\upalpha_2\scrN_{3,3} + xy^7p + y^6q.
\]
\item\label{thm!funnel4B} There exist $\upalpha_{i,j}\in\mathbb{C}$ such that $\Jac(g)\cong \Jac(k)$, where 
\[
k=f_{\stub}+\sum_{(i,j)\in\scrT_4}\upalpha_{i,j}\scrN_{i,j}.
\]
\end{enumerate}
\end{thm}
\begin{proof}
The proof first proceeds inductively through each degree $n=6,\hdots,15$ in order, examining all the monomials up to cyclic equivalence of degree $n$.   For any fixed $n$, \ref{Length4Neckpush}\eqref{Length4Neckpushm1} pushes all degree $n$ monomials divisible by $x^2$ and $y^3$ into higher degree. Up to cyclic permutation this leaves only those monomials of the form $xy^{a_1}\hdots xy^{a_t}=(a_1\hdots a_t)$ where all $a_i\in\{1,2\}$.  The second column in the following table summarises, up to cyclic equivalence, which monomials remain in degree $n$. The third column shows either how to clear these into higher degree using \ref{Length4Neckpush}, or how to push them onto other monomials on the same line (but to the right) again using \ref{Length4Neckpush}, or marks with a $*$ those monomials which need to be analysed further.
\[
\medmath{
\begin{array}{c@{\hspace{10pt}}ll}
 {\rm Deg}&{\rm Remaining}&\textnormal{Push using \ref{Length4Neckpush}(-)}\\
6 & (22), (111) &*,*\\
7 & (112) & \eqref{Length4Neckpush1}\\
8 & (122), (1111) &\eqref{Length4Neckpush1},*\\
9 & (222), (1112) &\eqref{Length4Neckpush2},\eqref{Length4Neckpush1}\\
10 & (1212), (1122), (11111) &\eqref{Length4Neckpush0},*,*\\
11 & (1222), (11112) &\eqref{Length4Neckpush1},\eqref{Length4Neckpush1} \\ 
12 & (2222), (11122), (11212), (111111) &\eqref{Length4Neckpush2},\eqref{Length4Neckpush0},\eqref{Length4Neckpush1},*\\
13 &  (11222), (12122), (111112)  & \eqref{Length4Neckpush0},\eqref{Length4Neckpush1},\eqref{Length4Neckpush1}\\
14 & (12222), (112112), (111212)\sim(121112), (111122),
(1111111) &\eqref{Length4Neckpush1},\eqref{Length4Neckpush0},\eqref{Length4Neckpush0},*,*\\
15 & (22222), (111222), (112212), (112122), (121212), (1111112)
&\eqref{Length4Neckpush2},\eqref{Length4Neckpush3},\eqref{Length4Neckpush3},\eqref{Length4Neckpush3},\eqref{Length4Neckpush3},\eqref{Length4Neckpush1}
\end{array}}
\]
The conclusion of the $n=6$ step is that $g\cong f_{\stub}+\tfrac{\upalpha_1}{3}(111)+\tfrac{\upalpha_2}{2}(22)+\scrO_7$ for some $\upalpha_1,\upalpha_2\in\C$.  Sending $x\mapsto x+\tfrac{\upalpha_1}{3}(xy^3+yxy^2+y^2xy)$ and $y\mapsto y + \tfrac{\upalpha_2}{4}(x^2y+xyx)$ gives $g\cong \sff_6\colonequals f_{\stub}+\upalpha_1\scrN_{3,3}+\upalpha_2\scrN_{2,4}+\scrO_{7}$. 

The proof then proceeds to the $n=7$ step.  By the above table everything in degree $n=7$ can be cleared into higher degree, so the conclusion of this step is that $g\cong \sff_7\colonequals f_{\stub}+\upalpha_1\scrN_{3,3}+\upalpha_2\scrN_{2,4}+\scrO_{8}$. 

The $n=8$ step is mildly more delicate.  By the above table everything in degree $8$ can be cleared into higher degree, except possibly $1111$, thus $g\cong f_{\stub}+\upalpha_1\scrN_{3,3}+\upalpha_2\scrN_{2,4}+\upbeta(1111) + \scrO_{9}$ for some $\upbeta$.  For (2), since every other monomial in $\scrN_{4,4}$ except $1111$ can be pushed into higher degree, by \ref{Length4NeckHowToNeck} $g\cong \sff_8\colonequals f_{\stub}+\upalpha_1\scrN_{3,3}+\upalpha_2\scrN_{2,4}+\upalpha_{4,4} \scrN_{4,4} + \scrO_{9}$. For (1), instead using \ref{lem:LeadsTo}\eqref{lem:LeadsTo1} (and \ref{not:LeadsTo}\eqref{not:LeadsToREV}\eqref{not:LeadsToCOMP}) we can replace $1111$ by the lower degree term $y^6$, without affecting (up to cyclic permutation) any other term of degree $\leq 8$.  Thus in (1) the conclusion of this step is that $g\cong \sff_8\colonequals f_{\stub}+\upalpha_1\scrN_{3,3}+\upalpha_2\scrN_{2,4}+\upalpha_3 y^6 + \scrO_{9}$.

The proof proceeds in this manner. At each step, inspecting the table then either everything is cleared and we simply move onto the next step, or alternatively for (1) we use \ref{lem:LeadsTo}, and for (2) we use \ref{Length4NeckHowToNeck}.  Indeed, at step $n=10$ by \ref{lem:LeadsTo}\eqref{lem:LeadsTo2} we can replace the remaining $1122$ by $y^8$, and the remaining $11111$ by $xy^7$.  Or, we can use \ref{Length4NeckHowToNeck} to replace $11111$ by $\scrN_{5,5}$ and replace $1122$ by $\scrN_{4,6}$ in a similar way to \ref{Length4NeckHowToNeck} since we can use \ref{Length4Neckpush}\eqref{Length4Neckpush0} $1212 \to 1221\sim1122$ to adjust the coefficients on $1212$ and $1122$ to match those dictated by the necklace, and add in the other terms at will.  

At step $n=12$ we can replace the remaining $111111$ by $y^9$ using \ref{lem:LeadsTo}\eqref{lem:LeadsTo4}, or replace $111111$ by $\scrN_{6,6}$ using \ref{Length4NeckHowToNeck}.  At step $n=14$  using \ref{lem:LeadsTo}\eqref{lem:LeadsTo5} we can replace $111122$ by $y^{11}$, and replace $1111111$ by either $xy^{10}$ using \ref{lem:LeadsTo}\eqref{lem:LeadsTo5} or $\scrN_{7,7}$ using  \ref{Length4NeckHowToNeck}.

The conclusion after step $n=15$ is either that $\Jac(g)\cong\Jac(\sff_{15})$ where $\sff_{15}\colonequals x^3 + y^4+\scrN_{2,3}+\upalpha_1\scrN_{2,4}+\upalpha_2\scrN_{3,3}+ xy^7p+y^6q +\scrO_{16}$ for some $p,q\in\C[y]$, or $\sff_{15}\colonequals x^3 + y^4+\sum_{(i,j)\in\scrT_4}\upalpha_{i,j}\scrN_{i,j}+\scrO_{16}$ for some $\upalpha_{i,j}\in\mathbb{C}$.  Applying \ref{prop:l4finite} to $\sff_{15}$, the $15$th power of the radical belongs to its Jacobi ideal. Since the $\scrO_{16}$ part of $\sff_{15}$ differentiates to degree $\geq15$, the tail $\scrO_{16}$ does not contribute to the relations, and so it can be set to zero. 
\end{proof}

\begin{remark}
The above is the only point in the whole paper where we deliberately take some term in degree $n$ and push it in a smaller degree.  The reason for this is the term $111122$ in degree $14$. The proof of \ref{thm!funnel4} replaces it with $y^{11}$, which may be in lower degree, but it does have the advantage of clearly being a necklace polynomial.  It is not clear to us whether $111122$ can be made into a necklace polynomial without dropping the degree.  In the length $\ell=5,6$ funnels, this problem does not arise.
\end{remark}

\subsection{Length Six}\label{subsec:l6funnel}
Throughout this subsection, $f_{\rm stub}\colonequals  x^3 + y^4+\scrN_{3,3}+\uplambda\scrN_{2,4}$ where $\la\neq 0$, and we again recycle the notation $\to$ from \ref{tonotation}, noting again that any previous clearing lemma is no longer valid since $f_{\rm stub}$ has changed.  Within the following clearing lemma, the assumption that $\uplambda\neq 0$ presents itself naturally in the proof of \eqref{Length6Neckpush0}, which is then used repeatedly.

In what follows, $A$ and $B$ are tuples of $1$s and $2s$, so that e.g.\ $A221B$ represents a tuple $(a_1,\hdots,a_t,2,2,1,b_1,\hdots,b_s)$ where we allow $A$ or $B$ (or both) to be empty.
\begin{lemma}\label{Length6Neckpush}
For $f_{\rm stub}\colonequals x^3 + y^4+\scrN_{3,3}+\uplambda\scrN_{2,4}$ where $\la\neq 0$, and for any $A,B$ which are (possibly empty) tuples of $1$s and $2s$, the following statements hold. 
\begin{enumerate}
\item\label{Length6Neckpush0} $A111B \to A22B$, if either $A\neq \emptyset$ or $B\neq \emptyset$.
\item\label{Length6Neckpush-1} $A122B \to A221B$.  %BA can be empty
\item\label{Length6Neckpush2} $A2112B\to\scrO_{n+1}$.  %BA can be empty
\item\label{Length6Neckpush1} $A111112B\to\scrO_{n+1}$.  %BA can be empty
\item\label{Length6Neckpush3} $A11111111B\to\scrO_{n+1}$.   %BA can be empty
\item\label{Length6Neckpush4} $A2222B\to\scrO_{n+1}$, provided that $\deg(A2222B)\geq 15$. 
\item\label{Length6Neckpush5} $A1221B\to\scrO_{n+1}$, provided that $\deg(A1221B)\geq 15$. 
\end{enumerate}
\end{lemma}
\begin{proof}
(1) By assumption $n\colonequals \deg(A111B)\geq 8$.  Write $BA=x\msf y$, and consider $\upphi$ which just sends $y\mapsto y+\upalpha yx\msf y-\upalpha x\msf y^2$.  Applying $\upphi$, rotating and cancelling terms gives
\[
\upphi(f_{\stub})\sim f_{\stub} +\upalpha A111B + \la\upalpha A22B + x^2g  + y^3h+\scrO_{n+1}
\]
for some $g, h$ with $\deg(g)=n-2$ and $\deg(h)=n-3$.  Clearing $x^2g$ and $y^3h$ into $\scrO_{n+1}$ in the obvious way (using the $x^3$ and $y^4$ terms of $f_{\stub}$), since $\la\neq 0$, the statement follows.\\
(2)  If $\sfm\to\sfm'$, then $\sfm'\to\sfm$. Hence applying  \eqref{Length6Neckpush0} twice, $A122B\to A1111B\to A221B$.\\
(3) Write $A2112B\sim 112BA2 = xyxyx(y^2\msf xy^2)$ for some possibly empty $\msf $, set $n\colonequals\deg(A2112B)$, and consider the automorphism $\uppsi$ which just sends $y\mapsto y + \upalpha y^2\msf xy^2 - \upalpha y\msf xy^3$. Applying $\uppsi$, rotating and cancelling terms gives
\[
\uppsi(f_{\stub})\sim f_{\stub} +\upalpha A2112B + x^2g  + y^3h+\scrO_{n+1}
\]
for some $g, h$ with $\deg(g)=n-2$ and $\deg(h)=n-3$.  Again as in \eqref{Length6Neckpush0} clearing $x^2g$ and $y^3h$ into $\scrO_{n+1}$ in the obvious way, the statement follows.\\
(4) Simply apply \eqref{Length6Neckpush0} to give $A111112B\to A22112B$, then apply \eqref{Length6Neckpush2}.\\
(5) Simply apply \eqref{Length6Neckpush0} to give $A11111111B\to A1111122B$, then apply \eqref{Length6Neckpush1}.\\
(6) Applying \eqref{Length6Neckpush0} twice, $A2222B \to 111111BA$. If $BA$ contains a $2$, then after possibly rotating, by \eqref{Length6Neckpush1} $111111BA\to\scrO_{n+1}$.  The only other case is that $BA$ contains only $1$s, and the degree restriction forces there to be at least two of them. But then $111111BA=111111(1\hdots1)\to\scrO_{n+1}$ by \eqref{Length6Neckpush3}.\\
(7) By \eqref{Length6Neckpush0} $A1221B \to A11111B\sim 11111BA$. If  $BA$ contains a $2$, then after possibly rotating, by \eqref{Length6Neckpush1} $11111BA\to\scrO_{n+1}$.  The only other case is that $BA$ consists only of $1$s, and now by the degree restriction there must be at least three of them. Then, again, $11111BA\to\scrO_{n+1}$ by \eqref{Length6Neckpush3}.
\end{proof}
In the following, e.g.\ $(12)^4\colonequals(12121212)$ and $(12)^322\colonequals(12121222)$. 

\begin{lemma}\label{Length6NeckLeft}
If $\msf=(a_1\hdots a_t)$ is a monomial of degree $n\geq 15$, where each $a_i\in\{1,2\}$, then either $\msf\to\scrO_{n+1}$, or $\msf=(12)^t$ or $\msf=(12)^{t-1}(22)$ for some $t\geq 3$.
\end{lemma}
\begin{proof}
Given \ref{Length6Neckpush}, the statement is entirely combinatorial, so we only sketch the proof. If $\msf$ consists only of $1$s then $\msf\to\scrO_{n+1}$ by \eqref{Length6Neckpush3}, and if $\msf$ consists only of $2$s then $\msf\to\scrO_{n+1}$ by \eqref{Length6Neckpush4}.  Hence we can assume $\msf$ contains both a $1$ and a $2$, and after rotation $\msf=\hdots 12$. The proof then analyses the possible ways that $\msf$ may start.

If (a) $\msf$ starts with $\geq5$ number $1$s, then $\msf\to\scrO_{n+1}$ by \eqref{Length6Neckpush1}. If (b) $\msf$ starts with precisely four $1$s, namely $11112$, then by \eqref{Length6Neckpush0} and rotation $\msf \to\hdots 122212$, which by \eqref{Length6Neckpush-1} gets pushed $\to \hdots (12)^222$.  If (c) $\msf$ starts with precisely three $1$s, namely $1112$, then by \eqref{Length6Neckpush0} and rotation $\msf \to\hdots 12222$, which by \eqref{Length6Neckpush4} gets pushed $\to \scrO_{n+1}$. If (d) $\msf$ starts with precisely two $1$s, namely $\msf  = 12\hdots12$, since $\deg(1212)=10$, more analysis is required. Going through the options of the third and fourth last entries of $\msf$, which are either $11$, $12$, $21$ and $22$, for the third option $\to\scrO_{n+1}$ by \eqref{Length6Neckpush2}, for the fourth option $\to \hdots (12)^222$ by \eqref{Length6Neckpush-1} twice, and for the second option $\to \hdots (12)^3$ by rotation. For the first option $\msf\sim \hdots111212\to\hdots 22212\to\hdots 21222$, and an easy argument which analyses the sixth last entry shows that either $\msf\to \scrO_{n+1}$ or $\msf\to \hdots (12)^222$.  Finally, if (e) $\msf$ starts with a $2$, so $\msf\sim\hdots 122$, then another similarly easy argument (now using \eqref{Length6Neckpush5}) shows that either $\msf\to \scrO_{n+1}$ or $\msf\to \hdots (12)^222$.

Since (a)--(e) exhaust all possibilities, the conclusion of (a)--(e) is that either $\msf\to\scrO_{n+1}$, or $\msf\to \hdots (12)^3$, or $\msf\to \hdots (12)^222$.  If $\deg(\msf)=15$ then the only non-clearing option is $(12)^3$, and if $\deg(\msf)=16$ (which cannot be obtained by extending $(12)^3$), the only non-clearing option is $(12)^222$.  Hence we can assume that $\deg(\msf)\geq 17$.

Case (i) is when $\msf = \hdots(12)^222$ is extended, to be either $1\hdots121222$ or $2\hdots121222$. The second option clears by  \eqref{Length6Neckpush4} after rotation, and from there another easy argument shows that either $\msf\to\scrO_{n+1}$ or $\msf\to \hdots (12)^322$.  Case (ii) is when $\msf = \hdots(12)^3$ is extended, to be either $1\hdots(12)^3$ or $2\hdots(12)^3$. In the first option, an easy argument shows that either $\msf\to\scrO_{n+1}$ or $\msf\to \hdots (12)^322$, and in the second option, an easy argument shows that either $\msf\to\scrO_{n+1}$, or $\msf\to \hdots (12)^322$, or $\msf\to \hdots (12)^4$.

The conclusion of (i) and (ii) is that if $\deg(\msf)=17,18,19,20$ then $(12)^4$ is the only non-clearing option, and if $\deg(\msf)=21$ then $(12)^322$ is the only non-clearing option.  Hence we can assume that $\deg(\msf)\geq 22$. Repeating the above with case (i) and (ii) extended in the obvious way establishes the induction step, and the statement follows.
\end{proof}

The proof of the following lemma is identical to \ref{Length4NeckHowToNeck}. We repeat the result here only since the notation $\to$ has changed, as it depends on the stub. 
\begin{lemma}\label{Length6NeckHowToNeck}
Suppose that $\msf$ is one of the monomials in $\scrN_{i,j}$, where $i+j\geq 7$.  If every other monomial $\msf'$ in $\scrN_{i,j}$ satisfies $\msf'\to\scrO_{n+1}$, then $\msf\to \scrN_{i,j}$.  
\end{lemma}

As notation in the following, consider the set $\scrT_6$  defined to be
\begin{equation}
\medmath{\scrT_6\colonequals \{ (3,4), (3,5), (3,6), (4,6), (4,7), (4,8), (5,9), (6,9), (6,10), (8,12), (8,13), (10,15),(10,16)\}}\label{eqn:DefS6}
\end{equation}
\begin{thm}\label{thm!funnel6}
Set $f_{\rm stub}=x^3 + y^4+\scrN_{3,3}+\uplambda\scrN_{2,4}$, where $\uplambda\neq 0$.  If $g= f_{\rm stub}+ \scrO_7$, then there exists $\upalpha_{i,j}\in\C$  such that $\Jac(g)\cong \Jac(h)$, where $h=f_{\stub}+\sum_{(i,j)\in\scrT_6}\upalpha_{i,j}\scrN_{i,j}$.
\end{thm}
\begin{proof}
The proof proceeds through each small degree $n=7,\hdots,14$ in order, then it proceeds through degrees $n=15,\hdots,25$ which is easier since for those we can use \ref{Length6NeckLeft}.  As in \ref{thm!funnel4}, for any fixed $n\geq 6$ Tschirnhausen \ref{tschirnhausen} pushes all monomials divisible by $x^2$ into higher degree, and similarly the $y^4$ term in $f_{\stub}$ can be used to push all monomials divisible by $y^3$ into higher degree. In degree $n$ this leaves only those monomials of the form $xy^{a_1}\hdots xy^{a_t}=(a_1\hdots a_t)$ where all $a_i\in\{1,2\}$.  After these moves, the second column in the following table summarises which monomials which remain (up to cyclic permutation) in degree $n$, for small $n$. The third column shows either how to clear these into higher degree using \ref{Length6Neckpush}, or how to push them onto other monomials on the same line again using \ref{Length6Neckpush}, or marks with a $*$ those monomials which need to be analysed further.
\[
\medmath{
\begin{array}{c@{\hspace{10pt}}lll}
 {\rm Deg}&{\rm Remaining}&\textnormal{Push using \ref{Length6Neckpush}(-)} & \textnormal{Necklace}\\
7 & (112) & *& \scrN_{3,4}\\
8 & (122), (1111)  &*,\eqref{Length6Neckpush0} & \scrN_{3,5}\\
9 & (222), (1112)  &*,\eqref{Length6Neckpush0} & \scrN_{3,6}\\
10 & (1212), (1122), (11111) &*,\eqref{Length6Neckpush2},\eqref{Length6Neckpush0} & \scrN_{4,6}\\
11 &  (1222), (11112) & *,\eqref{Length6Neckpush0}& \scrN_{4,7}\\ 
12 &  (2222), (11122), (11212), (111111) &*,\eqref{Length6Neckpush0},\eqref{Length6Neckpush2},\eqref{Length6Neckpush0} & \scrN_{4,8} \\ 
13 & (12122), (11222), (111112)  &\eqref{Length6Neckpush0},\eqref{Length6Neckpush2},\eqref{Length6Neckpush1}  \\
14 & (12222), (112112), (111122), (111212),
(1111111) &*,\eqref{Length6Neckpush2},\eqref{Length6Neckpush0},\eqref{Length6Neckpush0},\eqref{Length6Neckpush0} & \scrN_{5,9}
\end{array}}
\]
In each degree, we are left with a single monomial marked $*$. Either there is evidently no other monomial on the same line as a given $*$ which belongs to the same necklace, or (e.g.\ in degree $10$) every other monomial which belongs to the same necklace as $*$ can be cleared into higher degree. Hence, in both cases, every other monomial in the same necklace as $*$ can be cleared into higher degree, thus by \ref{Length6NeckHowToNeck} each $*$ can be made into a necklace, and this is summarised in the fourth column of the above table.  It follows that $\Jac(g)\cong \Jac(\sff)$ where $\sff=f_{\stub}+\upalpha_{3,4}\scrN_{3,4}+\hdots+\upalpha_{5,9}\scrN_{5,9}+\scrO_{15}$ and the sum is taken over the necklaces in the fourth column of the above.

We next proceed through degrees $n=15,\hdots,26$.  By \ref{Length6NeckLeft}, the clearing lemma \ref{subsec:l6funnel} is sufficient to clear every monomial of degree $n$ into $\scrO_{n+1}$, except possibly the following.
\[
\medmath{
\begin{array}{c@{\hspace{10pt}}ll}
 {\rm Deg}&{\rm Remaining}&\textnormal{Necklace}\\
15 & (121212) & \scrN_{6,9}\\
16 &  (121222)&\scrN_{6,10}\\
17 &  \emptyset&\\
18 &  \emptyset&\\
19 &  \emptyset&\\ 
20 &  (12121212)&\scrN_{8,12}\\
21 &  (12121222) & \scrN_{8,13}\\
22 &  \emptyset&\\
23 &  \emptyset&\\
24 &  \emptyset&\\ 
25 &  (1212121212)&\scrN_{10,15}\\
26 &  (1212121222)&\scrN_{10,16}
\end{array}}
\]
In every line, every other monomial except those stated clears into $\scrO_{n+1}$, so in particular every other monomial in the same necklace as those monomials above in degree $n=15, 16, 20, 21, 25, 26$ can be cleared into $\scrO_{n+1}$. Thus by \ref{Length6NeckHowToNeck} each can be made into a necklace, and this is summarised in the third column of the above table. 

Hence, proceeding through the degrees in increasing order, we have shown that $\Jac(g)\cong\Jac(f_{\stub}+\sum_{(i,j)\in\scrT_6}\upalpha_{i,j}\scrN_{i,j} + \scrO_{27})$ for some $\upalpha_{i,j}\in \C$.  Applying \ref{prop:l6finite}, the $26$th power of the radical belongs to its Jacobi ideal. Since the $\scrO_{27}$ part differentiates to terms of degree $\geq26$, the tail $\scrO_{27}$ does not contribute to the relations, and so can be set to zero. 
\end{proof}

\subsection{Summary}
Summarising the proceeding three subsections and \S\ref{subsec:E6funnel} gives the following, which reworks \ref{thm:ReductionPrecise} into a much more precise form suitable for geometric arguments.

\begin{cor}\label{cor:NecklaceSummary}
If $f\in\ringtwo$ has algebraic length $\ell$, then $\Jac(f)\cong \Jac(g)$ where $g$ is some element in the following list, where all $\upvarepsilon_*\in\{0,1\}$, and all $\upalpha_*\in\C$
\[
\renewcommand{\arraystretch}{1.3}
\medmath{\begin{array}{c@{\hspace{20pt}}ll}
1 & x^2 + \upvarepsilon_1y^n & n\geq 2 \\
2 &x^2y+\upvarepsilon_2 y^{2n}+\upvarepsilon_3y^{2m-1} &  2n-1\geq m\geq 2 \\
3 &x^3 + xy^3+ y^4p &p\in\C[y] \\
4 & x^3 + y^4+\scrN_{2,3}+\sum_{(i,j)\in\scrT_4}\upalpha_{i,j}\scrN_{i,j} &\\
5 & x^3+\scrN_{2,2}+y^5+xy^5p+y^6q& p,q\in\C[y]\\
6 & x^3 + y^4+\scrN_{3,3}+\uplambda\scrN_{2,4}+\sum_{(i,j)\in\scrT_6}\upalpha_{i,j}\scrN_{i,j} & \uplambda\neq 0\\
\end{array}}
\renewcommand{\arraystretch}{1}
\]
In particular, the following statements hold.
\begin{enumerate}
\item\label{cor:NecklaceSummary1} $\Jac(f)\cong\Jac(g)$ where $g$ is a finite linear combination of necklace polynomials.
\item\label{cor:NecklaceSummary2} $\dim_\C\Jac(g)=\infty\iff$ $\upvarepsilon_1=0$ in $\ell=1$, $\upvarepsilon_2=0$ in $\ell=2$, or $p = 0$ in $\ell=3$.
\end{enumerate}
\end{cor}
\begin{proof}
Since $f$ has algebraic length, $\Jac(f)\cong\Jac(g)$ for some $g$ appearing in \ref{cor:SemiFinalAlgLength}.   If $\ell=1,2$ there is nothing to prove.

\noindent
If $\ell=3$, then by \ref{cor:SemiFinalAlgLength} $g=x^3+xy^3+y^4p$ for some $p\in\C\llsq y\rrsq$. Either $p=0$, which falls within the statement, or $p\neq 0$ in which case by \ref{prop:l3finite} some high power of the radical belongs to the Jacobi ideal. Consequently, high degree terms in $g$ do not contribute to the relations, and so they can be set to zero. Thus we can assume that $g$ is a polynomial, so $p\in\C[y]$. \\
\noindent
If $\ell=4$ then by \ref{cor:SemiFinalAlgLength} $\Jac(f)\cong\Jac(g)$ where $g=x^3 + y^4 + xyxy^2+\scrO_6$.  Applying $x\mapsto x+ \frac{1}{3}y^3$ gives $\Jac(f)\cong\Jac(x^3 + y^4 + \scrN_{2,3}+\scrO_6)$. From here, \ref{thm!funnel4}\eqref{thm!funnel4B} proves the result.\\
\noindent
If $\ell=5$ then by \ref{cor:SemiFinalAlgLength} $\Jac(f)\cong\Jac(x^3 + xyxy + y^5+\scrO_6)$.  Applying \ref{lem:Ell5necksmall} gives $\Jac(f)\cong\Jac(x^3 + \scrN_{2,2} + y^5+\scrO_6)$, and from there \ref{thm!funnel5} proves the result.\\
\noindent
If $\ell=6$ after \ref{cor:SemiFinalAlgLength} and rescaling $\Jac(f)\cong\Jac(x^3 + y^4 + \tfrac{1}{3}(xy)^3+ \uplambda'(xy^2)^2+ \scrO_7)$ where still $\uplambda'\neq 0$.  Applying the automorphism $x\mapsto x + \tfrac{1}{3}(xy^3+yxy^2+y^2xy) + \tfrac{2\la'}{3}y^4$ and $y\mapsto y + \tfrac{\la'}{2}xyx$ shows that $\Jac(f)\cong\Jac(x^3 + y^4+\scrN_{3,3}+2\la'\scrN_{2,4}+\scrO_6)$ where $2\la' \neq 0$.  After renaming the scalar $2\la'$ to be $\la$, from here \ref{thm!funnel6} proves the result.

\noindent
(1) This is clear from the list of possible $g$, since $x^m, xy^n$ and $y^n$ are (up to scalar) necklace polynomials.\\
(2) The case $\ell=1$ is clear, whilst the case $\ell=2$ is \cite[6.19]{BW3}.  For $\ell=3$, simply combine \ref{lem:l3infinite} (dimension is infinite when $p=0$) with \ref{prop:l3finite} (finite when $p\neq0$). In all the remaining cases $\ell=4,5,6$, the corresponding $\Jac(f)$ is finite dimensional by \ref{prop:l4finite}, \ref{prop:l5finite} and \ref{prop:l6finite}. 
\end{proof}

\begin{remark}\label{rem:AlmostconverseAlgLength}
Note again, like in \ref{thm:ReductionPrecise}, that \ref{cor:NecklaceSummary} only claims the implication in one direction.  The converse, namely \emph{every} $g$ stated in the list has algebraic length, will shortly be obtained by appealing to geometric arguments; see \ref{rem:converseAlgLength} below.
\end{remark}

\section{The Main Results}\label{sec:MAIN}
 
The previous sections combine in \ref{cor:NecklaceSummary} to show that an arbitrary $\Jac(f)$, where $f$ has an algebraic length, can be manipulated into the form $\Jac(f)\cong\Jac(g)$ where $g$ is a finite linear combination of necklace polynomials.  By \cite{BW3} such Jacobi algebras are known to be the noncommutative deformation algebra of some rational curve in a CY $3$-fold. From this observation, and recent advances in contractibility of such curves from our companion paper \cite{BW4}, the classification of flops finally comes into view.
 
\subsection{Intrinsic Contraction Algebras}\label{subsec:IntrinsMain}
We first recall the following.
\begin{defin}\label{defin:flop}
Let $g\colon \scrX\to \Spec \scrR$ be a crepant projective birational morphism, where $\scrX$ is smooth, $(\scrR,\m)$ is a complete local \textnormal{cDV} singularity, and we assume in addition that the contraction is simple, namely after taking reduced scheme structure $g^{-1}(\m)\cong\mathbb{P}^1$. 
\begin{enumerate}
\item If $\scrR$ is isolated, then $g$ is called a smooth flopping contraction.  The flop of $g$ is $g^+\colon \scrX^+\to \Spec\scrR$ where for every line bundle $\scrL=\scrO_\scrX(D)$ on $\scrX$ such that $-D$ is $g$-nef, the proper transform of $D$ is $\mathbb{Q}$-Cartier, and $g^+$-nef. 
\item If $\scrR$ is non-isolated, then $g$ is called a crepant divisor-to-curve contraction.
\end{enumerate}
For generic $g\in \scrR$, then in the pullback diagram
\[
\begin{tikzpicture}
\node (A) at (0,0) {$\scrY$};
\node (B) at (2,0) {$\scrX$};
\node (a) at (0,-1) {$\Spec\scrR/g$};
\node (b) at (2,-1) {$\Spec\scrR$};
\draw[->] (A)--(B);
\draw[->] (A)--(a);
\draw[->] (a)--(b);
\draw[->] (B)--(b);
\end{tikzpicture}
\]
the morphism $\scrY\to\Spec\scrR/g$ is a partial resolution of an ADE surface singularity, and thus $\scrY$ is dominated by the minimal resolution \cite{Pagoda}.   Consequently, the curve $\mathbb{P}^1$ in $\scrY$ corresponds to some node in some ADE Dynkin diagram; the Koll\'ar length $\ell$ of $g^{-1}(\m)\cong\mathbb{P}^1$ is equal to the number associated to this node in \eqref{eqn:allDynkin}.
\end{defin} 
The following is the main result of this paper, from which all else follows.
\begin{thm}\label{thm:Summary}
If $\mathrm{B}$ is a $\mathbb{C}$-algebra, then the following statements are equivalent. 
\begin{enumerate}
\item $\mathrm{B}\cong\Jac(f)$ for some element $f\in\ringtwo$ that has an algebraic length $\ell$.
\item $\mathrm{B}\cong\Jac(g)$ for some element $g\in\ringtwo$ in the list in \ref{cor:NecklaceSummary}.
\item $\mathrm{B}\cong\CA$ where $\CA$ is the contraction algebra of a smooth $3$-fold flop, or a smooth crepant divisor-to-curve contraction, with Koll\'ar length $\ell$.
\end{enumerate}
\end{thm}
\begin{proof}
(1)$\Rightarrow$(2) is the statement of \ref{cor:NecklaceSummary}.\\
(2)$\Rightarrow$(3) Consider an arbitrary $g$ in the list in \ref{cor:NecklaceSummary}.  If $\ell=1,2$ it is already known (see e.g.\ \cite{BW3}) that $\Jac(g)$ is isomorphic to a contraction algebra corresponding to either a flop or crepant divisor-to-curve with Koll\'ar length equal to $\ell$.  If $\ell=3$ with $p=0$, the single geometric example constructed in \ref{E6divcurveMain} below shows that that $\Jac(g)$ is isomorphic to a contraction algebra of a crepant divisor-to-curve contraction with Koll\'ar length~$3$.  In all other cases, namely $\ell=3$ with $g=x^3+xy^3+p$ where $p\neq 0$, or any $\ell=4,5,6$, by \ref{cor:NecklaceSummary} all such $\Jac(g)$ are finite dimensional.  Further, since each such $g$ is a \emph{finite} sum of necklace polynomials of the particular form in \ref{cor:NecklaceSummary}, we can directly import the main result of our companion paper \cite[1.2]{BW5} (which was entirely motivated by these particular forms!) to conclude that each $\Jac(g)$ is isomorphic to a contraction algebra for a smooth $3$-fold flop. Furthermore, \cite[\S2]{BW5} computes that these have Koll\'ar length equal to~$\ell$.\\
(3)$\Rightarrow$(1) That all contraction algebras have algebraic length is already known, and is implicit in \cite[(3.C)]{DW1}. To spell out the proof, consider a generic element $g\in\scrR$, which by Reid's general elephant \cite{Pagoda} satisfies $\scrR/g$ is ADE.  This $g$ induces a central element in the NCCR and thus its quotient the contraction algebra $\CA$, which we also denote by $g$.  The proof of \cite[(3.C)]{DW1}, which does not need $\scrR$ to be isolated since $g$ is generic, then shows that $\CA/g\cong e\Uppi e$, and thus $\CA$ has algebraic length.
\end{proof}

\begin{remark}\label{rem:converseAlgLength}
As a consequence of \ref{thm:Summary}(2)$\Rightarrow$(1), it follows that \emph{every} $g$ in the list in \ref{thm:Summary} has an algebraic length.  This is somewhat remarkable, given that the proof does not explicitly write down the central element which realises the needed $e\Uppi e$.  In general, the central elements are very complicated; for context, the very easiest case within $\ell=3$ are the elements $g_{6,n}$ in \cite[\S7.3]{BW3}.
\end{remark}

\begin{cor}\label{cor:classifyAcon}
The following statements hold.
\begin{enumerate}
\item\label{cor:classifyAcon1} Contraction algebras of smooth $3$-fold flops are equivalently characterised, up to isomorphism as the following.
\begin{enumerate}
\item Those $\Jac(f)$ where $f$ has an algebraic length, and $\dim_\mathbb{C}\Jac(f)<\infty$.
\item Those $\Jac(g)$ where $g$ is one of the following
\[
\renewcommand{\arraystretch}{1.3}
\medmath{\begin{array}{c@{\hspace{20pt}}l@{\hspace{2pt}}l@{\hspace{20pt}}l@{\hspace{2pt}}}
 1&x^2 + y^n && n\geq 2 \\
2&x^2y+ y^{2n}+\upvarepsilon y^{2m-1} && 2n-1\geq m\geq 2  \\
3&x^3 + xy^3 + y^n + y^{n+1}q& &q\in\C[y],\,n\geq 4  \\
4&x^3 + y^4+\scrN_{2,3}+\sum_{(i,j)\in\scrT_4}\upalpha_{i,j}\scrN_{i,j}  && \\
5&x^3+\scrN_{2,2}+y^5+xy^5p+y^6q  && p,q\in\C[y]\\
6&x^3 + y^4+\scrN_{3,3}+\uplambda\scrN_{2,4}+\sum_{(i,j)\in\scrT_6}\upalpha_{i,j}\scrN_{i,j} && \uplambda\neq 0\\
\end{array}}
\renewcommand{\arraystretch}{1}
\]
where the length indexes the rows, $\upvarepsilon\in\{0,1\}$, $\upalpha_*\in\C$ and $\scrT_4,\scrT_6$ are defined in \eqref{eqn:DefS4} and \eqref{eqn:DefS6}.
\end{enumerate}
\item\label{cor:classifyAcon2} Contraction algebras of smooth crepant divisor-to-curve contractions are equivalently characterised up to isomorphism as the following.
\begin{enumerate}
\item Those $\Jac(f)$ where $f$ has an algebraic length, and $\dim_\mathbb{C}\Jac(f)=\infty$.
\item Those $\Jac(g)$ where $g$ is one of the following
\[
\renewcommand{\arraystretch}{1.3}
\medmath{\begin{array}{c@{\hspace{20pt}}lll}
1& x^2 &  \\
2&x^2y + \upvarepsilon y^{2m-1} & m\geq2  \kern 135pt\\
3&x^3 + xy^3 & 
\end{array}\kern 33pt}
\renewcommand{\arraystretch}{1}
\]
where the length indexes the rows, and $\upvarepsilon\in\{0,1\}$.
\end{enumerate}
\end{enumerate}
\end{cor}
\begin{proof}
(1) Contraction algebras of flops are characterised as the subset of contraction algebras which are finite dimensional \cite[1.4]{DW2}.  Given (1)$\Leftrightarrow$(2)$\Leftrightarrow$(3) in \ref{thm:Summary}, taking the subset throughout which are finite dimensional (using \ref{cor:NecklaceSummary}\eqref{cor:NecklaceSummary2} to identify the relevant subset within the given list) then proves the statement.  In that analysis, note that for $\ell=3$, when $p\neq 0$ we can scale so that the first non-zero term in $y^4p$ is a pure power of~$y$, and thus replace $y^4p$ with $y^{n}+y^{n+1}q$ for some (possibly zero) $q\in\C[y]$, where $n\geq 4$.\\ 
(2) Similarly, this follows by taking the subset of (1)$\Leftrightarrow$(2)$\Leftrightarrow$(3)  in \ref{thm:Summary} which are infinite dimensional, using \ref{cor:NecklaceSummary}\eqref{cor:NecklaceSummary2} to know which in the list are infinite dimensional.
\end{proof}

\begin{cor}\label{rem:KMok}
If $f\in\ringtwo$ has algebraic length, then there exists a central element $z$ for which $\Jac(f)/\lcl z\rcl$ is one of the six algebras in \eqref{eqn:6keyalgebras}.
\end{cor}
\begin{proof}
By \ref{thm:Summary} every such $\Jac(f)$ is the contraction algebra of either a flop, or a crepant divisor-to-curve contraction.  When it is a flop, we can appeal to Katz--Morrison \cite{KM} to say that the flop generically slices to one of the six marked curves illustrated in \ref{def!6algebras}.  But then, by \cite[(3.C)]{DW1}, this generic slice induces a central element $z$ for which $\Jac(f)/\lcl z\rcl$ is isomorphic one of the six algebras in the list in \ref{cor:classifyAcon}\eqref{cor:classifyAcon1}.

When it is a divisor-to-curve contraction, by \ref{cor:classifyAcon}\eqref{cor:classifyAcon2} there are very few options. In the first line, $y$ is clearly central and $\Jac(f)/\lcl y\rcl\cong\C$, which is the first algebra in \eqref{eqn:6keyalgebras}. In the second line, by \cite[7.7]{BW3} $x^2$ is central and $\Jac(f)/\lcl x^2\rcl$ is isomorphic to the second algebra in \eqref{eqn:6keyalgebras}. In the third line, $z=x^2+xyx+yx^2$ is central, and 
%Note that $x^2$ by itself does not work, by \cite[7.8]{BW2}} $x^2$ is clearly central in $\Jac(f)\cong\tfrac{\ringtwo}{\lcl x^2+y^3, xy^2+yxy+y^2x\rcl}$ since it clearly commutes with both generators $x$ and $y$.  Thus $xyxy-yxyx\equiv [x,\dcyc_y f]\in \lcl \dcyc f\rcl$, so exactly as in \cite[7.6]{BW2}  $xyx+yx^2$ is also central in $\Jac(f)$.  Thus the sum $z=x^2+xyx+yx^2$ is central.
the quotient 
\begin{equation}
\Jac(f)/\lcl z\rcl \cong \frac{\ringtwo}{\lcl x^2+y^3, xy^2+yxy+y^2x, x^2+xyx+yx^2\rcl}\label{eqn:E6previous}
\end{equation}
is precisely \cite[(7.D)]{BW3} with $\uplambda=1$, $\upmu=-1$ after rescaling $y\mapsto-y$.  It was shown in \cite[7.7]{BW3}  that \cite[(7.D)]{BW3} (for any choice of $\uplambda\neq 0$ and $\upmu\neq 0$), and thus \eqref{eqn:E6previous}, is isomorphic to the third algebra in \eqref{eqn:6keyalgebras}.  
\end{proof}
 
\subsection{The Classification of Flops}\label{subsec:classflops}

Leading up to the following, consider the following polynomials in $\C[a^{\pm1},z]$ defined as
\begin{equation}\label{def:Qs}
\renewcommand{\arraystretch}{1.5}
\medmath{\begin{array}{ll}
Q_3\colonequals \upalpha_{n+1}a^{-(n-1)}+\upalpha_{n+2}a^{-n}z^{1}+\hdots+\upalpha_{\kx}a^{2-\kx}z^{\kx-1-n} \\
Q_4\colonequals\sum_{(i,j)\in\scrT_4}\upalpha_{i,j}a^{2-j}z^{i+j-1} \\
Q_5\colonequals  \sum_{k\geq 6}(\upalpha_k a^{3-k} + \upbeta_k a^{2-k})  z^{k-1}  \\
Q_6\colonequals\sum_{(i,j)\in\scrT_6}\upalpha_{i,j}a^{2-j}z^{i+j-1}
\end{array}}
\renewcommand{\arraystretch}{1}
\end{equation}
for some choice of finitely many scalars $\upalpha_*,\upbeta_*\in\C$, where $\scrT_4,\scrT_6$ are from \eqref{eqn:DefS4} and \eqref{eqn:DefS6}.  

\begin{cor}\label{cor:flopsMain}
Let $\scrX\to\scrX_{\con}$ be a smooth \tf\ flopping contraction. 
Then the completion of $\scrX$ along $\Curve$ is isomorphic to the completion of the variety defined by gluing two copies of $\mathbb{A}^3$ via the rule
\[
(a,z,y) \stackrel{a\neq 0}{\longleftrightarrow} (a^{-1}, \,a^{-1}z, \,\,a^{3}y + \partial E)
\]
along the curve $y=z=0$, for some $\partial E\in\mathbb{C}[a^{\pm1},z]$ in the following list
\[
\renewcommand{\arraystretch}{1.5}
\medmath{\begin{array}{c@{\hspace{30pt}}r@{\hspace{2pt}}l@{\hspace{2pt}}c@{\hspace{2pt}}l@{\hspace{2pt}}c@{\hspace{2pt}}c@{\hspace{2pt}}c@{\hspace{20pt}}l@{\hspace{2pt}}l@{\hspace{2pt}}l@{\hspace{2pt}}c@{\hspace{2pt}}}
1 &a^2z &+& a^{2-n}z^{n-1}&&&&&
n\geq 2\\
2 &az^2 &+& a^{2-2n}z^{2n-1} &+& \upvarepsilon a^{3-2m}z^{2m-2}  &&&
 2n-1\geq m\geq 2,\, \upvarepsilon\in\{0,1\}\\
3&a^2z^2 &+& a^{-1}z^3 &+& a^{2-n}z^{n-1}&+&z^nQ_{3} & n\geq 4  \\ 
4& a^2z^2 &+& a^{-2}z^3 &+& a^{-1}z^4 &+& Q_4  &  \\ 
5&a^2z^2 &+& \phantom{a^{-2}}z^3 &+& a^{-3}z^4 &+& Q_5 &\\
6&a^2z^2 &+&  a^{-2}z^3 &+& (a^{-1} + \uplambda a^{-2})z^5  &+& Q_6  & \uplambda\in\C^* \\
\end{array}}
\renewcommand{\arraystretch}{1}
\]
where $Q_{3},Q_4,Q_5,Q_6\in\mathbb{C}[a^{\pm1},z]$ are polynomials in $z$ of the form described in \eqref{def:Qs}.  Conversely, every member of the above list, for every choice of coefficients in $Q_{3},Q_4,Q_5,Q_6$, describes a flopping neighbourhood of the given length. 
\end{cor}
\begin{proof}
That all such glues realise flops is the main result of the companion paper \cite{BW5}, the content here is the converse, which we claim is now a simple consequence of \ref{cor:classifyAcon}\eqref{cor:classifyAcon1} together with the fact that the Donovan--Wemyss conjecture is true \cite{JKM, KLW}. Indeed, by \cite{JKM, KLW} smooth (irreducible) $3$-fold flops are classified by the isomorphism classes of their contraction algebras.  Thus, smooth $3$-fold flops are classified by the isomorphism classes of the Jacobi algebras in  \ref{cor:classifyAcon}\eqref{cor:classifyAcon1}.  By the main result of our companion paper \cite[1.1(1)]{BW5}, \emph{each} of the $\Jac(h)$ in  \ref{cor:classifyAcon}\eqref{cor:classifyAcon1} is the contraction algebra of a flop $\scrY\to\Spec\scrR$ whose completion along the flopping curve is obtained via the glueing $\partial E$ precisely as in the statement.  Since by \ref{cor:classifyAcon}\eqref{cor:classifyAcon1} there are no further contraction algebras other than those given directly above, there can be no further flops other than those in the statement.

The final statement claiming that each flop has the stated length is \cite[1.1(2)]{BW5}.
\end{proof}

\begin{remark}
As a point of warning, the curve $\Curve$ in \ref{cor:flopsMain} contracts Zariski locally when $\ell\leq 3$ (by luck, essentially), but it only contracts analytically locally when $\ell\geq 4$. 
\end{remark}

\begin{remark}\label{rem:EquationsMessy}
The above \ref{cor:flopsMain} gives the classification of flops in clear geometric form, but notwithstanding this, it would be convenient to have additional geometric information. A more explicit description of the base $\scrR$ would be helpful, albeit the equations are in general very complicated (see e.g.\ \cite[p.16]{BW4}).  It is both convenient and remarkable that the noncommutative perspective allows the existence of continuous parameters to be established easily in \S\ref{sec:Moduli} below, without needing to explicitly describe $\scrR$.
\end{remark}

\subsection{First Applications}\label{subsec:GVfinally}
The first consequence is to numerical curve counting invariants.  Below, we focus on the Gopakumar--Vafa (GV) invariants of flopping contractions as defined by Katz \cite{KatzGV}, which are a tuple of numbers $\GV{1},\hdots,\GV{\ell}$ where $\ell$ is the length.  It is well-known (see e.g.\ \cite{BKL}) how to translate between these and all other types of numerical invariants, such as Gromov--Witten (GW) invariants. 

The main result of this subsection, \ref{cor:allDims}, will use known lower bounds on the GV invariants from \cite{DW5}, together with the upper bounds on the dimension of the contraction algebras from \S\ref{sec:FD}, to simultaneously extract the precise value of the dimension, and also the precise value of the invariants.  For this argument to work requires only the following two additional results.  Since both proofs are ultimately just $A_\infty$-calculations that are very similar in spirit to \cite{BW2}, we postpone the proofs until Appendix~\ref{appendix}.

\begin{lemma}[{\S\ref{subsec:ProofOfGVl3}}]\label{lem:GV2forl3}
Consider those $\ell=3$ flops, given in \S\ref{subsec:IntrinsMain}--\ref{subsec:classflops} by the elements $x^3 + xy^3+ y^n + y^{n+1}q$ where $n>4$ and $q\in\C[y]$.  For these, $\GV{2}=n-1$.
\end{lemma}
\begin{lemma}[{{\S\ref{subsec:ProofOfGVl4}}}]\label{lem:GV2forl4}
Consider $\ell=4$ flops, classified in \S\ref{subsec:IntrinsMain}--\ref{subsec:classflops}. For these, $\GV{2}\geq  5$.
\end{lemma}

\begin{cor}\label{cor:allDims}
Let $\scrX\to\Spec\scrR$ be a smooth $3$-fold flop as in \ref{cor:flopsMain}, with associated contraction algebra $\CA$.  Then $\dim_\mathbb{C}\CA$, and the \textnormal{GV} invariants, are the following.
\[
\medmath{
\begin{array}{l@{\hspace{20pt}}l@{\hspace{20pt}}l@{\hspace{10pt}}l@{\hspace{10pt}}l@{\hspace{10pt}}}
%\toprule
\ell & \dim_\C\CA & \textnormal{GV invariants}&\\[2pt]
1 & n-1 & n-1 &&\\[3pt]
2&2m+4n-4 & 2m,n-1 & \mbox{if }\upvarepsilon=1 \mbox{ and } n\geq m\\
2& 6n-3& 2n+1,n-1 & \mbox{if }\upvarepsilon=1 \mbox{ and } n<m\\
2& 6n-3& 2n+1,n-1 & \mbox{if }\upvarepsilon=0 & \\[3pt]
3& 27 & 6,3,1 & \mbox{if }n=4 & \\
3& 4n+12 & 7,n-1,1 &\mbox{if }n>4\\[3pt]
4 & 60 & 6,5,2,1 \\[2pt]
5 & 125 & 8,6,4,2,1  \\[2pt]
6 & 200 & 6,6,4,3,2,1
\end{array}}
\]
\end{cor}
\begin{proof}
The proof will make extensive use of Toda's formula \cite{TodaGV}, namely 
\begin{equation}
\dim_\mathbb{C}\CA=\sum_{i=1}^\ell i^2\cdot \GV{i},\label{eqn:Toda}
\end{equation}
and furthermore $\GV{1}=\dim_\mathbb{C}\CAab$ where $\CAab$ is the abelianisation of $\CA$.  The main point below is that we can now use the classification of contraction algebras in \ref{cor:classifyAcon}\eqref{cor:classifyAcon1} to calculate the abelianisation.

For $\ell=1,2$ the answer is already known (see e.g.\ \cite{Okke}, or \cite[Figure 4]{BW2}).   The case $\ell=3$ splits into two subcases.  

\noindent
(a) When $\upalpha_3\neq 0$, equivalently $\ord(p)=0$, then it is easy to check that $\GV{1}=\dim_\C\CAab=6$. Further $\GV{2}=3$ by \ref{lem:GV2forl3}, $\dim_\C\CA\leq 27$ by \ref{prop:l3finite4}, and $\GV{3}\geq 1$ by \cite[7.7]{DW5}. Hence Toda's formula \eqref{eqn:Toda} forces $\dim_\C\CA=27$, and $\GV{3}=1$.   

\noindent
(b) When $\upalpha_3= 0$, equivalently $\ord(p)\geq 1$. Now $\GV{1}=\dim_\C\CAab=7$. Further $\GV{2}=\ord(p)+3$ by \ref{lem:GV2forl3}, $\dim_\C\CA\leq 4\ord(p)+28$ by \ref{prop:l3finite}, and $\GV{3}\geq 1$ by \cite[7.7]{DW5}. Hence \eqref{eqn:Toda} forces $\dim_\C\CA=4\ord(p)+28=4n+12$, and $\GV{3}=1$. 

For $\ell=4$, by \cite[7.7]{DW5} the theoretical lower bounds for the GV invariants are $6,4,2,1$, and thus by Toda's formula \eqref{eqn:Toda} $\dim_\C\CA\geq 56$.  Using \ref{cor:classifyAcon}\eqref{cor:classifyAcon1} it is easy to see that $\GV{1}=\dim_\C\CAab=6$, and by \ref{lem:GV2forl4} $\GV{2}\geq 5$.  We know from \ref{prop:l4finite} that $\dim_\C\CA\leq 60$.  If either $\GV{3}$ or $\GV{4}$ is greater than the theoretical minimum, this contradicts Toda's formula.  This then fixes $\GV{1}$, $\GV{2}$, $\GV{3}$ and $\GV{4}$, and also forces $\dim_\C\CA=60$.

For $\ell=5$,  by \cite[7.7]{DW5} the theoretical lower bounds for the GV invariants are $7,6,4,2,1$, and thus by Toda's formula \eqref{eqn:Toda} $\dim_\C\CA\geq 124$.  Using \ref{cor:classifyAcon}\eqref{cor:classifyAcon1} it is easy to see that $\GV{1}=\dim_\C\CAab=8$, and thus $\dim_\C\CA\geq 125$.  But by \ref{prop:l5finite} $\dim_\C\CA\leq 125$, thus equality must hold, and the GV numbers must be as in the statement.

For $\ell=6$,  by \cite[7.7]{DW5} the theoretical lower bounds for the GV invariants are $6,6,4,3,2,1$, and thus by \eqref{eqn:Toda} $\dim_\C\CA\geq 200$. But by \ref{prop:l6finite} $\dim_\C\CA\leq 200$, thus equality must hold, and further the GV numbers must be as stated.
\end{proof}

Note that the precise dimension of the contraction algebra is obtained in \ref{cor:allDims} without ever verifying that the $\{\mathsf{g}_i\}$ constructed in \S\ref{subsec:l3bound}--\ref{subsec:l6bound} are in fact a standard basis, which is an extremely cumbersome verification.  In fact, these are a standard basis, but this follows only as a consequence of the geometry in \ref{cor:allDims}, not as an input.
\begin{cor}\label{cor:IsStandardBasis!}
Fix any $\ell\in\{3,4,5,6\}$ and any potential $g\in\ring$ of length~$\ell$ as specified in the respective \ref{prop:l3finite} or \ref{prop:l3finite4} for $\ell=3$, or \ref{prop:l4finite}, \ref{prop:l5finite}  \ref{prop:l6finite} for $\ell=4,5,6$.
Consider the set $\scrG=\{\sfg_1,\hdots,\sfg_{n_\ell}\}$ constructed by the respective proofs in {\rm \S}\ref{subsec:l3bound}--{\rm \S}\ref{subsec:l6bound}, where $n_3=7$, $n_4 = 5$, $n_5=n_6=11$.
Then $\scrG$ is a standard basis of $\lcl\dcyc g\rcl$.
\end{cor}
\begin{proof}
The results in \S\ref{subsec:l3bound}--\S\ref{subsec:l6bound} show that $\scrG\subset \lcl\dcyc g\rcl$ and compute $\dim_\C\ring/\lcl\scrS\rcl$, where $\scrS=\LT(\scrG)$. If $\scrS$ does not generate $\LT\lcl\dcyc g\rcl$, then we obtain the following contradiction:
\[
\dim_\C\Jac(g) \buildrel{\ref{sec:stdbasis}\eqref{GI3} }\over{=} \dim_\C\ring/\LT\lcl\dcyc g\rcl<\dim_\C\ring/\lcl\scrS\rcl\buildrel{\ref{cor:allDims}}\over{=}\dim_\C\Jac(g).
\]
Thus $\scrS$ generates $\LT\lcl\dcyc g\rcl$, which is to say that $\scrG$ is a standard basis.  \end{proof}

\subsection{Divisor-To-Curve Contractions} 
The following example realises $\Jac(x^3+xy^3)$ as a contraction algebra of a crepant divisor-to-curve contraction.
If the analogue of the Donovan--Wemyss conjecture holds in this setting, then this is the only example of such an $\ell\geq 3$ crepant divisor-to-curve contraction.

\begin{example}\label{ex:l3div}
Set $\A^4=\Spec\C[x,y,z,t]$ and $f=x^5 - 3x^2yz + xy^3 + z^3 - t^2$.
Consider the variety $Y = \Spec R\subset\A^4$ where $R=\C[x,y,z,t]/(f)$.
Clearly $Y$ has a $c{\mathrm{E}}_6$ singularity at the origin, setting $y=0$.

The reduced singular locus $C=\Sing(Y)_{\red}$ is defined by $I_C=(t,u,v,w)$ where $u=xz-y^2$, $v=x^2y-z^2$, $w=x^3-yz$ (saturating the Jacobian ideal easily against $x$ to remove the embedded point at the origin). This is the image of the monomial normalisation
\[
\begin{array}{rl}
\upnu\colon&\A^1\longrightarrow\A^4 \\
&s\mapsto(s^3,s^4,s^5,0)
\end{array}
\textrm{ and }
C = \upnu(\A^1) = 
\left(\bigwedge^2M=0
\right)
\textrm{ for }
M=
\begin{pmatrix}
x&y&z\\y&z&x^2
\end{pmatrix}
\]
so is an irreducible curve, smooth away from a singular point at the origin. 
Any $3\times3$ minor of the Hessian of $f$ at $\upnu(s)$ that involves $\partial^2f/\partial t^2$ is a power of $s$, so the singularity of $Y$ along $C$ away from the origin is of transverse ${\mathrm{A}}_1$ type.

We compute the blow up $\phi\colon X = \Bl_C Y\rightarrow Y$ of $C\subset Y$ by introducing ratios of the given generators of $I_C$.
For the blowup chart $U_1=(t\ne0)$, the ideal
\[
\left(t\cdot\begin{pmatrix}a_2,-a_3,a_1\end{pmatrix} - \bigwedge^2 M\right) \subset R[a_1,a_2,a_3]
\]
is easily saturated against $t$ by setting 
\[
\begin{array}{rcl}
\upxi = -a_1a_2 + a_3^2 &\quad& \upeta = xa_1^2 + a_2a_3 \\
\upzeta = xa_1a_3 + a_2^2 &\quad& \uptau = -x^2a_2 + xa_1 + za_2
\end{array}
\]
and observing that the easy identities
\[
\begin{array}{rcl}
t^2\upxi = x(f + t^2) &\quad& t^2\upeta = y(f + t^2) \\
t^2\upzeta = z(f + t^2) &\quad& t\uptau = -f + t^2
\end{array}
\]
give $U_1 = \left\{x=\upxi, y=\upeta, z = \upzeta, t = \uptau\right\}
\cong \Spec\C[a_1,a_2,a_3]$. Thus $U_1=\A^3$ is a smooth patch with blowdown map
$\phi(a_1,a_2,a_3) = (\upxi,\upeta,\upzeta,\uptau)$ (substituting in $z=\upzeta$ and $x=\upxi$).

It is enough to compute one other chart, $U_2=(u\ne0)$.
Setting $b_1 = t/u$, $b_2 = (x^2y-z^2)/u$, $b_3=(x^3-yz)/u$ it is immediate that
\[
u(xb_2-yb_3+z) = 0,\quad
u^2(x - yb_1^2 + b_2b_3) = u^2yf,\quad
u^2(xb_1^2 - b_3^2 + b_2) = -xu^2f.
\]
The raw blowup ideal $\left(f, b_1u - t, b_2u - (x^2y-z^2), b_3u - (x^3-yz) \right)  \subset R[b_1,b_2,b_3]$
saturated against $u$ is therefore
\[
\left(ub_1 - t, xb_1^2 - b_3^2 + b_2, x - yb_1^2 + b_2b_3, xb_2-yb_3+z \right)  \subset R[b_1,b_2,b_3]
\]
from which we eliminate $z,t,b_2$ at once to see
\[
U_2 = \Spec\C[x,y,b_1,b_3]/(x - yb_1^2 - xb_1^2b_3 + b_3^3).
\]
This is clearly smooth, with curve above the origin $\phi^{-1}(0,0,0,0) = \Spec\C[b_1,b_3]/b_3^3$, which has length~3.

%It doesn't seem helpful to calculate that the exceptional divisor is
%\[
%ya^2 + xa^2c - x - c^3 = -y^3 + x^4 =
% -y^2 + yxc + x^3a^2 - x^2c^2 = 0.
%\]

Woking in $U_1$ and denoting the standard regular $3$-forms by
$\upomega_{Y}=\frac{1}{t}dx\wedge dy\wedge dz$ and $\upomega_{X}=da_1\wedge da_2 \wedge da_3$,
it follows at once that
\[
\phi^*\left(\upomega_{Y}\right) = \frac{1}{\phi^*(t)}\det\begin{pmatrix}a_2&a_1&-2a_3\\-3a_1^2a_2 + 2a_1a_3^2 & -a_1^3 + a_3 & 2a_1^2a_3 + a_2\\
-2a_1a_2a_3 + a_3^3 & -a_1^2a_3 + 2a_2 & -a_1^2a_2 + 3a_1a_3^2
\end{pmatrix} \upomega_{X} = \upomega_{X}
\]
so $\phi$ is crepant, and it is necessarily a divisorial contraction to resolve the generic point of $C\subset Y$.
(In fact, it is easy to see that $\phi^{-1}(C)\subset X$ is defined by $\phi^*(t)=0$.)
\end{example}

\begin{prop}\label{E6divcurveMain}
Consider $\scrX_{\con}\colonequals\Spec\scrR$ where $\scrR\colonequals\C\llsq x,y,z,t\rrsq/(f)$
with $f=x^5 - 3x^2yz + xy^3 + z^3 - t^2$.
\begin{enumerate}
\item\label{E6divcurveMain1}
$C=(\Sing \scrX_\con)_{\mathrm{red}}=(t, xz-y^2, x^2y-z^2,x^3-yz)$, a curve germ, with a singular point at the origin.
\item\label{E6divcurveMain2} Blowing up $C\subset \scrX_\con$ gives rise to a crepant $\ell=3$ divisor-to-curve contraction $\scrX\to\scrX_\con$ where $\scrX$ is smooth with $\P^1$ as the reduced fibre above the origin.
\item\label{E6divcurveMain3} The contraction algebra $\CA$ of $\scrX\to\scrX_\con$ is isomorphic to $\Jac(x^3+xy^3)$.
\end{enumerate}
\end{prop}
\begin{proof}
Consider $\phi\colon X\rightarrow Y$ from \ref{ex:l3div}.
(1)~$\scrX_\con$ is the completion of $Y$ at the origin.
(2)~$\scrX$ is the formal fibre of $\phi\colon X\rightarrow Y$ above the origin.
(3)~By \eqref{E6divcurveMain2} $\scrX$ has Koll\'ar length~$3$, so  $\CA$ has algebraic length~$3$ by \ref{thm:Summary}. Further $\dim_\C\CA=\infty$ by \cite{DW2}, since the contraction is divisor-to-curve.  By \ref{cor:classifyAcon}\eqref{cor:classifyAcon2} there is, up to isomorphism, only one algebra with this property, namely $\Jac(x^3+xy^3)$. Hence  $\CA\cong\Jac(x^3+xy^3)$, as claimed.
\end{proof}

\begin{prop}\label{cor:NoLength456Main}
Smooth crepant divisor-to-curve contractions with $\ell\geq 4$ do not exist.
\end{prop}
\begin{proof}
Should such a contraction $\scrX\to\Spec\scrR$ exist, by \ref{thm:Summary}(3)$\Rightarrow$(1), $\CA\cong\Jac(f)$ where $f$ has algebraic length $\geq 4$.  Since the contraction is divisor-to-curve, $\CA$ is supported on the base curve, and so is infinite dimensional \cite{DW2}.  However, by \ref{cor:NecklaceSummary}, such an algebra does not exist.
\end{proof}

\begin{remark}
The example in \ref{E6divcurveMain} appears in the thesis of Curto \cite[p91]{Curto}, although there it is assumed that $\scrX\to\Spec \scrR$ is projective birational with isolated base $\scrR$.  It is the infinite dimensionality of the (noncommutative) contraction algebra which most readily illustrates this mistake; note that in \ref{E6divcurveMain}, $\dim_\C\CA=\infty$ but $\dim_\C\CAab=7$. See also \S\ref{subsec:physics} below.
\end{remark}

\begin{remark}\label{rem:Issue90s}
The result \cite[3.1]{WilsonSymp}, used in e.g.\ \cite{WilsonIII}, claims that every crepant divisor-to-curve contraction (i.e.\ primitive Type III contraction) contracts to a smooth curve, and further the exceptional locus is a conic bundle over this curve.  The example \ref{ex:l3div}--\ref{E6divcurveMain} shows that this is false.  However, we believe that when $\ell=1,2$ the claim in  \cite[3.1]{WilsonSymp} is true, so in those cases \cite{WilsonIII} remains valid.  This is to be expected: roughly speaking \cite{WilsonIII} claims that the GW invariants are determined by two numbers $(n_1,n_2)$, which is consistent with the fact $\ell\leq 2$.  We should expect the GW invariants when $\ell=3$ to be determined by three numbers. 
\end{remark}

\subsection{Relationship to Physics}\label{subsec:physics}
Typically the physics literature for rational curves (e.g.\ \cite{AspinwallKatz,AM,CM,Katz}) considers superpotentials viewed inside the commutative power series ring $\C\llsq x,y\rrsq$. Ferrari \cite[\S3.3.1]{Ferrari} did observe that ordering ambiguities arise, suggesting noncommutativity, but most physics-inspired predictions remain aimed squarely within commutative deformation theory. This subsection explains what these predictions were, why they are false as stated, whilst at the same time acknowledging that these predictions greatly influenced this paper.

The Intriligator--Wecht classification (ignoring `hat' geometries), summarised in \cite[\S6.2]{Curto}, is the following, where we have in some cases swapped the variables $x$ and $y$ to make the comparison below easier. 
\begin{equation}\label{eqn:physics}
\renewcommand{\arraystretch}{1.3}
\medmath{\begin{array}{lll}
A_k& \mathrm{Tr}(x^2+y^{k+1}) &k\geq 1\\
D_{k+2}&\mathrm{Tr}(x^2y+ y^{k+1})&k\geq 2 \\
E_6&\mathrm{Tr}(x^3+y^4) \\
E_7&\mathrm{Tr}(x^3+xy^3) \\
E_8&\mathrm{Tr}(x^3+y^5) 
\end{array}}
\renewcommand{\arraystretch}{1}
\end{equation}
The notation $\mathrm{Tr}(-)$ reminds us that the superpotential is commutative, but we will typically drop this below.  There was some speculation that the superpotentials in \eqref{eqn:physics} corresponded to some flopping geometries, and in fact \cite[Table 6.1]{Curto} even gives some predictions on which glues realise these flopping geometries.

There are three fundamental problems.  The first is that \eqref{eqn:physics} contains only five lines, whereas the classification of flops (via the length invariant) contains six. The second is that there is a unique example stated in Types $E_6$, $E_7$ and $E_8$ above, whereas after the classification of flops there are moduli in all these types (see \S\ref{sec:ExistModuli} below). The third, which is actually the most serious, is that by \cite{BW3} there are \emph{many} glues which realise the commutative potentials in \eqref{eqn:physics}, and these glues have radically different properties.  It is, for example, perfectly possible to construct two glues, each with commutative superpotential $x^3+y^4$, for which in one the curve contracts, and in the other it does not.

Ignoring these issues for the moment, with the classification of flops now known via contraction algebras (=noncommutative deformation theory) by \S\ref{subsec:IntrinsMain} and \S\ref{subsec:classflops}, taking the abelianisation allows us to fully determine all possible commutative deformation algebras for the sheaves $\scrO_\Curve$ arising in smooth $3$-fold flops.  This is straightforward: after commuting variables, $\scrN_{i,j}$ turns into $\frac{1}{i+j}\binom{i+j}{j} x^iy^j$, and thus viewed in $\C\llsq x,y\rrsq$ the noncommutative potentials of \ref{cor:classifyAcon}\eqref{cor:classifyAcon1} become (where in length four we now prefer \ref{thm!funnel4}\eqref{thm!funnel4A} over \ref{thm!funnel4}\eqref{thm!funnel4B})
\[
\renewcommand{\arraystretch}{1.4}
\medmath{\begin{array}{c@{\hspace{20pt}}l}
1& x^2 + y^n \\
2&x^2y+ y^{2n}+\upvarepsilon y^{2m-1} \\
3&x^3 + xy^3+ y^n+ y^{n+1}q  \\
4&x^3 + y^4 + \tfrac{10}{5}x^2y^3+\tfrac{15}{6}\upalpha_1x^2y^4+ \tfrac{20}{6}\upalpha_2 x^3y^3+ xy^7p +y^6q  \\
5&x^3+\tfrac{6}{4}x^2y^2+y^5+xy^5p+y^6q \\
6&x^3 + y^4+ \tfrac{20}{6}x^3y^3+\uplambda\tfrac{15}{6} x^2y^4+\sum_{(i,j)\in\scrT_6}\frac{\upalpha_{i,j}}{i+j}\binom{i+j}{j} x^iy^j 
\end{array}}
\renewcommand{\arraystretch}{1}
\]
The process of commuting variables opens out more changes in coordinates; for example in $\ell=5$ the $x^2y^2$ term can now be attacked using $x^3$, via $x\mapsto x-\tfrac{1}{3}y^2$. Doing this repeatedly in all cases, the above turn out to be isomorphic to
%% Type A obvious
%% Type D goes with the lowest degree x^power, and lowest is x^3 
%% l=3 then ff := x^3 + y^4 and f := x^3 + x*y^3 + y^4 are iso
%%		but for x^3 + x*y^3 + y^5 the xy^3 term dominates ( y to x+y, clears xy^3, with x2 terms which then clear via x3).
%% l=4 clear the x^2 terms using x^3. Clear y's using y^4. Easy to see can repeat
%% l=5 x--> x -  y^2   clears the x^2y^2 at the cost of xy^4 and y^6. As for l=3, the y^5 dominates the xy^4. Then peel off unit, I think.
%% l=6 use x^3 and y^4 to clear all the others. 
%% Dims of the ones below match the GV_1
\begin{equation}\label{eqn:reality}
\renewcommand{\arraystretch}{1.3}
\medmath{\begin{array}{c@{\hspace{20pt}}ll}
1&x^2 + y^n & n\geq 2\\
2&x^2y + y^{k+1} & k\geq 2 \\
\multirow{2}{*}{3}&x^3 + y^4&\textnormal{if }n=4\\[-3pt]& x^3 + xy^3&\textnormal{if }n>4 \\
4&x^3 + y^4  \\
5&x^3 + y^5 \\
6&x^3 + y^4 & \\
\end{array}}
\renewcommand{\arraystretch}{1}
\end{equation}
Comparing \eqref{eqn:physics} and \eqref{eqn:reality}, Type $A$ agrees, as does Type $D$, but there is far less correlation amongst the others. However, note that every commutative superpotential for flops is, up to isomorphism, contained in \eqref{eqn:reality}, and thus by inspection is contained within the list (somewhere) of the Intriligator--Wecht superpotentials in \eqref{eqn:physics}.  Further, all such Intriligator--Wecht superpotentials arise in this way, from some flop.

%-----------------------------------------------------------------------------------------
\section{Moduli for Exceptional Types}\label{sec:ExistModuli}

Flops with $\ell=1,2$ are classified by discrete parameters \cite[8.10]{BW3}.
By contrast, in this section we show that there are continuous parameters in all $\ell\geq 3$ cases.  By \ref{cor:classifyAcon}\eqref{cor:classifyAcon1} it suffices to show that the isomorphism classes of the (finite dimensional) Jacobi algebras with algebraic length $\ell\geq 3$ admit continuous parameters, and to achieve this requires us to determine when two such Jacobi algebras are not isomorphic.

%-----------------------------------------------------------------------------------------
\subsection{General Setup and Filtered Shirayanagi}\label{sec!modulisetup}
Below we will assume the existence of an algebra isomorphism $\phi\colon\Jac(g)\to\Jac(f)$, which is automatically an isomorphism of $\C$-vector spaces.  The standard basis of $\lcl\dcyc f\rcl$ determines a monomial $\fJ$-adic basis $\scrB$ of $\Jac(f)$. The isomorphism $\phi$ must therefore take the two generators of $\Jac(g)$ to a unique linear combinations of the elements in $\scrB$, namely
\begin{equation}\label{eq!mapeq}
\begin{array}{c@{\hspace{4pt}}c@{\hspace{4pt}}c@{\hspace{2pt}}c@{\hspace{2pt}}c@{\hspace{2pt}}c@{\hspace{2pt}}c@{\hspace{2pt}}c@{\hspace{2pt}}c@{\hspace{2pt}}c@{\hspace{2pt}}c@{\hspace{2pt}}c@{\hspace{2pt}}c@{\hspace{2pt}}c@{\hspace{2pt}}c@{\hspace{2pt}}c@{\hspace{2pt}}}
x&\mapsto&  a_1x &+& a_2y &+& a_3x^2 &+& a_4xy &+& a_5yx &+& a_6y^2 &+& \bigO_3 \\
y&\mapsto&  b_1x &+& b_2y &+& b_3x^2 &+& b_4xy &+& b_5yx &+& b_6y^2 &+& \bigO_3
\end{array}
\end{equation}
for $a_i,b_i\in\C$ satisfying
$\det\left(\begin{smallmatrix}a_1&a_2\\b_1&b_2\end{smallmatrix}\right)\not=0$, where as standard we have mildly abused notation and not written $+\lcl\dcyc g\rcl$ in the expressions on the left hand side, and not written $+\lcl\dcyc f\rcl$ on the right hand side. The two sums are taken over
the monomial $\C$-basis $\scrB$ of $\Jac(f)$, and below we will implicitly omit monomials not in $\scrB$.

\medskip

Our strategy to show that $\Jac(g)\ncong\Jac(f)$ relies on a  new `filtered' version of the Shirayanagi algorithm \cite{Shirayanagi}, which we now outline.

\begin{setup}
Suppose the existence of an isomorphism $\phi\colon\Jac(g)\rightarrow\Jac(f)$, which is necessarily described by \eqref{eq!mapeq} for some $a_i,b_i\in\C$.  
\begin{enumerate}
\item\label{shira1}
For any $k\geq 3$, consider $\phi_k$, which is defined to be the composition
\[
\phi_k\colon
\ring\buildrel{\nat}\over{\longrightarrow}
\Jac(g)\buildrel{\phi}\over{\longrightarrow}\Jac(f)\buildrel{\uppi_k}\over{\longrightarrow} A_k=\Jac(f)/\fJ^k
\]
where $\nat$ and $\uppi_k$ are the natural maps (see \S\ref{conventions} and \ref{not!jacqt}).
\item\label{shira2}
The $\fJ$-adic basis $\scrB$ of $\Jac(f)$ determines a monomial basis $\scrB_k$ of $A_k$, simply by omitting all monomials of degree $\ge k$. Now, on one hand $\phi_k(\dcyc_xg)=0=\phi_k(\dcyc_yg)$, since after applying $\nat$, both $\dcyc_xg$ and $\dcyc_yg$ map to zero. On the other hand $\dcyc_xg$ and $\dcyc_yg$ are a combination of the generators $x$ and $y$, and so applying the algebra homomorphism in \eqref{eq!mapeq} to this combination expresses $\phi_k(\dcyc_xg)$ and $\phi_k(\dcyc_yg)$ in terms of the basis $\scrB_k$, where the coefficients are expressions in $a_i$ and~$b_i$.  Combining, it follows that if the isomorphism $\phi$ exists, all the coefficients (which are expressions in $a_i$ and $b_i$) must be zero.
\end{enumerate}
\end{setup}

The algorithm starts with $k=3$ in \ref{lem:startShira} below, and at this step the first elementary constraint on $\phi$ is obtained (namely that in \eqref{eq!mapeq} $a_2=0$). We then simply increase $k$, one step at a time: crucially, since $\scrB$ is an $\fJ$-adic basis, the conditions determined in \eqref{shira2} for lower values of $k$ remain unchanged, while new ones arise at each step.
Thus repeating \eqref{shira2} for gradually increasing $k$ gives a (possibly infinite) sequence of conditions on the $a_i$ and $b_i$. For an arbitrarily-chosen isomorphism $\phi$ to exist, all these conditions must hold, thus if the conditions become inconsistent at some finite~$k$, no isomorphism can exist, and so $\Jac(g)\not\cong\Jac(f)$.  

As a bonus, since below we will always find a contradiction at some finite $k$, the above strategy only uses the initial (finite) portion of $\scrB$ whose leading terms are bounded by that fixed degree $k$. In particular, we never need to calculate a full standard basis.

The first part of this strategy applies to all $\ell\geq 3$ potentials, as follows.
\begin{lemma}\label{lem:startShira}
Given $f,g\in\ring$ with $f=x^3+\bigO_4$ and $g=x^3+\bigO_4$, suppose that there is an isomorphism $\phi\colon\Jac(g)\rightarrow\Jac(f)$.
Then the following statements hold.
\begin{enumerate}
\item\label{killyi}
$\Jac(f)$ has a $\fJ$-adic basis $\scrB=\{1,x,y,xy,yx,y^2\}\cup\scrB^{\ge3}$, where $\scrB^{\ge3}\subset\fJ^3$.
\item\label{killyii}
With respect to any $\fJ$-adic basis of $\Jac(f)$ of the form of $\scrB$ above, $\phi$ has the form
\begin{equation}\label{eq!killy}
\begin{array}{c@{\hspace{4pt}}c@{\hspace{4pt}}c@{\hspace{2pt}}c@{\hspace{2pt}}c@{\hspace{2pt}}c@{\hspace{2pt}}c@{\hspace{2pt}}c@{\hspace{2pt}}c@{\hspace{2pt}}}
x&\mapsto&  a_1x &+& 0       &+& \sum a_\msf\msf \\
y&\mapsto&  b_1x &+& b_2y &+& \sum b_\msf\msf
\end{array}
\end{equation}
for constants $a_\msf,b_\msf\in\C$ with $a_1b_2\not=0$, where each sum in \eqref{eq!killy} is taken over all $\msf\in\{xy,yx,y^2\}\cup\scrB^{\ge3}$.
\end{enumerate}
\end{lemma}
\begin{proof}
(1) The derivatives of $f$ are $\dcyc_xf = 3x^2 + \bigO_3$ and $\dcyc_yf\in\n^3$, so that $x^2$ is the only leading monomial
of the standard basis in degree~$\le2$.\\
(2) Since $\dcyc_xg = 3x^2 + \bigO_3$, applying the algebra homomorphism $\phi_2$ gives
\[
\phi_2(\dcyc_xg)=3(a_1x+a_2y)^2 \stackrel{\dcyc_xf}{\equiv} 3a_1a_2xy + 3a_1a_2yx + 3a_2^2y^2 
\]
where the right hand side is expressed in the basis $\scrB_3$. Thus
$a_1a_2=a_2^2=0$, so $a_2=0$. The condition $\det\left(\begin{smallmatrix}a_1&a_2\\b_1&b_2\end{smallmatrix}\right)\not=0$ under \eqref{eq!mapeq} then becomes $a_1b_2\neq 0$.
\end{proof}

%-----------------------------------------------------------------------------------------
\subsection{Moduli for Exceptional Types}\label{sec:Moduli}

We use the notation and strategy of \S\ref{sec!modulisetup} throughout the remainder. Recall from \S\ref{conventions} that the notation $h + \hot$ indicates that there may be unwritten terms beyond the terms of $h$ in the monomial order, which therefore cannot alter any coefficients of~$h$ even after reduction by a standard basis.

\begin{prop}[Moduli for $\ell=3$]\label{prop:moduliE65}
For any $a,b\in\C$, consider
\[
f = \tfrac{1}3x^3 + xy^3 + \tfrac{1}5y^5 + \tfrac{1}6ay^6
\quad\textrm{and}\quad
g = \tfrac{1}3x^3 + xy^3 + \tfrac{1}5y^5 + \tfrac{1}6by^6.
\]
Then $\Jac(f)\cong\Jac(g)$ if and only if $a=b$.
\end{prop}

Set $\sfg_1\colonequals \dcyc_xf = x^2 + y^3$ and $\sfg_2\colonequals \dcyc_yf = xy^2 + yxy + y^2x + y^4 + ay^5$. Working modulo $\n^6$, the standard basis of $\lcl\dcyc f\rcl=\lcl \sfg_1,\sfg_2 \rcl$ also includes
the reduction of the \Spoly{} $x\sfg_2 - \sfg_1y^2$, namely
\[
\sfg_3 \colonequals
x\sfg_2 - \sfg_1y^2 - \sfg_2 x + y^2\sfg_1  - \sfg_2 y^2 + y\sfg_2 y = 
xyxy - yxyx + y^3xy - y^4x + \bigO_6.
\]
The standard basis thus has (at least) leading terms $x^2, xy^2, xyxy$, and it is now quick to see that all other leading terms lie in~$\n^6$. We will also use the following lemma without further comment: it follows by noting $\sfg_1=0$ in $\Jac(f)$, and so $x^2 = -y^3$.
\begin{lemma}\label{lem:commuting}
Consider $f$ as in \ref{prop:moduliE65} and any $m\ge0$. Then in $\Jac(f)$ both $x^m$ and $y^3$ commute, and $x^2$ and $y^m$ commute.
\end{lemma}

\begin{proof}[Proof of \ref{prop:moduliE65}]
If $\phi\colon\Jac(g)\rightarrow\Jac(f)$ is an isomorphism, then by \eqref{eq!killy} and the above paragraph it may be written 
\[
\medmath{
\begin{array}{c@{\hspace{4pt}}c@{\hspace{4pt}}c@{\hspace{2pt}}c@{\hspace{2pt}}c@{\hspace{2pt}}c@{\hspace{2pt}}c@{\hspace{2pt}}c@{\hspace{2pt}}c@{\hspace{2pt}}c@{\hspace{2pt}}c@{\hspace{2pt}}c@{\hspace{2pt}}c@{\hspace{2pt}}c@{\hspace{2pt}}c@{\hspace{2pt}}c@{\hspace{2pt}}c@{\hspace{2pt}}c@{\hspace{2pt}}c@{\hspace{2pt}}c@{\hspace{2pt}}c@{\hspace{2pt}}c@{\hspace{2pt}}c@{\hspace{2pt}}}
x&\mapsto & a_1x &+& 0 &+& a_3xy &+& a_4yx &+& a_5y^2 &+& a_6xyx &+& a_7yxy &+& a_8y^2x &+& a_9y^3 &+& \bigO_4 \\
y&\mapsto & b_1x &+& b_2y &+& b_3xy &+& b_4yx &+& b_5y^2 &+& b_6xyx &+& b_7yxy &+& b_8y^2x &+& b_9y^3 &+& \bigO_4
\end{array}}
\]
for some $a_i,b_i\in\C$, with the  conditions $a_1b_2\not=0$ and
$\upvarphi(\dcyc_xg) = 0= \upvarphi(\dcyc_yg)$ in  $\Jac(f)$.

Working modulo $\fJ^4$ (so, $k=4$ in the algorithm), we have
\[
\medmath{0=\phi_4(\dcyc_yg) = \phi_4(xy^2 + yxy + y^2x + y^4 + by^5) = 3a_1b_1^2 x^3 + 2a_1b_1b_2x^2y + 2a_1b_1b_2xyx + \hot }
\]
 Thus $a_1b_1b_2=0$, since terms divisible by $x^2$ reduce under $\sfg_1$ into degree~4 and $xyx$ is not divisible by any leading term so is in $\scrB_4$. Since $a_1b_2\neq 0$ by \ref{lem:startShira}\eqref{killyii}, necessarily $b_1=0$.    After setting $b_1=0$ and reducing $x^2$ to $-y^3$ throughout,
\[
\medmath{
\begin{array}{r@{\hspace{2pt}}c@{\hspace{2pt}}l}
0=\phi_4(\dcyc_xg) = \phi_4(x^2 + y^3)
&=& a_1(a_3 + a_4)xyx + a_1a_5xy^2 + a_1a_5y^2x + (-a_1^2+b_2^3)y^3  \\
&\stackrel{\sfg_2}{\equiv}& a_1(a_3 + a_4)xyx - a_1a_5yxy + (-a_1^2+b_2^3)y^3 
\end{array}}
\]
so that (since $a_1\neq 0$ by \ref{lem:startShira}\eqref{killyii}) $a_3+a_4=0=a_5$ and $a_1^2 = b_2^3$.

Plugging in these conditions, we next consider $k=5$ in the algorithm, so work modulo~$\fJ^5$.  Reducing the degree~3 terms by~$\sfg_2$ and degree~4 terms by~$\sfg_1$, it follows that 
\[
\medmath{
\begin{array}{c@{\hspace{4pt}}c@{\hspace{4pt}}l}
 && \phi_5(\dcyc_yg)\\
&\stackrel{\sfg_2}{\equiv}& (2a_1b_2b_3 + a_1b_2b_4)xyxy + (a_1b_2b_3 + a_1b_2b_4)xy^2x + (2a_1b_2b_5 + 
    a_3b_2^2)xy^3 \\
&& \quad +  (a_1b_2b_3 + 2a_1b_2b_4)yxyx+   a_1b_2b_5yxy^2  + a_1b_2b_5y^2xy + (2a_1b_2b_5 - a_3b_2^2)y^3x + (-a_1b_2^2 + b_2^4)y^4\\
&\stackrel{\sfg_1}{\equiv}& 2a_1b_2(b_3 + b_4)yxyx + 3a_1b_2b_5y^3x +    (-a_1b_2^2 + b_2^4)y^4.
\end{array}}
\]
Thus $b_3+b_4=b_5=0$ and $a_1=b_2^2$, whence $a_1=b_2=1$.

The final step is to consider $k=6$, that is work modulo $\fJ^6$.
The conclusion will follow from the coefficient of~$y^5$.
We may ignore terms in degree~5 that have a factor of~$x$, since they may only reduce into degree $\ge6$ or to other
terms in degree~5 with the same number of factors of~$x$.
Doing this, 
\[
\medmath{
\begin{array}{c@{\hspace{4pt}}c@{\hspace{4pt}}l}
0=\phi_6(\dcyc_yg)&=& \phi_6(xy^2 + yxy + y^2x + y^4 + by^5) \\
&=& xy^2 + yxy + y^2x + b_3x^2y^2 + b_3xyxy + a_3xy^3 - b_3yxyx - b_3y^2x^2 - a_3y^3x + y^4 \\
&&\qquad\qquad+ \textrm{(degree 5 terms with a factor of $x$)} + by^5 \\
&\equiv&  \textrm{(degree 5 terms with a factor of $x$)} + (b-a)y^5 
\end{array}}
\]
where the last line is reduction by $\sfg_1,\sfg_2,\sfg_3$ and the commutation relations of~\ref{lem:commuting}.
Thus $a=b$, as claimed.  The converse direction is clear.
\end{proof}

\begin{prop}[Moduli for $\ell=4$]\label{prop:modForl4}
For any $a,b\in\C$ consider
\[
f = \tfrac{1}3 x^3 + \tfrac{1}4 y^4 + xyxy^2 + \tfrac{1}3 a(xy)^3 
\quad\textrm{and}\quad
g = \tfrac{1}3 x^3 + \tfrac{1}4 y^4 + xyxy^2 + \tfrac{1}3 b(xy)^3
\]
Then $\Jac(f)\cong\Jac(g)$ if and only if $a=\upomega b$ for some $\upomega^5 = 1$.
\end{prop}
The derivatives of~$f$ are $\sfg_1 \colonequals \dcyc_xf =  x^2 + yxy^2 + y^2xy + ayxyxy$ and $\sfg_2 \colonequals \dcyc_yf =  y^3 + xyxy + xy^2x + yxyx + axyxyx$, where the respective derivatives of $g$ replace $a$ by $b$ as the degree~5 coefficient.
The standard basis of $\lcl\dcyc f\rcl$ includes $\sfg_3 = x\sfg_1-\sfg_1x=xyxy^2 + xy^2xy-yxy^2x-y^2xyx + \bigO_6$,
and it is quick to see that all other standard basis leading terms lie in~$\n^6$.

\begin{proof}
If $\phi\colon\Jac(g)\rightarrow\Jac(f)$ is an isomorphism, then by \eqref{eq!killy} and the above paragraph  it may be written
\[
\medmath{
\begin{array}{c@{\hspace{4pt}}c@{\hspace{4pt}}c@{\hspace{2pt}}c@{\hspace{2pt}}c@{\hspace{2pt}}c@{\hspace{2pt}}c@{\hspace{2pt}}c@{\hspace{2pt}}c@{\hspace{2pt}}c@{\hspace{2pt}}c@{\hspace{2pt}}c@{\hspace{2pt}}c@{\hspace{2pt}}c@{\hspace{2pt}}c@{\hspace{2pt}}c@{\hspace{2pt}}c@{\hspace{2pt}}c@{\hspace{2pt}}c@{\hspace{2pt}}c@{\hspace{2pt}}c@{\hspace{2pt}}c@{\hspace{2pt}}c@{\hspace{2pt}}}
x&\mapsto&  a_1x &+& 0 &+&a_3xy &+& a_4yx &+& a_5y^2 &+& a_6xyx &+& a_7xy^2 &+& a_8yxy &+& a_9y^2x &+& \bigO_4 \\
y&\mapsto&  b_1x &+& b_2y &+& b_3xy &+& b_4yx &+& b_5y^2 &+& b_6xyx &+& b_7xy^2 &+& b_8yxy &+& b_9y^2x &+& \bigO_4
\end{array}}
\]
for some $a_i,b_i\in\C$ with $a_1b_2\not=0$.

Consider the case $k=4$, that is work modulo $\fJ^4$, where $x^2=0=y^3$. Thus
\[
\medmath{
\begin{array}{l@{\hspace{2pt}}l@{\hspace{2pt}}l@{\hspace{2pt}}l@{\hspace{2pt}}l@{\hspace{2pt}}}
\phi_4(\dcyc_xg) &=& \phi_4(x^2) &=& (a_1x + a_3xy + a_4yx + a_5y^2)^2 \\
&&&=& a_1^2x^2 + a_1a_3x^2y + a_1(a_4 + a_3)xyx + a_1a_5xy^2 + a_1a_4yx^2 + a_1a_5y^2x  \\
&&&\equiv& a_1(a_4 + a_3)xyx +     a_1a_5(xy^2 + y^2x).
\end{array}}
\]
Thus (since $a_1\neq 0$ by \ref{lem:startShira}\eqref{killyii}) $a_5=0$ and $a_4=-a_3$. Similarly (since still $x^2=y^3=0$)
\[
\medmath{\phi_4(\dcyc_yg) = \phi_4(y^3) = (b_1x+b_2y)^3 = b_1^2b_2xyx + b_1b_2^2(xy^2 + yxy + y^2x)}
\]
and so $b_1=0$. 

We plug in these conditions, and next consider $k=5$, so work modulo $\fJ^5$. Here
\[
\medmath{
\begin{array}{c@{\hspace{4pt}}c@{\hspace{4pt}}l}
\phi_5(\dcyc_xg) &=& \phi_5(x^2 + yxy^2 + y^2xy) \\
&=& \big(a_1x + a_3(xy-yx)+ a_6xyx + a_7xy^2 + a_8yxy + a_9y^2x\big)^2 +  a_1b_2^3(yxy^2 + y^2xy) \\
&\stackrel{\sfg_1,\sfg_2}{\equiv}&
(a_3^2+a_1a_8)xyxy + (-a_3^2+a_1a_9 + a_1a_7)xy^2x {} \\
&& \qquad\qquad + (a_3^2+a_1a_8)yxyx + a_1(-a_1 + b_2^3)yxy^2 + a_1(-a_1 + b_2^3)y^2xy
\end{array}}
\]
so $a_1=b_2^3$, while $\phi_5(\dcyc_yg) =\phi_5(y^3+ xyxy + xy^2x + yxyx) $ equals
\[
\medmath{
\begin{array}{l}
\big(b_2y + b_3xy + b_4yx + b_5y^2\big)^3 + a_1^2b_2^2\big(xyxy + xy^2x + yxyx\big) \\
\equiv b_2^2(-b_2+a_1^2)xyxy + b_2^2(-b_2+a_1^2)xy^2x +  b_2^2(-b_2+a_1^2)yxyx + b_2^2(b_3 + b_4)yxy^2 +b_2^2(b_3 + b_4)y^2xy
\end{array}}
\]
so that $b_4=-b_3$ and $b_2 = a_1^2$. Therefore $a_1=\upomega$ and $b_2=\upomega^2$ for some $\upomega^5=1$. 

To conclude, consider $k=6$ and the coefficient of $xyxyx$ in
\[
\phi_6(\dcyc_yg) 
= \phi_6(y^3) + \phi_6(xyxy+yxyx) + \phi_6(xy^2x) + \phi_6(bxyxyx).
\]
Taking each summand in turn,
\[
\phi_6(y^3) = \upomega^6 y^3 +
\upomega^4b_3\big(y^2(xy-yx) + y(xy-yx)y + (xy-yx)y^2\big) +
\hdots
\]
where we omit any terms that cannot reduce to $xyxyx$: in fact only the terms
$\upomega^6y^3 - \upomega^4b_3y^3x + \upomega^4b_3xy^3$
may reduce to $xyxyx$ (under $\sfg_2$), and the coincidental $\pm b_3$ cancels two of these, so the $xyxyx$ term is exactly $-\upomega^6a xyxyx$.
The contribution of $\phi_6(xyxy+yxyx)$ involves no reduction: 
the only way to achieve an $xyxyx$ term is when
the first and last factors of either term are replaced suitably by $xy$ or $yx$.
Since $xy-yx$ appears in both $\phi(x)$ and $\phi(y)$ all such terms cancel and there is no contribution.
Similarly $\phi_6(xy^2x)$ makes no contribution to the $xyxyx$ term.
Finally $\phi_6(bxyxyx) = \upomega^7bxyxyx$.
Therefore the coefficient of $xyxyx$ in $\phi_6(\dcyc_yg)$ is equal to $-\upomega^6a + \upomega^7b$,
so that $a=\upomega b$ as claimed.

Conversely, if $a=\upomega b$ in $f$, observe that the scaling automorphism $\phi$ defined by $x\mapsto\upomega^4x$ and $y\mapsto\upomega^3y$ gives $\phi(f) = \upomega^2g$, so that $\Jac(f)\cong\Jac(g)$ by \ref{prop:scaling}.
\end{proof}

To exhibit moduli for $\ell=5$ the standard basis requires a little more work.
\begin{lemma}\label{lem:stdl5trun}
Let $f = \tfrac{1}3 x^3 + \tfrac{1}2 xyxy + \tfrac{1}5 y^5 + \tfrac{1}6 ay^6$ for any $a\in\C$.
Then the standard basis of $\lcl \dcyc f\rcl$ consists of the four named elements
\[
\begin{array}{lcl}
   \sfg_1 \colonequals \dcyc_xf = x^2 + yxy &&
   \sfg_2 \colonequals \dcyc_yf = xyx + y^4 + ay^5  \\
   \sfg_3 \colonequals xy^2xy - y^4x + \bigO_6 &&
   \sfg_4 \colonequals xy^4 - yxy^2x + \bigO_6
\end{array}
\]
with others of order $\ge6$.
\end{lemma}
\begin{proof} 
$\sfg_3=xy\sfg_1-\sfg_2x$ and $\sfg_4=x\sfg_2-\sfg_1yx$ ($\sfg_4$ could self-reduce its $axy^5$ term, but we do not need to consider that within $\bigO_6$) and all other \Spolys\ reduce into~$\n^6$.
\end{proof}
\begin{prop}[Moduli for $\ell=5$]\label{prop:modForl5}
For any $a,b\in\C$ consider
\[
f = \tfrac{1}3 x^3 + \tfrac{1}2 xyxy + \tfrac{1}5 y^5 + \tfrac{1}6 ay^6
\quad\textrm{and}\quad
g = \tfrac{1}3 x^3 + \tfrac{1}2 xyxy + \tfrac{1}5 y^5 + \tfrac{1}6 by^6.
\]
Then $\Jac(f)\cong\Jac(g)$ if and only if $a=b$.
\end{prop}

\begin{proof}
Here, $\dcyc_xg=x^2+yxy$ and $\dcyc_yg=xyx+y^4+by^5$.
Now if $\phi\colon\Jac(g)\rightarrow\Jac(f)$ is some isomorphism, then by \eqref{eq!killy} and \ref{lem:stdl5trun} it may be written
\begin{equation}\label{eq:l5phi}
\medmath{
\begin{array}{c@{\hspace{4pt}}c@{\hspace{4pt}}c@{\hspace{2pt}}c@{\hspace{2pt}}c@{\hspace{2pt}}c@{\hspace{2pt}}c@{\hspace{2pt}}c@{\hspace{2pt}}c@{\hspace{2pt}}c@{\hspace{2pt}}c@{\hspace{2pt}}c@{\hspace{2pt}}c@{\hspace{2pt}}c@{\hspace{2pt}}c@{\hspace{2pt}}c@{\hspace{2pt}}c@{\hspace{2pt}}c@{\hspace{2pt}}c@{\hspace{2pt}}c@{\hspace{2pt}}c@{\hspace{2pt}}c@{\hspace{2pt}}c@{\hspace{2pt}}}
x&\mapsto&  a_1x &+& 0 &+& a_3xy &+& a_4yx &+& a_5y^2 &+& a_6xy^2 &+& a_7yxy &+& a_8y^2x &+& a_9y^3 &+& \bigO_4 \\
y&\mapsto&  b_1x &+& b_2y &+& b_3xy &+& b_4yx &+& b_5y^2 &+& b_6xy^2 &+& b_7yxy &+& b_8y^2x &+& b_9y^3 &+& \bigO_4
\end{array}}
\end{equation}
for unique $a_i,b_i\in\C$ with $a_1b_2\not=0$.
Working modulo $\fJ^4$, we have%reducing by $\sfg_1,\sfg_2$,
\[
\medmath{
\begin{array}{c@{\hspace{4pt}}c@{\hspace{4pt}}l}
\phi_4(\dcyc_xg) &=& a_1^2x^2 + a_1b_1^2x^3 + (a_1a_3 + a_1b_1b_2)x^2y + (a_1a_3 + 
    a_1a_4)xyx + a_1a_5xy^2  \\
    && \qquad\qquad+ (a_1a_4 + a_1b_1b_2)yx^2 + a_1b_2^2yxy + a_1a_5y^2x \\
&\stackrel{\sfg_1,\sfg_2}{\equiv}& a_1a_5xy^2 + (-a_1^2 + a_1b_2^2)yxy + a_1a_5y^2x
\end{array}}
\]
(using $\sfg_1$, $\sfg_2$ from \ref{lem:stdl5trun})
so $a_5=0$ and $a_1=b_2^2$.
Modulo $\fJ^5$ (keeping $a_1$ in place for clarity),
\begin{equation}
\medmath{
\begin{array}{c@{\hspace{2pt}}c@{\hspace{2pt}}l}\label{eq:l5phi5}
\phi_5(\dcyc_yg)&=&a_1^2x(b_1x + b_2y)x + a_1x(b_1x+b_2y)(a_3xy + a_4yx) \\
&&\qquad\qquad{}+ a_1(a_3xy + a_4yx)(b_1x+b_2y)x + a_1^2x(b_3xy + b_4yx + b_5y^2)x + (b_1x+b_2y)^4\\
&\equiv&b_2(b_2^3-a_1^2)y^4 + \big(a_1(a_3+a_4)b_2+a_1^2b_5 + b_1^2b_2^2\big)xy^2x + b_1b_2^3(xy^3+yxy^2+yxy^2+y^3x)
\end{array}}
\end{equation}
so $b_1=0$ and $b_2^3=a_1^2=b_2^4$, whence $a_1=b_2=1$ and then $a_3+a_4+b_5=0$.
With these in \eqref{eq:l5phi},
\begin{eqnarray*}
\begin{array}{c@{\hspace{2pt}}c@{\hspace{2pt}}l}
\phi_5(\dcyc_xg)&=& a_3x^2y + (a_3+a_4)xyx + a_4yx^2 + \textrm{degree 4 terms with an $x$} \\\nonumber
& \equiv &  -(a_3+a_4)y^4 + \textrm{degree 4 terms with an $x$} 
\end{array}
\end{eqnarray*}
and so $a_3+a_4=b_5=0$.

To conclude, we need consider only the $y^5$ term in $\phi_6(\dcyc_yg)$, and so we ignore
any terms in $\phi_6(xyx)$ that cannot reduce to a $y^5$ term under \ref{lem:stdl5trun}.
In particular, no monomial of degree~5 can reduce to a $y^5$ term except $y^5$ itself, and so
when evaluating $\phi_6(xyx)$ there is no need to treat terms in \eqref{eq:l5phi} of degree $\ge3$.
Thus, for $\phi_6(xyx)$ modulo $\fJ^6$ we consider
\begin{align*}
&(x + a_3xy + a_4yx)(y + b_3xy+b_4yx)(x + a_3xy + a_4yx)  \\
&\qquad  \equiv xyx + xy(a_3xy + a_4yx) + (a_3xy+a_4yx)yx \\
&\qquad  \equiv xyx + a_3xyxy + (a_4 + a_3)xy^2x + a_4yxyx \\
&\qquad  \equiv -y^4 - ay^5 
\end{align*}
Similarly, because no degree~5 term can reduce to a $y^5$ term, the only terms of
$\phi_6(y^4+by^5)$ we need treat are $(y+b_5y^2)^4 + by^5 = y^4 + by^5$.
Added together, these two show that the $y^5$ term in $\phi_6(\dcyc_yg)$ is exactly $(b-a)y^5$.
Since $y^5$ does not reduce under \ref{lem:stdl5trun} we must have $b-a=0$, as required.  

The converse, that $a=b$ implies that $\Jac(f)\cong\Jac(g)$, is clear.
\end{proof}

Perhaps counterintuitively, $\ell=6$ is in fact the easiest case.  
\begin{prop}[Moduli for $\ell=6$]\label{prop:modForl6}
For any $a,b\in\C$ consider
\[
f = \tfrac{1}3 x^3 + \tfrac{1}4 y^4 + \tfrac{1}3 (xy)^3 + \tfrac{1}2 a(xy^2)^2
\quad\textrm{and}\quad
g = \tfrac{1}3 x^3 + \tfrac{1}4 y^4 + \tfrac{1}3 (xy)^3 + \tfrac{1}2 b(xy^2)^2.
\]
Then $\Jac(f)\cong\Jac(g)$ if and only if $a=\upzeta b$ where $\upzeta^9=1$.
\end{prop}
Setting $\sfg_1\colonequals x^2 + yxyxy + ay^2xy^2$ and $\sfg_2\colonequals y^3 + xyxyx + axy^2xy + ayxy^2x$
to be the two derivatives of $f$, note that all \Spolys{} have order at least~6 and so, working modulo~$\fJ^k$ for $k\le6$,
this is the only part of the standard basis of $\lcl\dcyc f\rcl$ that is used.

\begin{proof}
Here $\dcyc_xg=x^2 + yxyxy + by^2xy^2$ and $\dcyc_yg=y^3 + xyxyx + axy^2xy + byxy^2x$. Now if $\phi\colon\Jac(g)\rightarrow\Jac(f)$ is an isomorphism, then by \eqref{eq!killy} and the above paragraph it may be written
\[
\medmath{
\begin{array}{c@{\hspace{4pt}}c@{\hspace{4pt}}c@{\hspace{2pt}}c@{\hspace{2pt}}c@{\hspace{2pt}}c@{\hspace{2pt}}c@{\hspace{2pt}}c@{\hspace{2pt}}c@{\hspace{2pt}}c@{\hspace{2pt}}c@{\hspace{2pt}}c@{\hspace{2pt}}c@{\hspace{2pt}}c@{\hspace{2pt}}c@{\hspace{2pt}}c@{\hspace{2pt}}c@{\hspace{2pt}}c@{\hspace{2pt}}c@{\hspace{2pt}}c@{\hspace{2pt}}c@{\hspace{2pt}}c@{\hspace{2pt}}c@{\hspace{2pt}}}
x&\mapsto&  a_1x &+& 0 &+& a_3xy &+& a_4yx &+& a_5y^2 &+& a_6xyx &+& a_7xy^2 &+& a_8yxy &+& a_9y^2x &+& \bigO_4 \\
y&\mapsto&  b_1x &+& b_2y &+& b_3xy &+& b_4yx &+& b_5y^2 &+& b_6xyx &+& b_7xy^2 &+& b_8yxy &+& b_9y^2x &+& \bigO_4
\end{array}}
\]
for unique $a_i,b_i\in\C$ with $a_1b_2\not=0$.

Working modulo $\fJ^4$ (crossing out terms that reduce into $\bigO_5$ under $\sfg_1$)
\[
\phi_4(\dcyc_xg) = \phi_4(x^2) = \cancel{a_1^2x^2} + a_1\cancel{a_3x^2}y + a_1(a_3 + a_4)xyx + a_1a_5xy^2 + \hot
\]
so that $a_3+a_4=a_5=0$, and
\[
\phi_4(\dcyc_yg) = \phi_4(y^3) = \cancel{b_1^3x^3} + b_1^2\cancel{b_2x^2}y + b_1^2b_2xyx + \hot
\]
so that $b_1=0$.
Imposing these conditions on $a_i,b_j$ and working modulo $\fJ^5$ (with the reductions above done,
and further reductions under $\sfg_2$ indicated) we have
\[
\phi_5(\dcyc_yg) = \phi_5(y^3) = \cancel{b_2^2y^3} + b_2^2\cancel{b_3x}y^3 + b_2^2(b_3 + b_4)yxy^2 + \hot
\]
so that $b_3+b_4=0$.

Finally, impose these conditions and work modulo $\fJ^6$.
Reductions by $\sfg_2$ push terms into $\bigO_6$ with the exception of $b_2^3y^3$ which reduces to $-b_2^3xyxyx + \hot$,
so that (omitting terms of degree $\le4$, none of which can reduce into degree~5 to alter the $xyxyx$ coefficient),
\[
\phi_6(\dcyc_yg) - b_2^3\sfg_2  = \textrm{(omitted terms of degree $\le4$)} + (a_1^3b_2^2 - b_2^3)xyxyx + \hot
\]
whence $b_2=a_1^3$.
Similarly the only nontrivial reduction under $\sfg_1$ is $a_1^2x^2$ which reduces to $-a_1^2(yxyxy + ay^2xy^2)$, so after omitting non-reducible terms of degree $\leq 4$ and any terms of degree~5 with an initial or terminal~$x$ (which include products involving the degree~4 part of $\phi(x)$, since they cannot reduce to alter the coefficients of $yxyxy$ or $y^2xy^2$),  
\[
\phi_6(\dcyc_xg) - a_1^2\sfg_1=(a_1^2b_2^3 - a_1^2)yxyxy + (-aa_1^2 + ba_1b_2^4)y^2xy^2 + \hot
\]
Thus $b_2^3=1$, whence $a_1^9=1$  and $a = ba_1^2$, proving ($\Rightarrow$).

Conversely, with $a=\upzeta b$ in $f$, observe that the scaling automorphism $\phi$ defined by $x\mapsto\upzeta^4x$ and $y\mapsto\upzeta^3y$ gives $\phi(f) = \upzeta^3g$, so that $\Jac(f)\cong\Jac(g)$ by \ref{prop:scaling}.
\end{proof}

\appendix
\section{Two GV Calculations}\label{appendix}
This appendix uses the main result \ref{cor:flopsMain} to compute the second Gopakumar--Vafa invariant $\GV{2}$ for flops of length $3$ and $4$, which is required in \S\ref{subsec:GVfinally}. 
As the proof technique closely mirrors that in \cite{BW4}, we now change notation to match \emph{loc.\ cit.}.  It is worth remaking that similar calculations can now be performed to compute any $\GV{n}$ for any length $\ell$, but such calculations are both lengthy and tedious: all such calculations except the two contained in this appendix are now unnecessary, due to \ref{cor:allDims}.

\subsection{Notation}\label{subsec:AppNot}
Given a flopping contraction $\scrX\to\Spec\scrR$ of $\ell=3,4$, by \ref{cor:flopsMain} (and a reindexing of the variables to match \cite{BW4}) the completion of $\scrX$ along the flopping curve $\Curve$, which we will temporarily denote $\mathcal{X}$, is isomorphic to the completion of the variety $\mathrm{X}$ defined by gluing $\mathrm{U}_1=\mathbb{A}_{a,v_1,v_2}^3$ with $\mathrm{U}_2=\mathbb{A}_{b,w_1,w_2}^3$ via
\[
(a,v_2,v_1) \stackrel{a\neq 0}{\longleftrightarrow} (a^{-1}, \,a^{-1}v_2, \,\,a^{3}v_1 + \partial E)=(b,w_2,w_1)
\]
whenever $a\neq 0$, where for $\ell=3,4$ the term $\partial E$ is defined to be
\begin{equation}
\medmath{\begin{array}{c@{\hspace{20pt}}llll}
3&a^2v_2^2 + a^{-1}v_2^3 + a^{-(n-2)}v_2^{n-1} +  \upalpha_{n+1}a^{-(n-1)}v_2^n+\hdots+\upalpha_{\kx}a^{2-\kx}v_2^{\kx-1}&   \\[5pt] 
4& a^2v_2^2 + a^{-2}v_2^3 + a^{-1}v_2^4 + Q_4  &  \\ [5pt]
\end{array}}\label{eqn:appendixGlue}
\end{equation}
As in \cite{BW2}, the ideal $(v_1,v_2)$ glues to the ideal $(w_1,w_2)$ to give a sheaf of ideals $\scrI$ in $\mathrm{X}$ such that $\scrO/\scrI=\scrO_{\Curve}$.   Further note that, as in \cite{BW2}, the inverse glue is given by 
\[
(a,v_2,v_1)=(b^{-1}, \,b^{-1}w_2, \,\,b^{3}w_1 - \partial F) \stackrel{b\neq 0}{\longleftrightarrow} (b,w_2,w_1)
\]
where $\partial F$ is given by one of the following
\begin{equation}
\medmath{\begin{array}{c@{\hspace{20pt}}llll}
3&b^{-1}w_2^2 + bw_2^3 + b^{2}w_2^{n-1} + \upalpha_{n+1}b^{2}w_2^{n}+\hdots+\upalpha_{\kx}b^{2}v_2^{\kx-1}&   \\[5pt] 
4& b^{-1}w_2^2 + b^{2}w_2^3 + w_2^4 + W_4 &  \\ [5pt]
\end{array}}\label{eqn:appendixGlue_inv}
\end{equation}
where $W_4$ is some expression in $b$ and $w_2$ described explicitly in \cite[(2.B)]{BW2}, which we suppress here since it will play almost no part in the calculation below.

As in \cite{BW2}, the following constants $\kx, \ky\in\N$ and polynomials $A, B$ will naturally appear throughout the analysis.

\begin{notation}\label{not: key abc notation}
Write
\[
\partial E = a^{2-\kx}v_2^{2}A\quad\textnormal{and}\quad
\partial F = b^{2-\ky}w_2^{2}B
\]
where $A=A(a,v_2)\in\scrO_{\scrU_1}$ and $B=B(b,w_2)\in\scrO_{\scrU_2}$ are \emph{polynomials},
neither $a$ nor $v_2$ divides $A$, and neither $b$ nor $w_2$ divides $B$.
In fact $2-\kx\in\Z$ is just the most negative power of $a$ appearing in \eqref{eqn:appendixGlue}, and $2-\ky$ is the most negative power of $b$ in the inverse glue.  By construction, both $\kx\geq 0$ and $\ky\geq 0$.
\end{notation}

\subsection{GV Strategy}
It is very well-known \cite{KatzGV} that $\GV{2}$ can be computed using the commutative deformation theory of a particular sheaf $\scrO_{2\Curve}\in\coh \scrX$ described below. The key point is that the categorical equivalences
\begin{equation}
\coh_\Curve\scrX \simeq \coh \mathcal{X}\simeq \coh_\Curve \mathrm{X} \label{eqn:TransferDGA}
\end{equation}
allow us to shift the problem to computing the deformation theory of $\scrO_{2\Curve}$ to that of a corresponding sheaf on $\mathrm{X}$, which being algebraic gives access the tools of \cite{BW4} involving locally free resolutions (described below).  Note in particular that this avoids having to compute anything on the formal scheme $\mathcal{X}$, where locally free resolutions are problematic, or to compute anything on $\scrX$, where we have little information.

By Katz \cite{KatzGV} there exists $\scrO_{2\Curve}\in\coh\scrX$ which can be characterised as the sheaf which sits as the middle term of the unique non-split short exact sequence
\[
0\to\scrO_\Curve(-1)\to\scrO_{2\Curve}\to\scrO_{\Curve}\to 0
\]
Automatically $\scrO_{2\Curve}$ is set-theoretically supported on $\Curve$, and so $\scrO_{2\Curve}\in\coh_\Curve \scrX$. Thus, across \eqref{eqn:TransferDGA}, below in the cases $\ell=3,4$ we will first construct the corresponding sheaf $\scrO_{2\Curve}\in\coh \mathrm{X}$ explicitly. Then, to compute its deformation theory, we use the \v{C}ech DG-model, which we now briefly review.  We first compute a locally free resolution
\[
0\to\scrE_3\to\scrE_2\to\scrE_1\to\scrE_0\to\scrO_{2\Curve}\to 0
\] 
then consider the DG-algebra $\scrC$ described \cite[\S3.4]{BW4}, where a homogenous element $\asf\in\scrC_i$ is by definition a triple of vectors of homomorphisms $\asf = (\asf_1, \asf_2, \asf_{12})$, where 
\[
\asf_1 \in \bigoplus_{n\in\Z} \scrHom_{\mathrm{U}_1}(\scrE_n,\scrE_{n-i}),\,\,
\asf_2 \in \bigoplus_{n\in\Z} \scrHom_{\mathrm{U}_2}(\scrE_n,\scrE_{n-i}),\,\,
\asf_{12} \in \bigoplus_{n\in\Z} \scrHom_{\mathrm{U}_{12}}(\scrE_n,\scrE_{n-i+1}) 
\]
The differential $\dd\colon \scrC_i\rightarrow\scrC_{i+1}$ is defined by
\begin{equation}
(\asf_1,\asf_2,\asf_{12})\mapsto (\updelta \asf_1, \updelta \asf_2, (\asf_{1|\mathrm{U}_{12}} - \asf_{2|\mathrm{U}_{12}}) - \updelta \asf_{12})\label{eqn:defdd}
\end{equation}
where $\updelta$ is the natural differential on $\scrHom$.  Multiplication $\star$ is defined on homogeneous elements $\asf=(\asf_1,\asf_2,\asf_{12})\in\scrC_i$
and $\bsf=(\bsf_1,\bsf_2,\bsf_{12})\in\scrC_j$ by
\begin{equation}\label{eqn:star}
\asf\star \bsf \colonequals
(\asf_1\circ\bsf_1, \,\,\,\asf_2\circ\bsf_2,\,\,\, \asf_{12}\circ {\bsf_2}_{|\scrU_{12}} + (-1)^i{\asf_1}_{|\scrU_{12}}\circ \bsf_{12})
\in\scrC_{i+j}
\end{equation}
and extended by linearity.  As is well-known \cite[\S8.2.1]{Dbook}, $\mathrm{H}^i(\scrC)\cong\Ext^i_\mathrm{X}(\scrO_{2\Curve},\scrO_{2\Curve})$, and the DG-algebra $(\scrC,\dd,\star)$ acts as a DG-model which can be used to compute the deformation theory of $\scrO_{2\Curve}$.

\medskip
In the subsections below, we will use the $A_\infty$ notation from \cite[\S5]{BW2} freely, in particular the chain-level $\mathsf{U}_n$ described in \cite[(5.A)]{BW2}, which will determine the products $\msf_n$ in cohomology.

\subsection{Proof of \ref{lem:GV2forl3}}\label{subsec:ProofOfGVl3}
Consider the $\ell=3$ flops given by the elements $x^3 + xy^3+ y^n + y^{n+1}q$ where $n>4$ and $q\in\C[y]$. After harmless rescaling, these correspond to the top line glues in \eqref{eqn:appendixGlue}.  To make the proof below uniform, it will be convenient to write
\[
\partial E = a^2v_2^2 + a^{-1}v_2^3 + \upalpha_5a^{-3}v_2^{4}+ \upalpha_6a^{-4}v_2^{5}+\hdots+\upalpha_{\kx}a^{2-\kx}v_2^{\kx-1}
\]
with the understanding that $\upalpha_5=\hdots=\upalpha_{n-1}=0$ (if $n>5$).
Below we will compute the sheaf $\scrO_{2\Curve}$, and its commutative deformation theory, showing that its prorepresenting object is given by $\mathbb{C}\llsq t\rrsq/t^{n-1}$. It then follows immediately that $\GV{2}$, which is the dimension of this object, equals $n-1$.

\medskip
First note that $(v_1,v_2^2)$ glues to $(w_1,w_2^2)$, and thus defines a sheaf of ideals $\scrI_2$ on $\mathrm{X}$, which is clearly a subsheaf of $\scrI$ from \S\ref{subsec:AppNot}.   Also, clearly $\scrI_2/\scrI\cong\scrO_{\Curve}(-1)$, and so the short exact sequence $0\to \scrI/\scrI_2 \to \scrO/\scrI_2\to \scrO/\scrI\to 0$ becomes
\[
0\to\scrO_{\Curve}(-1)\to \scrO/\scrI_2\to\scrO_{\Curve}\to 0
\]
Thus $\scrO_{2\Curve}\colonequals  \scrO/\scrI_2$ is the sheaf that we desire.  We will calculate the $A_\infty$-products on the endomorphism DG-algebra of $\scrO_{2\Curve}$, and this requires a locally free resolution.  

Here, by inspection of the top line of the inverse glue \eqref{eqn:appendixGlue_inv}, in the notation of \ref{not: key abc notation} we have $\ky=3$, and $A=a^{\kx}+a^{\kx-3}v_2+ \upalpha_5a^{\kx-5}v_2^{2}+ \upalpha_6a^{\kx-6}v_2^{3}+\hdots+\upalpha_{\kx} v_2^{\kx-3}$. With this notation set, consider the following locally free resolution of $\scrO_{2\Curve}$
\begin{equation}
\begin{array}{c}
\begin{tikzpicture}
\node (B) at (-0.5,0) {$\scrO({\scriptstyle -4-\kx})$};
\node (C1) at (2.5,0.75) {$\scrO({\scriptstyle -2-\kx})$};
\node () at (2.5,0.4) {$\oplus$};
\node (C2) at (2.5,0) {$\scrO({\scriptstyle -3})$};
\node () at (2.5,-0.35) {$\oplus$};
\node (C3) at (2.5,-0.75) {$\scrO({\scriptstyle-\kx})$};
\node (D1) at (7,0.8) {$\scrO({\scriptstyle-1})$};
\node () at (7,0.4) {$\oplus$};
\node (D2) at (7,0) {$\scrO({\scriptstyle-2})$};
\node () at (7,-0.4) {$\oplus$};
\node (D3) at (7,-0.8) {$\scrO({\scriptstyle2-\kx})$};
\node (E) at (10,0) {$\scrO$};
\draw[->] (B)--
node[above]{$\begin{psmallmatrix}v_2^2 \\ -a^{\kx+1}\\ -1\end{psmallmatrix}$}
node[below]{$\begin{psmallmatrix}w_2^2 \\ -1\\ -b^{4}\end{psmallmatrix}$}
(C2);
\draw[->] (C2)--
node[above]{$\begin{psmallmatrix}
a^{\kx+1}&v_2^2&0\\
A&-v_1&b^{2-\kx}w_1\\
-1&0&-v_2^2\end{psmallmatrix}$}
node[below]{$\begin{psmallmatrix}
1&w_2^2&0\\
B&-a^{-1}v_1&w_1\\
-b^{4}&0&-w_2^2\end{psmallmatrix}$}
(D2);
\draw[->] (D2)--
node[above]{$\begin{psmallmatrix}v_1 & v_2^2& b^{2-\kx}w_1\end{psmallmatrix}$}
node[below]{$\begin{psmallmatrix}a^{-1}v_1 &w_2^2& w_1\end{psmallmatrix}$}
(E);
\end{tikzpicture}
\end{array}\label{eqn:LFres}
\end{equation}
where $b^{2-\kx}w_1=a^{\kx+1}v_1+Av_2^2$ and $a^{-1}v_1=b^4w_1-w_2^2B$, and we adopt the convention of \cite[\S2.3]{BW2}, where the matrix above an arrow describes the map in the first chart, and below an arrow describes the map in the second chart.

To define certain elements below it will be convenient to write
\[
aA=a^{\kx+1}+v_2(\underbrace{a^{\kx-2}+ \upalpha_5a^{\kx-4}v_2+ \upalpha_6a^{\kx-5}v_2^{2}+\hdots+\upalpha_{\kx} av_2^{\kx-4}}_{\colonequals F})
\]
Then, using small case letters for degree one elements, consider $\xsf_1$ defined to be 
\[
\begin{tikzpicture}
\node () at (-3.5,0) {$\xsf _1=$};
\node (B1) at (-1,1) {$\scrE_3$};
\node (C1) at (3.25,1) {$\scrE_2$};
\node (D1) at (7,1) {$\scrE_1$};
\node (E1) at (8,1) {$\scrE_0$};
\draw[->] (B1)--
node[above]{$\begin{psmallmatrix}v_2^2 \\ -a^{\kx+1}\\ -1\end{psmallmatrix}$}
(C1);
\draw[->] (C1)--
node[above]{$\begin{psmallmatrix}
a^{\kx+1}&v_2^2&0\\
A&-v_1&b^{2-\kx}w_1\\
-1&0&-v_2^2\end{psmallmatrix}$}
(D1);
\draw[->] (D1)--(E1);

\node (B2) at (-2,-1.5) {$\scrE_3$};
\node (C2) at (-1,-1.5) {$\scrE_2$};
\node (D2) at (3.25,-1.5) {$\scrE_1$};
\node (E2) at (7,-1.5) {$\scrE_0$};
\draw[->] (B2)--(C2);
\draw[->] (C2)--
node[above]{$\begin{psmallmatrix}
a^{\kx+1}&v_2^2&0\\
A&-v_1&b^{2-\kx}w_1\\
-1&0&-v_2^2\end{psmallmatrix}$}
(D2);
\draw[->] (D2)--
node[above]{$\begin{psmallmatrix}v_1 & v_2^2& b^{2-\kx}w_1\end{psmallmatrix}$}
(E2);

\draw[->](B1)--node[left,pos=0.35]{$\begin{psmallmatrix}-av_2\\0\\0\end{psmallmatrix}$}(C2);
\draw[->](C1)--node[gap,pos=0.35]{$\begin{psmallmatrix}0&-av_2&0\\ F&-v_2&0\\0&0&av_2\end{psmallmatrix}$}(D2);
\draw[->](D1)--node[right,pos=0.35]{$\begin{psmallmatrix}v_2&-av_2&0\end{psmallmatrix}$}(E2);
\end{tikzpicture}
\]
There is a very similar map $\xsf_2$, defined on the other chart, such that $\xsf_{1|\mathrm{U}_{12}} - \xsf_{2|\mathrm{U}_{12}}=0$.  Under the sign convention defining the homological differential $\updelta$, which for odd degree is down-across+across-down,  $\updelta(\xsf_1)=0=\updelta(\xsf_2)$. Thus $\xsf=(\xsf_1,\xsf_2,0)\in\scrC_1$ satisfies $\dd(\xsf)=0$.

It is easy to verify, in a similar way to \cite[4.1]{BW2} that $\xsf$ descends to give a non-trivial element in $\mathrm{H}^1(\scrC)$, since the right hand vertical map in $\xsf_1$ does not belong to the ideal $(v_1,v_2^2)$ but the horizontal maps do, so there can be no homotopy.  Then since $\mathrm{H}^1(\scrC)\cong\mathrm{Ext}^1(\scrO_{2\Curve},\scrO_{2\Curve})$, and the neighbourhood is $\ell=3$ flopping, it follows from \cite[7.4]{DW5} that $\mathrm{Ext}^1(\scrO_{2\Curve},\scrO_{2\Curve})\cong\mathbb{C}$, and thus $\xsf$ in fact gives a basis.

Using capital letters for degree two elements, consider $\Xsf_1$ defined to be
\[
\begin{tikzpicture}[yscale=0.75,xscale=0.8]
\node () at (0,0) {$\Xsf_1=$};
\node (C1) at (3.25,1) {$\scrE_3$};
\node (D1) at (6.25,1) {$\scrE_2$};
\node (E1) at (7.75,1) {$\scrE_1$};
\node (F1) at (9.25,1) {$\scrE_0$};
\draw[->] (E1)--(F1);
\draw[->] (C1)--
node[above]{$\begin{psmallmatrix}v_2^2 \\ -a^{\kx+1}\\ -1\end{psmallmatrix}$}
(D1);
\draw[->] (D1)--(E1);
\node (B2) at (0.25,-1) {$\scrE_3$};
\node (C2) at (1.75,-1) {$\scrE_2$};
\node (D2) at (3.25,-1) {$\scrE_1$};
\node (E2) at (6.25,-1) {$\scrE_0$};
\draw[->] (B2)--(C2);
\draw[->] (C2)--(D2);
\draw[->] (D2)--node[above]{$\begin{psmallmatrix}v_1 & v_2^2& b^{2-\kx}w_1\end{psmallmatrix}$}(E2);
\draw[->](C1)--node[left]{$\scriptstyle a^{\kx-1}\begin{psmallmatrix}0\\1\\0\end{psmallmatrix}$}(D2);
\draw[->](D1)--node[right]{$\scriptstyle a^{\kx-1}\begin{psmallmatrix}1&0&0\end{psmallmatrix}$}(E2);
\end{tikzpicture}
\]
Again there is a similar map $\Xsf_2$ defined on the other chart, such that $\Xsf_{1|\mathrm{U}_{12}} - \Xsf_{2|\mathrm{U}_{12}}=0$.  Under the sign convention on  $\updelta$, which for even degree is (down-across)$-$(across-down),  $\updelta(\Xsf_1)=0=\updelta(\Xsf_2)$.  Thus $\Xsf=(\Xsf_1,\Xsf_2,0)\in\scrC_2$ satisfies $\dd(\Xsf)=0$.

It is easy to verify, in a similar way to \cite[4.5]{BW2} that $\Xsf$ descends to give a non-trivial element in $\mathrm{H}^2(\scrC)$.  By Calabi--Yau duality $\mathrm{Ext}^2(\scrO_{2\Curve},\scrO_{2\Curve})\cong\mathbb{C}$, and so $\Xsf$ gives a basis.

\medskip
From here, we compute (most of) the $A_\infty$-structure on $\scrC$, using the above generators. By definition $\xsf\star\xsf$ is computed by composition on both charts, and by direct verification  on chart one 
\begin{equation}
\xsf_1\circ\xsf_1=
\begin{array}{c}
\begin{tikzpicture}[yscale=0.75,xscale=1.25]
\node (C1) at (3,1) {$\scrE_3$};
\node (D1) at (4,1) {$\scrE_2$};
\node (E1) at (5,1) {$\scrE_1$};
\node (F1) at (6,1) {$\scrE_0$};
\draw[->] (E1)--(F1);
\draw[->] (C1)--(D1);
\draw[->] (D1)--(E1);
\node (B2) at (1,-1) {$\scrE_3$};
\node (C2) at (2,-1) {$\scrE_2$};
\node (D2) at (3,-1) {$\scrE_1$};
\node (E2) at (4,-1) {$\scrE_0$};
\draw[->] (B2)--(C2);
\draw[->] (C2)--(D2);
\draw[->] (D2)--(E2);
\draw[->](C1)--node[left]{$\scriptstyle \scriptstyle -av_2F\begin{psmallmatrix}0\\1\\0\end{psmallmatrix}$}(D2);
\draw[->](D1)--node[right]{$\scriptstyle -av_2F\begin{psmallmatrix}1&0&0\end{psmallmatrix}$}(E2);
\end{tikzpicture}
\end{array}\label{eqn:l3first}
\end{equation}
It is direct to verify that $\xsf_1\circ\xsf_1=\updelta(\asf_1)$ where $\asf_1$ is the following homotopy. 
\[
\begin{array}{cc}
\begin{array}{c}
\begin{tikzpicture}[yscale=0.75,xscale=1.25]
\node (B1) at (0,1) {$\scrE_3$};
\node (C1) at (1,1) {$\scrE_2$};
\node (D1) at (2,1) {$\scrE_1$};
\node (E1) at (3,1) {$\scrE_0$};
\draw[->] (B1)--(C1);
\draw[->] (C1)--(D1);
\draw[->] (D1)--(E1);
\node (B2) at (-1,-1) {$\scrE_3$};
\node (C2) at (0,-1) {$\scrE_2$};
\node (D2) at (1,-1) {$\scrE_1$};
\node (E2) at (2,-1) {$\scrE_0$};
\draw[->] (B2)--(C2);
\draw[->] (C2)--(D2);
\draw[->] (D2)--(E2);
\draw[->](B1)--node[left]{$\scriptstyle f_1$}(C2);
\draw[->](C1)--node[gap]{$\scriptstyle f_2$}(D2);
\draw[->](D1)--node[right]{$\scriptstyle f_3$}(E2);
\end{tikzpicture}
\end{array}
&
\begin{array}{lccc}
&{\scriptstyle f_1}&{\scriptstyle f_2}&{\scriptstyle f_3}\\[2mm]
\asf_{1}\colonequals &\begin{psmallmatrix}-a^2\\0\\0\end{psmallmatrix}
&\begin{psmallmatrix}0&-a^2&0\\ 0&-a&0\\0&0&a^2\end{psmallmatrix}
&\begin{psmallmatrix}a&-a^2&0\end{psmallmatrix} \\[3mm]
\end{array}
\end{array}
\]
There is a map $\asf_2$ on the other chart which glues to give $\asf_1$ (meaning $\asf_{1|\mathrm{U}_{12}} - \asf_{2|\mathrm{U}_{12}}=0$); from this we obtain $\asf=(\asf_1,\asf_2,0)\in\scrC_1$.  A similar calculation to the above, on the other chart, shows that $\xsf\star\xsf=\dd(\asf)$, and so $\msf_2(\xsf,\xsf)=0$.  The defining equation for $\fsf_2(\xsf,\xsf)$ in e.g.\ \cite[(5.C)]{BW2} then becomes $-\mathsf{U}_{2}(\xsf,\xsf)=\dd\fsf_2(\xsf,\xsf)$, so we can choose $\fsf_2(\xsf,\xsf)=-\asf$.

\medskip
To set notation for the following, consider
\[
\begin{array}{l@{\hspace{4pt}}l@{\hspace{4pt}}l@{\hspace{4pt}}l@{\hspace{4pt}}l@{\hspace{4pt}}l@{\hspace{4pt}}l@{\hspace{4pt}}l@{\hspace{4pt}}l@{\hspace{4pt}}l@{\hspace{4pt}}l@{\hspace{4pt}}}
G&\colonequals& \upalpha_5a^{\kx-5}&+&\upalpha_6a^{\kx-4}v_2&+&\upalpha_7a^{\kx-3}v_2^2&+&\hdots &+& \upalpha_\kx v_2^{\kx-5}
\end{array}
\]
so that $A=a^{\kx}+a^{\kx-3}v_2 + v_2^2G$. Further, set $\deg(a^zv_2^n)\colonequals z+n$, and define the degree of an element of $\mathbb{C}[a^{\pm},v_2]$ to be the maximum $\deg(m)$ where $m$ are its monomials. The following homotopies will be used extensively, two of which depend on elements $J,K\in\mathbb{C}[a,v_2]$.  
\[
\begin{array}{cc}
\begin{array}{c}
\begin{tikzpicture}[yscale=0.75,xscale=1.25]
\node (B1) at (0,1) {$\scrE_3$};
\node (C1) at (1,1) {$\scrE_2$};
\node (D1) at (2,1) {$\scrE_1$};
\node (E1) at (3,1) {$\scrE_0$};
\draw[->] (B1)--(C1);
\draw[->] (C1)--(D1);
\draw[->] (D1)--(E1);
\node (B2) at (-1,-1) {$\scrE_3$};
\node (C2) at (0,-1) {$\scrE_2$};
\node (D2) at (1,-1) {$\scrE_1$};
\node (E2) at (2,-1) {$\scrE_0$};
\draw[->] (B2)--(C2);
\draw[->] (C2)--(D2);
\draw[->] (D2)--(E2);
\draw[->](B1)--node[left]{$\scriptstyle f_1$}(C2);
\draw[->](C1)--node[gap]{$\scriptstyle f_2$}(D2);
\draw[->](D1)--node[right]{$\scriptstyle f_3$}(E2);
\end{tikzpicture}
\end{array}
&
\begin{array}{lccc}
&{\scriptstyle f_1}&{\scriptstyle f_2}&{\scriptstyle f_3}\\[2mm]
\gsf_{1}\colonequals &\begin{psmallmatrix}1\\0\\0\end{psmallmatrix}
&\begin{psmallmatrix}0&0&0\\ 0&1&0\\0&0&0\end{psmallmatrix}
&\begin{psmallmatrix}0&1&0\end{psmallmatrix} \\[3mm]
\hsf_{1}\colonequals &\begin{psmallmatrix}a\\0\\0\end{psmallmatrix}
&\begin{psmallmatrix}0&a&0\\ -aG&1&0\\0&0&-a\end{psmallmatrix}
&\begin{psmallmatrix}-1&a&0\end{psmallmatrix} \\[3mm]
\jsf_{1,J}\colonequals &\begin{psmallmatrix}0\\0\\0\end{psmallmatrix}
&\begin{psmallmatrix}0&0&0\\ H&0&0\\0&0&0\end{psmallmatrix}
&\begin{psmallmatrix}0&0&0\end{psmallmatrix} \\[3mm]
\ksf_{1,K}\colonequals &\begin{psmallmatrix}0\\0\\0\end{psmallmatrix}
&\begin{psmallmatrix}0&0&0\\ 0&0&-K\\0&0&0\end{psmallmatrix}
&\begin{psmallmatrix}0&0&-K\end{psmallmatrix} \\[3mm]

\end{array}
\end{array}
\]
There is a map $\gsf_2$ on the other chart which glues to give $\gsf_1$ (meaning $\gsf_{1|\mathrm{U}_{12}} - \gsf_{2|\mathrm{U}_{12}}=0$); from this we obtain $\gsf=(\gsf_1,\gsf_2,0)\in\scrC_1$. Similarly, the map $\hsf_{1}$ glues, producing $\hsf\colonequals (\hsf_{1},\hsf_{2},0)\in\scrC_1$. Further, if $\deg(J)\leq \kx$ and $\deg (K)\leq \kx-2$, those maps also glue and so we similarly obtain $\jsf_J\colonequals (\jsf_{1,J},\jsf_{2,J},0)$ and $\ksf_K\colonequals (\ksf_{1,K},\ksf_{2,K},0)$.

\begin{lemma}\label{lem:l3homot}
With notation as above, in particular $\deg(J)\leq \kx$ and $\deg (K)\leq \kx-2$, then  $\dd(\gsf), \dd(\hsf), \dd(\jsf_J)$ and $\dd(\ksf_K)$ are described as follows
\[
\begin{array}{cc}
\Ysf\colonequals
\begin{array}{c}
\begin{tikzpicture}[yscale=0.75,xscale=1.25]
\node (C1) at (3,1) {$\scrE_3$};
\node (D1) at (4,1) {$\scrE_2$};
\node (E1) at (5,1) {$\scrE_1$};
\node (F1) at (6,1) {$\scrE_0$};
\draw[->] (E1)--(F1);
\draw[->] (C1)--(D1);
\draw[->] (D1)--(E1);
\node (B2) at (1,-1) {$\scrE_3$};
\node (C2) at (2,-1) {$\scrE_2$};
\node (D2) at (3,-1) {$\scrE_1$};
\node (E2) at (4,-1) {$\scrE_0$};
\draw[->] (B2)--(C2);
\draw[->] (C2)--(D2);
\draw[->] (D2)--(E2);
\draw[->](C1)--node[left]{$\scriptstyle \begin{psmallmatrix}0\\1\\0\end{psmallmatrix} $}(D2);
\draw[->](D1)--node[right]{$\scriptstyle \begin{psmallmatrix}1&0&0\end{psmallmatrix}$}(E2);
\end{tikzpicture}
\end{array}
\begin{array}{rl}
\dd(\gsf)=&A\cdot\Ysf\\[3mm]
\dd(\hsf)=&a^{\kx-2}v_2\cdot\Ysf\\[3mm]
\dd(\jsf_J)=&v_2^2J\cdot\Ysf\\[3mm]
\dd(\ksf_K)=& K \cdot\Ysf
\end{array}
\end{array}
\]
where for e.g.\ $\dd(\gsf)$ both matrices in $\Ysf$ are multiplied by $A$.
\end{lemma}

Given everything is now set up in both charts, to ease notation we suppress the calculation in the second chart, and only display that in the first chart. 

Since $\msf_2(\xsf,\xsf)=0$, it is direct to calculate that
\begin{equation}
\mathsf{U}_3(\xsf,\xsf,\xsf)=\fsf_2(\xsf,\xsf)\star\xsf+\xsf\star\fsf_2(\xsf,\xsf)=a^2F\cdot
\Ysf\label{eqn:l3second}
\end{equation}
If $n=5$ set $R_6=0$, else set
\[
\begin{array}{l@{\hspace{4pt}}l@{\hspace{4pt}}l@{\hspace{4pt}}l@{\hspace{4pt}}l@{\hspace{4pt}}l@{\hspace{4pt}}l@{\hspace{4pt}}l@{\hspace{4pt}}l@{\hspace{4pt}}l@{\hspace{4pt}}l@{\hspace{4pt}}}
R_6&\colonequals& \upalpha_6a^{\kx-3}&+&\upalpha_7a^{\kx-4}v_2&+&\hdots &+& \upalpha_\kx a^3v_2^{\kx-6}\\
\end{array}
\]
Then  $a^2F=a^{\kx}+\upalpha_5a^{\kx-2}v_2+v_2^2R_6$, so 
\begin{align*}
a^2F&=a^{\kx}+\upalpha_5a^{\kx-2}v_2+v_2^2R_6\\
&= A+\upalpha_5a^{\kx-2}v_2+v_2^2R_6+(a^\kx-A)
\end{align*}
Since $\deg(a^\kx-A)=\kx-2$, combining with \ref{lem:l3homot} and \eqref{eqn:l3second} it follows that
\[
\mathsf{U}_3(\xsf,\xsf,\xsf)=\dd(\gsf+\upalpha_5\hsf+\jsf_{R_6}+\ksf_{a^{\kx}-A})
\]
Thus $\msf_3(\xsf,\xsf,\xsf)=0$, and since $\dd\fsf_3(\xsf,\xsf,\xsf)=-\mathsf{U}_3(\xsf,\xsf,\xsf)$, we can choose $\fsf_3(\xsf,\xsf,\xsf)=-\gsf-\upalpha_5\hsf+\jsf_{-R_6}+\ksf_{A-a^{\kx}}$

Since all previous $\msf_i(\xsf,\hdots,\xsf)$ are zero, it is now direct to calculate that
\begin{align*}
\mathsf{U}_4(\xsf,\hdots,\xsf)&=\fsf_3(\xsf,\xsf,\xsf)\star\xsf+\fsf_2(\xsf,\xsf)\star\fsf_2(\xsf,\xsf)+\xsf\star\fsf_3(\xsf,\xsf,\xsf)\\
&=(-(1+\upalpha_5a)F-av_2(\upalpha_5aG-R_6))\cdot
\Ysf
\end{align*}
Now if $\upalpha_5\neq 0$ (i.e.\ $n=5$) then since $aF$ contains $a^{\kx-1}$, the right hand side can be written in the form $-\upalpha_5\Xsf+\dd(-)$, and so up to rescaling $\msf_4(\xsf,\hdots,\xsf)=\Xsf$.  In this case, it follows that the prorepresenting object is isomorphic to  $\mathbb{C}\llsq t\rrsq/t^{4}$, as required.

We can thus assume that $n>5$, so that $\upalpha_5=0$ and $\msf_4(\xsf,\hdots,\xsf)=0$.  This simplifies $\fsf_3$, and further
\begin{equation}
-\mathsf{U}_4(\xsf,\hdots,\xsf)=(F-av_2R_6)\cdot\Ysf=(F-\upalpha_6a^{\kx-2}v_2-v_2^2R_7)\cdot\Ysf \label{eqn:l3third}
\end{equation}
where either $R_7=0$ or
\[
\begin{array}{l@{\hspace{4pt}}l@{\hspace{4pt}}l@{\hspace{4pt}}l@{\hspace{4pt}}l@{\hspace{4pt}}l@{\hspace{4pt}}l@{\hspace{4pt}}l@{\hspace{4pt}}l@{\hspace{4pt}}l@{\hspace{4pt}}l@{\hspace{4pt}}}
R_7&\colonequals& \upalpha_7a^{\kx-3}&+&\upalpha_8a^{\kx-4}v_2&+&\hdots &+& \upalpha_\kx a^4v_2^{\kx-7}\\
\end{array}
\]
Thus combining \eqref{eqn:l3third} and \ref{lem:l3homot}, we can choose $\fsf_4(\xsf,\hdots,\xsf)=-\upalpha_6\hsf+\jsf_{-R_7}+\ksf_F$.

\medskip
The proof proceeds in this manner, and we illustrate the next few steps so that the mod 3 pattern presents itself.  Since all previous $\msf_i(\xsf,\hdots,\xsf)$ are zero, using the previous expressions for $\fsf_i$ (simplified since $\upalpha_5=0$), it is direct to calculate that
\begin{align*}
\mathsf{U}_5(\xsf,\hdots,\xsf)&=\fsf_4(\xsf,\hdots,\xsf)\star\xsf+\fsf_3(\xsf,\xsf,\xsf)\star\fsf_2(\xsf,\xsf)+\fsf_2(\xsf,\xsf)\star\fsf_3(\xsf,\xsf,\xsf)+\xsf\star\fsf_4(\xsf,\hdots,\xsf)\\
&=(-av_2(\upalpha_6aG-R_7)-a^2R_6-\upalpha_6aF)\cdot
\Ysf
\end{align*}
Now if $\upalpha_6\neq 0$ (i.e.\ $n=6$) then the right hand side can be written in the form $-2\upalpha_6\Xsf+\dd(-)$, and so up to rescaling $\msf_5(\xsf,\hdots,\xsf)=\Xsf$.  In this case, it follows that the prorepresenting object is isomorphic to  $\mathbb{C}\llsq t\rrsq/t^{5}$, as required.

We can thus assume that $n>6$, so that $\upalpha_5=\upalpha_6=0$  and $\msf_5(\xsf,\hdots,\xsf)=0$.  This simplifies $\fsf_3$ and $\fsf_4$, and further
\begin{equation}
-\mathsf{U}_5(\xsf,\hdots,\xsf)=(a^2R_6-av_2R_7)\cdot\Ysf=0 \label{eqn:l3fourth}
\end{equation}
Thus we can choose $\fsf_5(\xsf,\hdots,\xsf)=0$. Given all this, the only non-zero entries for $\mathsf{U}_6$ are
\begin{align*}
\mathsf{U}_6(\xsf,\hdots,\xsf)&=\fsf_4\star\fsf_2+\fsf_3\star\fsf_3+\fsf_2\star\fsf_4=(R_6-a^2R_7)\cdot\Ysf
\end{align*}
Now if $\upalpha_7\neq 0$ (i.e.\ $n=7$) then the right hand side can be written in the form $-\upalpha_7\Xsf+\dd(-)$, and so up to rescaling $\msf_6(\xsf,\hdots,\xsf)=\Xsf$.  In this case, it follows that the prorepresenting object is isomorphic to  $\mathbb{C}\llsq t\rrsq/t^{6}$, as required.

We can thus assume that $n>7$, so that $\upalpha_5=\upalpha_6=\upalpha_7=0$ and $\msf_6(\xsf,\hdots,\xsf)=0$.  This simplifies previous expressions, including 
\[
-\mathsf{U}_6(\xsf,\hdots,\xsf)=(\upalpha_8a^{\kx-2}v_2+v_2^2R_9-R_6)=\dd(\upalpha_8\hsf+\jsf_{R_9}+\ksf_{-R_6})
\]
where either $R_9=0$ or
\[
\begin{array}{l@{\hspace{4pt}}l@{\hspace{4pt}}l@{\hspace{4pt}}l@{\hspace{4pt}}l@{\hspace{4pt}}l@{\hspace{4pt}}l@{\hspace{4pt}}l@{\hspace{4pt}}l@{\hspace{4pt}}l@{\hspace{4pt}}l@{\hspace{4pt}}}
R_9&\colonequals& \upalpha_9a^{\kx-3}&+&\upalpha_{10}a^{\kx-4}v_2&+&\hdots &+& \upalpha_\kx a^6v_2^{\kx-9}
\end{array}
\]
Thus we choose $\fsf_6(\xsf,\hdots,\xsf)=\upalpha_8\hsf+\jsf_{R_9}+\ksf_{-R_6}$.  Substituting into the expression for $\mathsf{U}_7$, using the fact that all previous $\msf_i=0$, a direct calculation shows that
\[
\mathsf{U}_7(\xsf,\hdots,\xsf)=(R_7+\upalpha_8aF-av_2(R_9-\upalpha_8aG))\cdot \Ysf
\]
Now if $\upalpha_8\neq 0$ (i.e.\ $n=8$) then the right hand side can be written in the form $\upalpha_8\Xsf+\dd(-)$, and so up to rescaling $\msf_7(\xsf,\hdots,\xsf)=\Xsf$.  In this case, it follows that the prorepresenting object is isomorphic to  $\mathbb{C}\llsq t\rrsq/t^{7}$, as required.

We can thus assume that $n>8$, so that $\upalpha_5=\hdots=\upalpha_8=0$ and $\msf_7(\xsf,\hdots,\xsf)=0$.  This simplifies previous expressions, including 
\[
-\mathsf{U}_7(\xsf,\hdots,\xsf)=(av_2R_9-R_7)\cdot \Ysf=(\upalpha_9a^{\kx-2}v_2+v_2^2R_{10}-R_7)\cdot \Ysf
\]
Thus by \ref{lem:l3homot} we can choose $\fsf_7(\xsf,\hdots,\xsf)=\upalpha_9\hsf+\jsf_{R_{10}}+\ksf_{-R_7}$. Substituting into the expression for $\mathsf{U}_7$, using the fact that all previous $\msf_i=0$, a direct calculation shows that
\[
\mathsf{U}_8(\xsf,\hdots,\xsf)=(av_2(\upalpha_9aG-R_{10})+a^2R_9+\upalpha_9aF)\cdot \Ysf
\]
This should be compared to $\mathsf{U}_5$, where (up to sign) there is now a clear pattern. Regardless, if $\upalpha_9\neq 0$ (i.e.\ $n=9$) then the right hand side can be written in the form $2\upalpha_9\Xsf+\dd(-)$, and so up to rescaling $\msf_8(\xsf,\hdots,\xsf)=\Xsf$.  In this case, it follows that the prorepresenting object is isomorphic to  $\mathbb{C}\llsq t\rrsq/t^{8}$, as required.

We can thus assume that $n>9$, so that $\upalpha_5=\hdots=\upalpha_9=0$ and $\msf_8(\xsf,\hdots,\xsf)=0$.  This simplifies previous expressions, including 
\[
\mathsf{U}_8(\xsf,\hdots,\xsf)=(a^2R_9-av_2R_{10})\cdot \Ysf=0
\]
and thus we can choose $\fsf_8(\xsf,\hdots,\xsf)=0$. 
 
 We are finally able to assume the inductive hypothesis that $n>9$, that $\fsf_2,\fsf_3,\fsf_4$ are determined as above, that for all $i$ such that $5\leq i\leq n-3$ we have
\[
\fsf_{i}(\xsf,\hdots,\xsf)=\begin{cases}
(-1)^{i\phantom{+1}}(\jsf_{R_{i+3}}+\ksf_{-R_i}) & i\equiv 0 \mbox{ mod } 3\\
(-1)^{i+1}(\jsf_{R_{i+3}}+\ksf_{-R_i}) & i\equiv 1 \mbox{ mod } 3\\
0 & i\equiv 2 \mbox{ mod } 3\\
\end{cases}
\]
and $\msf_2(\xsf,\xsf)=\hdots=\msf_{n-2}(\xsf,\hdots,\xsf)=0$, and furthermore
\[
\fsf_{n-2}(\xsf,\hdots,\xsf)=\begin{cases}
(-1)^{n-2}(\upalpha_{n}\hsf+\jsf_{R_{n+1}}+\ksf_{-R_{n-2}}) & n\equiv 2 \mbox{ mod } 3\\
(-1)^{n-1}(\upalpha_{n}\hsf+\jsf_{R_{n+1}}+\ksf_{-R_{n-2}}) & n\equiv 0 \mbox{ mod } 3\\
0 & n\equiv 1 \mbox{ mod } 3\\
\end{cases}
\]
From the above information, a direct calculation using the defining equation of $\mathsf{U}_{n-1}$ (see e.g.\ \cite[(5.A)]{BW2}) shows that
\[
\mathsf{U}_{n-1}(\xsf,\hdots,\xsf)=\begin{cases}
(-1)^{n-1}(R_{n-1}-a^2R_n)\cdot \Ysf & n\equiv 1 \mbox{ mod } 3\\
(-1)^{n\phantom{-1}}(R_{n-1}+\upalpha_{n}aF-av_2(R_{n+1}-\upalpha_naG))\cdot \Ysf & n\equiv 2 \mbox{ mod } 3\\
(-1)^{n-1}(a^2R_n+\upalpha_naF-av_2(R_{n+1}-\upalpha_naG))\cdot \Ysf & n\equiv 0 \mbox{ mod } 3\\
\end{cases}
\]
In all cases the right hand side equals a non-zero scalar times $\Xsf$, plus $\dd(-)$.  Thus up to rescaling $\msf_{n-1}(\xsf,\hdots,\xsf)=\Xsf$, and so the prorepresenting object is isomorphic to $\mathbb{C}\llsq t\rrsq/t^{n-1}$.

\subsection{Proof of \ref{lem:GV2forl4}}\label{subsec:ProofOfGVl4}
Consider an $\ell=4$ flop, classified in \S\ref{subsec:IntrinsMain}--\ref{subsec:classflops}. In this section we will first prove the claim that the noncommutative deformation algebra $\mathrm{A}_{\mathrm{def}}$ of $\scrO_{2\Curve}$ satisfies
\begin{equation}
\mathrm{A}_{\mathrm{def}}\cong\frac{\ring}{\lcl x^2+r_3, xy+yx+s_3\rcl},\label{eqn:l4aim}
\end{equation}
where $r_3,s_3$ have order $\geq 3$.  Since the abelianisation $(\mathrm{A}_{\mathrm{def}})^{\mathrm{ab}}$ is obtained by formally commuting variables, setting $\m=(x,y)$ it then follows immediately that
\[
\GV{2} = \dim_\mathbb{C} (\mathrm{A}_{\mathrm{def}})^{\mathrm{ab}}\geq \dim_\mathbb{C}\left(\frac{(\mathrm{A}_{\mathrm{def}})^{\mathrm{ab}}}{\m^4}\right)=
 \dim_\mathbb{C}\left(\frac{\mathbb{C}[x,y]}{(x^2+s_3,2xy+r_3,\m^4)}\right)=5
\]
which proves \ref{lem:GV2forl4}. Thus it suffices to prove the claim, \eqref{eqn:l4aim}.

\medskip
Consider the glue in the second line of \eqref{eqn:appendixGlue}.  Exactly as in \S\ref{subsec:ProofOfGVl3} the ideal $(v_1,v_2^2)$ glues to the ideal $(w_1,w_2^2)$ to give a sheaf of ideals $\scrI_2$ on $X$, and $\scrO_{2\Curve}\colonequals \scrO/\scrI_2$ is the sheaf we seek.  In the glue we are considering, again $\ky=3$, %highest x is 3, in x^3
 so \eqref{eqn:LFres} is still a locally free resolution of $\scrO_{2\Curve}$ (albeit in the differentials, $\kx, A, B$ etc now gets replaced by those coming from the new glue).  Now we have $A=a^{\kx}+a^{\kx-4}v_2+v_2^2C$, where $C=a^{\kx-3}+v_2D$ for some $D$.

Again using small case letters for degree one elements, consider $\xsf_1, \ysf_1$ defined to be 
\[
\begin{tikzpicture}
\node () at (-3.5,0) {$\xsf _1=$};
\node (B1) at (-1,1) {$\scrE_3$};
\node (C1) at (3.25,1) {$\scrE_2$};
\node (D1) at (7,1) {$\scrE_1$};
\node (E1) at (8,1) {$\scrE_0$};
\draw[->] (B1)--
node[above]{$\begin{psmallmatrix}v_2^2 \\ -a^{\kx+1}\\ -1\end{psmallmatrix}$}
(C1);
\draw[->] (C1)--
node[above]{$\begin{psmallmatrix}
a^{\kx+1}&v_2^2&0\\
A&-v_1&b^{2-\kx}w_1\\
-1&0&-v_2^2\end{psmallmatrix}$}
(D1);
\draw[->] (D1)--
(E1);

\node (B2) at (-2,-1.5) {$\scrE_3$};
\node (C2) at (-1,-1.5) {$\scrE_2$};
\node (D2) at (3.25,-1.5) {$\scrE_1$};
\node (E2) at (7,-1.5) {$\scrE_0$};
\draw[->] (B2)--
(C2);
\draw[->] (C2)--
node[above]{$\begin{psmallmatrix}
a^{\kx+1}&v_2^2&0\\
A&-v_1&b^{2-\kx}w_1\\
-1&0&-v_2^2\end{psmallmatrix}$}
(D2);
\draw[->] (D2)--
node[above]{$\begin{psmallmatrix}v_1 & v_2^2& b^{2-\kx}w_1\end{psmallmatrix}$}
(E2);

\draw[->](B1)--node[left,pos=0.35]{$\begin{psmallmatrix}-av_2\\0\\0\end{psmallmatrix}$}(C2);
\draw[->](C1)--node[gap,pos=0.35]{$\begin{psmallmatrix}0&-av_2&0\\ a^{\kx-3}+av_2C&-v_2&0\\0&0&av_2\end{psmallmatrix}$}(D2);
\draw[->](D1)--node[right,pos=0.35]{$\begin{psmallmatrix}v_2&-av_2&0\end{psmallmatrix}$}(E2);
\end{tikzpicture}
\]
\[
\begin{tikzpicture}
\node () at (-3.5,0) {$\ysf _1=$};
\node (B1) at (-1,1) {$\scrE_3$};
\node (C1) at (3.25,1) {$\scrE_2$};
\node (D1) at (7,1) {$\scrE_1$};
\node (E1) at (8,1) {$\scrE_0$};
\draw[->] (B1)--
node[above]{$\begin{psmallmatrix}v_2^2 \\ -a^{\kx+1}\\ -1\end{psmallmatrix}$}
(C1);
\draw[->] (C1)--
node[above]{$\begin{psmallmatrix}
a^{\kx+1}&v_2^2&0\\
A&-v_1&b^{2-\kx}w_1\\
-1&0&-v_2^2\end{psmallmatrix}$}
(D1);
\draw[->] (D1)--
(E1);

\node (B2) at (-2,-1.5) {$\scrE_3$};
\node (C2) at (-1,-1.5) {$\scrE_2$};
\node (D2) at (3.25,-1.5) {$\scrE_1$};
\node (E2) at (7,-1.5) {$\scrE_0$};
\draw[->] (B2)--
(C2);
\draw[->] (C2)--
node[above]{$\begin{psmallmatrix}
a^{\kx+1}&v_2^2&0\\
A&-v_1&b^{2-\kx}w_1\\
-1&0&-v_2^2\end{psmallmatrix}$}
(D2);
\draw[->] (D2)--
node[above]{$\begin{psmallmatrix}v_1 & v_2^2& b^{2-\kx}w_1\end{psmallmatrix}$}
(E2);

\draw[->](B1)--node[left,pos=0.35]{$\begin{psmallmatrix}-a\\0\\0\end{psmallmatrix}$}(C2);
\draw[->](C1)--node[gap,pos=0.35]{$\begin{psmallmatrix}0&-a&0\\ aC&-1&-a^{\kx-3}v_2\\0&0&a\end{psmallmatrix}$}(D2);
\draw[->](D1)--node[right,pos=0.35]{$\begin{psmallmatrix}1&-a&-a^{\kx-3}v_2\end{psmallmatrix}$}(E2);
\end{tikzpicture}
\]
As in the previous section, there are elements $\xsf_2,\ysf_2$ on the other chart for which $\xsf\colonequals (\xsf_1,\xsf_2,0)\in\scrC_1$ and  $\ysf\colonequals (\ysf_1,\ysf_2,0)\in\scrC_1$ satisfy $\dd(\xsf)=0=\dd(\ysf)$. For very similar reasons as those in \cite[4.1]{BW2} and in the last subsection, $\xsf$ and $\ysf$ descend to give linearly independent elements of $\mathrm{H}^1(\scrC)$. Since $\mathrm{H}^1(\scrC)\cong\mathrm{Ext}^1(\scrO_{2\Curve},\scrO_{2\Curve})$, and the neighbourhood is $\ell=4$ flopping, it follows from \cite[7.4]{DW5} that $\mathrm{Ext}^1(\scrO_{2\Curve},\scrO_{2\Curve})\cong\mathbb{C}^2$, and thus $\xsf$ and $\ysf$ in fact give a basis.

Consider the degree two elements $\Xsf_1$ and $\Ysf_1$, defined as follows by specifying the two maps $F_1$ and $F_2$. 
\[
\begin{array}{cc}
\begin{array}{c}
\begin{tikzpicture}[yscale=0.75,xscale=1.25]
\node (C1) at (3,1) {$\scrE_3$};
\node (D1) at (4,1) {$\scrE_2$};
\node (E1) at (5,1) {$\scrE_1$};
\node (F1) at (6,1) {$\scrE_0$};
\draw[->] (E1)--(F1);
\draw[->] (C1)--(D1);
\draw[->] (D1)--(E1);
\node (B2) at (1,-1) {$\scrE_3$};
\node (C2) at (2,-1) {$\scrE_2$};
\node (D2) at (3,-1) {$\scrE_1$};
\node (E2) at (4,-1) {$\scrE_0$};
\draw[->] (B2)--(C2);
\draw[->] (C2)--(D2);
\draw[->] (D2)--(E2);
\draw[->](C1)--node[left]{$\scriptstyle F_1$}(D2);
\draw[->](D1)--node[right]{$\scriptstyle F_2$}(E2);
\end{tikzpicture}
\end{array}
&
%%%%% SECOND
\begin{array}{lcc}
&{\scriptstyle F_1}&{\scriptstyle F_2}\\[2mm]
\Xsf_1\colonequals&\begin{psmallmatrix}0\\-a^2C\\0\end{psmallmatrix} 
&\begin{psmallmatrix}-a^2C&0&0\end{psmallmatrix}\\[3mm]
\Ysf_1\colonequals&\begin{psmallmatrix}0\\-a^2v_2C\\0\end{psmallmatrix} 
&\begin{psmallmatrix}-a^2v_2C&0&0\end{psmallmatrix}
\end{array}
\end{array}
\]
As in the previous section, there are elements $\Xsf_2,\Ysf_2$ on the other chart for which $\Xsf\colonequals (\Xsf_1,\Xsf_2,0)\in\scrC_2$ and  $\Ysf\colonequals (\Ysf_1,\Ysf_2,0)\in\scrC_2$ satisfy $\dd(\Xsf)=0=\dd(\Ysf)$. It is easy to verify, in a similar way to \cite[4.5]{BW2} that $\Xsf$ and $\Ysf$ descend to give linearly independent elements in $\mathrm{H}^2(\scrC)$.  By CY duality $\mathrm{H}^2(\scrC)\cong\mathrm{Ext}^2(\scrO_{2\Curve},\scrO_{2\Curve})\cong\mathbb{C}^2$, and thus $\Xsf,\Ysf$ form a basis.

The lemma below will require homotopies, and for this consider the three degree one elements $\hsf, \jsf, \ksf$, defined as follows by specifying the three maps $f_1, f_2$ and $f_3$. 
\[
\begin{array}{cc}
\begin{array}{c}
\begin{tikzpicture}[yscale=0.75,xscale=1.25]
\node (B1) at (0,1) {$\scrE_3$};
\node (C1) at (1,1) {$\scrE_2$};
\node (D1) at (2,1) {$\scrE_1$};
\node (E1) at (3,1) {$\scrE_0$};
\draw[->] (B1)--(C1);
\draw[->] (C1)--(D1);
\draw[->] (D1)--(E1);
\node (B2) at (-1,-1) {$\scrE_3$};
\node (C2) at (0,-1) {$\scrE_2$};
\node (D2) at (1,-1) {$\scrE_1$};
\node (E2) at (2,-1) {$\scrE_0$};
\draw[->] (B2)--(C2);
\draw[->] (C2)--(D2);
\draw[->] (D2)--(E2);
\draw[->](B1)--node[left]{$\scriptstyle f_1$}(C2);
\draw[->](C1)--node[gap]{$\scriptstyle f_2$}(D2);
\draw[->](D1)--node[right]{$\scriptstyle f_3$}(E2);
\end{tikzpicture}
\end{array}
&
\begin{array}{lccc}
&{\scriptstyle f_1}&{\scriptstyle f_2}&{\scriptstyle f_3}\\[2mm]
\hsf_1\colonequals &\begin{psmallmatrix}-a^2\\0\\0\end{psmallmatrix}
&\begin{psmallmatrix}0&-a^2&0\\ 0&-a&0\\0&0&a^2\end{psmallmatrix}
&\begin{psmallmatrix}1&-a^2&0\end{psmallmatrix} \\[3mm]
\jsf_1\colonequals &\begin{psmallmatrix}0\\0\\0\end{psmallmatrix}
&\begin{psmallmatrix}0&0&0\\ 0&0&0\\0&0&0\end{psmallmatrix}
&\begin{psmallmatrix}0&0&a^{\kx-2}\end{psmallmatrix} \\[3mm]
\ksf_1\colonequals &\begin{psmallmatrix}0\\0\\0\end{psmallmatrix}
&\begin{psmallmatrix}0&0&0\\ 0&0&a^{\kx-2}\\0&0&0\end{psmallmatrix}
&\begin{psmallmatrix}1&0&0\end{psmallmatrix} \\[3mm]

\end{array}
\end{array}
\]
Again there are maps $\hsf_2,\jsf_2,\ksf_2$ on the other chart which glue to give the above; from this we obtain the three elements $\hsf=(\hsf_1,\hsf_2,0)$, $\jsf=(\jsf_1,\jsf_2,0)$, $\ksf=(\ksf_1,\ksf_2,0)\in\scrC_1$. 

\begin{lemma}\label{lem:AppenComp}
With notation as above, the following statements hold.
\begin{enumerate}
\item $\xsf\star \xsf= 0 +\dd(\hsf)$.
\item $\ysf\star \xsf= \Ysf +\dd(\jsf)$.
\item $\xsf\star \ysf= \Ysf +\dd(\ksf)$.
\item $\ysf\star \ysf= \Xsf +\dd(0)$.
\end{enumerate}
\end{lemma}
\begin{proof}
Each is a simple and direct verification, using only the fact that $\star$ is defined using composition of functions. For example, $\xsf\star \ysf$ on chart one is $\ysf_1$ followed by $\xsf_1$, which is
\[
\begin{tikzpicture}[yscale=0.75]
\node (C1) at (3,1) {$\scrE_3$};
\node (D1) at (4.5,1) {$\scrE_2$};
\node (E1) at (5.5,1) {$\scrE_1$};
\node (F1) at (6.5,1) {$\scrE_0$};
\draw[->] (E1)--(F1);
\draw[->] (C1)--(D1);
\draw[->] (D1)--(E1);
\node (B2) at (1,-1) {$\scrE_3$};
\node (C2) at (2,-1) {$\scrE_2$};
\node (D2) at (3,-1) {$\scrE_1$};
\node (E2) at (4.5,-1) {$\scrE_0$};
\draw[->] (B2)--(C2);
\draw[->] (C2)--(D2);
\draw[->] (D2)--(E2);
\draw[->](C1)--node[left]{$\scriptstyle \begin{psmallmatrix}0\\-a^{\kx-2}-a^2v_2C\\0\end{psmallmatrix} 
$}(D2);
\draw[->](D1)--node[right]{$\scriptstyle \begin{psmallmatrix}-a^2v_2C&0&a^{\kx-2}v_2^2\end{psmallmatrix}$}(E2);
\end{tikzpicture}
\]
This evidently equals $\Ysf_1+\updelta(\ksf_1)$.  All other statements are similar.
\end{proof}
The above is the only part of the $A_\infty$-structure that is necessary to verify the claim \eqref{eqn:l4aim}, and thus prove \ref{lem:GV2forl4}.  Indeed, abusing notation slightly and now writing $\xsf,\ysf$ etc for their images in cohomology, it follows from \ref{lem:AppenComp} that 
\[
\msf_2(\xsf,\xsf)=0, \quad\msf_2(\xsf,\ysf)=\Ysf, \quad\msf_2(\ysf,\xsf)=\Ysf, \quad\msf_2(\ysf,\ysf)=\Xsf
\]
where $\msf_2$ is the first non-trivial operation in the $A_\infty$-minimal model of $\Ext^*(\scrO_{2\Curve},\scrO_{2\Curve})$ (see e.g.\ \cite[\S5.1]{BW2}).  As the $A_\infty$-operations dualise to give a presentation of $\mathrm{A}_{\mathrm{def}}$ (see e.g.\ \cite[below 4.3]{Kawamata_survey}), the claim \eqref{eqn:l4aim} follows immediately.

%-----------------------------------------------------------------------------------------

\end{document}